\documentclass[oneside,final, letterpaper]{ucr}

\usepackage{amsmath, amssymb, amsfonts, mathrsfs, adjustbox} 
\usepackage[toc, page]{appendix}

\usepackage{xcolor}
\definecolor{darkgreen}{rgb}{0,0.3,0}
\definecolor{darkblue}{rgb}{0,0,0.6}
\usepackage[colorlinks, citecolor=darkgreen, urlcolor=blue, final, hyperindex, pagebackref, linkcolor = darkblue]{hyperref}

\usepackage{amsthm}

\usepackage[capitalize]{cleveref}
\crefname{equation}{Equation}{Equations}

\usepackage{multicol}

\theoremstyle{plain}
\newtheorem{thm}{Theorem}
\newtheorem{prop}[thm]{Proposition}
\newtheorem{cor}[thm]{Corollary}
\newtheorem{lem}[thm]{Lemma}

\theoremstyle{definition}
\newtheorem{defn}[thm]{Definition}
\newtheorem{expl}[thm]{Example}

\theoremstyle{remark}

\newtheorem*{question}{Question}

\numberwithin{thm}{chapter}

\newcommand{\define}[1]{{\bf \boldmath{#1}}\index{#1}}
\newcommand{\maps}{\colon}
\newcommand{\To}{\Rightarrow}
\newcommand{\inv}{^{-1}}
\newcommand{\inta}{\textstyle\int\hspace{-.03in}}

\newcommand{\cod}{\mathrm{cod}}
\newcommand{\colim}{\mathrm{colim}}

\newcommand{\lax}{\mathrm{lax}}
\newcommand{\ob}{\mathrm{ob}}
\newcommand{\op}{^\mathrm{op}}
\newcommand{\rev}{^\mathrm{rev}}
\newcommand{\opl}{\mathrm{opl}}
\newcommand{\pse}{\mathrm{ps}}
\newcommand{\spl}{\mathrm{s}}

\newcommand{\cardinal}[1]{\mathbf{#1}}
\newcommand{\1}{\cardinal 1}
\newcommand{\2}{\cardinal 2}
\newcommand{\n}{\cardinal n}
\newcommand{\m}{\cardinal m}

\newcommand{\category}[1]{\mathcal{#1}}
\newcommand{\A}{\category A}
\newcommand{\B}{\category B}
\newcommand{\C}{\category C}
\newcommand{\D}{\category D}
\newcommand{\J}{\category J}
\newcommand{\K}{\category K}
\renewcommand{\L}{\category L}
\newcommand{\V}{\category V}
\newcommand{\X}{\category X}
\newcommand{\Y}{\category Y}

\newcommand{\pseudofunctor}[1]{\mathcal{#1}}
\newcommand{\F}{\pseudofunctor F}
\newcommand{\G}{\pseudofunctor G}
\newcommand{\M}{\pseudofunctor M}
\newcommand{\psN}{\pseudofunctor N}
\newcommand{\N}{\mathbb N}
\renewcommand{\O}{\mathcal O}

\newcommand{\namedcat}[1]{\mathsf{#1}}

\newcommand{\Cat}{\namedcat{Cat}}
\newcommand{\CAT}{\namedcat{CAT}}
\newcommand{\Cart}{\namedcat{Cart}}
\newcommand{\Cocart}{\namedcat{Cocart}}

\newcommand{\Fam}{\namedcat{Fam}}
\newcommand{\Fib}{\namedcat{Fib}}
\newcommand{\cMonFib}{\namedcat{cMonFib}}
\newcommand{\cocMonOpFib}{\namedcat{cocMonOpFib}}

\newcommand{\Grph}{\namedcat{Grph}}
\newcommand{\ICat}{\namedcat{ICat}}
\newcommand{\OpFib}{\namedcat{OpFib}}
\newcommand{\Mon}{\namedcat{Mon}}
\newcommand{\CMon}{\namedcat{CMon}}
\newcommand{\Comon}{\namedcat{Comon}}
\newcommand{\CoComon}{\namedcat{CoComon}}
\newcommand{\GMon}{\namedcat{GMon}}
\newcommand{\IMon}{\namedcat{IMon}}
\newcommand{\MonTCat}{\namedcat{Mon2Cat}}
\newcommand{\OpICat}{\namedcat{OpICat}}
\newcommand{\PsMon}{\namedcat{PsMon}}

\newcommand{\Set}{\namedcat{Set}}
\newcommand{\TCat}{\namedcat{2Cat}}
\newcommand{\Br}{\namedcat{Br}}
\newcommand{\Sym}{\namedcat{Sym}}
\newcommand{\BrPsMon}{\Br\PsMon}
\newcommand{\SymPsMon}{\Sym\PsMon}

\newcommand{\BrMonTCat}{\Br\MonTCat}
\newcommand{\SymMonTCat}{\Sym\MonTCat}

\newcommand{\MonFib}{\namedcat{MonFib}}
\newcommand{\MonOpFib}{\namedcat{MonOpFib}}
\newcommand{\BrMonFib}{\namedcat{BrMonFib}}
\newcommand{\SymMonFib}{\namedcat{SymMonFib}}

\newcommand{\BrMonOpFib}{\namedcat{BrMonOpFib}}
\newcommand{\SymMonOpFib}{\namedcat{SymMonOpFib}}
\newcommand{\MonICat}{\namedcat{MonICat}}

\newcommand{\MonOpICat}{\namedcat{MonOpICat}}

\newcommand{\NetMod}{\namedcat{NetMod}}

\newcommand{\Maf}{\namedcat{Maf}}
\newcommand{\Mod}{\namedcat{Mod}}
\newcommand{\Comod}{\namedcat{Comod}}

\newcommand{\inp}[1]{{#1}^\mathrm{in}}
\newcommand{\out}[1]{{#1}^\mathrm{out}}

\newcommand{\R}{\mathbb R}

\newcommand{\Int}{\textstyle{\int}}

\newcommand{\Boole}{\mathbb{B}}

\newcommand{\Op}{\namedcat{Op}}
\newcommand{\Opd}{\namedcat{Opd}}

\newcommand{\E}{\category E}

\newcommand{\Inj}{\namedcat{Inj}}
\newcommand{\FinInj}{\namedcat{FinInj}}
\newcommand{\FinBij}{\namedcat{FinBij}}
\newcommand{\SC}{{\S(C)}}

\newcommand{\Ob}{\mathrm{Ob}}
\newcommand{\gn}{\mathfrak g}

\newcommand{\sGrph}{\namedcat{SimpleGph}}
\newcommand{\Quiv}{\namedcat{Quiv}}
\newcommand{\Aut}{\namedcat{Aut}}
\newcommand{\Alg}{\namedcat{Alg}}
\newcommand{\KG}{\namedcat{KG}}

\newcommand{\bink}[1]{\binom{#1}{k}}
\newcommand{\GMV}{\Gamma_{M,\V}}
\newcommand{\Mor}{\mathsf{Mor}}
\newcommand{\PreMonCat}{\mathsf{PreMonCat}}
\newcommand{\Petri}{\namedcat{Petri}}

\renewcommand{\S}{\mathsf{S}}

\newcommand{\assoc}{\alpha}
\newcommand{\braid}{\beta}
\newcommand{\leftor}{\lambda}
\newcommand{\rightor}{\rho}

\newcommand{\mlt}{\mu}
\newcommand{\uni}{\epsilon}

\newcommand{\Ring}{\namedcat{Ring}}

\newcommand{\SG}{\mathrm{SG}}
\newcommand{\DG}{\mathrm{DG}}
\newcommand{\MG}{\mathrm{MG}}
\newcommand{\MGplus}{{\mathrm{MG}^{+}}}
\newcommand{\DMG}{\mathrm{DMG}}
\newcommand{\HG}{\mathrm{HG}}

\newcommand{\Day}{\mathrm{Day}}

\mathchardef\mhyphen="2D

\usepackage{tikz, tikz-cd}
\usetikzlibrary{positioning}
\definecolor {processblue}{cmyk}{0.9, 0.5, 0, 0}

\tikzstyle{simple} = [-, line width = 2.000]
\tikzstyle{arrow} = [-, postaction = {decorate}, decoration = {markings, mark = at position .5 with {\arrow{>}}}, line width = 1.100]
\pgfdeclarelayer{edgelayer}
\pgfdeclarelayer{nodelayer}
\pgfdeclarelayer{bg}
\pgfsetlayers{bg, edgelayer, nodelayer, main}

\tikzstyle{none} = [inner sep = -1pt]
\tikzstyle{species} = [circle, fill = none, draw = black, scale = 1.0]
\tikzstyle{transition} = [rectangle, fill = none, draw = black, scale = 1.15]
\tikzstyle{empty} = [circle, fill = none, draw = none]
\tikzstyle{inputdot} = [circle, fill = black, draw = black, scale = .5]
\tikzstyle{dot} = [circle, fill = black, draw = black]
\tikzstyle{triplebounding} = [rectangle, dashed, fill = none, draw = black, scale = 30.00]
\tikzstyle{bounding} = [rectangle, dashed, fill = none, draw = black, scale = 11.00]
\tikzstyle{simple} = [-, draw = black, line width = 1.000]
\tikzstyle{inarrow} = [-, draw = black, postaction = {decorate}, decoration = {markings, mark = at position .5 with {\arrow{>}}}, line width = 1.000]
\tikzstyle{tick} = [-, draw = black, postaction = {decorate}, decoration = {markings, mark = at position .5 with {\draw (0, -0.1) -- (0, 0.1);}}, line width = 1.000]
\tikzstyle{inputarrow} = [->, draw = black, shorten > = .05cm]

\tikzset{main node/.style = {circle, fill = blue!20, draw, minimum size = 1cm, inner sep = 0pt}, }

\tikzstyle{construct} = [fill = white, draw = black, shape = rectangle]
\tikzstyle{universal} = [fill = black, draw = black, shape = circle]

\tikzstyle{reddot} = [fill = red, draw = red, shape = circle, inner sep = 0pt, minimum size = 4pt]
\tikzstyle{bluedot} = [fill = blue, draw = blue, shape = circle, inner sep = 0pt, minimum size = 4pt]
\tikzstyle{blackdot} = [fill = black, draw = black, shape = circle, inner sep = 0pt, minimum size = 4pt]

\usetikzlibrary{backgrounds, circuits, circuits.ee.IEC, shapes, fit, matrix}

\tikzstyle{tri} = [regular polygon, regular polygon sides = 3, shape border rotate = 1
80, fill = none, draw = black]

\tikzstyle{simple} = [-, line width = 2.000]
\tikzstyle{arrow} = [-, postaction = {decorate}, decoration = {markings, mark = at position .5 with {\arrow{>}}}, line width = 1.100]
\pgfdeclarelayer{edgelayer}
\pgfdeclarelayer{nodelayer}
\pgfsetlayers{edgelayer, nodelayer, main}

\tikzstyle{none} = [inner sep = -1pt]

\definecolor{lblue}{rgb}{0, 250, 255}
\definecolor{llblue}{HTML}{a1ddde}
\definecolor{red}{rgb}{0.8, 0, 0}

\tikzstyle{species} = [circle, fill = yellow, draw = black, scale = 1]
\tikzstyle{catalyst} = [circle, fill = yellow, draw = red, scale = 1]
\tikzstyle{transition} = [rectangle, fill = lblue, draw = black, scale = 1]
\tikzstyle{morphism} = [rectangle, fill = llblue, draw = black, scale = 1]

\tikzstyle{empty} = [circle, fill = none, draw = none]
\tikzstyle{inputdot} = [circle, fill = black, draw = black, scale = .5]
\tikzstyle{dot} = [circle, fill = black, draw = black]

\tikzstyle{simple} = [-, draw = black, line width = 1.000]
\tikzstyle{inarrow} = [-, draw = black, postaction = {decorate}, decoration = {markings, mark = at position .5 with {\arrow{>}}}, line width = 1.000]
\tikzstyle{tick} = [-, draw = black, postaction = {decorate}, decoration = {markings, mark = at position .5 with {\draw (0, -0.1) -- (0, 0.1);}}, line width = 1.000]
\tikzstyle{inputarrow} = [->, draw = black, shorten > = .05cm]

\usepackage{tikz-cd}
\usetikzlibrary{arrows, positioning, fit, matrix, shapes.geometric, external}

\newcommand*\pgfdeclareanchoralias[3]{
  \expandafter\def\csname pgf@anchor@#1@#3\expandafter\endcsname
     \expandafter{\csname pgf@anchor@#1@#2\endcsname}}

\tikzset{
    circnode/.style = {
      circle, draw = red, very thin, outer sep = 0.025em, minimum size = 2em, 
      fill = red, text centered}, 
    integral/.style = {
      circle, draw = black, very thick, outer sep = 0.025em, 
      minimum size = 2em, fill = blue!5, text centered}, 
    multiply/.style = {
      circle, draw = black, very thick, outer sep = 0.025em, 
      minimum size = 2em, fill = blue!5, text centered}, 
    zero/.style = {
      circle, draw = black, very thick, minimum size = 0.15cm, fill = black, 
      inner sep = 0, outer sep = 0}, 
    bang/.style = {
      circle, draw = black, very thick, minimum size = 0.15cm, fill = green!10, 
      inner sep = 0, outer sep = 0}, 
    delta/.style = {
      regular polygon, regular polygon sides = 3, minimum size = 0.4cm, inner
      sep = 0, outer sep = 0.025em, draw = black, very thick, fill = green!10}, 
    codelta/.style = {
      regular polygon, regular polygon sides = 3, shape border rotate = 180, minimum size = 0.4cm, 
      inner sep = 0, outer sep = 0.025em, draw = black, very thick, fill = green!10}, 
    plus/.style = {
      regular polygon, regular polygon sides = 3, shape border rotate = 180, minimum size = 0.4cm, 
      inner sep = 0, outer sep = 0.025em, draw = black, very thick, fill = black}, 
    coplus/.style = {
      regular polygon, regular polygon sides = 3, minimum size = 0.4cm, 
      inner sep = 0, outer sep = 0.025em, draw = black, very thick, fill = black}, 
    sqnode/.style = {
      regular polygon, regular polygon sides = 4, minimum size = 2.6em, 
      draw = black, very thick, inner sep = 0.2em, outer sep = 0.025em, 
      fill = yellow!10, text centered}, 
    bigcirc/.style = {
      circle, draw = black, very thick, text width = 1.6em, outer sep = 0.025em, 
      minimum height = 1.6em, fill = blue!5, text centered}
}

\tikzstyle{tri} = [regular polygon, regular polygon sides = 3, shape border rotate = 1
80, fill = none, draw = black]

\pgfdeclareanchoralias{regular polygon}{corner 2}{left copy}
\pgfdeclareanchoralias{regular polygon}{corner 3}{right copy}
\pgfdeclareanchoralias{regular polygon}{corner 3}{addend}
\pgfdeclareanchoralias{regular polygon}{corner 2}{summand}

\tikzset{
   oriented WD/.style = {
      every to/.style = {out = 0, in = 180, draw}, 
      label/.style = {
         font = \everymath\expandafter{\the\everymath\scriptstyle}, 
         inner sep = 0pt, 
         node distance = 2pt and -2pt}, 
      semithick, 
      node distance = 1 and 1, 
      decoration = {markings, mark = at position .5 with {\arrow{stealth};}}, 
      ar/.style = {postaction = {decorate}}, 
      execute at begin picture = {\tikzset{
         x = \bbx, y = \bby, 
         every fit/.style = {inner xsep = \bbx, inner ysep = \bby}}}
      }, 
   bbx/.store in = \bbx, 
   bbx = 1.5cm, 
   bby/.store in = \bby, 
   bby = 1.75ex, 
   bb port sep/.store in = \bbportsep, 
   bb port sep = 2, 
   bb port length/.store in = \bbportlen, 
   bb port length = 4pt, 
   bb min width/.store in = \bbminwidth, 
   bb min width = 1cm, 
   bb rounded corners/.store in = \bbcorners, 
   bb rounded corners = 2pt, 
   bb small/.style = {bb port sep = 1, bb port length = 2.5pt, bbx = .4cm, bb min width = .4cm, bby = .7ex}, 
   bb Small/.style = {bb port sep = 1, bb port length = 2.5pt, bbx = .5cm, bb min width = .5cm, bby = 1ex}, 
   bb/.code 2 args = {
      \pgfmathsetlengthmacro{\bbheight}{\bbportsep * (max(#1, #2)+1) * \bby}
      \pgfkeysalso{draw, minimum height = \bbheight, minimum width = \bbminwidth, outer sep = 0pt, 
         rounded corners = \bbcorners, thick, 
         prefix after command = {\pgfextra{\let\fixname\tikzlastnode}}, 
         append after command = {\pgfextra{\draw
            \ifnum #1 = 0{} \else foreach \i in {1, ..., #1} {
               ($(\fixname.north west)!{\i/(#1+1)}!(\fixname.south west)$) +(-\bbportlen, 0) coordinate
               (\fixname_in\i) -- +(\bbportlen, 0) coordinate (\fixname_in\i')}\fi 
            \ifnum #2 = 0{} \else foreach \i in {1, ..., #2} {
               ($(\fixname.north east)!{\i/(#2+1)}!(\fixname.south east)$) +(-\bbportlen, 0) coordinate
               (\fixname_out\i') -- +(\bbportlen, 0) coordinate (\fixname_out\i)}\fi;
         }}}
   }, 
   bb name/.style = {append after command = {\pgfextra{\node[anchor = north] at (\fixname.north) {#1};}}}
}

\tikzstyle{downtri}=[regular polygon, regular polygon sides=3, draw, thick, fill=blue!20, shape border rotate = 180]

\begin{document}

\title{The Grothendieck Construction in Categorical Network Theory}
\author{Joseph Patrick Moeller
}
\degreemonth{December}
\degreeyear{2020}
\degree{Doctor of Philosophy}
\chair{Dr. John C.\ Baez}
\othermembers{
Dr. Wee Liang Gan\\
Dr. Carl Mautner}
\numberofmembers{3}
\field{Mathematics}
\campus{Riverside}

\maketitle
\copyrightpage{}
\approvalpage{}

\degreesemester{Fall}

\begin{frontmatter}

{\ssp\begin{acknowledgements}
First of all, I owe all of my achievements to my wife, Paola. I couldn't have gotten here without my parents: Daniel, Andrea, Tonie, Maria, and Luis, or my siblings: Danielle, Anthony, Samantha, David, and Luis.

I would like to thank my advisor, John Baez, for his support, dedication, and his unique and brilliant style of advising. I could not have become the researcher I am under another's instruction. I would also like to thank Christina Vasilakopoulou, whose kindness, energy, and expertise cultivated a deeper appreciation of category theory in me. My experience was also greatly enriched by my academic siblings: Daniel Cicala, Kenny Courser, Brandon Coya, Jason Erbele, Jade Master, Franciscus Rebro, and Christian Williams, and by my cohort: Justin Davis, Ethan Kowalenko, Derek Lowenberg, Michel Manrique, and Michael Pierce.

I would like to thank the UCR math department. Professors from whom I learned a ton of algebra, topology, and category theory include Julie Bergner, Vyjayanthi Chari, Wee-Liang Gan, Jos\'e Gonzalez, Jacob Greenstein, Carl Mautner, Reinhard Schultz, and Steffano Vidussi. Special thanks goes to the department chair Yat-Sun Poon, as well as Margarita Roman, Randy Morgan, and James Marberry, and many others who keep the whole thing together.

The material in \cref{ch:NetworkModels} consists of work from both \emph{Network models} joint with John Baez, John Foley, and Blake Pollard \cite{NetworkModels}. \cref{ch:NNM} consists of work done in my paper \emph{Noncommutative network models} \cite{NoncommNetMods}. \cref{ch:PetriNets} arose from \emph{Network models from Petri nets with catalysts} joint with Baez and Foley \cite{Catalysts}. \cref{ch:MonGroth} consists of joint work with Christina Vasilakopoulou appearing in our paper \emph{Monoidal Grothendieck construction} \cite{MonGroth}. Part of this work was performed with funding from a subcontract with Metron Scientific Solutions working on DARPA’s Complex Adaptive System Composition and Design Environment (CASCADE) project.
\end{acknowledgements} }

\begin{dedication}
    \null\vfil
    {\large
    \begin{center}
    To Teresa Danielle Moeller.
    \end{center}}
    \vfil\null
\end{dedication}

{\ssp\begin{abstract}
    In this thesis, we present a flexible framework for specifying and constructing operads which are suited to reasoning about network construction. The data used to present these operads is called a \emph{network model}, a monoidal variant of Joyal's combinatorial species. The construction of the operad required that we develop a monoidal lift of the Grothendieck construction. We then demonstrate how concepts like priority and dependency can be represented in this framework. For the former, we generalize Green's graph products of groups to the context of universal algebra. For the latter, we examine the emergence of monoidal fibrations from the presence of catalysts in Petri nets.
\end{abstract}}

\setcounter{tocdepth}{1}
\tableofcontents
\end{frontmatter}

{\ssp
\chapter{Introduction}

\section*{Search and Rescue}
Imagine that you have a network of boats, planes, and drones tasked with rescuing sailors who have fallen overboard in a hurricane. You want to be able to task these agents to search certain areas for survivors in an intelligent way. You do not want to waste time and resources by double searching some areas while other areas get neglected. Also, if one of the searchers gets taken out by the storm, you must update the tasking so that other agents can cover the areas which the downed agent has yet to search, as well as recording that there is a new known person in need of rescue.

In 2015, DARPA launched a program called Complex Adaptive System Composition and Design Environment, or CASCADE. The goal of this program was to write software that would be able to handle this sort of tasking of agents in a network in a flexible and responsive way. The bulk of this thesis was developed while I was working on this project with Metron Scientific Solutions Inc., developing a mathematically principled foundation around which this software could be designed. John Baez, John Foley, Blake Pollard, and I developed the theory of \emph{network models} to address this challenge \cite{NetworkModels}. 

\section*{Network Operads}
Large complex networks can be viewed as being built up from small simple pieces. This sort of many-to-one composition is perfectly suited to being modeled using \emph{operads}. While a category can be described as a system of composition for a collection of arrows which have a specified input type and a specified output type, an operad is a system of composition for a collection of trees which have a specified \emph{family} of input types and a single specified output type. 
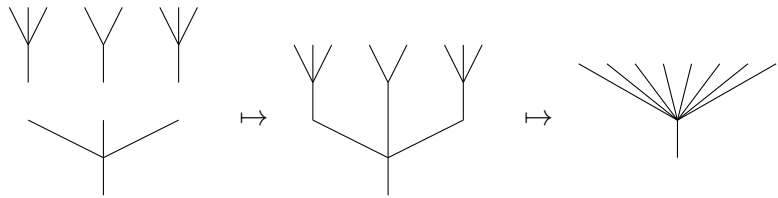
\[
\begin{tikzpicture}
\begin{pgfonlayer}{nodelayer}
	\node [style=none] (1) at (0.25, 2.5) {};
	\node [style=none] (2) at (0.5, 2.5) {};
	\node [style=none] (3) at (0.75, 2.5) {};
	\node [style=none] (4) at (1.25, 2.5) {};
	\node [style=none] (5) at (1.75, 2.5) {};
	\node [style=none] (6) at (2.25, 2.5) {};
	\node [style=none] (7) at (2.5, 2.5) {};
	\node [style=none] (8) at (2.75, 2.5) {};
	\node [style=none] (9) at (0.5, 2) {};
	\node [style=none] (10) at (1.5, 2) {};
	\node [style=none] (11) at (2.5, 2) {};
	\node [style=none] (12) at (0.5, 1.5) {};
	\node [style=none] (13) at (1.5, 1.5) {};
	\node [style=none] (14) at (2.5, 1.5) {};
	\node [style=none] (15) at (0.5, 1) {};
	\node [style=none] (16) at (1.5, 1) {};
	\node [style=none] (17) at (2.5, 1) {};
	\node [style=none] (18) at (1.5, 0.5) {};
    \node [style=none] (19) at (1.5, 0) {};
	\node [style=none] () at (3.5, 1) {$\mapsto$};
	\node [style=none] () at (4, 0.5) {}; 
\end{pgfonlayer}
\begin{pgfonlayer}{edgelayer}
	\draw (18.center) to (19.center);
	\draw (15.center) to (18.center);
	\draw (16.center) to (18.center);
	\draw (17.center) to (18.center);
	\draw (9.center) to (12.center);
	\draw (10.center) to (13.center);
	\draw (11.center) to (14.center);
	\draw (1.center) to (9.center);
	\draw (2.center) to (9.center);
	\draw (3.center) to (9.center);
	\draw (4.center) to (10.center);
	\draw (5.center) to (10.center);
	\draw (6.center) to (11.center);
	\draw (7.center) to (11.center);
	\draw (8.center) to (11.center);
\end{pgfonlayer}
\end{tikzpicture}
\begin{tikzpicture}
\begin{pgfonlayer}{nodelayer}
	\node [style=none] (1) at (0.25, 2) {};
	\node [style=none] (2) at (0.5, 2) {};
	\node [style=none] (3) at (0.75, 2) {};
	\node [style=none] (4) at (1.25, 2) {};
	\node [style=none] (5) at (1.75, 2) {};
	\node [style=none] (6) at (2.25, 2) {};
	\node [style=none] (7) at (2.5, 2) {};
	\node [style=none] (8) at (2.75, 2) {};
	\node [style=none] (9) at (0.5, 1.5) {};
	\node [style=none] (10) at (1.5, 1.5) {};
	\node [style=none] (11) at (2.5, 1.5) {};
	\node [style=none] (15) at (0.5, 1) {};
	\node [style=none] (16) at (1.5, 1) {};
	\node [style=none] (17) at (2.5, 1) {};
	\node [style=none] (18) at (1.5, 0.5) {};
    \node [style=none] (19) at (1.5, 0) {};
	\node [style=none] () at (3.5, 1) {$\mapsto$};
	\node [style=none] () at (4, 0.5) {}; 
\end{pgfonlayer}
\begin{pgfonlayer}{edgelayer}
	\draw (18.center) to (19.center);
	\draw (15.center) to (18.center);
	\draw (16.center) to (18.center);
	\draw (17.center) to (18.center);
	\draw (9.center) to (15.center);
	\draw (10.center) to (16.center);
	\draw (11.center) to (17.center);
	\draw (1.center) to (9.center);
	\draw (2.center) to (9.center);
	\draw (3.center) to (9.center);
	\draw (4.center) to (10.center);
	\draw (5.center) to (10.center);
	\draw (6.center) to (11.center);
	\draw (7.center) to (11.center);
	\draw (8.center) to (11.center);
\end{pgfonlayer}
\end{tikzpicture}
\begin{tikzpicture}
\begin{pgfonlayer}{nodelayer}
	\node [style=none] (1) at (0, 1.25) {};
	\node [style=none] (2) at (3/8, 1.25) {};
	\node [style=none] (3) at (6/8, 1.25) {};
	\node [style=none] (4) at (9/8, 1.25) {};
	\node [style=none] (5) at (12/8, 1.25) {};
	\node [style=none] (6) at (15/8, 1.25) {};
	\node [style=none] (7) at (18/8, 1.25) {};
	\node [style=none] (8) at (21/8, 1.25) {};
	\node [style=none] (18) at (21/16, 0.5) {};
    \node [style=none] (19) at (21/16, 0) {};
    \node [style=none] () at (1, -0.5) {};
\end{pgfonlayer}
\begin{pgfonlayer}{edgelayer}
	\draw (18.center) to (19.center);
	\draw (1.center) to (18.center);
	\draw (2.center) to (18.center);
	\draw (3.center) to (18.center);
	\draw (4.center) to (18.center);
	\draw (5.center) to (18.center);
	\draw (6.center) to (18.center);
	\draw (7.center) to (18.center);
	\draw (8.center) to (18.center);
\end{pgfonlayer}
\end{tikzpicture}
\]
We use the word ``operations'' instead of ``trees'', hence the name \emph{operad}. Like categories, operads were originally developed in algebraic topology \cite{Mayoperad, BV}. Also like categories, operads have since found applications elsewhere, including physics and computer science \cite{MarklShniderStasheff, SetOperads}. We include a review of the basics of operads needed for this thesis in \cref{app:operads}.

In a network operad, the operations describe ways of sticking together a collection of networks to form a new larger network. To get a network operad, we treat a network as one of these operations and define the composition as overlaying a bunch of small networks on top of a large base network. For example, in the following picture, we are considering simple graphs as a sort of network. On the left, we are starting with a base network consisting of nine nodes and four edges, and we are attempting to attach more edges by overlaying three smaller graphs. The result of the operadic composition is on the right.
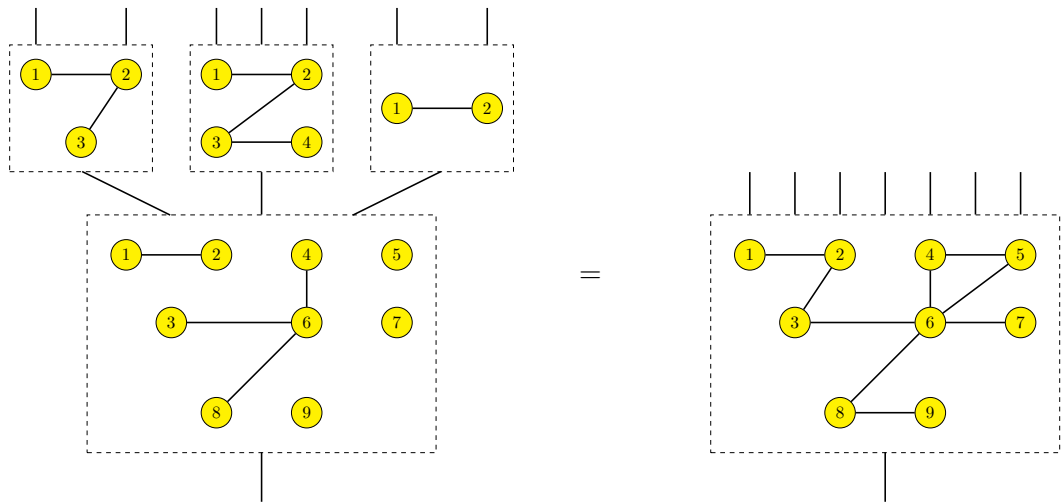
\[
\scalebox{0.6}{
\begin{tikzpicture}
\begin{pgfonlayer}{nodelayer}
	\node [style=species] (1) at (1, 2.5) {$1$};
	\node [style=species] (2) at (3, 2.5) {$2$};
	\node [style=species] (3) at (2, 1) {$3$};
	\node [style=species] (4) at (5, 2.5) {$4$};
	\node [style=species] (5) at (7, 2.5) {$5$};
	\node [style=species] (6) at (5, 1) {$6$};
	\node [style=species] (7) at (7, 1) {$7$};
	\node [style=species] (8) at (3, -1) {$8$};
	\node [style=species] (9) at (5, -1) {$9$};
	\node [rectangle, dashed, fill = none, draw = black, minimum width = 220, minimum height = 150] () at (4, 0.75) {};
	\node [style=species] (11) at (-1, 6.5) {$1$};
	\node [style=species] (12) at (1, 6.5) {$2$};
	\node [style=species] (13) at (0, 5) {$3$};
	\node [rectangle, dashed, fill = none, draw = black, minimum width = 90, minimum height = 80] () at (0, 5.75) {};
	\node [style=species] (14) at (3, 6.5) {$1$};
	\node [style=species] (15) at (5, 6.5) {$2$};
	\node [style=species] (16) at (3, 5) {$3$};
	\node [style=species] (17) at (5, 5) {$4$};
	\node [rectangle, dashed, fill = none, draw = black, minimum width = 90, minimum height = 80] () at (4, 5.75) {};
	\node [style=species] (18) at (7, 5.75) {$1$};
	\node [style=species] (19) at (9, 5.75) {$2$};
	\node [rectangle, dashed, fill = none, draw = black, minimum width = 90, minimum height = 80] () at (8, 5.75) {};
	\node [style = none] (a) at (-1,8) {};
	\node [style = none] (b) at (1,8) {};
	\node [style = none] (c) at (3,8) {};
	\node [style = none] (d) at (4,8) {};
	\node [style = none] (e) at (5,8) {};
	\node [style = none] (f) at (7,8) {};
	\node [style = none] (g) at (9,8) {};
	\node [style = none] (h) at (-1,7.14) {};
	\node [style = none] (i) at (1,7.14) {};
	\node [style = none] (j) at (3,7.14) {};
	\node [style = none] (k) at (4,7.14) {};
	\node [style = none] (l) at (5,7.14) {};
	\node [style = none] (m) at (7,7.14) {};
	\node [style = none] (n) at (9,7.14) {};
	\node [style = none] (o) at (0,4.36) {};
	\node [style = none] (p) at (4,4.36) {};
	\node [style = none] (q) at (8,4.36) {};
	\node [style = none] (r) at (2,3.37) {};
	\node [style = none] (s) at (4,3.37) {};
	\node [style = none] (t) at (6,3.37) {};
	\node [style = none] (u) at (4,-1.85) {};
	\node [style = none] (v) at (4,-3) {};
\end{pgfonlayer}
\begin{pgfonlayer}{edgelayer}
	\draw [style=simple] (1) to (2);
	\draw [style=simple] (3) to (6);
	\draw [style=simple] (4) to (6);
	\draw [style=simple] (6) to (8);
	\draw [style=simple] (11) to (12);
	\draw [style=simple] (12) to (13);
	\draw [style=simple] (14) to (15);
	\draw [style=simple] (15) to (16);
	\draw [style=simple] (16) to (17);
	\draw [style=simple] (18) to (19);
	\draw [style=simple] (a) to (h);
	\draw [style=simple] (b) to (i);
	\draw [style=simple] (c) to (j);
	\draw [style=simple] (d) to (k);
	\draw [style=simple] (e) to (l);
	\draw [style=simple] (f) to (m);
	\draw [style=simple] (g) to (n);
	\draw [style=simple] (o) to (r);
	\draw [style=simple] (p) to (s);
	\draw [style=simple] (q) to (t);
	\draw [style=simple] (u) to (v);
\end{pgfonlayer}
\end{tikzpicture}}
\begin{tikzpicture}
\begin{pgfonlayer}{nodelayer}
	\node [style = none] () at (0,0) {=};
	\node [style = none] () at (1.5,-3) {};
	\node [style = none] () at (-1,0) {};
\end{pgfonlayer}
\begin{pgfonlayer}{edgelayer}
\end{pgfonlayer}
\end{tikzpicture}
\scalebox{0.6}{
\begin{tikzpicture}
\begin{pgfonlayer}{nodelayer}
	\node [style=species] (1) at (1, 2.5) {$1$};
	\node [style=species] (2) at (3, 2.5) {$2$};
	\node [style=species] (3) at (2, 1) {$3$};
	\node [style=species] (4) at (5, 2.5) {$4$};
	\node [style=species] (5) at (7, 2.5) {$5$};
	\node [style=species] (6) at (5, 1) {$6$};
	\node [style=species] (7) at (7, 1) {$7$};
	\node [style=species] (8) at (3, -1) {$8$};
	\node [style=species] (9) at (5, -1) {$9$};
	\node [rectangle, dashed, fill = none, draw = black, minimum width = 220, minimum height = 150] () at (4, 0.75) {};
	\node [style = none] (a) at (1,4.36) {};
	\node [style = none] (b) at (2,4.36) {};
	\node [style = none] (c) at (3,4.36) {};
	\node [style = none] (d) at (4,4.36) {};
	\node [style = none] (e) at (5,4.36) {};
	\node [style = none] (f) at (6,4.36) {};
	\node [style = none] (g) at (7,4.36) {};
	\node [style = none] (h) at (1,3.37) {};
	\node [style = none] (i) at (2,3.37) {};
	\node [style = none] (j) at (3,3.37) {};
	\node [style = none] (k) at (4,3.37) {};
	\node [style = none] (l) at (5,3.37) {};
	\node [style = none] (m) at (6,3.37) {};
	\node [style = none] (n) at (7,3.37) {};
	\node [style = none] (u) at (4,-1.85) {};
	\node [style = none] (v) at (4,-3) {};
\end{pgfonlayer}
\begin{pgfonlayer}{edgelayer}
	\draw [style=simple] (1) to (2);
	\draw [style=simple] (2) to (3);
	\draw [style=simple] (3) to (6);
	\draw [style=simple] (4) to (5);
	\draw [style=simple] (4) to (6);
	\draw [style=simple] (5) to (6);
	\draw [style=simple] (6) to (7);
	\draw [style=simple] (6) to (8);
	\draw [style=simple] (8) to (9);
	\draw [style=simple] (a) to (h);
	\draw [style=simple] (b) to (i);
	\draw [style=simple] (c) to (j);
	\draw [style=simple] (d) to (k);
	\draw [style=simple] (e) to (l);
	\draw [style=simple] (f) to (m);
	\draw [style=simple] (g) to (n);
	\draw [style=simple] (u) to (v);
\end{pgfonlayer}
\end{tikzpicture}}
\]

This example is fairly elementary, and it is probably not too difficult for someone comfortable with the notions to define this operad. However, it is not just simple graphs that one needs when talking about managing and tasking complex networks of various sorts of agents with various forms of communication and capabilities. One could continue replicating the procedure for constructing network operads for each type of network whenever needed, but this is an inefficient strategy. Instead, we devised a general recipe for constructing such an operad for a given network type, and a general method for specifying a network type in an efficient way, using what we call a \emph{network model}. All of this is done in the language of category theory, so we also have a theory of how morphisms between network models give morphisms between their operads.

\section*{Constructing Network Operads}
There is a well-known trick for extracting an operad from any symmetric monoidal category. An operation in the operad is defined to be a morphism from a tensor product of a finite family of objects to a single object. This is called the \emph{underlying operad} of the symmetric monoidal category. So now we have shifted the problem of defining an operad where the operations are networks to defining a symmetric monoidal category where the morphisms are networks. To achieve this, we can use the famous \emph{Grothendieck construction}---though we need to enhance it to suit our purposes. 

\section*{Monoidal Grothendieck Construction}
The Grothendieck construction is a well-known trick for turning a family of categories indexed by the objects of some other category into a single category in an intelligent way \cite{SGAI}. What we really would like is that morphisms in the indexing category translate into morphisms in our total category between objects from the corresponding indices. The classic example is the family of categories $\Mod_R$ of $R$-modules, indexed by the objects of $\Ring$, the category of rings. Sometimes, one would like to talk about a single category of modules over all possible rings to study the interactions between such modules. The naive thing to do would be to just take the coproduct, defining \[\Mod = \coprod_{R \in \Ring} \Mod_R.\] However, in this category an $R$-module and an $S$-module would have no morphisms between them. This runs counter to the goal of having a single category for reasoning about the interactions of modules over potentially different rings. If $f \maps R \to S$ is a ring homomorphism, there is a way of turning $S$-modules into $R$-modules using $f$, called \emph{pullback}. If $M$ is an $S$-module, $m \in M$, and $r \in R$, pulling back $M$ along $f$ defines an $R$-module structure on the underlying abelian group of $M$. We define the action of $r \in R$ on $m \in M$ by the following formula.
\[
    r \cdot m = f(r) \cdot m
\]
This construction turns out to give a functor 
\[f^\ast \maps \Mod_S \to \Mod_R.\]
We should hope also that the data of these functors is included in the total category we construct. Indeed, the Grothendieck construction accomplishes precisely this. 

However, it is not simply a category that we need, but a symmetric monoidal category. So we built an enhanced version of the Grothendieck construction, which takes family of categories indexed by a \emph{symmetric monoidal} category and constructs a \emph{symmetric monoidal} category \cite{NetworkModels}. Christina Vasilakopoulou and I extended this modification to solve the monoidality problem in the Grothendieck correspondence \cite{MonGroth}.

These two steps constitute the construction of the desired network operad: we start with a monoidal indexed category, use the monoidal Grothendieck construction to produce a symmetric monoidal category, and then take its underlying operad. This leads to another question though: what monoidal indexed categories should we feed into this construction in order to produce network operads?

\section*{Network Models}
The answer is that we should take a monoidal version of Joyal's combinatorial species \cite{especies}. A combinatorial species is a functor $F \maps \FinBij \to \Set$. One way of looking at this is as a family of symmetric group actions, one for each natural number. Another way of looking at it is as a particular type of indexed category, where there is a family of discrete categories (sets) indexed by the natural numbers, and functors (functions) between them corresponding to the morphisms in $\FinBij$. So this is something to which we can apply the Grothendieck construction. The resulting total category is a groupoid which has all the elements in all the sets as the objects, and an isomorphism between these elements if they are in the same orbit under the symmetric group action.

Recall our example of a network operad where an operation is a simple graph. To build this, we can start with the species of simple graphs $\SG \maps \FinBij \to \Set$. We give this the structure of a lax monoidal functor $(\FinBij, +) \to (\Set, \times)$ by equipping it with a natural map $\SG(m) \times \SG(n) \to \SG(m+n)$ given by disjoint union. We include the data of the overlaying of graphs as a monoid structure on the set $\SG(n)$ of simple graphs on $n$ nodes. The product of two graphs on $n$ nodes is another graph on $n$ nodes given by identifying corresponding nodes, and including an edge wherever either of the original graphs had one. So now we have a lax symmetric monoidal functor $(\SG, \sqcup) \maps (\FinBij, +) \to (\Mon, \times)$. We call such a map a \emph{network model}. When we take the Grothendieck construction of this, we treat the monoids as one-object categories. By doing this, the resulting category has objects given by finite sets, a morphism $n \to n$ is given by a simple graph on $n$ nodes, composition overlays the graphs, and tensor sets them side by side. 

\section*{Constructing Network Models}

Network operads are constructed from network models. How do we get our hands on some network models? We know about a few examples of network models: simple graphs, directed graphs, multigraphs, colored vertices, etc. Ideally, we would have a (functorial) way to generate network models from some simple description of what we want a network to look like. 

We can begin by examining the basic example: simple graphs. It consists of a family of monoids $\SG(n)$ where the elements are simple graphs on $n$ nodes, with symmetric group actions which permute the nodes, and a ``disjoint union'' operation $\sqcup \maps \SG(m) \times \SG(n) \to \SG(m+n)$. The level-$0$ and level-$1$ monoids are both trivial. The first interesting one is level-$2$, where the monoid is isomorphic to the Boolean truth values with the ``or'' operation. The rest of the monoids in this network model can be seen as built from $\SG(2)$. A simple graph with $n$ nodes has $n\choose2$ places where it can either have or not have an edge. We can define the monoid $\SG(n)$ to be the product of $n\choose2$ copies of $\SG(1)$, indexed by distinct pairs of nodes. This leads to the general construction: given a monoid $M$, let $\overline M(n)$ be the monoid given by the product $n \choose2$ copies of $M$. Then the collection of these monoids $\overline M$ is a network model, where a network has an element of $M$ between every pair of nodes, and overlaying two networks simply requires performing the monoid operation at every pair of nodes. This construction covers the example of simple graphs by design, but also includes multigraphs, directed graphs, graphs with colored edges, and many other examples. 

\section*{Noncommutative Network Models}

Another property we wanted to be able to represent within the network operads framework was forms of communication which had a built-in limitation on the number of connections. This is a natural issue in the search and rescue domain problem \cite{NoncommNetMods}.

There is no natural way to decide which edges not to include when the limit of connections is reached. This means that the network must have some extra data built into it. In particular, it must remember the order in which the connections were added to each node. For this, we need the edge components of the constituent monoids to not commute with each other. Due to a variant of the Eckmann-Hilton argument, edge components of a network model's constituent monoid actually must commute with each other if they do not share any of their nodes. This means the most we can ask for is that edge components of the network model do not commute with each other when edges have a node in common. 

We cannot simply take iterated products of the monoid as we did before because the edge components of the resulting monoids always commute with each other. We also cannot simply take coproducts because the edge components do not commute with each other in way that are necessary for a network model. Therefore, we must have a mix of products and coproducts depending on which edges share a node and which do not. Specifically, if two edges share a node, then elements of the corresponding edge components of the monoid must not commute with each other, and if they do not share a node, they must commute with each other. Such a monoid can be constructed using \emph{graph products of monoids}, introduced for groups in Elisabeth Green's thesis \cite{Green}. The idea is to produce a new monoid from a finite set of monoids by assigning them to the nodes in a graph, taking the coproduct of them all, then imposing commutativity relations between elements coming from monoids which had an edge between them in the graph.

What indexing graph should we use though? We want a copy of the monoid for every possible edge. So our indexing graph should have $n\choose2$ nodes, one for every subset of cardinality $2$. We want to impose commutativity between two edge components whenever the corresponding edges do not share a node, so we add an edge for each pair of cardinality $2$ subsets which have empty intersection. This is precisely the definition of what are called the \emph{Kneser graphs}! The first few non-empty ones are depicted below.
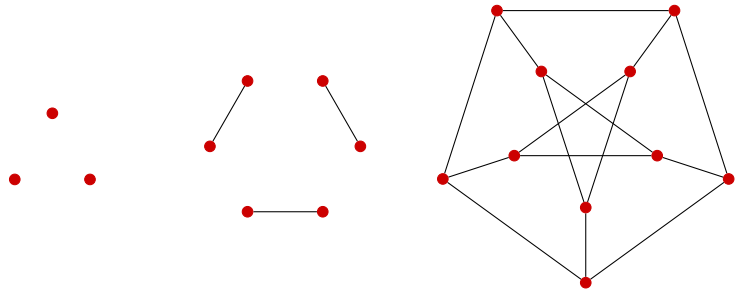
\[
\vcenter{\hbox{\begin{tikzpicture} 
	\begin{pgfonlayer}{nodelayer}
		\node [style=reddot] (a) at (0, 0.44) {};
		\node [style=reddot] (e) at (-0.5, -.44) {};
		\node [style=reddot] (f) at (0.5, -.44) {};
		\node [style=none] () at (2, 0) {};
	\end{pgfonlayer}
	\begin{pgfonlayer}{edgelayer}
	\end{pgfonlayer}
\end{tikzpicture}}}
\vcenter{\hbox{\begin{tikzpicture} 
	\begin{pgfonlayer}{nodelayer}
		\node [style=reddot] (a) at (1, 0) {};
		\node [style=reddot] (b) at (0.5, .87) {};
		\node [style=reddot] (c) at (-0.5, .87) {};
		\node [style=reddot] (d) at (-1, 0) {};
		\node [style=reddot] (e) at (-0.5, -.87) {};
		\node [style=reddot] (f) at (0.5, -.87) {};
		\node [style=none] () at (2, 0) {};
	\end{pgfonlayer}
	\begin{pgfonlayer}{edgelayer}
		\draw (a) to (b);
		\draw (c) to (d);
		\draw (e) to (f);
	\end{pgfonlayer}
\end{tikzpicture}}}
\vcenter{\hbox{\begin{tikzpicture} 
	\begin{pgfonlayer}{nodelayer}
		\node [style=reddot] (a) at (0, -1) {};
		\node [style=reddot] (b) at (0.95, -0.31) {};
		\node [style=reddot] (c) at (0.59, 0.81) {};
		\node [style=reddot] (d) at (-0.59, 0.81) {};
		\node [style=reddot] (e) at (-0.95, -0.31) {};
		\node [style=reddot] (A) at (0, -2) {};
		\node [style=reddot] (B) at (1.9, -0.62) {};
		\node [style=reddot] (C) at (1.18, 1.62) {};
		\node [style=reddot] (D) at (-1.18, 1.62) {};
		\node [style=reddot] (E) at (-1.9, -0.62) {};
	\end{pgfonlayer}
	\begin{pgfonlayer}{edgelayer}
		\draw (a) to (A);
		\draw (b) to (B);
		\draw (c) to (C);
		\draw (d) to (D);
		\draw (e) to (E);
		\draw (A) to (B);
		\draw (B) to (C);
		\draw (C) to (D);
		\draw (D) to (E);
		\draw (E) to (A);
		\draw (a) to (c);
		\draw (b) to (d);
		\draw (c) to (e);
		\draw (d) to (a);
		\draw (e) to (b);
	\end{pgfonlayer}
\end{tikzpicture}}}
\]

For a given monoid $M$, we thus define the corresponding network model to be the graph product of $M$ with itself indexed by the corresponding Kneser graph. In fact, this construction gives the free network model on $M$, forming a left adjoint to the functor which evaluates a network model at $2$. This provides a solution to the problem of representing degree limited networks in the language of network operads. This construction gives a network operad where the networks are graphs such that every vertex has degree $\leq N$, and the network does not take an edge if this limit would be exceeded.

\section*{Petri Nets with Catalysts}

Network models are also able to describe scenarios where there is an agent or agents that can manipulate and transport resources within the network \cite{Catalysts}. Baez, Foley, and I use a simple structure called a \emph{Petri net} to represent resources and processes that transform them \cite{BaezBiamonte}. A Petri net can be drawn as a directed graph with vertices of two kinds: \emph{places} or \emph{species}, which we draw as yellow circles below, and \emph{transitions}, which we draw as blue squares:
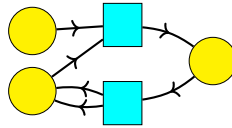
\[
\vcenter{\hbox{\scalebox{0.8}{
\begin{tikzpicture}
\begin{pgfonlayer}{nodelayer}
    \node [style=species] (C) at (-1, 0) {$\quad\;$};
    \node [style=species] (B) at (-4, -0.5) {$\quad\;$};
    \node [style=species] (A) at (-4, 0.5) {$\quad\;$};
    \node [style=transition] (tau1) at (-2.5, 0.6) {$\big.\;\;\;\, $};
    \node [style=transition] (tau2) at (-2.5, -0.7) {$\big.\;\;\;\, $};
\end{pgfonlayer}
\begin{pgfonlayer}{edgelayer}
    \draw [style=inarrow] (A) to (tau1);
    \draw [style=inarrow] (B) to (tau1);
    \draw [style=inarrow, bend left=15, looseness=1.00] (tau1) to (C);
    \draw [style=inarrow, bend left=15, looseness=1.00] (C) to (tau2);
    \draw [style=inarrow, bend left=15, looseness=1.00] (tau2) to (B); 
    \draw [style=inarrow, bend right =15, looseness=1.00] (tau2) to (B); 
\end{pgfonlayer}
\end{tikzpicture}
}}}
\]
Petri nets are intended to model resources in a network of processes. Sometimes, we represent the resources by a finite number of \emph{tokens} in each place:
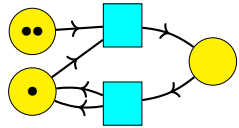
\[
\vcenter{\hbox{\scalebox{0.8}{
\begin{tikzpicture}
\begin{pgfonlayer}{nodelayer}
	\node [style=species] (C) at (-1, 0) {$\quad\;$};
	\node [style=species] (B) at (-4, -0.5) {$\;\bullet\;$};
	\node [style=species] (A) at (-4, 0.5) {$\bullet\bullet$};
	\node [style=transition] (tau1) at (-2.5, 0.6) {$\big.\;\;\;\, $};
    \node [style=transition] (tau2) at (-2.5, -0.7) {$\big.\;\;\;\, $};
\end{pgfonlayer}
\begin{pgfonlayer}{edgelayer}
	\draw [style=inarrow] (A) to (tau1);
	\draw [style=inarrow] (B) to (tau1);
	\draw [style=inarrow, bend left=15, looseness=1.00] (tau1) to (C);
    \draw [style=inarrow, bend left=15, looseness=1.00] (C) to (tau2);
    \draw [style=inarrow, bend left=15, looseness=1.00] (tau2) to (B); 
    \draw [style=inarrow, bend right =15, looseness=1.00] (tau2) to (B); 
\end{pgfonlayer}
\end{tikzpicture}
}}}
\]
This is called a \emph{marking}. We can then ``run'' the Petri net by repeatedly changing the marking using the transitions. For example, the above marking can change to this:
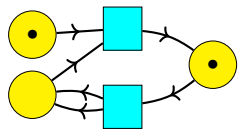
\[
\vcenter{\hbox{\scalebox{0.8}{
\begin{tikzpicture}
\begin{pgfonlayer}{nodelayer}
    \node [style=species] (C) at (-1, 0) {$\;\bullet\;$};
    \node [style=species] (B) at (-4, -0.5) {$\quad\;$};
    \node [style=species] (A) at (-4, 0.5) {$\;\bullet\;$};
    \node [style=transition] (tau1) at (-2.5, 0.6) {$\big.\;\;\;\, $};
    \node [style=transition] (tau2) at (-2.5, -0.7) {$\big.\;\;\;\, $};
\end{pgfonlayer}
\begin{pgfonlayer}{edgelayer}
	\draw [style=inarrow] (A) to (tau1);
	\draw [style=inarrow] (B) to (tau1);
	\draw [style=inarrow, bend left=15, looseness=1.00] (tau1) to (C);
    \draw [style=inarrow, bend left=15, looseness=1.00] (C) to (tau2);
    \draw [style=inarrow, bend left=15, looseness=1.00] (tau2) to (B); 
    \draw [style=inarrow, bend right =15, looseness=1.00] (tau2) to (B); 
\end{pgfonlayer}
\end{tikzpicture}
}}}
\]
and then this:
\[
\vcenter{\hbox{\scalebox{0.8}
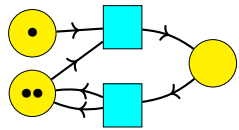
{
\begin{tikzpicture}
\begin{pgfonlayer}{nodelayer}
	\node [style=species] (C) at (-1, 0) {$\quad\;$};
	\node [style=species] (B) at (-4, -0.5) {$\bullet\bullet$};
	\node [style=species] (A) at (-4, 0.5) {$\;\bullet\;$};
	\node [style=transition] (tau1) at (-2.5, 0.6) {$\big.\;\;\;\, $};
    \node [style=transition] (tau2) at (-2.5, -0.7) {$\big.\;\;\;\, $};
\end{pgfonlayer}
\begin{pgfonlayer}{edgelayer}
	\draw [style=inarrow] (A) to (tau1);
	\draw [style=inarrow] (B) to (tau1);
	\draw [style=inarrow, bend left=15, looseness=1.00] (tau1) to (C);
    \draw [style=inarrow, bend left=15, looseness=1.00] (C) to (tau2);
    \draw [style=inarrow, bend left=15, looseness=1.00] (tau2) to (B); 
    \draw [style=inarrow, bend right =15, looseness=1.00] (tau2) to (B); 
\end{pgfonlayer}
\end{tikzpicture}
}}}
\]
Thus, the places represent different \emph{types} of resource, and the transitions describe ways that one collection of resources of specified types can turn into another such collection. 

An agent might pick up a box and carry it over to a truck, and then drive the truck over to a new warehouse, and then unload the box. In this scenario, the gasoline in the truck might be a resource that considered to be consumed by this process, but the agent is not. This qualitative difference between the agent as a resource and the gasoline as a resource leads to a quantitative difference. Specifically, the number of agents in this network is never changing, but the number of gallons of gasoline is. What this means for the Petri net model of this network is that there is no combination of transition firing that change the number of agents. This gives us a fibration of the commutative monoidal category of executions for the Petri net. However, unlike the monoidal fibrations described earlier, the fibres here are only \emph{pre}monoidal in general, not quite monoidal. This gives an example of a generalized network model, one where the monoids in the original definition are replaced with categories.

\section*{Outline of the Thesis}
I begin by laying out the theory of network models and network operads in \cref{ch:NetworkModels}. \cref{sec:netmod} and \cref{sec:netmod_C} contain basic definitions and examples. The construction of a network operad from a network model and several examples of algebras of network operads are given in \cref{sec:netoperads}. 

In \cref{ch:NNM}, more constructions of network models are given. The construction of free network models from a given monoid is detailed in \cref{sec:funnetmod}. This depends on a generalization of Green's graph products given in \cref{sec:graphs}. In section 3.4, an example of an algebra for a noncommutative network model arising from limitations on communication networks is given.

\cref{ch:PetriNets} discusses the construction of network models from Petri nets with catalysts. In \cref{sec:petri}, the basic notions for the categorical treatment of Petri nets are recalled. \cref{sec:catalysts} explains what it means for a Petri net to have catalysts. \cref{sec:premonoidal} describes how catalysts induce a premonoidal fibration on the category of executions, and explain how this gives an example of a generalized network model. 

I finish with a self-contained treatment of the monoidal Grothendieck construction in \cref{ch:MonGroth}. As the theoretical underpinning of the theory of network models, it is the most technically dense, and thus saved for the most enduring of readers. \cref{sec:monfib} and \cref{sec:monicat} describe monoidal fibrations and indexed categories. \cref{sec:monequiv} details the corresponding Grothendieck constructions for each monoidal variant. \cref{sec:cartesiancase} discusses the special case of when the base category is co/cartesian. In \cref{sec:summary}, we give a nuts-and-bolts description of the monoidal structures constructed in various scenarios. \cref{sec:applications} demonstrates the potential usefulness of the construction with examples from categorical algebra and dynamical systems. 

I wanted to include my own explanations and several references for preliminary materials, but did not want this to clutter the primary narrative of the thesis. I have included much of this in several appendices. I discuss some of the basic theory of monoidal categories in \cref{app:monoidalcats}; monoidal 2-categories and Gray monoids in \cref{app:Monoidal2cats}; fibrations, indexed categories, and the Grothendieck construction in \cref{app:fibicat}; and combinatorial species and operads in \cref{app:speciesandoperads}.

}{\ssp
\setcounter{chapter}{1}
\chapter{Network Models}
\label{ch:NetworkModels}

\section{Introduction}
\label{sec:intro}

In this chapter, we study operads suited for designing networks. These could be networks where the vertices represent fixed or moving agents and the edges represent communication channels. More generally, they could be networks where the vertices represent entities of various types, and the edges represent relationships between these entities, e.g.\ that one agent is committed to take some action involving the other. The work done is this chapter arose from an example where the vertices represent planes, boats and drones involved in a search and rescue mission in the Caribbean \cite{CommNet, CompTask}. However, even for this one example, we want a flexible formalism that can handle networks of many kinds, described at a level of detail that the user is free to adjust.

To achieve this flexibility, we introduce a general concept of \emph{network model}. Simply put, a network model is a \emph{kind} of network. Any network model gives an operad whose operations are ways to build larger networks of this kind by gluing smaller ones. This operad has a canonical algebra where the operations act to assemble networks of the given kind. But it also has other algebras, where it acts to assemble networks of this kind \emph{equipped with extra structure and properties}. This flexibility is important in applications. 

What exactly is a kind of network?  At the crudest level, we can model networks as simple graphs. If the vertices are agents of some sort and the edges represent communication channels, this means we allow at most one channel between any pair of agents. However, simple graphs are too restrictive for many applications. If we allow multiple communication channels between a pair of agents, we should replace simple graphs with multigraphs. Alternatively, we may wish to allow directed channels, where the sender and receiver have different capabilities: for example, signals may only be able to flow in one direction. This requires replacing simple graphs with directed graphs. To combine these features we could use directed multigraphs. It is also important to consider graphs with colored vertices, to specify different types of agents, and colored edges, to specify different types of channels. This leads us to colored directed multigraphs. All these are examples of what we mean by a kind of network. Even more complicated kinds, such as hypergraphs or Petri nets, are likely to become important as we proceed. Thus, instead of separately studying all these kinds of networks, we introduce a unified notion that subsumes all these variants: a \emph{network model}. Namely, given a set $C$ of vertex colors, a \define{network model} $F$ is a lax symmetric monoidal functor $F \maps \S(C) \to \Cat$, where $\S(C)$ is the free strict symmetric monoidal category on $C$ and $\Cat$ is the category of small categories, considered with its cartesian monoidal structure. Unpacking this definition takes a little work. It simplifies in the special case where $F$ takes values in $\Mon$, the category of monoids. It simplifies further when $C$ is a singleton, since then $\S(C)$ is the groupoid $\S$, where objects are natural numbers and morphisms from $m$ to $n$ are bijections $\sigma \maps \{1,\dots,m\} \to \{1,\dots,n\}$. If we impose both these simplifying assumptions, we have what we call a \define{one-colored network model}: a lax symmetric monoidal functor $F \maps \S \to \Mon$. As we shall see, the network model of simple graphs is a one-colored network model, and so are many other motivating examples.

Joyal began an extensive study of functors $F \maps \S \to \Set$, which are now commonly called \emph{species} \cite{especies, AnalyticFunctors, BLL}. Any type of extra structure that can be placed on finite sets and transported along bijections defines a species if we take $F(n)$ to be the set of structures that can be placed on the set $\{1, \dots, n\}$. From this perspective, a one-colored network model is a species with some extra operations. 

This perspective is helpful for understanding what a one-colored network model $F \maps \S \to \Mon$ is actually like. If we call elements of $F(n)$ \emph{networks with $n$ vertices}, then:
\begin{enumerate}
    \item Since $F(n)$ is a monoid, we can \define{overlay} two networks with the same number of vertices and get a new one. We denote this operation by
    \[
        \cup \colon F(n) \times F(n) \to F(n). 
    \]
    For example:
    \[\scalebox{0.8}
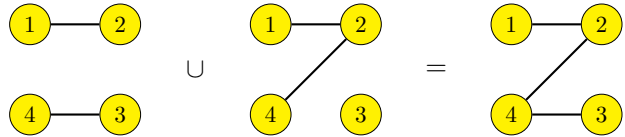
{
    \begin{tikzpicture}
	\begin{pgfonlayer}{nodelayer}
	    \node [style=none, scale = 1.2] () at (5,2) {$\cup$};
	    \node [style=none, scale = 1.2] () at (9,2) {=};
		\node [style=species]  (1) at (3.75, 2.75) {2};
		\node [style=species]  (2) at (2.25, 2.75) {1};
		\node [style=species]  (3) at (2.25, 1.25) {4};
		\node [style=species]  (4) at (3.75, 1.25) {3};
		\node [style=species]  (5) at (7.75, 2.75) {2};
		\node [style=species]  (6) at (6.25, 2.75) {1};
		\node [style=species]  (7) at (6.25, 1.25) {4};
		\node [style=species]  (8) at (7.75, 1.25) {3};
		\node [style=species]  (9) at (11.75, 2.75) {2};
		\node [style=species]  (10) at (10.25, 2.75) {1};
		\node [style=species]  (11) at (10.25, 1.25) {4};
		\node [style=species]  (12) at (11.75, 1.25) {3};
	\end{pgfonlayer}
	\begin{pgfonlayer}{edgelayer}
		\draw [style=simple] (2) to (1);
		\draw [style=simple] (3) to (4);
		\draw [style=simple] (6) to (5);
		\draw [style=simple] (5) to (7);
		\draw [style=simple] (10) to (9);
		\draw [style=simple] (9) to (11);
		\draw [style=simple] (11) to (12);
	\end{pgfonlayer}
    \end{tikzpicture}}\]
    \item Since $F$ is a functor, the group $S_n$ acts on the monoid $F(n)$. Thus, for each $\sigma \in S_n$, we have a monoid automorphism that we call
    \[\sigma \maps F(n) \to F(n)  . \]
    For example, if $\sigma = (2\,3) \in S_3$, then
    \[\scalebox{0.8}
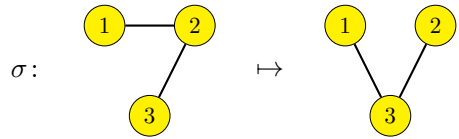
{
    \begin{tikzpicture}
	\begin{pgfonlayer}{nodelayer}
	    \node [style=none, scale = 1.2] () at (5,2) {$\sigma\maps$};
	    \node [style=none, scale = 1.2] () at (9,2) {$\mapsto$};
		\node [style=species]  (1) at (7.75, 2.75) {2};
		\node [style=species]  (2) at (6.25, 2.75) {1};
		\node [style=species]  (3) at (7, 1.25) {3};
		\node [style=species]  (13) at (11.75, 2.75) {2};
		\node [style=species]  (14) at (10.25, 2.75) {1};
		\node [style=species]  (15) at (11, 1.25) {3};
	\end{pgfonlayer}
	\begin{pgfonlayer}{edgelayer}
		\draw [style=simple] (2) to (1);
		\draw [style=simple] (1) to (3);
		\draw [style=simple] (14) to (15);
		\draw [style=simple] (15) to (13);
	\end{pgfonlayer}
    \end{tikzpicture}}\]
    \item Since $F$ is lax monoidal, we have an operation
    \[
        \sqcup \colon F(m) \times F(n) \to F(m+n)
    \]
    We call this operation the \define{disjoint union} of networks. For example:
    \[\scalebox{0.8}
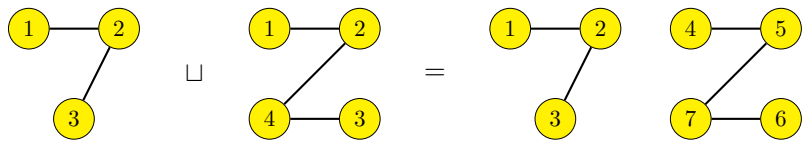
{
    \begin{tikzpicture}
    	\begin{pgfonlayer}{nodelayer}
    	    \node [style=none, scale = 1.2] () at (5,2) {$\sqcup$};
    	    \node [style=none, scale = 1.2] () at (9,2) {=};
    		\node [style=species]  (1) at (3.75, 2.75) {2};
    		\node [style=species]  (2) at (2.25, 2.75) {1};
    		\node [style=species]  (3) at (3, 1.25) {3};
    		\node [style=species]  (5) at (7.75, 2.75) {2};
    		\node [style=species]  (6) at (6.25, 2.75) {1};
    		\node [style=species]  (7) at (6.25, 1.25) {4};
    		\node [style=species]  (8) at (7.75, 1.25) {3};
    		\node [style=species]  (9) at (14.75, 2.75) {5};
    		\node [style=species]  (10) at (13.25, 2.75) {4};
    		\node [style=species]  (11) at (13.25, 1.25) {7};
    		\node [style=species]  (12) at (14.75, 1.25) {6};
    		\node [style=species]  (13) at (11.75, 2.75) {2};
    		\node [style=species]  (14) at (10.25, 2.75) {1};
    		\node [style=species]  (15) at (11, 1.25) {3};
    	\end{pgfonlayer}
    	\begin{pgfonlayer}{edgelayer}
    		\draw [style=simple] (2) to (1);
    		\draw [style=simple] (1) to (3);
    		\draw [style=simple] (6) to (5);
    		\draw [style=simple] (5) to (7);
    		\draw [style=simple] (7) to (8);
    		\draw [style=simple] (10) to (9);
    		\draw [style=simple] (9) to (11);
    		\draw [style=simple] (11) to (12);
    		\draw [style=simple] (14) to (13);
    		\draw [style=simple] (15) to (13);
    	\end{pgfonlayer}
    \end{tikzpicture}}\]
\end{enumerate}
The first two operations are present whenever we have a functor $F \maps \S \to \Mon$. The last two are present whenever we have a lax symmetric monoidal functor $F \maps \S \to \Set$. When $F$ is a one-colored network model we have all three---and unpacking the definition further, we see that they obey some equations, which we list in \cref{thm:equations}. For example, the interchange law 
\[(g \cup g') \sqcup (h \cup h') = (g \sqcup h) \cup (g' \sqcup h') \]
holds whenever $g,g' \in F(m)$ and $h, h' \in F(n)$.

In \cref{sec:netmod} we study one-colored network models more formally, and give many examples. In \cref{sec:models_from_monoids} we describe a systematic procedure for getting one-colored network models from monoids. In \cref{sec:netmod_C} we study general network models and give examples of these. In \cref{sec:cat_netmod} we describe a category $\NetMod$ of network models, and show that the procedure for getting network models from monoids is functorial. We also make $\NetMod$ into a symmetric monoidal category, and give examples of how to build new networks models by tensoring old ones.

Our main result is that any network model gives a typed operad, also known as a \emph{colored operad} or \emph{symmetric multicategory} \cite{Yau}. A typed operad describes ways of sticking together things of various types to get new things of various types. An algebra of the operad gives a particular specification of these things and the results of sticking them together. We review the definitions of operads and their algebras in \cref{app:operads}. A bit more precisely, a typed operad $O$ has:
\begin{itemize}
    \item a set $T$ of \define{types},
    \item sets of \define{operations} $O(t_1,...,t_n ; t)$ where $t_i, t \in T$,
    \item ways to \define{compose} any operation
    \[f \in O(t_1,\dots,t_n ;t) \]
    with operations 
    \[g_i \in O(t_{i1},\dots,t_{i k_i}; t_i)   \qquad (1 \le i \le n) \] 
    to obtain an operation 
    \[f \circ (g_1,\dots,g_n) \in O(t_{1i}, \dots, t_{1k_1}, \dots, 
    t_{n1}, \dots t_{n k_n}; t), \]
    \item and ways to permute the arguments of operations,
    \end{itemize}
which obey some rules \cite{Yau}. An algebra $A$ of $O$ specifies a set $A(t)$ for each type $t \in T$ such that the operations of $O$ act on these sets. Thus, it has:
\begin{itemize}
    \item for each type $t \in T$, a set $A(t)$ of things of type $t$,
    \item ways to apply any operation
    \[f \in O(t_1, \dots, t_n ; t) \]
    to things
    \[a_i \in A(t_i)  \qquad (1 \le i \le n) \]
    to obtain a thing
    \[\alpha(f)(a_1, \dots, a_n) \in A(t). \]
\end{itemize}
Again, we demand that some rules hold \cite{Yau}. 

 In \cref{thm:one-colored_network_operads} we describe the typed operad $\O_F$ arising from a one-colored network model $F$. The set of types is $\N$, since we can think of `network with $n$ vertices' as a type. The sets of operations are given as follows:
\[
    \O_F(n_1, \dots, n_k; n) = \left\{ 
    \begin{array}{cl}  S_n \times F(n) & \textrm{if } 
    n_1 + \cdots + n_k = n \\
    \emptyset & \textrm{otherwise.} 
    \end{array}  \right. 
\]
The key idea here is that we can overlay a network in $F(n)$ on the disjoint union of networks with $n_1, \dots, n_k$ vertices and get a new network with $n$ vertices as long as $n_1 + \cdots n_k = n$. We can also permute the vertices; this accounts for the group $S_n$. But the most important fact is that \emph{networks serve as operations to assemble networks}, thanks to our ability to overlay them.

Using this fact, we show in \cref{ex:canonical_algebra} that the operad $\O_F$ has a canonical algebra $A_F$ whose elements are simply networks of the kind described by $F$:
\[A_F(n) = F(n) .\]
In this algebra any operation
\[
    (\sigma,g) \in  \O_F(n_1, \dots , n_k; n) = S_n \times F(n) 
\] 
acts on a $k$-tuple of networks
\[
    h_i \in A_F(n_i) = F(n_i)   \qquad (1 \le i \le k) 
\]
to give the network
\[
    \alpha(\sigma,g)(h_1, \dots, h_k) =  g \cup \sigma(h_1 \sqcup \cdots \sqcup h_k) \in A_F(n).
\]
In other words, we first take the disjoint union of the networks $h_i$, then permute their vertices with $\sigma$, and then overlay the network $g$.

An example is in order, since the generality of the formalism may hide the simplicity of the idea. The easiest example of our theory is the network model for simple graphs. In \cref{ex:simple_graph} we describe a one-colored network model $\SG \maps \S \to \Mon$ such that $\SG(n)$ is the collection of simple graphs with vertex set $\n = \{1,\dots,n\}$. Such a simple graph is really a collection of 2-element subsets of $\n$, called \emph{edges}. Thus, we may overlay simple graphs $g,g' \in \SG(n)$ by taking their union $g \cup g'$. This operation makes $\SG(n)$ into a monoid. 

Now consider an operation $f \in \O_\SG(3,4,2;9)$. This is an element of $S_9 \times \SG(9)$: a permutation of the set $\{1,\dots, 9\}$ together with a simple graph having this set of vertices. If we take the permutation to be the identity for simplicity, this operation is just a simple graph $g \in \SG(9)$. We can draw an example as follows:
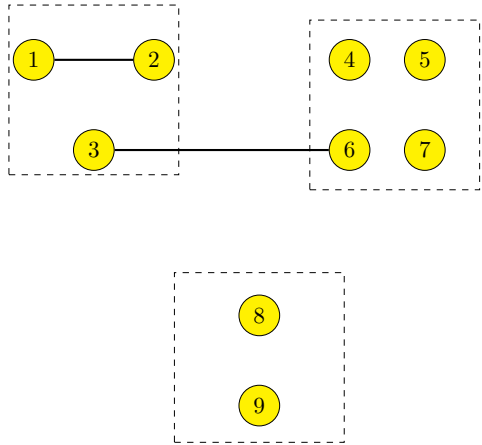
\[\scalebox{0.8}{
\begin{tikzpicture}
	\begin{pgfonlayer}{nodelayer}
		\node [style=species] (0) at (2, 1) {$3$};
		\node [style=species] (1) at (4.75, -3.25) {$9$};
		\node [style=species] (2) at (7.5, 2.5) {$5$};
		\node [style=species] (3) at (1, 2.5) {$1$};
		\node [style=none] (4) at (0, 3) {};
		\node [style=none] (5) at (7, -2.5) {};
		\node [style=species] (6) at (4.75, -1.75) {$8$};
		\node [style=bounding] (7) at (4.75, -2.45) {};
		\node [style=bounding] (8) at (2, 2) {};
		\node [style=none] (9) at (8.75, 3) {};
		\node [style=species] (10) at (6.25, 2.5) {$4$};
		\node [style=species] (11) at (6.25, 1) {$6$};
		\node [style=bounding] (12) at (7, 1.75) {};
		\node [style=species] (13) at (3, 2.5) {$2$};
		\node [style=species] (14) at (7.5, 1) {$7$};
	\end{pgfonlayer}
	\begin{pgfonlayer}{edgelayer}
		\draw [style=simple] (0) to (11);
		\draw [style=simple] (3) to (13);
	\end{pgfonlayer}
\end{tikzpicture}
}\]
The dashed circles indicate that we are thinking of this simple graph as an element of $\O(3,4,2;9)$: an operation that can be used to assemble simple graphs with 3, 4, and 2 vertices, respectively, to produce one with 9 vertices.

Next let us see how this operation acts on the canonical algebra $A_\SG$, whose elements are simple graphs. Suppose we have elements $a_1 \in A_\SG(3)$, $a_2 \in A_\SG(4)$ and $a_3 \in A_\SG(2)$:
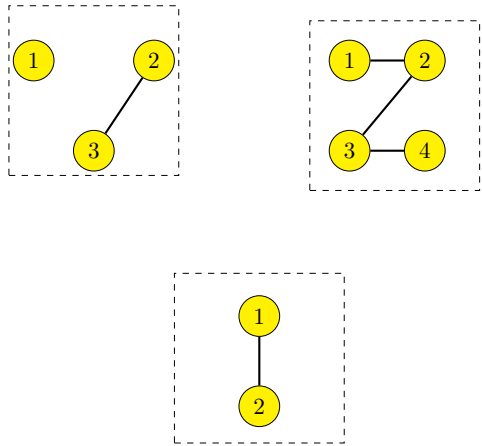
\[
\scalebox{0.8}{
\begin{tikzpicture}
	\begin{pgfonlayer}{nodelayer}
		\node [style=species] (0) at (2, 1) {$3$};
		\node [style=species] (1) at (4.75, -3.25) {$2$};
		\node [style=species] (2) at (7.5, 2.5) {$2$};
		\node [style=species] (3) at (1, 2.5) {$1$};
		\node [style=none] (4) at (0, 3) {};
		\node [style=none] (5) at (7, -2.5) {};
		\node [style=species] (6) at (4.75, -1.75) {$1$};
		\node [style=bounding] (7) at (4.75, -2.45) {};
		\node [style=bounding] (8) at (2, 2) {};
		\node [style=none] (9) at (8.75, 3) {};
		\node [style=species] (10) at (6.25, 2.5) {$1$};
		\node [style=species] (11) at (6.25, 1) {$3$};
		\node [style=bounding] (12) at (7, 1.75) {};
		\node [style=species] (13) at (3, 2.5) {$2$};
		\node [style=species] (14) at (7.5, 1) {$4$};
	\end{pgfonlayer}
	\begin{pgfonlayer}{edgelayer}
		\draw [style=simple] (13) to (0);
		\draw [style=simple] (10) to (2);
		\draw [style=simple] (2) to (11);
		\draw [style=simple] (11) to (14);
		\draw [style=simple] (6) to (1);
	\end{pgfonlayer}
\end{tikzpicture}
}
\]
We can act on these by the operation $f$ to obtain $\alpha(f)(a_1,a_2,a_3) \in A_\SG(9)$. It looks like this:
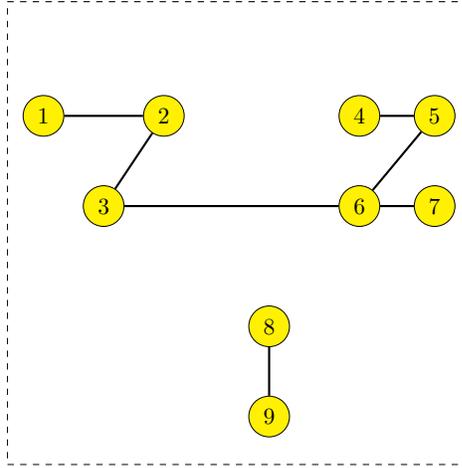
\[\scalebox{0.8}{
\begin{tikzpicture}
	\begin{pgfonlayer}{nodelayer}
		\node [style=species] (0) at (2, 1) {$3$};
		\node [style=species] (1) at (4.75, -2.5) {$9$};
		\node [style=species] (2) at (7.5, 2.5) {$5$};
		\node [style=species] (3) at (1, 2.5) {$1$};
		\node [style=none] (4) at (0, 3) {};
		\node [style=none] (5) at (7, -2.5) {};
		\node [style=species] (6) at (4.75, -1) {$8$};
		\node [style=none] (9) at (8.75, 3) {};
		\node [style=species] (10) at (6.25, 2.5) {$4$};
		\node [style=species] (11) at (6.25, 1) {$6$};
		\node [style=species] (13) at (3, 2.5) {$2$};
		\node [style=species] (14) at (7.5, 1) {$7$};
		\node [style=triplebounding] (15) at (4.25, .55) {};
	\end{pgfonlayer}
	\begin{pgfonlayer}{edgelayer}
		\draw [style=simple] (3) to (13);
		\draw [style=simple] (13) to (0);
		\draw [style=simple] (10) to (2);
		\draw [style=simple] (0) to (11);
		\draw [style=simple] (2) to (11);
		\draw [style=simple] (11) to (14);
		\draw [style=simple] (6) to (1);
	\end{pgfonlayer}
\end{tikzpicture}}\]
We have simply taken the disjoint union of $a_1$, $a_2$, and $a_3$ and then overlaid $g$, obtaining a simple graph with 9 vertices.

The canonical algebra is one of the simplest algebras of the operad $O_\SG$. We can define many more interesting algebras for this operad. For example, we might wish to use this operad to describe communication networks where the communicating entities have locations and the communication channels have limits on their range. To include location data, we can choose $A(n)$ for $n \in \N$ to be the set of all graphs with $n$ vertices where each vertex is an actual point in the plane $\R^2$. To handle range-limited communications, we could instead choose $A(n)$ to be the set of all graphs with $n$ vertices in $\R^2$ where an edge is permitted between two vertices only if their Euclidean distance is less than some specified value. This still gives a well-defined algebra: when we apply an operation, we simply omit those edges from the resulting graph that would violate this restriction.

Besides the plethora of interesting algebras for the operad $O_\SG$, and useful homomorphisms between these, one can also modify the operad by choosing another network model. This provides additional flexibility in the formalism. Different network models give different operads, and the construction of operads from network models is functorial, so morphisms of network models give morphisms of operads. 

In \cref{sec:netoperads} we apply the machinery provided by \cref{ch:MonGroth} to build operads from network models. We also describe some algebras of these operads, and in \cref{ex:range_limit_algebra} we discuss an algebra whose elements are networks of range-limited communication channels. 
\section{One-Colored Network Models}
\label{sec:netmod}

We begin with a special class of network models: those where the vertices of the network have just one color. To define these, we use $\S$ to stand for a skeleton of the groupoid of finite sets and bijections:

\begin{defn}
    Let $\S$, the \define{symmetric groupoid}, be the groupoid for which:
    \begin{itemize}
        \item objects are natural numbers $n \in \N$,
        \item a morphism from $m$ to $n$ is a bijection $\sigma \colon \{1,\dots,m\} \to
        \{1,\dots,n\}$
    \end{itemize}
    and bijections are composed in the usual way.
\end{defn}

There are no morphisms in $\S$ from $m$ to $n$ unless $m = n$. For each $n \in \N$, the endomorphisms of $n$ form the symmetric group $S_n$. It is convenient to write $\n$ for the set $\{1,\dots,n\}$, so that a morphism $\sigma \maps n \to n$ in $\S$ is the same as a bijection $\sigma \maps \n \to \n$.

There is a functor $+ \maps \S \times \S \to \S$ defined as follows. Given $m, n \in \N$ we let $m + n$ be the usual sum, and given $\sigma \in S_m$ and $\tau \in S_n$, let $\sigma+\tau \in S_{m+n}$ be as follows:
\begin{equation}
\label{eq:plus}
    (\sigma + \tau)(j)=
    \begin{cases}
        \sigma(j)&\text{if } j\leq m
        \\\tau(j-m)+m&\text{otherwise.}
    \end{cases}
\end{equation}
For objects $m, n \in \S$, let $B_{m,n}$ be the block permutation of $m+n$ which swaps the first $m$ with the last $n$. For example $B_{4,3} \maps 7 \to 7$ is the permutation $(1473625)$:
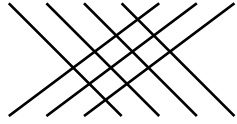
\[
\begin{tikzpicture}
	\begin{pgfonlayer}{nodelayer}
		\node [style=empty] (1) at (1, 1.5) {};
		\node [style=empty] (2) at (1.5, 1.5) {};
		\node [style=empty] (3) at (2, 1.5) {};
		\node [style=empty] (4) at (2.5, 1.5) {};
		\node [style=empty] (5) at (3, 1.5) {};
		\node [style=empty] (6) at (3.5, 1.5) {};
		\node [style=empty] (7) at (4, 1.5) {};
		\node [style=empty] (1a) at (1, 0) {};
		\node [style=empty] (2a) at (1.5, 0) {};
		\node [style=empty] (3a) at (2, 0) {};
		\node [style=empty] (4a) at (2.5, 0) {};
		\node [style=empty] (5a) at (3, 0) {};
		\node [style=empty] (6a) at (3.5, 0) {};
		\node [style=empty] (7a) at (4, 0) {};
	\end{pgfonlayer}
	\begin{pgfonlayer}{edgelayer}
		\draw [style=simple] (1.center) to (4a.center);
		\draw [style=simple] (2.center) to (5a.center);
		\draw [style=simple] (3.center) to (6a.center);
		\draw [style=simple] (4.center) to (7a.center);
		\draw [style=simple] (5.center) to (1a.center);
		\draw [style=simple] (6.center) to (2a.center);
		\draw [style=simple] (7.center) to (3a.center);
	\end{pgfonlayer}
\end{tikzpicture} \]
The tensor product $+$ and braiding $B$ give $\S$ the structure of a strict symmetric monoidal category. This follows as a special case of \cref{prop:free}. 

\begin{defn}
\label{defn:network_model}
    A \define{one-colored network model} is a lax symmetric monoidal functor 
    \[F \maps  \S \to \Mon .\]
    Here $\Mon$ is the category with monoids as objects and monoid homomorphisms as morphisms, considered with its cartesian monoidal structure.
\end{defn}

Algebraically, a network model is a family of monoids $\{M_n\}_{n \in \N}$ each with a group action of the corresponding symmetric group $S_n$, such that the product of any two embed into the one indexed by the sum of their indices equivariantly, i.e.\ in a way which respects the group action: $M_m \times M_n \hookrightarrow M_{m+n}$. 

Many examples of network models are given below. A pedestrian way to verify that these examples really are network models is to use the following result:

\begin{thm}
\label{thm:equations}
    A one-colored network model $F \maps \S \to \Mon$ is the same as:
    \begin{itemize}
        \item a family of sets $\{F(n)\}_{n\in \N}$
        \item distinguished \define{identity} elements $e_n \in F(n)$
        \item a family of \define{overlay} functions $\cup \maps F(n) \times F(n) \to F(n)$
        \item a bijection $\sigma \maps F(n) \to F(n)$ for each $\sigma \in S_n$
        \item a family of \define{disjoint union} functions $\sqcup \maps F(m) \times F(n) \to F(m+n)$
    \end{itemize}
    satisfying the following equations:
    \begin{multicols}{2}
    \begin{enumerate}
        \item $e_n \cup g = g \cup e_n = g$
        \item $g_1 \cup (g_2 \cup g_3) = (g_1 \cup g_2) \cup g_3$
        \item $\sigma(g_1 \cup g_2) = \sigma g_1 \cup \sigma g_2$
        \item $\sigma e_n = e_n$
        \item $(\sigma_2 \sigma_1) g = \sigma_2 (\sigma_1 g)$
        \item $(g_1 \cup g_2) \sqcup (h_1 \cup h_2) = (g_1 \sqcup h_1) \cup (g_2 \sqcup h_2)$
        \item $1 (g) = g$
        \item $e_m \sqcup e_n = e_{m+n}$
        \item $\sigma g \sqcup \tau h = (\sigma + \tau) (g \sqcup h)$
        \item $g_1 \sqcup (g_2 \sqcup g_3) = (g_1 \sqcup g_2) \sqcup g_3$
        \item $e_0 \sqcup g = g \sqcup e_0 = g$
        \item $B_{m,n} (h \sqcup g) = g \sqcup h$
    \end{enumerate}
    \end{multicols}
    \noindent
    for $g, g_i \in F(n)$, $h, h_i \in F(m)$, $\sigma, \sigma_i \in S_n$, $\tau \in S_m$, and $1$ the identity of $S_n$.
\end{thm}

\begin{proof}
    Having a functor $F \maps \S \to \Mon$ is equivalent to having the first four items satisfying Equations 1--6. The binary operation $\cup$ gives the set $F(n)$ the structure of a monoid, with $e_n$ acting as the identity. Equation 1 tells us $e_n$ acts as an identity, and Equation 2 gives the associativity of $\cup$. Equations 3 and 4 tell us that $\sigma$ is a monoid homomorphism. Equations 5 and 6 say that the map $(\sigma,g) \mapsto \sigma g$ defines an action of $S_n$ on $F(n)$ for each $n$. All of these actions together give us the functor $F \maps \S \to \Mon$.
    
    That the functor is lax monoidal is equivalent to having item 5 satisfying Equations 7--11. Equations 7 and 8 tell us that $\sqcup$ is a family of monoid homomorphisms. Equation 9 tells us that it is a natural transformation. Equation 10 tells us that the associativity hexagon diagram for lax monoidal functors commutes for $F$. Equation 11 implies the commutativity of the left and right unitor square diagrams. That the lax monoidal functor is symmetric is equivalent to Equation 12. It tells us that the square diagram for symmetric monoidal functors commutes for $F$. 
\end{proof}

This is one of the simplest examples of a network model:

\begin{expl}[\textbf{Simple graphs}] \label{ex:simple_graph}
    Let a \define{simple graph} on a set $V$ be a set of 2-element subsets of $V$, called \define{edges}. There is a one-colored network model $\SG \maps \S \to \Mon$ such that $\SG(n)$ is the set of simple graphs on $\n$.
    
    To construct this network model, we make $\SG(n)$ into a monoid where the product of simple graphs $g_1, g_2 \in \SG(n)$ is their union $g_1 \cup g_2$. Intuitively speaking, to form their union, we `overlay' these graphs by taking the union of their sets of edges. The simple graph on $\n$ with no edges acts as the unit for this operation. The groups $S_n$ acts on the monoids $\SG(n)$ by permuting vertices, and these actions define a functor $\SG \maps \S \to \Mon$.
    
    Given simple graphs $g \in \SG(m)$ and $h \in \SG(n)$ we define $g \sqcup h \in \SG(m + n)$ to be their disjoint union. This gives a monoid homomorphism $\sqcup \maps \SG(m) \times \SG(n) \to \SG(m + n)$ because 
    \[(g_1 \cup g_2) \sqcup (h_1 \cup h_2) = (g_1 \sqcup h_1) \cup (g_2 \sqcup h_2). \]
    This in turn gives a natural transformation with components
    \[\sqcup_{m, n} \maps \SG(m) \times \SG(n) \to \SG(m + n), \] 
    which makes $\SG$ into lax symmetric monoidal functor. 
    
    One can prove this construction really gives a network model using either \cref{thm:equations}, which requires verifying a list of equations, or \cref{thm:graph_model}, which gives a  general procedure for getting a network model from a monoid $M$ by letting elements of $\Gamma_M(n)$ be maps from the complete graph on $\n$ to $M$. If we take $M  = \Boole = \{F,T\}$ with `or' as the monoid operation, this procedure gives the network model $\SG = \Gamma_\Boole$. We explain this in \cref{ex:simple_graph_2}.
\end{expl}
 
There are many other kinds of graph, and many of them give network models:

\begin{expl}[\textbf{Directed graphs}] \label{ex:directed_graph}
    Let a \define{directed graph} on a set $V$ be a collection of ordered pairs $(i,j) \in V^2$ such that $i \ne j$. These pairs are called \define{directed edges}. There is a network model $\DG \maps \S \to \Mon$ such that $\DG(n)$ is the set of directed graphs on $\n$. As in \cref{ex:simple_graph}, the monoid operation on $\DG(n)$ is union.
\end{expl}

\begin{expl}[\textbf{Multigraphs}] \label{ex:multigraph}
    Let a \define{multigraph} on a set $V$ be a multiset of 2-element subsets of $V$. If we define $\MG(n)$ to be the set of multigraphs on $\n$, then there are at least two natural choices for the monoid operation on $\MG(n)$. 
    The most direct generalization of $\SG$ of \cref{ex:simple_graph} is the network model $\MG \maps \S \to \Mon$ with 
     values $(\MG(n), \cup)$ where $\cup$  is now union of edge multisets.
     That is, the multiplicity of $\{ i, j \}$ in  $g \cup h$ is maximum of the  multiplicity of $\{ i, j \}$ in $g$ and the  multiplicity of $\{ i, j \}$ in $h$.
    Alternatively, there is another network model  $\MGplus \maps \S \to \Mon$ with 
     values $(\MG(n), +)$ where $+$  is multiset sum. That is, $g + h$ obtained by adding multiplicities of corresponding edges. 
\end{expl}
 
\begin{expl}[\textbf{Directed multigraphs}] \label{ex:directed_multigraph}
    Let a \define{directed multigraph} on a set $V$ be a multiset of ordered pairs $(i,j) \in V^2$ such that $i \ne j$. There is a network model $\DMG \maps \S \to \Mon$ such that $\DMG(n)$ is the set of directed multigraphs on $\n$ with monoid operation the union of multisets. Alternatively, there is a network model with values $(\DMG(n), +)$ where $+$  is multiset sum.
 \end{expl}
 
\begin{expl}[\textbf{Hypergraphs}]
 \label{ex:hypergraph}
   Let a \define{hypergraph} on a set $V$ be a set of nonempty subsets of $V$, called \define{hyperedges}. There is a network model $\HG \maps \S \to \Mon$ such that $\HG(n)$ is the set of hypergraphs on $\n$. The monoid operation $\HG(n)$ is union.
\end{expl}

\begin{expl}[\textbf{Graphs with colored edges}]
\label{ex:graphs_with_colored_edges}
    Fix a set $B$ of \define{edge colors} and let $\SG \maps \S \to \Mon$ be the network model of simple graphs as in \cref{ex:simple_graph}. Then there is a network model $H \maps \S \to \Mon$ with
    \[
        H(n) = \SG(n)^B
    \]
    making the product of $B$ copies of the monoid $\SG(n)$ into a monoid in the usual way. In this model, a network is a $B$-tuple of simple graphs, which we may view as a graph with at most one edge of each color between any pair of distinct vertices. We describe this construction in more detail in \cref{ex:graphs_with_colored_edges_2}.
\end{expl}

There are also examples of network models not involving graphs:

\begin{expl}[\textbf{Partitions}]
\label{ex:partitions}
    A poset is a lattice if every finite subset has both an infimum and a supremum. If $L$ is a lattice, then $(L, \vee)$ and $(L, \wedge)$ are both monoids, where $x \vee y$ is the supremum of $\{x,y\} \subseteq L$ and $x \wedge y$ is the infimum. 
    
    Let $P(n)$ be the set of partitions of the set $\n$. This is a lattice where $\pi \le \pi'$ if the partition $\pi$ is finer than $\pi'$. Thus, $P(n)$ can be made a monoid in either of the two ways mentioned above. Denote these monoids as $P^{\vee}(n)$ and $P^{\wedge}(n)$. These monoids extend to give two network models $P^{\vee}, P^{\wedge} \maps \S \to \Mon$. 
\end{expl}
 
\subsection{One-colored network models from monoids}
\label{sec:models_from_monoids}

There is a systematic procedure that gives many of the network models we have seen so far. To do this, we take networks to be ways of labelling the edges of a complete graph by elements of some monoid $M$. The operation of overlaying two of these networks is then described using the monoid operation. 

For example, consider the Boolean monoid $\Boole$: that is, the set $\{F,T\}$ with `inclusive or' as its monoid operation. A complete graph with edges labelled by elements of $\Boole$ can be seen as a simple graph if we let $T$ indicate the presence of an edge between two vertices and $F$ the absence of an edge. To overlay two simple graphs $g_1, g_2$ with the same set of vertices we simply take the `or' of their edge labels. This gives our first example of a network model, \cref{ex:simple_graph}. 

To formalize this we need some definitions. Given $n \in \N$, let $\E(n)$ be the set of 2-element subsets of $\n = \{1, \dots, n\}$. We call the members of $\E(n)$ \define{edges}, since they correspond to edges of the complete graph on the set $\n$. We call the elements of an edge $e \in \E(n)$ its \define{vertices}.

Let $M$ be a monoid. For $n \in \N$, let $\Gamma_M(n)$ be the set of functions $g \maps \E(n) \to M$. Define the operation $\cup \colon \Gamma_M(n) \times \Gamma_M(n) \to \Gamma_M(n)$ by $(g_1 \cup g_2)(e) = g_1(e) g_2(e)$ for $e \in \E(n)$.  Define the map $\sqcup \colon \Gamma_M(m) \times \Gamma_M(n) \to \Gamma_M(m+n)$ by
\[
    (g_1 \sqcup g_2)(e) =
    \left\{\begin{array}{cl}
    g_1(e) & {\rm if \; both \; vertices \; of \;} e {\rm \; are \;}\leq m 
    \\
    g_2(e) & {\rm if \; both \; vertices \; of \;} e {\rm \; are \;} > m 
    \\
    {\rm the \; identity \; of \; } M & {\rm otherwise} 
    \\
    \end{array}
    \right.
\]
The symmetric group $S_n$ acts on $\Gamma_M(n)$ by $\sigma(g)(e) = g(\sigma^{-1}(e))$.

\begin{thm}
\label{thm:graph_model}
    For each monoid $M$ the data above gives a one-colored network model $\Gamma_M \maps \S \to \Mon$.
\end{thm}
\begin{proof}
We can define $\Gamma_M$ as the composite of two functors, $\E \maps \S \to \Inj$ and $M^{-} \maps \Inj \to \Mon$, where $\Inj$ is the category of sets and injections.

The functor $\E \maps \S \to \Inj$ sends each object $n \in \S$ to $\E(n)$, and it sends each morphism $\sigma \maps n \to n$ to the permutation of $\E(n)$ that maps any edge $e = \{x,y\} \in \E(n)$ to $\sigma(e) = \{\sigma(x), \sigma(y)\}$. The category $\Inj$ does not have coproducts, but it is closed under coproducts in $\Set$. It thus becomes symmetric monoidal with $+$ as its tensor product and the empty set as the unit object. For any $m, n \in \S$ there is an injection
\[\mu_{m,n} \maps \E(m) + \E(n) \to \E(m+n) \]
expressing the fact that a 2-element subset of either $\m$ or $\n$ gives a 2-element subset of $\m+\n$. The functor $\E \maps \S \to \Inj$ becomes lax symmetric monoidal with these maps $\mu_{m,n}$ giving the lax preservation of the tensor product.

The functor $M^- \maps \Inj \to \Mon$ sends each set $X$ to the set $M^X$ made into a monoid with pointwise operations, and it sends each function $f \maps X \to Y$ to the monoid homomorphism $M^f \maps M^X \to M^Y$ given by
\[(M^f g)(y) = \left\{ \begin{array}{ccl}
g(f^{-1}(y)) & \textrm{if } y \in \mathrm{im}(f) \\
1 & \textrm{otherwise} 
\end{array} \right.\]
for any $g \in M^X$. Using the natural isomorphisms $M^{X + Y} \cong M^X \times M^Y$ and $M^{\emptyset} \cong 1$ this functor can be made symmetric monoidal. 

As the composite of the lax symmetric monoidal functor $\E \maps \S \to \Inj$ and the symmetric monoidal functor $M^- \maps \Inj \to \Mon$, the functor $\Gamma_M \maps \S \to \Mon$ is lax symmetric monoidal, and thus a network model. With the help of \cref{thm:equations}, it is easy to check that this description of $\Gamma_M$ is equivalent to that in the theorem statement.
\end{proof}

\begin{expl}[\textbf{Simple graphs, revisited}]
\label{ex:simple_graph_2}
    Let $\Boole = \{F,T\}$ be the Boolean monoid. If we interpret $T$ and $F$ as `edge' and `no edge' respectively, then $\Gamma_{\Boole}$ is just $\SG$, the network model of simple graphs discussed in \cref{ex:simple_graph}.
\end{expl}

Recall from \cref{ex:multigraph} that a multigraph on the set $\n$ is a multisubset of $\E(n)$, or in other words, a function $g \maps \E(n) \to \N$. There are many ways to create a network model $F \maps \S \to \Mon$ for which $F(n)$ is the set of multigraphs on the set $\n$, since $\N$ has many monoid structures. Two of the most important are these:

\begin{expl}[\textbf{Multigraphs with addition for overlaying}]
\label{ex:multigraph_2}
    Let $(\N, +)$ be $\N$ made into a monoid with the usual notion of addition as $+$. In this network model, overlaying two multigraphs $g_1, g_2 \maps \E(n) \to \N$ gives a multigraph $g \maps \E(n) \to \N$ with $g(e) = g_1(e) + g_2(e)$. In fact, this notion of overlay corresponds to forming the multiset sum of edge multisets and $\Gamma_{(\N,+)}$ is the network model of multigraphs called $\MGplus$ in \cref{ex:multigraph}. 
\end{expl}

\begin{expl}[\textbf{Multigraphs with maximum for overlaying}]
\label{ex:multigraph_3}
    Let $(\N, \max)$ be $\N$ made into a monoid with $\max$ as the monoid operation. Then $\Gamma_{(\N,\max)}$ is a network model where overlaying two multigraphs $g_1, g_2 \maps \E(n) \to \N$ gives a multigraph $g \maps \E(n) \to \N$ with $g(e) = g_1(e) \max g_2(e)$.
    For this monoid structure overlaying two copies of the same multigraph gives the same multigraph. In other words, every element in each monoid $\Gamma_{(\N,\max)}(n)$ is idempotent and $\Gamma_{(\N,\max)}$ is the network model of multigraphs called $\MG$ in \cref{ex:multigraph}. 
\end{expl}

\begin{expl}[\textbf{Multigraphs with at most $k$ edges between vertices}]
\label{ex:multigraph_with_at_most_k}
    For any $k \in \N$, let $\Boole_k$ be the set $\{0,\dots,k\}$ made into a monoid with the monoid operation $\oplus$ given by 
    \[x \oplus y = (x + y) \min k \]
    and $0$ as its unit element. For example, $\Boole_0$ is the trivial monoid and $\Boole_1$ is isomorphic to the Boolean monoid. There is a network model $\Gamma_{\Boole_k}$ such that $\Gamma_{\Boole_k}(n)$ is the set of multigraphs on $\n$ with at most $k$ edges between any two distinct vertices. 
\end{expl}

\section{General Network Models}
\label{sec:netmod_C}

The network models described so far allow us to handle graphs with colored edges, but not with colored vertices. Colored vertices are extremely important for applications in which we have a network of agents of different types. Thus, network models will involve a set $C$ of vertex colors in general. This requires that we replace $\S$ by the free strict symmetric monoidal category generated by the color set $C$. Thus, we begin by recalling this category.

For any set $C$, there is a category $\SC$ for which:
\begin{itemize}
    \item Objects are formal expressions of the form
    \[
        c_1 \otimes \cdots \otimes c_n 
    \]
    for $n \in \N$ and $c_1, \dots, c_n \in C$. 
    We denote the unique object with $n = 0$ as $I$.
    \item There exist morphisms from $c_1 \otimes \cdots \otimes c_m$ to $c'_1 \otimes \cdots \otimes c'_n$ only if $m = n$, and in that case a morphism is a permutation $\sigma \in S_n$ such that $c'_{\sigma(i)} = c_i$ for all $i$.
    \item Composition is the usual composition of permutations. 
\end{itemize}

Note that elements of $C$ can be identified with certain objects of $\S(C)$, namely the one-fold tensor products. We do this in what follows.

\begin{prop}
\label{prop:free}
    $\S(C)$ can be given the structure of a strict symmetric monoidal category making it into the free strict symmetric monoidal category on the set $C$. Thus, if $\A$ is any strict symmetric monoidal category and $f \maps C \to \Ob(\A)$ is any function from $C$ to objects of the $\A$, there exists a unique strict symmetric monoidal functor $F \maps \S(C) \to \A$ with $F(c) = f(c)$ for all $c \in C$.
\end{prop}

\begin{proof}
This is well-known; see for example Sassone \cite[Sec.\ 3]{Sassone} or Gambino and Joyal \cite[Sec.\ 3.1]{GambinoJoyal}. The tensor product of objects is $\otimes$, the unit for the tensor product is $I$, and the braiding 
\[(c_1 \otimes \cdots \otimes c_m) \otimes (c'_1 \otimes \cdots \otimes c'_n) \to (c'_1 \otimes \cdots \otimes c'_n)  \otimes (c_1 \otimes \cdots \otimes c_m) \]
is the block permutation $B_{m, n}$. Given $f \maps C \to \Ob(\A)$, we define $F\maps \S(C) \to \A$ on objects by
 \[F(c_1 \otimes \cdots \otimes c_n) = f(c_1) \otimes \cdots \otimes f(c_n) , \]
and it is easy to check that $F$ is strict symmmetric monoidal, and the unique functor with the required properties.
\end{proof}

\begin{defn}
\label{defn:colored_network_model}
    Let $C$ be a set, called the set of \define{vertex colors}. A
    $C$\define{-colored network model} is a lax symmetric monoidal functor 
     \[F \maps  \SC \to \Cat. \] 
    A \define{network model} is a $C$-colored network model for some set $C$.
\end{defn}

If $C$ has just one element, $\S(C) \cong \S$ and a $C$-colored network model is a one-colored network model in the sense of \cref{defn:network_model}. Here are some more interesting examples:

\begin{expl}[\textbf{Simple graphs with colored vertices}]
\label{ex:simple_graphs_with_colored_vertices}
    There is a network model of simple graphs with $C$-colored vertices. To construct this, we start with the network model of simple graphs $\SG \maps \S \to \Mon$ given in \cref{ex:simple_graph}. There is a unique function from $C$ to the one-element set. By \cref{prop:free}, this function extends uniquely to a strict symmetric monoidal functor 
    \[F \maps \S(C) \to \S . \]
    An object in $\S(C)$ is formal tensor product of $n$ colors in $C$; applying $F$ to this object we forget the colors and obtain the object $n \in \S$. Composing $F$ and $\SG$, we obtain a lax symmetric monoidal functor
    \[
        \S(C) \stackrel{F}{\longrightarrow} \S  \stackrel{\SG}{\longrightarrow} \Mon
    \]
    which is the desired network model. We can use the same idea to `color' any of the network models in \cref{sec:netmod}.
    
    Alternatively, suppose we want a network model of simple graphs with $C$-colored vertices where an edge can only connect two vertices of the same color. For this we take a cartesian product of $C$ copies of the functor $\SG$, obtaining a lax symmetric monoidal functor 
    \[{\SG}^C \maps \S^C \to \Mon^C. \]  
    There is a function $h \maps C \to \Ob(\S^C)$ sending each $c \in C$ to the object of $S^\C$ that equals $1 \in \S$ in the $c$th place and $0 \in \S$ elsewhere. Thus, by \cref{prop:free}, $h$ extends uniquely to a strict symmetric monoidal functor 
    \[H_C \maps \S(C) \to \S^C  .\]
    Furthermore, the product in $\Mon$ gives a symmetric monoidal functor
    \[\Pi \maps \Mon^C \to \Mon .\]
    Composing all these, we obtain a lax symmetric monoidal functor
    \[
        \SC \stackrel{H_C}{\longrightarrow} \S^C \stackrel{\SG^C}{\longrightarrow} \Mon^C \stackrel{\Pi}{\longrightarrow} \Mon
    \]
    which is the desired network model. 
    
    More generally, if we have a network model $F_c \maps \S \to \Mon$ for each color $c \in C$, we can use the same idea to create a network model:
    \[
    \begin{tikzcd}
        \S(C) 
        \arrow[r, "H_C"]
        &
        \S^C
        \arrow[r, "\prod_{c \in C} F_c"]
        &
        \Mon^C
        \arrow[r, "\prod"]
        &
        \Mon
    \end{tikzcd}\]
    in which the vertices of color $c \in C$ partake in a network of type $F_c$.
    \label{ex:colors}
\end{expl}

\begin{expl}[\textbf{Petri nets}]
    Petri nets are a kind of network widely used in computer science, chemistry and other disciplines \cite{RxNet}. A \define{Petri net} $(S, T, i, o)$ is a pair of finite sets and a pair of functions $i, o \maps S \times T \to \N$. Let $P(m, n)$ be the set of Petri nets $(\m, \n, i, o)$. This becomes a monoid with product
    \[
        (\m, \n, i, o) \cup (\m, \n, i', o')
        = (\m, \n, i+i', o+o')
    \]
    The groups $S_m\times S_n$ naturally act on these monoids, so we have a functor 
    \[P \maps \S^2 \to \Mon . \]
    There are also `disjoint union' operations
    \[\sqcup \maps P(m, n) \times P(m', n') \to P(m+m', n+n') \]
    making $P$ into a lax symmetric monoidal functor. In \cref{ex:simple_graphs_with_colored_vertices} we described a strict symmetric monoidal functor $H_C \maps \S(C) \to \S^C$ for any set $C$. In the case of the 2-element set this gives
    \[H_2 \maps \S(2) \to \S^2 .\]
    We define the network model of Petri nets to be the composite
    \[\S(2) \stackrel{H_2}{\longrightarrow} \S^2 \stackrel{P}{\longrightarrow} \Mon .\]
\end{expl}

\subsection{Categories of network models}
\label{sec:cat_netmod}

For each choice of the set $C$ of vertex colors, we can define a category $\NetMod_C$ of $C$-colored network models. However, it is useful to create a larger category $\NetMod$ containing all these as subcategories, since there are important maps between network models that involve changing the vertex colors. 

\begin{defn}
\label{defn:NM_C}
For any set $C$, let $\NetMod_C$ be the category for which:
\begin{itemize}
    \item an object is a $C$-colored network model, that is, a lax symmetric monoidal functor $F \maps \SC\to \Cat$, 
    \item a morphism is a monoidal natural transformation between such functors:
    \[\begin{tikzcd}
    \SC
    \arrow[r, "F", bend left=40]
    \arrow[bend left=40]{r}[name=LUU, below]{}
    \arrow[r, "F'", bend right=40, swap, pos=0.45]
    \arrow[bend right=40, pos = 0.53]{r}[name=LDD]{}
    \arrow[Rightarrow, to path=(LUU) -- (LDD)\tikztonodes]{r}{\gn}
    & 
    \Cat
    \end{tikzcd}\]
    and composition is the usual composition of monoidal natural transformations.
    \end{itemize}
\end{defn}

In particular, $\NetMod_1$ is the category of one-colored network models. For an example involving this category, consider the network models built from monoids in \cref{sec:models_from_monoids}. Any monoid $M$ gives a one-colored network model $\Gamma_M$ for which an element of $\Gamma_M(n)$ is a way of labelling the edges of the complete graph on $\n$ by elements of $M$. Thus, we should expect any homomorphism of monoids $f \maps M \to M'$ to give a morphism of network models $\Gamma_f \maps \Gamma_M \to \Gamma_{M'}$ for which 
\[\Gamma_f(n) \maps \Gamma_M(n) \to \Gamma_{M'}(n)  \]
applies $f$ to each edge label. 

Indeed, this is the case. As explained in the proof of \cref{thm:graph_model}, the network model $\Gamma_M$ is the composite
\[\S \stackrel{\E}{\longrightarrow} \Inj 
\stackrel{M^{-}}{\longrightarrow} \Mon .\]
The homomorphism $f$ gives a natural transformation
\[f^{-} \maps M^{-} \To M'^{-}  \]
that assigns to any finite set $X$ the monoid homomorphism
\[\begin{array}{rccl} 
f^X \maps & M^X & \to     & M'^X  \\
          &  g  & \mapsto & f \circ g .
\end{array}
\]      
It is easy to check that this natural transformation is monoidal. Thus, we can whisker it with the lax symmetric monoidal functor $\E$ to get a morphism of network models:
\[\begin{tikzcd}
    \S \arrow[r, "\E"] &
    \Inj
    \arrow[r, "M^-", bend left=40]
    \arrow[bend left=40]{r}[name=LUU, below]{}
    \arrow[r, "M'^-", bend right=40, swap, pos=0.45]
    \arrow[bend right=40]{r}[name=LDD]{}
    \arrow[Rightarrow, to path=(LUU) -- (LDD)\tikztonodes]{r}{f^-}
    & 
    \Mon
    \end{tikzcd}\]
and we call this $\Gamma_f \maps \Gamma_M \to \Gamma_{M'}$.

\begin{thm}
\label{thm:models_from_monoids}
There is a functor 
\[\Gamma \maps \Mon \to \NetMod_1  \]
sending any monoid $M$ to the network model $\Gamma_M$ and any homomorphism of monoids $f \maps M \to M'$ to the morphism of network models $\Gamma_f \maps \Gamma_M \to \Gamma_{M'}$.
\end{thm}

\begin{proof}
To check that $\Gamma$ preserves composition, note that
\[\begin{tikzcd}[column sep=huge]
    \S 
    \arrow[r, "\E"] 
    &
    \Inj
    \arrow[r, "M^-", bend left=80]
    \arrow[r, ""{name=TOP}, bend left=80, swap, pos=0.455]
    \arrow[r, "M'^-"{name=Ml}] 
    \arrow[Rightarrow, from=TOP, to=Ml, "f^-", pos=0.3]
    \arrow[r, ""{name=M}, swap]
    \arrow[r, "M''^-", bend right=80, swap]
    \arrow[r, ""{name=BOT}, bend right=80, pos=0.45]
    \arrow[Rightarrow, from=M, to=BOT, "f'^-", pos=0.5]
    & 
    \Mon
\end{tikzcd}\]
equals
\[\begin{tikzcd}[column sep=huge]
    \S \arrow[r, "\E"] &
    \Inj
    \arrow[r, "M^-", bend left=80]
    \arrow[bend left=80, pos=0.47]{r}[name=LUU, below]{}
    \arrow[r, "M''^-", bend right=80, swap, pos=0.5]
    \arrow[bend right=80, pos=0.44]{r}[name=LDD]{}
    \arrow[Rightarrow, from=LUU, to=LDD, "(f'f)^-"]
    & 
    \Mon
    \end{tikzcd}   \]
since $f'^- f^- = (f'f)^-$. Similarly $\Gamma$ preserves identities. \end{proof}

It has been said that category theory is the subject in which even the examples need examples. So, we give an example of the above result:

\begin{expl}[\textbf{Imposing a cutoff on the number of edges}]
\label{ex:cutoff}
In \cref{ex:multigraph_2} we described the network model of multigraphs $\MGplus$ as $\Gamma_{(\N, +)}$. In \cref{ex:multigraph_with_at_most_k} we described a network model $\Gamma_{\Boole_k}$ of multigraphs with at most $k$ edges between any two distinct vertices. There is a homomorphism of monoids
\begin{align*}
    f \maps (\N, +) &\to \Boole_k\\
    n &\mapsto n \min k
\end{align*}
and this induces a morphism of network models
\[\Gamma_f \maps \Gamma_{(\N, +)} \to \Gamma_{\Boole_k} .\]
This morphism imposes a cutoff on the number of edges between any two distinct vertices: if there are more than $k$, this morphism keeps only $k$ of them. In particular, if $k = 1$, $\Boole_k$ is the Boolean monoid, and 
\[\Gamma_f \maps \MGplus \to \SG \]
sends any multigraph to the corresponding simple graph.
\end{expl}

One useful way to combine $C$-colored networks is by `tensoring' them. This makes $\NetMod_C$ into a symmetric monoidal category:

\begin{thm}
\label{thm:tensor_product_of_network_models}
For any set $C$, the category $\NetMod_C$ can be made into a symmetric monoidal category with the tensor product defined pointwise, so that for objects $F, F' \in \NetMod_C$ we have 
\[(F \otimes F')(x) = F(x) \times F'(x) \]
for any object or morphism $x$ in $\S(C)$, and for morphisms $\phi, \phi'$ in $\NetMod_C$ we have
\[(\phi \otimes \phi')_x = \phi_x \times \phi'_x \]
for any object $x \in \S(C)$.
\end{thm}

\begin{proof}
    More generally, for any symmetric monoidal categories $\A$ and $\B$, there is a symmetric monoidal category $\Sym\Mon\Cat(\A, \B)$ whose objects are lax symmetric monoidal functors from $\A$ to $\B$ and whose morphisms are monoidal natural transformations, with the tensor product defined pointwise. The proof in the `weak' case was given by Hyland and Power \cite{HP}, and the lax case works the same way.
\end{proof}

If $F, F' \maps \S(C) \to \Mon$ then their tensor product again takes values in $\Mon$. There are many interesting examples of this kind:

\begin{expl}[\textbf{Graphs with colored edges, revisited}]
\label{ex:graphs_with_colored_edges_2}
    In \cref{ex:graphs_with_colored_edges} we described network models of simple graphs with colored edges. The above result lets us build these network models starting from more basic data. To do this we start with the network model for simple graphs, $\SG \maps \S \to \Mon$, discussed in \cref{ex:simple_graph}. Fixing a set $B$ of `edge colors', we then take a tensor product of copies of $\SG$, one for each $b \in B$. The result is a network model $\SG^{\otimes B} \maps \S \to \Mon$ with 
    \[\SG^{\otimes B}(n) = \SG(n)^B  \]
    for each $n \in \N$.
\end{expl}

\begin{expl}[\textbf{Combined networks}]
\label{ex:mixed_networks}
    We can also combine networks of different kinds. For example, if $\DG \maps \S \to \Mon$ is the network model of directed graphs given in \cref{ex:directed_graph} and $\MG \maps \S \to \Mon$ is the network model of multigraphs given in \cref{ex:multigraph}, then   
    \[\DG \otimes \MG \maps \S \to \Mon \]
    is another network model, and we can think of an element of $(\DG \otimes \MG)(n)$ as a directed graph with red edges together with a multigraph with blue edges on the set $\n$.
\end{expl}

Next we describe a category $\NetMod$ of network models with arbitrary color sets, which includes all the categories $\NetMod_C$ as subcategories. To do this, first we introduce `color-changing' functors. Recall that elements of $C$ can be seen as certain objects of $\S(C)$, namely the 1-fold tensor products. If $f \maps C \to C'$ is a function, there exists a unique strict symmetric monoidal functor $f_* \maps \S(C) \to \S(C')$ that equals $f$ on objects of the form $c \in C$. This follows from \cref{prop:free}.

Next, we define an indexed category $\NetMod_{-} \maps \Set\op \to \CAT$ that sends any set $C$ to $\NetMod_C$ and any function $f \maps C \to D$ to the functor that sends any $D$-colored network model $F \maps \S(D) \to \Cat$ to the $C$-colored network model given by the composite
\[\S(C) \xrightarrow{f_*} \S(D) \xrightarrow{F} \Cat .\]
Applying the Grothendieck construction (see \cref{app:fibicat}) to this indexed category, we define the category of network models to be
\[\NetMod = \int \NetMod_-. \]
In elementary terms, $\NetMod$ has:
\begin{itemize}
    \item pairs $(C, F)$ for objects, where $C$ is a set and $F \maps \S(C) \to \Cat$ is a $C$-colored network model.
    \item pairs $(f, g) \maps (C, F) \to (D, G)$ for morphisms, where $f \maps C \to D$ is a function and $g \maps F \Rightarrow G \circ f_*$ is a morphism of network models.
\end{itemize}

\begin{expl}[\textbf{Simple graphs with colored vertices, revisited}]
    In \cref{ex:simple_graphs_with_colored_vertices} we constructed the network model of simple graphs with colored vertices. We started with the network model for simple graphs, which is a one-colored network model $\SG \maps \S \to \Mon$. The unique function $! \maps C \to 1$ gives a strict symmetric monoidal functor $!_* \maps \S(C) \to \S(1) \cong \S$. The network model of simple graphs with $C$-colored vertices is the composite 
     \[
       \S(C) \xrightarrow{!_*} \S \xrightarrow{\SG} \Mon
    \]
    and there is a morphism from this to the network model of simple graphs, which has the effect of forgetting the vertex colors.
\end{expl}

In fact, $\NetMod$ can be understood as a subcategory of the following category:

\begin{defn} 
\label{defn:SMICat}
Let $\Sym\Mon\ICat$ be the category where:
\begin{itemize}
    \item objects are pairs $(\C, F)$ where $\C$ is a small symmetric monoidal category and $F \maps \C \to \Cat$ is a lax symmetric monoidal functor, where $\Cat$ is considered with its cartesian monoidal structure.
    \item morphisms from $(\C, F)$ to $(\C', F')$ are pairs $(G, \gn)$ where $G \maps \C \to \C'$ is a lax symmetric monoidal functor and $\gn \maps F \To F' \circ G $ is a symmetric monoidal natural transformation:
    \[\begin{tikzcd}
        \C
        \arrow[dr, "F"]
        \arrow[dr, ""{name=F}, swap]
        \arrow[dd, "G", swap]
        \\&
        \Cat
        \\
        \C'
        \arrow[ur, "F'", swap]
        \arrow[ur, ""{name=F'}, pos=0.43]
        \arrow[Rightarrow, from = F, to = F', "\gn", swap]
    \end{tikzcd}\]
\end{itemize}
\end{defn}

We shall use this way of thinking in the next two sections to build operads from network models. It must be said that $\Sym\Mon\ICat$  is naturally a 2-category where a 2-morphism $\xi \maps (G, \gn) \To (G', \gn')$ is a natural transformation $\xi \maps G \to G'$ such that 
\[\begin{tikzcd}
    \C 
    \arrow[dd, bend right=90, "G", pos=0.495, swap]
    \arrow[dd, bend right=90, ""{name=L, right}, swap, phantom, pos = 0.49]
    \arrow[dd, "G'"{name=R, left}, swap]
    \arrow[dr, "F", pos=0.4]
    \arrow[dr, ""{name=U, below}, pos=0.4, phantom]
    & & &  
    \C'
    \arrow[dd, "G"{name=R2, left}, pos=0.52, swap]
    \arrow[dr, "F", pos=0.4]
    \arrow[dr, ""{name=U2, below}, pos=0.44, phantom]
    \\
    & \Cat & = & &  \Cat.
    \\
    \C
    \arrow[ur, "F'", swap, pos=0.4]
    \arrow[ur, ""{name=D}, pos=0.35, phantom]
    \arrow[Rightarrow, from=U, to=D, "\gn'", swap]
    \arrow[Rightarrow, from=L, to=R, "\xi"{above}, swap]
    & & & 
    \C'
    \arrow[ur, "F'", swap, pos=0.4]
    \arrow[ur, ""{name=D2}, phantom, pos = 0.4]
    \arrow[Rightarrow, from=U2, to=D2, "\gn", swap, pos=0.55]
\end{tikzcd}
\]
and here we are considering its 1-dimensional truncation. The 2-dimensional structure is detailed in \cref{app:fibicat}, and utilized in \cref{ch:MonGroth}. This lets us define 2-morphisms between network models, extending $\NetMod$ to a 2-category. We do not seem to need these 2-morphisms in our applications, so we suppress 2-categorical considerations in most of what follows. 

\section{Operads from Network Models}
\label{sec:netoperads}

Next we describe the operad associated to a network model. This construction is given in two steps. For the first step, we can use the strict symmetric Grothendieck construction of \cref{sec:monequiv} to define a strict symmetric monoidal category $\Int F$ from a given network model $F \maps \S(C) \to \Cat$. For the second step, we then use the \emph{underlying operad} construction (recalled in \cref{prop:operad_from_symmoncat}) to build an operad $\O_F$.
\begin{defn}
\label{defn:CN}
    Given a network model $F \maps \S(C) \to \Cat$, define the \define{network operad} $\O_F$ to be $\Op(\Int F)$. 
\end{defn} 
For the sake of the unfamiliar reader, we give a brief description of these constructions in the specific context of network models, which does not assume prior knowledge. We recall the ordinary Grothendieck construction in \cref{app:fibicat}, and \cref{ch:MonGroth} is entirely dedicated to studying the (braided/symmetric) monoidal variants of it. We give a nuts-and-bolts description of the symmetric monoidal category $(\int F, \otimes_\phi)$ built from a network model $(F, \phi) \maps (\S(C), \otimes ) \to (\Mon, \times)$. 

The objects of $\Int F$ correspond to objects of $\S(C)$, which are formal expressions of the form $c_1 \otimes \cdots \otimes c_n$ with $n \in \N$ and $c_i \in C$. The morphisms of $\int F$ are pairs $(\sigma, g)$ where $\sigma \maps c_1 \otimes \dots \otimes c_n \to c_{\sigma1} \otimes \dots \otimes c_{\sigma n}$ is a morphism in $\S(C)$, and $g$ is an element of the monoid $F(c_{\sigma1} \otimes \dots \otimes c_{\sigma n})$. Composition is given by $(\sigma, g) \circ (\tau, h) = (\sigma \tau, g \cdot F \sigma(h))$. The tensor product of two objects is given by concatenation. The tensor product of two morphisms is given by $(\sigma, g) \otimes (\tau, h) = (\sigma \otimes \tau, \phi(g,h))$. The unit object is $(I, \phi_0)$, where $I$ is the monoidal unit for $\S(C)$ and $\phi_0$ is the unit laxator for $F$. 

For a one-object network model $F$, a more compact description of the category $\int F$ can be given by the following formula, where monoids and groups are being considered as one-object categories by default.
\[\int F \cong \coprod_{n \in \N} F(n) \rtimes S_n\]

The network operad $\O_F$ is a typed operad where the types are ordered $k$-tuples of elements of $C$. For objects $x_i, x$ of $\int F$, the operations in $\O_F$ are given by $\O_F(x_1, \dots, x_n; x) = \Int F(x_1 \otimes \cdots \otimes x_n, x)$.

Now suppose that $F$ is a one-colored network model, so that $F \maps \S \to \Mon$. Then the objects of $\S$ are simply natural numbers, so $\O_F$ is an $\N$-typed operad. Given $n_1, \dots, n_k, n \in \N$, we have
\[
    \O_F(n_1, \dots, n_k; n) = \hom_{\Int \! F}(n_1 + \cdots + n_k, n). 
\]
By definition, a morphism in this homset is a pair consisting of a bijection $\sigma \maps n_1 + \cdots + n_k \to n$ and an element of the monoid $F(n)$. So, we have
\begin{equation}
\label{eq:operations_in_CN}
    \O_F(n_1, \dots, n_k; n) = \left\{ 
    \begin{array}{cl}  S_n \times F(n) & \textrm{if } n_1 + \cdots n_k = n \\
    \emptyset & \textrm{otherwise.} \\
    \end{array} \right. 
\end{equation}

Here is the basic example:

\begin{expl}[\textbf{Simple network operad}]
\label{ex:simple_network_operad} 
    If $\SG \maps \S \to \Mon$ is the network model of simple graphs in \cref{ex:simple_graph}, we call $\O_\SG$ the \define{simple network operad}. By \cref{eq:operations_in_CN}, an operation in $\O_\SG(n_1, \dots, n_k; k)$ is an element of $S_n$ together with a simple graph having $\mathbf{n} = \{1, \dots, n\}$ as its set of vertices. 
\end{expl}

The operads coming from other one-colored network models work similarly. For example, if $\DG \maps \S \to \Mon$ is the network model of directed graphs from \cref{ex:directed_graph}, then an operation in $\O_\SG(n_1, \dots, n_k; n)$ is an element of $S_n$ together with a directed graph having $\n$ as its set of vertices.

In \cref{thm:equations} we gave a pedestrian description of one-colored network models. We can describe the corresponding network operads in the same style:

\begin{thm}
\label{thm:one-colored_network_operads}
    Suppose $F$ is a one-colored network model. Then the network operad $\O_F$ is the $\N$-typed operad for which the following hold:
    \begin{enumerate}
        \item The sets of operations are given by
        \[\O_F(n_1, \dots, n_k; n) = \left\{ 
        \begin{array}{cl}  S_n \times F(n) & \textrm{if } 
        n_1 + \cdots n_k = n \\
        \emptyset & \textrm{otherwise.} 
        \end{array}  \right. \]
        \item Composition of operations is given as follows. Suppose that
        \[(\sigma, g) \in S_n \times F(n) = \O_F(n_1, \dots, n_k; n) \]
        and for $1 \le i \le k$ we have
        \[(\tau_i, h_i) \in S_{n_i} \times F(n_i) =
        \O_F(n_{i1}, \dots, n_{ij_i}; n_i). \]
        Then 
        \[(\sigma, g) \circ ((\tau_1, h_1), \dots, (\tau_k, h_k)) = (\sigma (\tau_1 + \cdots + \tau_k), g \cup \sigma(h_1 \sqcup \cdots \sqcup h_k)) \]
        where $+$ is defined in \cref{eq:plus}, while $\cup$ and $\sqcup$ are defined in \cref{thm:equations}.
        \item The identity operation in $\O_F(n;n)$ is 
        $(1, e_n)$, where $1$ is the identity in $S_n$ and $e_n$ is the identity in the monoid $F(n)$.
        \item The right action of the symmetric group $S_k$ on $\O_F(n_1, \dots, n_k;n)$ is given as follows. Given $(\sigma, g) \in \O_F(n_1, \dots, n_k;n)$ and $\tau \in S_k$, we have
        \[(\sigma, g) \tau = (\sigma\tau, g) . \]
    \end{enumerate}
\end{thm}

\begin{proof}
    This is a straightforward combination of the underlying operad of a symmetric monoidal category and the symmetric monoidal structure on $\int F$.
\end{proof}

The construction of operads from symmetric monoidal categories described in \cref{prop:operad_from_symmoncat} is functorial, so the construction of operads from network models is as well. 

\begin{thm} 
\label{thm:O}
    The assignment of a network model $F \maps \S(C) \to \Cat$ to the operad $\O_F = \Op(\Int G)$ and a morphism of network models $(G, \gn) \maps (C, F) \to (C', F' G')$ to the operad morphism $\O_G = \Op(\widehat\Gamma)$ is a functor
    \[
        \O \maps \NetMod \to \Opd.
    \]
\end{thm}

\begin{proof}
    There is a functor 
    \[
        \textstyle{\Int} \maps \NetMod \to \Sym\Mon\Cat 
    \]
    given by restricting the strict symmetric monoidal Grothendieck construction of \cref{thm:mainthm} to $\NetMod$. Composing this with the functor
    \[
        \Op \maps \Sym\Mon\Cat \to \Opd 
    \]
    constructed in \cref{prop:functoriality_of_operads_from_ssmcs} we obtain a functor $\O \maps \NetMod \to \Opd$ with the properties stated in the theorem. Since these properties specify how $\O$ acts on objects and morphisms, it is unique.
\end{proof}

\subsection{Algebras of network operads}
\label{sec:algebras}

Our interest in network operads comes from their use in designing and tasking networks of mobile agents. The operations in a network operad are ways of assembling larger networks of a given kind from smaller ones. To describe how these operations act in a concrete situation we need to specify an algebra of the operad. The flexibility of this approach to system design takes advantage of the fact that a single operad can have many different algebras, related by homomorphisms.  

An algebra $A$ of a typed operad $O$ specifies a set $A(t)$ for each type $t \in T$ such that the operations of $O$ can be applied to act on these sets. That is, each algebra $A$ specifies: 

\begin{itemize}
    \item for each type $t \in T$, a set $A(t)$, and
    \item for any types $t_1, \dots, t_n, t \in T$, a function
    \[
        \alpha \maps O(t_1, \dots, t_n;t) \to \hom(A(t_1) \times \cdots \times A(t_n), A(t)) 
    \] 
\end{itemize}
obeying some rules that generalize those for the action of a monoid on a set \cite{Yau}. All the examples in this section are algebras of network operads constructed from one-colored network models $F \maps \S \to \Mon$. This allows us to use \cref{thm:one-colored_network_operads}, which describes $\O_F$ explicitly.

The most basic algebra of such a network operad $\O_F$ is its `canonical algebra', where it acts on the kind of network described by the network model $F$:

\begin{expl}[\textbf{The canonical algebra}]
\label{ex:canonical_algebra}
    Let $F \maps \S \to \Mon$ be a one-colored network model. Then the operad $\O_F$ has a \define{canonical algebra} $A_F$ with
    \[
        A_F(n) = F(n) 
    \]
    for each $n \in N$, the type set of $\O_F$.
    In this algebra any operation
    \[
        (\sigma, g) \in  \O_F(n_1, \dots , n_k; n) = S_n \times F(n) 
    \] 
    acts on a $k$-tuple of elements
    \[
        h_i \in A_F(n_i) = F(n_i)   \qquad (1 \le i \le k)
    \]
    to give
    \[
        \alpha(\sigma, g)(h_1, \dots, h_k) = g \cup \sigma(h_1 \sqcup \cdots \sqcup h_k) \in A(n) .
    \]
    Here we use \cref{thm:equations}, which gives us the ability to overlay networks using the monoid structure $\cup \maps F(n) \times F(n) \to F(n)$, take their `disjoint union' using maps $\sqcup \maps F(m) \times F(m') \to F(m + m')$, and act on $F(n)$ by elements of $S_n$. Using the equations listed in this theorem one can check that $\alpha$ obeys the axioms of an operad algebra.
\end{expl}

When we want to work with networks that have more properties than those captured by a given network model, we can equip elements of the canonical algebra with extra attributes. Three typical kinds of network attributes are vertex attributes, edge attributes, and `global network' attributes. For our present purposes, we focus on vertex attributes. Vertex attributes can capture internal properties (or states) of agents in a network such as their locations, capabilities, performance characteristics, etc.

\begin{expl}[\textbf{Independent vertex attributes}]
\label{ex:vertex_attribute_algebra}
    For any one-colored network model $F\maps \S \to \Mon$ and any set $X$, we can form an algebra $A_X$ of the operad $\O_F$ that consists of networks whose vertices have attributes taking values in $X$. To do this, we define 
    \[
        A_X(n) = F(n) \times X^n  .
    \]
    In this algebra, any operation
    \[
        (\sigma, g) \in  \O_F(n_1, \dots , n_k; n) = S_n \times F(n) 
    \] 
    acts on a $k$-tuple of elements
    \[
        (h_i, x_i) \in F(n_i) \times X^{n_i} \qquad (1 \le i \le k) 
    \]
    to give
    \[
        \alpha_X(\sigma, g) = (g \cup \sigma(h_1 \sqcup \cdots \sqcup h_k), \sigma(x_1, \dots, x_k)). 
    \]
    Here $(x_1, \dots, x_k) \in X^n$
    is defined using the canonical bijection
    \[
        X^n \cong \prod_{i=1}^k X^{n_i} 
    \]
    when $n_1 + \cdots + n_k = n$, and $\sigma \in S_n$ acts on $X^n$ by permutation of coordinates. In other words, $\alpha_X$ acts via $\alpha$ on the $F(n_i)$ factors while permuting the vertex attributes $X^n$ in the same way that the vertices of the network $h_1 \sqcup \cdots \sqcup h_k$ are permuted.
    
    One can easily check that the projections 
    $F(n) \times X^n \to F(n)$ 
    define a homomorphism of $\O_F$-algebras, which we call
    \[
        \pi_X \maps A_X \to A  .
    \]
    This homomorphism `forgets the vertex attributes' taking values in the set $X$.
\end{expl}

\begin{expl}[\textbf{Simple networks with a rule obeyed by edges}]
\label{ex:edge_exception_algebra}
    Let $\O_\SG$ be the simple network operad as explained in \cref{ex:simple_network_operad}. We can form an algebra of the operad $\O_\SG$ that consists of simple graphs whose vertices have attributes taking values in some set $X$, but where an edge is permitted between two vertices only if their attributes obey some condition. We specify this condition using a symmetric function 
    \[
        p \maps X\times X \to \Boole 
    \] 
    where $\Boole = \{F, T\}$. An edge is not permitted between vertices with attributes $(x_1, x_2) \in X \times X$ if this function evaluates to $F$.
    
    To define this algebra, which we call $A_p$, we let $A_p(n) \subseteq \SG(n) \times X^n$ be the set of pairs $(g, x)$ such that for all edges $\{i, j\} \in g$ the attributes of the vertices $i$ and $j$ make $p$ true:
    \[
        p(x(i), x(j)) = T .
    \]
    There is a function 
    \[
        \tau_p \maps A_X(n) \to A_p(n) 
    \]
    that discards edges $\{i, j\}$ for which $p(x(i), x(j)) = F$. Recall that $A_X(n) = \SG(n) \times X^n$, and recall from \cref{ex:simple_graph_2} that we can regard $\SG(n)$ as the set of functions $g \maps \E(n) \to \Boole$. Then we define $\tau_p$ by
    \[ 
        \tau_p(g, x) = (\overline{g}, x) 
    \]
    where
    \[
        \overline{g}\{i, j\} = 
        \left\{  \begin{array}{ccl} g\{i, j\} & \textrm{if} & p(x(i), x(j)) = T \\
        F & \textrm{if} & p(x(i), x(j)) = F.
        \end{array} \right.
    \]
    We can define an action $\alpha_p$ of $\O_\SG$ on the sets $A_p(n)$ with the help of this function. Namely, we take $\alpha_p$ to be the composite
    \[
    \begin{tikzcd}
        \O_\SG(n_1, \dots, n_k ; n) \times A_p(n_1) \times \cdots \times A_p(n_k)  \arrow[d, hookrightarrow]
        \\
        \O_\SG(n_1, \dots, n_k ; n) \times A_X(n_1) \times \cdots \times A_X(n_k)  \arrow[d, "\alpha_X"]
        \\
        A_X(n)\arrow[d, "\tau_p"]
        \\
        A_p(n)
    \end{tikzcd}\]
    where the action $\alpha_X$ was defined in \cref{ex:vertex_attribute_algebra}. One can check that $\alpha_p$ makes the sets $A_p(n)$ into an algebra of $\O_\SG$, which we call $A_p$. One can further check that the maps $\tau$ define a homomorphism of $\O_\SG$-algebras, which we call
    \[
        \tau_p \maps A_X \to A_p  .
    \]
\end{expl}

\begin{expl}[\textbf{Range-limited networks}]
\label{ex:range_limit_algebra}
    We can use the previous examples to model range-limited communications between entities in a plane. First, let $X = \R^2$ and form the algebra $A_X$ of the simple network operad $\O_\SG$. Elements of $A_X(n)$ are simple graphs with vertices in the plane. 
    
    Then, choose a real number $L \ge 0$ and let $d$ be the usual Euclidean distance function on the plane. Define $p \maps X \times X \to \Boole$ by setting $p(x, y)=T$ if $d(x, y) \le L$ and $p(x, y) = F$ otherwise. Elements of $A_p(n)$ are simple graphs with vertices in the plane such that no edge has length greater than $L$.
\end{expl}
   
\begin{expl}[\textbf{Networks with edge count limits}]
\label{ex:edge_count_algebra}
    Recall the network model for multigraphs $\MGplus$, defined in \cref{ex:multigraph} and clarified in \cref{ex:multigraph_2}. An element of $\MGplus(n)$ is a multigraph on the set $\n$, namely a function $g \maps \E(n) \to \N$ where $\E(n)$ is the set of 2-element subsets of $\n$.
    If we fix a set $X$ we obtain an algebra $A_X$ of $\O_{\MGplus}$ as in \cref{ex:vertex_attribute_algebra}. The set $A_X(n)$ consists of multigraphs on $\n$ where the vertices have attributes taking values in $X$. 
    
    Starting from $A_X$ we can form another algebra  where there is an upper bound on how many edges are allowed between two vertices, depending on their attributes. We specify this upper bound using a symmetric function
    \[b \maps X \times X \to \N. \]
    
    To define this algebra, which we call $A_b$, let
    $A_b(n) \subseteq \MGplus(n) \times X^n$ be the set of pairs $(g, x)$ such that for all $\{i, j\} \in \E(n)$ we have 
    \[g(i, j) \le b(x(i), x(j)) .\]
    Much as in \cref{ex:edge_exception_algebra} there is function 
    \[\pi \maps A_X(n) \to A_b(n) \]
    that enforces this upper bound: for each $g \in A_X(n)$ its image  $\pi(g)$ is obtained by reducing the number of edges between vertices $i$ and $j$ to the minimum of $g(i, j)$ and $\beta(i, j)$:
    \[\pi(g)(i, j) = g(i, j) \min \beta(i, j) .\]
    We can define an action $\alpha_b$ of $\O_\MG$ on the sets $A_b(n)$ as follows:
    \[\begin{tikzcd}
    \O_\MG(n_1, \dots, n_k ; n) \times A_p(n_1) \times \cdots \times A_p(n_k)  \arrow[d, hookrightarrow]
    \\
    \O_\MG(n_1, \dots, n_k ; n) \times A_X(n_1) \times \cdots \times A_X(n_k)  \arrow[d, "\alpha_X"]
    \\
    A_X(n)\arrow[d, "\pi"]
    \\
    A_p(n)
    \end{tikzcd}\]
    One can check that $\alpha_b$ indeed makes the sets $A_b(n)$ into an algebra of $\O_\MGplus$, which we call $A_b$, and that the maps $\pi_p$ define a homomorphism of $\O_\MGplus$-algebras, which we call
    \[\pi_p \maps A_X \to A_b  .\]
\end{expl}
    
\begin{expl}[\textbf{Range-limited networks, revisited}]
\label{ex:range_limit_algebra_2}
    We can use \cref{ex:edge_count_algebra} to model entities in the plane that have two types of communication channel, one of which has range $L_1$ and another of which has a lesser range $L_2 < L_1$. To do this, take $X = \R^2$ and define $b \maps X \times X \to \N$ by 
    \[b(x, y)= \left\{ \begin{array}{cl}
    0 & \textrm{if } d(x, y) > L_1 \\
    1 & \textrm{if } L_2 < d(x, y) \le L_1  \\
    2 & \textrm{if } d(x, y) \le L_2 
    \end{array}  \right.
    \]
    Elements of $A_b(n)$ are multigraphs with vertices in the plane having no edges between vertices whose distance is $> L_1$, at most one edge between vertices whose distance is $\le L_1$ but $> L_2$, and at most two edges between vertices whose distance is $\le L_2$. 
    
    Moreover, the attentive reader may notice that the action $\alpha_b$ of $\O_\MGplus$ for this specific choice of $b$ factors through an action of $\O_{\Gamma_{\Boole_2}}$, where $\Gamma_{\Boole_2}$ is the network model defined in \cref{ex:multigraph_with_at_most_k}. That is, operations $\O_{\Gamma_{\Boole_2}}(n_1, \dots , n_k; n) = S_n \times \Gamma_{\Boole_2}(n)$ where
    $\Gamma_{\Boole_2}(n)$ is the set of multigraphs on $\n$ \emph{with at most $2$ edges between vertices} are sufficient to compose these range-limited networks. This is due to the fact that the values of this $b \maps X \times X \to \N$ are at most 2. 
\end{expl}

These examples indicate that vertex attributes and constraints can be systematically added to the canonical algebra to build more interesting algebras, which are related by homomorphisms. \cref{ex:vertex_attribute_algebra} illustrates how adding extra attributes to the networks in some algebra $A$ can give networks that are elements of an algebra $A'$ equipped with a homomorphism $\pi \maps A' \to A$ that forgets these extra attributes. \cref{ex:edge_count_algebra} illustrates how imposing extra constraints on the networks in some algebra $A$ can give an algebra $A'$ equipped with a homomorphism $\tau \maps A \to A'$ that imposes these constraints: this works only if there is a well-behaved systematic procedure, defined by $\tau$, for imposing the constraints on any element of $A$ to get an element of $A'$.

The examples given so far scarcely begin to illustrate the rich possibilities of network operads and their algebras. In particular, it is worth noting that all the specific examples of network models described here involve commutative monoids. However, noncommutative monoids are also important. Suppose, for example, that we wish to model entities with a limited number of point-to-point communication interfaces---e.g. devices with a finite number $p$ of USB ports. More formally, we wish to act on sets of degree-limited networks $A_{\rm deg} (n)\subset \SG(n) \times \N^n$ made up  of pairs $(g, p)$ such that the degree of each vertex $i$, ${\rm deg}(i), $ is at most the degree-limiting attribute of $i$: ${\rm deg}(i) \le p(i)$. Na\"ively, we might attempt to construct a map $\tau_{\rm deg} \maps A_\N \to A_{\rm deg}$ as in \cref{ex:edge_count_algebra} to obtain an action of the simple network operad $\O_\SG$. However, this is turns out to be impossible. For example, if attempt to build a network from devices with a single USB port, and we attempt to connect multiple USB cables to one of these devices, the relevant network operad must include a rule saying which attempts, if any, are successful. Since we cannot prioritize links from some vertices over others---which would break the symmetry built into any network model---the order in which these attempts are made must be relevant. Since the monoids $\SG(n)$ are commutative, they cannot capture this feature of the situation.

The solution is to use a class of noncommutative monoids dubbed `graphic monoids' by Lawvere \cite{GeneralizedGraphs}: namely, those that obey the identity $aba = ab$. These allow us to construct a one-colored network model $\Gamma \maps \S \to \Mon$ whose network operad $\O_\Gamma$ acts on $A_{\rm deg}$. For our USB device example, the relation $aba = ab$ means that first attempting to connect some USB cables between some devices ($a$), second attempting to connect some further USB cables ($b$), and third attempting to connect some USB cables \emph{precisely as attempted in the first step} ($a$, again) has the same result as only performing the first two steps ($ab$). We explore more applications of noncommutativity in network models in \cref{ch:NNM}.

}{\ssp
\setcounter{chapter}{2}
\chapter{Noncommutative Network Models}
\label{ch:NNM}

\section{Introduction}
\label{sec:NNMIntro}

In \cref{thm:graph_model}, we gave a functorial construction of a network model from a monoid, which we call the \emph{ordinary network model for weighted graphs}. In this chapter, we provide a different construction in order to realize a larger class of networks as algebras of network operads, which we call the \emph{free varietal network model for weighted graphs}. In Section \ref{sec:commitment}, we give an example of a family of networks which cannot form an algebra for any ordinary network model for weighted graphs, but does for a varietal one.  In this chapter, we give a construction for the free network model on a given monoid. This describes networks which look like the given monoid when you restrict to looking at the combinatorial behavior at a single pair of nodes. In \cref{sec:funnetmod}, we give a concrete construction of a left adjoint to the functor which evaluates a network model at its second level. This requires a categorical treatment and generalization of Green's theory of products of groups indexed by a graph, (i.e.\ \emph{graph products of groups}) \cite{Green}, which we give in \cref{sec:graphs}. 

This construction is designed to model networks which carry information on the edges. For example, with $\N$ a monoid under addition, $\Gamma_\N$ is a network model for loopless undirected multigraphs where overlaying is given by adding the number of edges. A similar example is $\Gamma_\Boole = \SG$. There is a monoid homomorphism $\N \to \Boole$ which sends all but $0$ to $T$. This induces a map of network models $\Gamma_\N \to \Gamma_\Boole$. Essentially this map reduces the information of a graph from the number of connections between each pair of vertices to just the existence of any connection.

\begin{expl}[\bf Algebra for range-limited communication]
    Consider a communication network where each node represents a boat and an edge between two nodes represents a working communication channel between the corresponding boats. Some forms of communication are restricted by the distance between those communicating. Assume that there is a known maximal distance over which our boats can communicate. Networks of this sort form an algebra of the simple graphs operad in the following way.
    
    Let $(X,d)$ be a metric space, and $0 \leq L \in \R$. Our boats will be located at points in this space. The operad $\O_\SG$ has an algebra $(A_{d,L}, \alpha)$ defined as follows. The set $A_{d,L}(\n)$ is the set of pairs $(h,f)$ where $h \in \SG(\n)$ is a simple graph and $f \maps \n \to X$ is a function such that if $\{v_1,v_2\}$ is an edge in $g$ then $d(f(v_1),f(v_2)) \leq L$. The number $L$ represents the maximal distance over which the boat's communication channels operate. Notice that this condition does not demand that all connections within range must be made. An operation $(\sigma, g)\in \O_\SG(\n_1, \dots, \n_k; \n)$ acts on a $k$-tuple $(h_i, f_i) \in A_{d,L}(\n_i)$ by 
    \[
        \alpha(\sigma, g)((h_1, f_1), \dots, (h_k, f_k)) = (g \cup \sigma(h_1 \sqcup \dots \sqcup h_k), f_1 \sqcup \dots \sqcup f_k).
    \] 
    Elements of this algebra are simple graphs in the space $X$ with an upper limit on edge lengths. When an operation acts on one of these, it tries to put new edges into the graph, but fails to when the range limit is exceeded \cite{NetworkModels}.
\end{expl}

A characteristic of the construction given in Theorem \ref{thm:graph_model} is that elements of the resulting monoids that correspond to different edges automatically commute with each other. 
For example, for a monoid $M$, the fourth constituent monoid of the ordinary $M$ network model is $\Gamma_M(4) = M^6$.
Then the element $(m_1, 0, 0, 0, 0, 0)$ represents a graph with one edge with weight $m_1 \in M$, the element $(0, m_2, 0, 0, 0, 0)$ represents a graph with a different edge with weight $m_2 \in M$, and 
\begin{align*}
    (m_1, 0, 0, 0, 0, 0) \cup (0, m_2, 0, 0, 0, 0) 
    &= (m_1, m_2, 0, 0, 0, 0) 
    \\&= (0, m_2, 0, 0, 0, 0) \cup (m_1, 0, 0, 0, 0, 0).
\end{align*}

This commutativity between edges means that networks given by ordinary network models cannot record information about the order in which edges were added to it. 
The ability to record such information about a network is desirable, for example, if one wishes to model networks which have a limit on the number of connections each agent can make to other agents. 

The \define{degree} of a vertex in a simple graph is the number of edges which include that vertex. The \define{degree} of a graph is the maximum degree of its vertices. A graph is said to have \define{degree bounded by $k$}, or simply \define{bounded degree}, if the degree of each vertex is less than or equal to $k$. Let $B_k(\n)$ denote the set of networks with $\n$ vertices and degree bound $k$. One might guess that the family of such networks could form an algebra for the simple graphs operad. 

\begin{question}
    Does the collection of networks of bounded degree form an algebra of a network operad? If so, is there such an algebra which is useful in applications?
\end{question}

Specifically, can networks of bounded degree form an algebra of $\O_\SG$, the simple graph operad? Setting two graphs next to each other will not change the degree of any of the vertices. Overlaying them almost definitely will, which makes defining an action of $\SG(\n)$ on $B_k(\n)$ less obvious. 

Ordinary network models are not sufficient to model this type of network because the graph monoids it produced could not remember the order that edges were added into a network. Even if $M$ is a noncommutative monoid, since $\Gamma_M$ is a product of several copies of $M$, one for each pair of vertices, it cannot distinguish the order that two different edges touching $v_1$ were added to a network if their other endpoints are different. 

Instead of taking the product of $\binom{n}{2}$ copies of $M$, we consider taking the coproduct, so as not to impose any commutativity relations between the edges. Since the lax structure map $\sqcup \maps F(\m) \times F(\n) \to F(\m+\n)$ associated to a network model $F \maps \S \to \Mon$ must be a monoid homomorphism, then \[(a \sqcup b) \cup (c \sqcup d) = (a \cup c) \sqcup (b \cup d).\] In particular, if we let $\emptyset$ denote the the identity of $F(\n)$ for any $\n$, then 
\begin{align*}
    (a \sqcup \emptyset) \cup (\emptyset \sqcup b) 
    &= (a \cup \emptyset) \sqcup (\emptyset \cup b)
    \\&= (\emptyset \cup a) \sqcup (b \cup \emptyset)
    \\&= (\emptyset \sqcup b) \cup (a \sqcup \emptyset).
\end{align*}
This is reminiscent of the Eckmann--Hilton argument (see \cref{app:monoidalcats}), but notice that the domains of the operations $\cup$ and $\sqcup$ are not the same. This equation says that elements which correspond to disjoint edges must commute with each other. Simply taking the coproduct of $\binom{n}{2}$ copies of $M$ cannot give the constituent monoids of a network model.

For a collection of monoids $\{M_i\}_{i \in I}$, elements of the product monoid which come from different components always commute with each other. In the coproduct, they never do. A \emph{graph product} (in the sense of Green \cite{Green}) of such a collection allows one to impose commutativity between certain components and not others by indicating such relations via a simple graph. The calculation above shows that the constituent monoids of a network model must satisfy certain partial commutativity relations. We use graph products to construct a family of monoids with the right amount of commutativity to both answer the question above and satisfy the conditions of being a network model. The following theorems are proven in Section \ref{sec:funnetmod}.

\begin{thm}
    The functor $\NetMod \to \Mon$ defined by $F \mapsto F(\2)$ has a left adjoint \[\Gamma_{-,\Mon} \maps \Mon \to \NetMod.\]
\end{thm}

The fact that this construction is a left adjoint tells us that the network models constructed are ones in which the only relations that hold are those that follow from the defining axioms of network models. 

A \emph{variety} of monoids is the class of all monoids satisfying a given set of identities. For example, $\Mon$ has subcategories $\CMon$ of commutative monoids and $\GMon$ of graphic monoids which are varieties of monoids satisfying the equations \[ab=ba \quad\text{and}\quad aba=ab\] respectively. Given a variety of monoids $\V$, let $\NetMod_\V$ be the subcategory of $\NetMod$ consisting of $\V$-valued network models. We recreate graph products in varieties of monoids to obtain a more general result.

\begin{thm}
    The functor $\NetMod_\V \to \V$ defined by $F \mapsto F(\2)$ has a left adjoint $\Gamma_{-,\V} \maps \V \to \NetMod_\V$.
\end{thm}

In particular, if $\V = \CMon$, since products and coproducts are the same in $\CMon$, the ordinary $M$ network model and the $\CMon$ varietal $M$ network model are also the same. Note that this does not indicate that $\Gamma_{-,\V}$ is a complete generalization of $\Gamma_-$ from Theorem \ref{thm:graph_model}, since $\Gamma_M$ is not an example of $\Gamma_{-,\V}$ when $M$ is not commutative. 

The ordinary construction for a network model given a monoid $M$ has constituent monoids given by finite cartesian powers of $M$. To include the networks described in the question above into the theory of network models, we must construct a network model from a given monoid which does not impose as much commutativity as the ordinary construction does, specifically among elements corresponding to different edges. The first attempt at a solution is to use coproducts instead of products. However, in this section we saw that we cannot create the constituent monoids of a network model simply by taking them to be coproducts of $M$ instead of products. There must be some commutativity between different edges, specifically between edges which do not share a vertex. 

Given a monoid $M$, we want to create a family of monoids indexed by $\N$, the $n$th of which looks like a copy of $M$ for each edge in the complete graph on $\n$, has minimal commutativity relations between these edge components, but does have commutativity relations between disjoint edges. Partial commutativity like this can be described with Green's graph products, which we describe in Section \ref{sec:graphprod}. The type of graph which describes disjointness of edges in a graph as we need is called a \emph{Kneser graph}, which we describe in Section \ref{sec:kneser}. Besides concerning ourselves with relations between edge components, sometimes we also want the constituent monoids in a network model to obey certain relations which $M$ obeys. In Section \ref{sec:varmon} we describe \emph{varieties of monoids} and a construction which produces monoids in a chosen variety. In Section \ref{sec:funnetmod} we prove this construction is functorial, and in Section \ref{sec:commitment} we use this construction to give a positive answer to the question.

\section{Graph Products}
\label{sec:graphs}

This section is dedicated to constructing the constituent monoids for the network models we want. In this section there are two different ways that graphs are being used. It is important that the reader does not get these confused. One way is the graphs which are elements of the constituent monoids of the network models we are constructing. The other way we use graphs is to index the \emph{Green product} (which we define in \cref{sec:graphprod}) to describe commutativity relations in the constituent monoids of the network models we are constructing. 

A network model is essentially a family of monoids with properties similar to the simple graphs example, so we think of the elements of these monoids as graphs, and we think of the operation as overlaying the graphs. 
These monoids have partial commutativity relations they must satisfy, as we see in \cref{sec:netmod}.
The graphs we use in the Green product, the Kneser graphs, are there to describe the partial commutativity in the constituent monoids. 

\subsection{Green Products}
\label{sec:graphprod}

Given a family of monoids $\{M_v\}_{v\in V}$ indexed by a set $V$, there are two obvious ways to combine them to get a new monoid, the product and the coproduct.
From an algebraic perspective, a significant difference between these two is whether or not elements that came from different components commute with each other. In the product they do. In the coproduct they do not. 
\emph{Green products}, or commonly \emph{graph products}, of groups were introduced in 1990 by Green \cite{Green}, and later generalized to monoids by Veloso da Costa \cite{Veloso}.
The idea provides something of a sliding scale of relative commutativity between components. We follow \cite{Fountain} in the following definitions.

By a \define{simple graph} $G=(V, E)$, we mean a set $V$ which we call the set of vertices, and a set $E \subseteq \binom{V}{2}$, which we call the set of edges. A \define{map of simple graphs} $f \maps (V, E) \to (V', E')$ is
a function $f \maps V \to V'$ such that if $\{u, v\} \in E$ then $\{f(u), f(v)\} \in E'$. 
Let $\sGrph$ denote the category of simple graphs and maps of simple graphs.

For a set $V$, a family of monoids $\{M_v\}_{v\in V}$, and a simple graph $G = (V, E)$, the \define{$G$ Green product} (or simply \define{Green product} when unambiguous) of $\{M_v\}_{v\in V}$, denoted $G(M_v)$, is 
\[
    G(M_v) = \left( \coprod_{v\in V} M_v \right) /R_G
\]
where $R_G$ is the congruence generated by the relation 
\[
    \{ (m n, n m) |\, m \in M_v, n\in M_u, u, v \text{ are adjacent in }G \}
\]
where the operation in the free product is denoted by concatenation. 
If $G$ is the complete graph on $n$ vertices, then $G(M_v) \cong \prod M_v$. If $G$ is the $n$-vertex graph with no edges, then $G(M_v) \cong \coprod M_v$. 

We call each $M_v$ a \define{component} of the Green product. Elements of $G(M_v)$ are written as \define{expressions} as in the free product, $m^{v_1}_1\dots m^{v_k}_k \in G(M_v)$ where the superscript indicates that $m_i \in M_{v_i}$. We often consider Green products of several copies of the same monoid, so this notation allows one to distguish elements coming from different components of the product, even if they happen to come from the same monoid. The intention and result of the imposed relations is that for an expression $m^{v_1}_1\dots m^{v_k}_k$ of an element, if there is an $i$ such that $\{v_i, v_{i+1}\} \in E$, then we can rewrite the expression by replacing $m^{v_i}_i m^{v_{i+1}}_{i+1}$ with $m^{v_{i+1}}_{i+1} m^{v_i}_i$. This move is called a \define{shuffle}, and two expressions are called \define{shuffle equivalent} if one can be obtained from the other by a sequence of shuffles. An expression $m^{v_1}_1\dots m^{v_k}_k$ is \define{reduced} if whenever $i<j$ and $v_i = v_j$, there exists $l$ with $i<l<j$ and $\{v_i, v_l\} \notin E$. If two reduced expressions are shuffle equivalent, they are clearly expressions of the same element. The converse is also true.

\begin{thm}[\cite{Fountain}, Thm.\ 1.1]
\label{thm:shuffle}
    Every element of $M$ is represented by a reduced expression. Two reduced expressions represent the same element of $M$ if and only if they are shuffle equivalent.
\end{thm}

In this section, we use a categorical description of Green products to define a similar construction in a more general context. The relevant property of $\Mon$ that we need for this generalization is that $\Mon$ is a \emph{pointed category}.

Let $\C$ be a category. An object of $\C$ which is both initial and terminal is called a \define{zero object}. If $\C$ has such an object, $\C$ is called a \define{pointed category} \cite{HomotAlg}. For any two objects $A, B$ of a pointed category, there is a unique map $0 \maps A \to B$ which is the composite of the unique map from $A$ to the zero object, and the unique map from the zero object to $B$.
If $\C$ is a pointed category with finite products, then for two objects $A, B$ of $\C$, the objects admit canonical maps $A \to A \times B$.
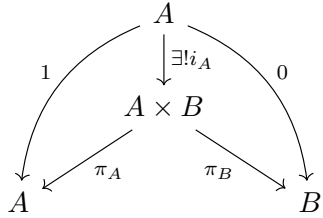
\[
\begin{tikzcd}
    &
    A
    \arrow[ddl, "1", bend right, swap]
    \arrow[d, "\exists!i_A"]
    \arrow[ddr, "0", bend left]
    \\&
    A \times B
    \arrow[dl, "\pi_A"]
    \arrow[dr, "\pi_B", swap]
    \\
    A
    &&
    B
\end{tikzcd}\]
So we have the following maps
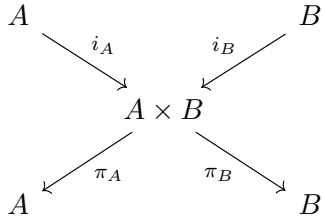
\[
\begin{tikzcd}
    A
    \arrow[dr, "i_A"]
    &&
    B
    \arrow[dl, "i_B", swap]
    \\&
    A \times B
    \arrow[dr, "\pi_B", swap]
    \arrow[dl, "\pi_A"]
    \\
    A
    &&
    B
\end{tikzcd}\]
satisfying the following properties.
\begin{align*}
    \pi_A i_A &= 1_A&
    \pi_B i_B &= 1_B\\
    \pi_B i_A &= 0&
    \pi_A i_B &= 0
\end{align*}
This is suggestive of a biproduct, but in a general pointed category $A \times B$ is not necessarily isomorphic to $A + B$.

In \cref{sec:funnetmod}, we use a generalized Green product to construct network models. 
A generalized Green product is a colimit of a diagram whose shape is derived from a given graph. We describe the shapes of the diagrams here with directed multi-graphs. We refer to them here as \emph{quivers} to help distinguish them from other variants of graphs and the role they play in this chapter. A \define {quiver} is a pair of sets $E$, $V$, respectively called the \emph{set of edges} and \emph{set of vertices}, and a pair of functions $s, t \maps E \to V$ assigning to each edge its \emph{starting} vertex and its \emph{terminating} vertex respectively. A \define{map of quivers} is a pair of functions 
\begin{displaymath}
\begin{tikzcd}
    E_1
    \arrow[d, shift right = 1, "s_1", swap]
    \arrow[d, shift left = 1, "t_1"]
    \arrow[r, "f_E"]
    &
    E_2
    \arrow[d, shift right = 1, "s_2", swap]
    \arrow[d, shift left = 1, "t_2"]
    \\
    V_1
    \arrow[r, "f_V", swap]
    &
    V_2
\end{tikzcd}
\end{displaymath}
such that the $s$-square and the $t$-square both commute.

We will use the word \define{cospan} to refer to the quiver with the following shape. \[\bullet \to \bullet \leftarrow \bullet\]
Define a functor $IC \maps \sGrph \to \Quiv$ which replaces every edge with a cospan ($IC$ stands for `insert cospan'). Specifically, given a simple graph $(V, E)$ where $E \subseteq \binom{V}{2}$, define the quiver $Q_1 \rightrightarrows Q_0$ where $Q_0 = V \sqcup E$ and $Q_1 = \{(v, e) \in V \times E |\, v \in e\}$, then define the source map $s \maps Q_1 \to Q_0$ by projection onto the first component, and the target map $t \maps Q_1 \to Q_0$ by projection onto the second component. For example, the simple graph 
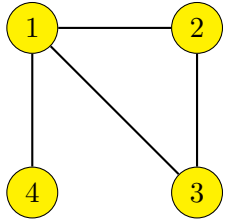
\[
\begin{tikzpicture}
    \node[style=species] (1) {$1$};
    \node[style=species] (2) [right = 1.5cm of 1]  {$2$};
    \node[style=species] (3) [below = 1.5cm and 1.5cm of 2] {$3$};
    \node[style=species] (4) [below = 1.5cm of 1] {$4$};

    \path[draw, thick]
    (1) edge node {} (2)
    (2) edge node {} (3)
    (3) edge node {} (1)
    (4) edge node {} (1);
\end{tikzpicture}\] 
gives the quiver 
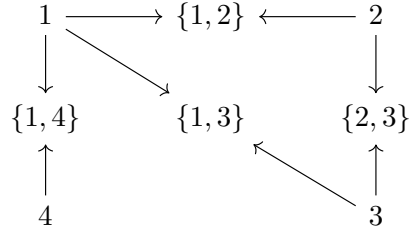
\[
\begin{tikzcd}
    1
    \arrow[r]
    \arrow[dr]
    \arrow[d]
    &
    \{1, 2\}
    &
    2
    \arrow[l]
    \arrow[d]
    \\
    \{1, 4\}
    &
    \{1, 3\}
    &
    \{2, 3\}
    \\
    4
    \arrow[u]
    &&
    3
    \arrow[u]
    \arrow[ul]
\end{tikzcd}\] 

Let $G=(V, E)$ and $G' = (V', E')$ be simple graphs, and $f \maps G \to G'$ a map of simple graphs. Define a map of quivers $ICf \maps IC(G) \to IC(G')$ by $ICf_0 = f_V \sqcup f_E$ and $ICf_1(v, e) = (f_V(v), f_E(e))$. 

\[
\begin{tikzcd}[row sep = large]
    IC(G)_1
    \arrow[d, "s_G", bend right, swap]
    \arrow[d, "t_G", bend left]
    \arrow[r, "ICf_1"]
    &
    IC(G')_1
    \arrow[d, "s_{G'}", bend right, swap]
    \arrow[d, "t_{G'}", bend left]
    \\
    IC(G)_0
    \arrow[r, "IC f_0", swap]
    &
    IC(G')_0
\end{tikzcd}\]
This construction gives a coproduct preserving functor $IC \maps \sGrph \to \Quiv$.

Let $F \maps \Quiv \to \Cat$ denote the free category (or path category) functor \cite{MacLane}. Since $F$ is a left adjoint, it preserves colimits. Notice that any quiver of the form $IC(G)$ would never have a path of length greater than 1. Thus the free path category on $IC(G)$ simply has identity morphisms adjoined. 

The objects in the category $F(IC(G))$ come from two places. There is an object for each vertex of $G$, and there is an object at the apex of the cospan for each edge in $G$. We call these two subsets of objects \define{vertex objects} and \define{edge objects}. We abuse notation and refer to the object given by the vertex $u$ by the same name, and similar for edge objects.

If $\{M_v\}_{v\in V}$ is a family of monoids indexed by the set $V$, that means that there is a functor $M \maps V \to \Mon$ from the set $V$ thought of as a discrete category. Notice that if $G$ is a simple graph with vertex set $V$, then the discrete category $V$ is a subcategory of $F(IC(G))$. 
We can then extend the functor $M$ to \[D \maps F(IC(G)) \to \Mon\] in the following way. Obviously we let $D(u) = M_u$ for a vertex object $u$. If $\{u, v\}$ is an edge in $G$, then $D(\{u, v\}) = M_u \times M_v$. The morphism $(u, \{u, v\})$ is sent to the canonical map $M_u \to M_u \times M_v$. For example, for a family of monoids $\{M_1, \dots M_4\}$, we have the following diagram.

\[
\begin{tikzcd}
    M_1
    \arrow[r]
    \arrow[dr]
    \arrow[d]
    &
    M_1 \times M_2
    &
    M_2
    \arrow[l]
    \arrow[d]
    \\
    M_1 \times M_4
    &
    M_1\times M_3
    &
    M_2\times M_3
    \\
    M_4
    \arrow[u]
    &&
    M_3
    \arrow[u]
    \arrow[ul]
\end{tikzcd}\] 
Since there are no non-trivial pairs of composable morphisms in categories of the form $F(IC(G))$, nothing further needs to be checked to confirm $D$ is a functor.

Despite the way we are denoting these products, we are not considering them to be ordered products. Alternatively, we could have used a more cumbersome notation that does not suggest any order on the factors.

\begin{thm}
    Let $V$ be a set, $\{M_v\}_{v \in V}$ be a family of monoids indexed by $V$, and $G = (V, E)$ be a simple graph with vertex set $V$. The $G$ Green product of $M_v$ is the colimit of the diagram $D \maps F(IC(G)) \to \Mon$ defined as above.
    \[
        G(M_v) \cong \colim D.
    \]
\end{thm}
\begin{proof}
    We show that $G(M_v)$ satisfies the necessary universal property. The vertex objects in the diagram have inclusion maps into the edge objects $i_{u, v} \maps M_u \to M_u \times M_v$, and all the objects have inclusion maps into $G(M_v)$, $j_u \maps M_u \to G(M_v)$ and $j_{u, v} \maps M_u \times M_v \to G(M_v)$ such that $j_{u, v} \circ i_{u, v} = j_u$. Note that due to the fact that we have unordered products for objects, there is some redundancy in our notation, namely $j_{u, v} = j_{v, u}$.
    If we have a monoid $Q$ and maps $f_u \maps M_u \to Q$ and $f_{u, v} \maps M_u \times M_v \to Q$ such that
    \begin{align*}
        f_{u, v} &= f_{v, u}\\
        f_{u, v} \circ i_{u, v} &= f_u, 
    \end{align*}
    then we define a map $\phi \maps G(M_v) \to Q$ by $\phi(m^{v_1}_1 \dots m^{v_k}_k) = f_{v_1}(m_1) \dots f_{v_k}(m_k)$. Since this map is defined via expressions of elements, \cref{thm:shuffle} tells us that to check this map is well-defined, we need only check that the values of two expressions that differ by a shuffle are the same. Let $m^{v_1}_1 \dots m^{v_k}_k$ be an expression, and $i$ such that $\{v_i, v_{i+1}\} \in E$. 
    \begin{align*}
        \phi (m^{v_i}_i m^{v_{i+1}}_{i+1})
        &= f_{v_i}(m_i) f_{v_{i+1}}(m_{i+1})
        \\&= f_{v_i, v_{i+1}} (m_i, m_{i+1})
        \\&= f_{v_{i+1}}(m_{i+1}) f_{v_i}(m_i)
        \\&= \phi (m^{v_{i+1}}_{i+1} m^{v_i}_i)
    \end{align*}
    It is clear that
    \[
        \phi (m^{v_1}_1 \dots m^{v_k}_k) = \phi (m^{v_1}_1 \dots m^{v_{i-1}}_{i-1}) \phi( m^{v_i}_im^{v_{i+1}}_{i+1}) \phi(m^{v_{i+2}}_{i+2} \dots m^{v_k}_k), 
    \] 
    so two shuffle equivalent expressions have the same value under $\phi$, and $\phi$ is well-defined. It is clearly a monoid homomorphism, and has the property $\phi \circ j_u = f_u$ and $\phi \circ j_{u, v} = f_{u, v}$. To show this map is unique, assume there is another such map $\psi \maps G(M_v) \to Q$. Since $\psi \circ j_u = f_u$, then $\psi(m_u) = f(u)$, and
    \begin{align*}
        \psi(m^{v_1}_1 \dots m^{v_k}_k) 
        &= \psi(m^{v_1}_1) \dots \psi(m^{v_k}_k)
        \\&= f_{v_1}(m_1) \dots f_{v_k}(m_k)
        \\&= \phi(m^{v_1}_1 \dots m^{v_k}_k).\qedhere
    \end{align*}
\end{proof}

This result makes it reasonable to generalize Green products in the following way.

\begin{defn}
    Let $\C$ be a pointed category with finite products and finite colimits, $V$ a set, $\{A_v\}_{v \in V}$ a family of objects of $\C$ indexed by $V$, and $G$ a simple graph with vertex set $V$. Let $D \maps F(IC(G)) \to \C$ be the diagram defined by $v \mapsto A_v$, $\{u, v\} \mapsto A_u \times A_v$, and the morphism $(u, \{u, v\})$ is mapped to the inclusion $A_u \to A_u \times A_v$ as above. The \define{$G$ Green product} of $\{A_v\}_{v \in V}$ is the colimit of $D$ in $\C$, 
    \[
        G^\C(A_v) = \colim D.
    \]
    If $\C=\Mon$, we denote the Green product simply as $G(A_v)$.
\end{defn}

In \cref{sec:funnetmod}, we use this general notion of graph products in \emph{varieties of monoids} to construct network models whose constituent monoids are in those varieties.
Note that since $F \circ IC$ is a functor, the group $\Aut(G)$ of graph automorphisms of $G$ naturally acts on $G^\C(A_v)$.

\subsection{Kneser Graphs}
\label{sec:kneser}

We focus here on a special family of simple graphs known as the \emph{Kneser graphs} \cite{Lovasz}. The \define{Kneser graph} $KG_{n,m}$ has vertex set $\binom{n}{m}$, the set of $m$-element subsets of an $n$-element set, and an edge between two vertices if they are disjoint subsets. Since a simple graph is defined as a collection of two-element subsets of an $n$-element set, the Kneser graph $KG_{n,2}$ has a vertex for each edge in the complete graph on $n$, and has an edge between every pair of vertices which correspond to disjoint edges. So the Kneser graph $\KG_{n,2}$ can be thought of as describing the disjointness of edges in the complete graph on $n$. For instance, the complete graph on $5$ is 
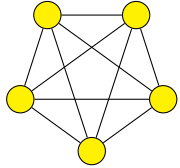
\[
\begin{tikzpicture} 
\begin{pgfonlayer}{nodelayer}
	\node [style=species] (a) at (0, -1) {};
	\node [style=species] (b) at (0.95, -0.31) {};
	\node [style=species] (c) at (0.59, 0.81) {};
	\node [style=species] (d) at (-0.59, 0.81) {};
	\node [style=species] (e) at (-0.95, -0.31) {};
\end{pgfonlayer}
\begin{pgfonlayer}{edgelayer}
	\draw (a) to (b);
	\draw (b) to (c);
	\draw (c) to (d);
	\draw (d) to (e);
	\draw (e) to (a);
	\draw (a) to (c);
	\draw (b) to (d);
	\draw (c) to (e);
	\draw (d) to (a);
	\draw (e) to (b);
\end{pgfonlayer}
\end{tikzpicture}\]
and the corresponding Kneser graph $\KG_{5,2}$ is the Petersen graph:
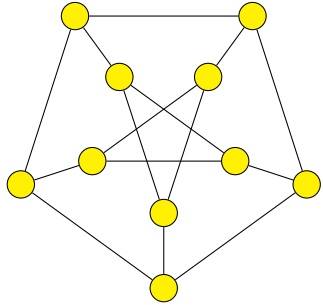
\[
\begin{tikzpicture} 
\begin{pgfonlayer}{nodelayer}
	\node [style=species] (a) at (0, -1) {};
	\node [style=species] (b) at (0.95, -0.31) {};
	\node [style=species] (c) at (0.59, 0.81) {};
	\node [style=species] (d) at (-0.59, 0.81) {};
	\node [style=species] (e) at (-0.95, -0.31) {};
	\node [style=species] (A) at (0, -2) {};
	\node [style=species] (B) at (1.9, -0.62) {};
	\node [style=species] (C) at (1.18, 1.62) {};
	\node [style=species] (D) at (-1.18, 1.62) {};
	\node [style=species] (E) at (-1.9, -0.62) {};
\end{pgfonlayer}
\begin{pgfonlayer}{edgelayer}
	\draw (a) to (A);
	\draw (b) to (B);
	\draw (c) to (C);
	\draw (d) to (D);
	\draw (e) to (E);
	\draw (A) to (B);
	\draw (B) to (C);
	\draw (C) to (D);
	\draw (D) to (E);
	\draw (E) to (A);
	\draw (a) to (c);
	\draw (b) to (d);
	\draw (c) to (e);
	\draw (d) to (a);
	\draw (e) to (b);
\end{pgfonlayer}
\end{tikzpicture}\]
For sets $X,Y$ and a function $f \maps X \to Y$, let $f[U] = \{f(x) |\, x \in U\}$ for $U \subseteq X$. Let $\FinInj$ denote the category of finite sets and injective functions.

\begin{lem}
    For $k \in \N$, there is a functor $\bink{-} \maps \FinInj \to \FinInj$ which sends $X$ to $\bink{X}$ the set of $k$-element subsets of $X$, and injections $f \maps X \to Y$ to the functions $\bink{f} \maps \bink{X} \to \bink{Y}$ defined by $\bink{f}(U) = f[U]$. 
\end{lem}

Note that this result holds for $\Inj$ the category of sets and injective functions, but we only require $\FinInj$ for our purposes.

\begin{proof}
    If $f \maps X \to Y$ is an injection, then $|f[U]| = |U|$ for $U \subseteq X$. It then makes sense to restrict the induced map on power sets to subsets of a fixed cardinality. The map $\bink{f} \maps \bink{m} \to \bink{n}$ defined by $\bink{f}(U) = f[U]$ is then well defined. If $f[U]=f[V]$ and $x \in U$, then $f(x) \in f[U] = f[V]$, which implies there is a $y \in V$ such that $f(y) = f(x)$. Since $f$ is injective, then $x = y \in V$. Thus $U = V$ by symmetry.
\end{proof}

Let $i_X$ and $i_Y$ denote the following inclusion maps.
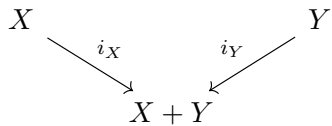
\[
\begin{tikzcd}
    X
    \arrow[dr,"i_X"]
    &&
    Y
    \arrow[dl,"i_Y",swap]
    \\&
    X+Y
\end{tikzcd}\]
Since these maps are injective, they induce maps $\bink{i_X}, \bink{i_Y}$, and we get a map $\Phi_{X,Y} \maps \bink{X}+\bink{Y} \to \bink{X+Y}$ by the universal property in the following way.
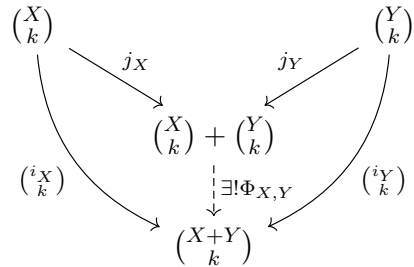
\[
\begin{tikzcd}
    \bink{X}
    \arrow[dr,"j_X"]
    \arrow[ddr,"\bink{i_X}",bend right,swap]
    &&
    \bink{Y}
    \arrow[dl,"j_Y",swap]
    \arrow[ddl,"\bink{i_Y}",bend left]
    \\&
    \bink{X}+\bink{Y}
    \arrow[d,"\exists!\Phi_{X,Y}",dashed]
    \\&
    \bink{X+Y}
\end{tikzcd}\]

\begin{lem}
\label{lem:chooselax}
    The functor $\bink{-}$ is made lax symmetric monoidal \[(\bink{-},\Phi, \phi) \maps (\FinInj,+, \emptyset) \to (\FinInj, +, \emptyset)\] where the components of $\Phi$ are defined as above.
\end{lem}
\begin{proof}
    The family of maps $\{\Phi_{X,Y}\}$ is clearly a natural transformation.
    There is no choice for the map $\phi \maps \emptyset \to \bink{\emptyset}$.
    The left and right unitor laws hold trivially. Checking the coherence conditions for the associator and the symmetry are straightforward computations.
\end{proof}

For $n,k \in \N$, the simple graph $KG_{n,k}$ has vertex set $V = \bink{n}$ and edge set $\{\{u,v\} \subseteq \binom{V}{2} |\, u \cap v = \emptyset\}$. 
If $f \maps m \to n$ is injective, then we get a map $\bink{f}$ between the vertex sets of $KG_{m,k}$ and $KG_{n,k}$. Let $\{u,v\} \in \binom{V}{2}$ be an edge in $KG_{m,k}$. Then $f[u] \cap f[v] = \emptyset$ by injectivity, so $\{f[u],f[v]\}$ is an edge of $KG_{n,k}$. An injection $f$ then induces a map of graphs, denoted $KG_{f,k} \maps KG_{m,k} \to KG_{n,k}$. 
Since $\bink{f}$ is injective, $KG_{f,k}$ is an embedding.
Nothing about this construction requires finiteness of the sets involved, but our applications only call for finite graphs. 

\begin{prop}
    For $k \in \N$, there is functor $KG_{-,k} \maps \FinInj \to \sGrph$ which sends $n$ to $KG_{n,k}$ and $f\maps m \to n$ to $KG_{f,k}$.
\end{prop}

Not only does $KG_{m,k}$ embed into $KG_{n,k}$ when $m<n$, but $KG_{m,k} + KG_{n,k}$ embeds into $KG_{m+n,k}$. 
We construct the embedding $KG_{m,k} + KG_{n,k} \to KG_{m+n,k}$ by using the lax structure map from \cref{lem:chooselax} for the vertex map, $\Phi_{m, n} \maps \bink{m} + \bink{n} \to \bink{m+n}$. Restricting this map to either $\bink{m}$ (resp. $\bink{n}$) gives the map $\bink{i_m}$ (resp. $\bink{i_n}$) which we already know induces a map of graphs. Thus $\Phi_{m, n}$ induces a map of graphs, which we call $\Psi_{m, n}$.

\begin{prop}
    The functor $KG_{-,k}$ is made lax (symmetric) monoidal \[(KG_{-,k},\Psi) \maps (\Inj,+) \to (\sGrph, +)\] where the components of $\Psi$ are defined as above.
\end{prop}
\begin{proof}
    All the necessary properties for $\Psi$ are inherited immediately from $\Phi$.
\end{proof}

Let $(L, \Lambda) \maps (\Inj,+) \to (\Cat,+)$ be the composite $L = F \circ IC \circ KG_{-,2}$ with the obvious laxator. Let $M$ be a monoid. Then from the construction given in the previous subsection, for each $n$ we get a diagram $D_n \maps L(n) \to \Mon$ which sends all vertex objects to $M$, all edge objects to $M \times M$, and all nontrivial morphisms to inclusions $M \to M \times M$. 
Taking the colimit of $D_n$ then gives the Green product $\KG_{n,2}(M)$. Note that we identify constituent monoids with the corresponding submonoid of the graph product when this can be done without confusion.

\begin{prop}
\label{prop:knesercomm}
    Let $M_{p,q}$ be a $\binom{m+n}{2}$ family of monoids, and $G_1$ and $G_2$ be graphs with $m$ and $n$ vertices respectively. Let $a_1 \in M_{p_1,q_1}$ with $p_1,q_1 \leq m$ and $a_2 \in M_{p_2,q_2}$ with $p_2,q_2>m$, and let $\overline a_1, \overline a_2$ be their values under the canonical inclusions $M_{p,q} \hookrightarrow (G_1 \sqcup G_2)(M_{p,q})$. Then $\overline a_1 \overline a_2 = \overline a_2 \overline a_1$ in $(G_1 \sqcup G_2)(M_{p,q})$.
\end{prop}

\begin{proof}
    By definition, there is an edge in the Kneser graph $KG_{m+n,2}$ between the vertices ${p_1,q_1}$ and ${p_2,q_2}$. This imposes the desired commutativity relation.
\end{proof}

\subsection{Varieties of Monoids}

A \define{finitary algebraic theory} or \define{Lawvere theory} is a category $T$ with finite products in which every object is isomorphic to a finite cartesian power $x^n = \prod^n x$ of a distinguished object $x$ \cite{LawvereThesis, ALR}. An \define{algebra} of a theory $T$, or \define{$T$-algebra}, is a product preserving functor $T \to \Set$. Let $T\Alg$ denote the category of $T$-algebras with natural transformations for morphisms. We are primarily concerned with monoids in this chapter. The theory of monoids $T_\Mon$ has morphisms $m \maps x \times x \to x$ and $e \maps x^0 \to x$, which makes the following diagrams commute.
\[
\begin{tikzcd}
    x^3
    \arrow[d, "m \times 1_x", swap]
    \arrow[r, "1_x \times m"]
    &
    x^2
    \arrow[d, "m"]
    &
    x \times x^0
    \arrow[r, "1_x \times e"]
    \arrow[dr, "\simeq", swap]
    &
    x^2
    \arrow[d, "m"]
    &
    x^0 \times x
    \arrow[l, "e \times 1_x", swap]
    \arrow[dl, "\simeq"]
    \\
    x^2
    \arrow[r, "m", swap]
    &
    x
    &&
    x
\end{tikzcd}\]

A \define{variety} of $T$-algebras is a full subcategory of $T\Alg$ which is closed under products, subobjects, and homomorphic images. Birkhoff's theorem implies that this is equivalent to the category $T'\Alg$ of algebras of another theory $T'$ which has the same morphisms, but satisfies more commutative diagrams \cite{universal}. For example, commutative monoids are given by algebras of the theory of commutative monoids $T_\CMon$, which has morphisms $m,e$ as in $T_\Mon$, satisfies the same commutative diagrams as $T_\Mon$, but also satisfies the following commutative diagram
\[
\begin{tikzcd}
    x^2
    \arrow[dr, "m", swap]
    \arrow[r, "b"]
    &
    x^2
    \arrow[d, "m"]
    \\&
    x
\end{tikzcd}\]
where $b: x^2 \to x^2$ is the braid isomorphism. We only use varieties of monoids in this chapter, so we give these ``extra'' conditions by equations, e.g.\ commutative monoids are those which satisfy the equation $ab=ba$ for all elements $a,b$. We call the extra equations the \emph{defining equations} of the variety.

A \define{graphic monoid} is a monoid which satisfies the \emph{graphic identity}: $aba=ab$ for all elements $a,b$. Graphic monoids are algebras of a theory $T_\GMon$. A semigroup obeying this relation is known as a \emph{left regular band} \cite{MSS}. The term \emph{graphic monoid} was introduced by Lawvere \cite{taco}. Let $M$ be a graphic monoid. If we let $b$ be the unit of $M$, then the graphic relation says that $a^2=a$. Every element of $M$ is idempotent. If $a,c \in M$, then $ca=c$ if $c$ already has $a$ as a factor. 

Graphic monoids are present when talking about types of information where a piece of information cannot contain the same piece of information twice. A simple example can be seen in the powerset of a given set $X$, given the structure of a monoid by union. Of course, this example is overly simple because the operation is commutative idempotent, which is stronger than graphic. A more interesting example can be seen by considering the following simple graph.
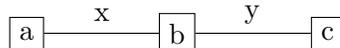
\[
\begin{tikzpicture}
	\begin{pgfonlayer}{nodelayer}
		\node [style=construct] (1) at (-2, 0) {a};
		\node [style=construct] (2) at (0, 0) {b};
		\node [style=construct] (3) at (2, 0) {c};
		\node [style=none] (4) at (-1, 0.25) {x};
		\node [style=none] (5) at (1, 0.25) {y};
	\end{pgfonlayer}
	\begin{pgfonlayer}{edgelayer}
		\draw (1) to (2);
		\draw (2) to (3);
	\end{pgfonlayer}
\end{tikzpicture}\]
We will define a monoid structure on the set $M = \{1, a, b, c, x, y\}$ in the following way. First, $1$ is a freely adjoined identity element. For $p,q \in M \setminus \{1\}$, define $pq$ as follows. Pick a generic point $f$ in $p$ and a generic point $g$ in $q$. Then move a small distance along a straight line path from $f$ to $g$. We define the product $pq$ to be the component of the graph you land in. Here are some example computations:
\begin{align*}
    &ab = x
    &aa = a
    \\
    &bc = y
    &xb = x
    \\
    &ac = x
    &ca = y
\end{align*}
The last two demonstrate that this monoid is not commutative. More complicated examples can be constructed by using the same idea for the operation, but applying it to different spaces.

The following fact is critical in \cref{ch:NetworkModels}. It follows immediately from the definitions.
\begin{lem}
\label{lem:varietiesarepointed}
    Every variety of monoids is a pointed category and has finite colimits.
\end{lem}

\subsection{Varietal Network Models}
\label{sec:varmon}

Our motivation for using graphic monoids is that we use the graphic relation to model ``commitment'' in the following way. Let $M$ be a graphic monoid, where we think of an element of $M$ as a task or list of tasks. If we first commit to doing task $x$, and then commit to doing task $y$, then we have the element $xy$ as our task list, indicating that we committed to $x$ before $y$. If we then try to commit to to doing $x$, the graphic relation saves us from recording this information twice. The relation also preserves the order in which we committed to $x$ and $y$: if $x$ is a task list of the form $x = ab$, and we have committed to $xy$, and then try to commit to $bc$, we get $(xy)(bc) = (aby)(bc) = a (byb) c = a (by) c = abyc = xyc$.

We want to construct a network model from a monoid in a variety $\V$ which has constituent monoids that are also in $\V$. If $M$ is a monoid in a variety $\V$, then each constituent monoid $\Gamma_M(n)$ is a product of several copies of $M$, and so is also in $\V$ by definition. Thus the ordinary network model (given in \cref{thm:graph_model}) restricted to a variety gives a functor $\V \to \NetMod_\V$, where $\NetMod_\V$ denotes the category of $\V$-valued network models.

The free product of two monoids is a monoid, $M+N$ an element of which is given by a list with entries in the set $M \sqcup N$ such that if two consecutive entries of a list are either both elements of $M$ or both elements of $N$, then the list is identified with the list that is the same everywhere except that those two entries are reduced to one entry occupied by their product. Note that the empty list is identified with both the singleton list consisting of the identity element of $M$, and the singleton list consisting of the identity element of $N$. Free products of monoids gives the coproduct in the category of monoids $\Mon$. Free products of monoids are very similar to free products of groups, which can be found in most books introducing group theory \cite{Hungerford}. 

If two monoids $M$ and $N$ are in a variety $\V$, taking their free product will not necessarily produce a monoid in $\V$, i.e.\ varieties are not necessarily closed under the coproduct of $\Mon$. It is easy to find an example demonstrating this. 
Consider $\IMon$, the variety of idempotent monoids, i.e.\ monoids satisfying the equation $x^2 = x$ for all elements $x$. The boolean monoid $\B$ is an object in $\IMon$. The free product of $\B$ with itself $\B+\B$ can be generated by elements $a$ and $b$ which correspond to the element $1$ in each copy of $\B$. The element $ab \in \B+\B$ is not idempotent, as $abab \neq ab$. However, every variety $\V$ does have coproducts. The coproduct in a variety of monoids is the quotient of the free product by the congruence relation generated by the variety's defining equations. In \cref{sec:funnetmod} we give a construction $\V \to \NetMod_\V$ which uses colimits in order to impose minimal relations.

\cref{lem:varietiesarepointed} tells us that it makes sense to talk about Green products in a variety, which we call \emph{varietal Green products}. In the next section, we use varietal Green products with Kneser graphs to construct network models.

\section{Free Network Models}
\label{sec:funnetmod}

In this section, we state and prove the main result of this chapter.
It says that given a monoid $M$ in a variety $\V$, we can construct a network model whose constituent monoids are also in $\V$, while avoiding to impose commutativity relations when possible. In the following section, we see how this construction resolves the dilemma presented in the question.

Let $M$ be a monoid in a variety $\V$. Define $\GMV(n)$ to be the $KG_{n,2}$ Green product of $\binom{n}{2}$ copies of $M$.

\begin{thm} 
\label{thm:main}
    For $\V$ a variety of monoids, $\Gamma_{-,\V} \maps \V \to \NetMod_\V$ is a functor, as given above. The network model $\GMV$ is called the \define{$\V$-varietal network model for $M$-weighted graphs}, or just the \define{varietal $M$ network model}.
\end{thm}

In order to prove this, we must first show that a monoid $M$ gives a network model, i.e.\ a lax symmetric monoidal functor. The laxator for $\GMV$ is canonically defined, but perhaps it is not as immediate as the one for the ordinary $M$ network model. We treat this first before returning to the proof of the main theorem.

Let $A$ and $B$ be objects in a pointed category with finite products and coproducts. Let $p_A \maps A \times B \to A$ and $p_B \maps A \times B \to B$ denote the canonical projections, and $i_A \maps A \to A+B$ and $i_B \maps B \to A + B$ the canonical inclusions.
The category $\CMon$ of commutative monoids is such a category. Recall that the operation of a monoid is a monoid homomorphism if and only if the monoid is commutative.
We have
\[
\begin{tikzcd}
    &
    A \times B
    \arrow[dl, "p_A", swap]
    \arrow[dr, "p_B"]
    \arrow[d, dashed]
    \\
    A
    \arrow[d, "i_A", swap]
    &
    (A + B) \times (A + B)
    \arrow[d,"\ast"]
    \arrow[dl]
    \arrow[dr]
    &
    B
    \arrow[d, "i_B"]
    \\
    A + B
    &
    A + B
    &
    A + B
\end{tikzcd}
\]
where $\ast$ denotes the operation in the commutative monoid $A+B$, and the dashed arrow is $<i_A p_A, i_B p_B>$ given by universal property.
The composite of the two maps going down the middle is the inverse to the canonical map $A + B \to A \times B$.
The operation in a noncommutative monoid is not a monoid homomorphism, but all the above maps still exist \emph{as functions}.
Recall that we let $\cup$ denote the operation in the monoids $\GMV(n)$. There is always a homomorphism $\phi_{m, n} \maps \GMV(m) + \GMV(n) \to \GMV(m+n)$ by universal property of coproducts. 
Let \[\gamma \maps (\GMV(m) + \GMV(n)) \times (\GMV(m) + \GMV(n)) \to \GMV(m) + \GMV(n)\] denote the monoid operation of the coproduct.

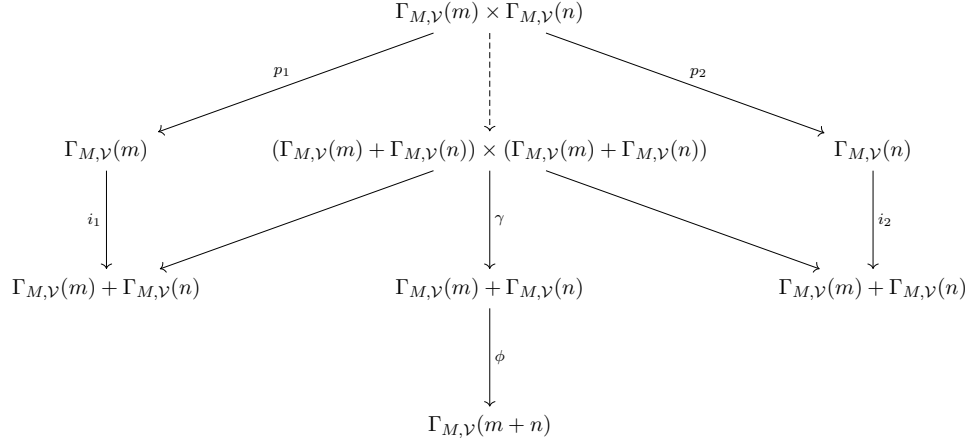
\[\adjustbox{scale = 0.75}{
\begin{math}\begin{tikzcd}[row sep = 50]
    &
    \GMV(m) \times \GMV(n)
    \arrow[dl, "p_1", swap]
    \arrow[dr, "p_2"]
    \arrow[d, dashed]
    \\
    \GMV(m)
    \arrow[d, "i_1", swap]
    &
    (\GMV(m)+\GMV(n))\times(\GMV(m)+\GMV(n))
    \arrow[d,"\gamma"]
    \arrow[dl]
    \arrow[dr]
    &
    \GMV(n)
    \arrow[d, "i_2"]
    \\
    \GMV(m)+\GMV(n)
    &
    \GMV(m)+\GMV(n)
    \arrow[d, "\phi"]
    &
    \GMV(m)+\GMV(n)
    \\&
    \GMV(m+n)
\end{tikzcd}\end{math}
}\]
The monoids $\GMV(n)$ are constructed specifically so that $\phi \circ \gamma \circ <i_1 \circ p_1, i_2 \circ p_2>$ is a monoid homomorphism despite the fact that $\gamma$ is not.

In the proof of the following theorem, we utilize a string diagrammatic calculus suited for reasoning in a symmetric monoidal category. We refer the reader to Selinger's thorough exposition of such string diagramatic languages and their use in category theory \cite{Selinger}.

\begin{lem}
    The function $\GMV(m) \times \GMV(n) \to \GMV(m + n)$ given by $\phi \circ (i_1 \circ p_1 \cup i_2 \circ p_2)$ is a monoid homomorphism. Moreover, the family of maps of this form gives a natural transformation, denoted $\sqcup$.
\end{lem}

\begin{proof}
    We have the following actors in play:
    \begin{itemize}
        \item the monoid operations $\cup_k \maps \GMV(\mathbf k)$ for $k = m, n, m+n$ (we leave off the subscripts below)
        \item the monoid operation of the coproduct \[\gamma \maps (\GMV(m) + \GMV(n)) \times (\GMV(m) + \GMV(n)) \to \GMV(m) + \GMV(n)\]
        \item the canonical inclusion maps $i_1 \maps \GMV(m) \to \GMV(m) \to \GMV(n)$ and $i_2 \maps \GMV(n) \to \GMV(m) \to \GMV(n)$
        \item the canonical map $\phi \maps \GMV(m) + \GMV(n) \to \GMV(m+n)$
    \end{itemize}
    We represent these string diagramatically (read from top to bottom) as follows. Note that these are digrams in $\Set$ with its cartesian monoidal structure, because the monoid operations $\cup_k$ and $\gamma$ are not necessarily monoid homomorphisms.
    \[
    \begin{tikzpicture}
    \begin{pgfonlayer}{nodelayer}
    	\node [style=construct] (U) at (0, 0) {$\cup$};
        \node [style=none] (1) at (-0.5, 1) {};
    	\node [style=none] (2) at (-0.5, 0.5) {};
    	\node [style=none] (3) at (0.5, 0.5) {};
    	\node [style=none] (4) at (0, -1) {};
    	\node [style=none] (5) at (0.5, 1) {};
        \node [style=none] () at (0.75, 0) {,};
        \node [style=none] () at (1, 0) {}; 
    \end{pgfonlayer}
    \begin{pgfonlayer}{edgelayer}
    	\draw [bend right] (2.center) to (U);
    	\draw [bend left] (3.center) to (U);
    	\draw [] (2.center) to (1);
    	\draw [] (3.center) to (5);
    	\draw (4.center) to (U);
    \end{pgfonlayer}
    \end{tikzpicture}
    \begin{tikzpicture}
    \begin{pgfonlayer}{nodelayer}
    	\node [style=construct] (g) at (0, 0) {$\gamma$};
        \node [style=none] (1) at (-0.5, 1) {};
    	\node [style=none] (2) at (-0.5, 0.5) {};
    	\node [style=none] (3) at (0.5, 0.5) {};
    	\node [style=none] (4) at (0, -1) {};
    	\node [style=none] (5) at (0.5, 1) {};
        \node [style=none] () at (0.75, 0) {,};
        \node [style=none] () at (1, 0) {}; 
    \end{pgfonlayer}
    \begin{pgfonlayer}{edgelayer}
    	\draw [bend right] (2.center) to (g);
    	\draw [bend left] (3.center) to (g);
    	\draw [] (2.center) to (1);
    	\draw [] (3.center) to (5);
    	\draw (4.center) to (g);
    \end{pgfonlayer}
    \end{tikzpicture}
    \begin{tikzpicture}
    \begin{pgfonlayer}{nodelayer}
    	\node [style=construct] (1) at (0, 0) {$i_1$};
    	\node [style=none] (2) at (0, 1) {};
    	\node [style=none] (3) at (0, -1) {};
        \node [style=none] (5) at (0.75, 0) {,};
        \node [style=none] (5) at (1, 0) {}; 
    \end{pgfonlayer}
    \begin{pgfonlayer}{edgelayer}
    	\draw (2.center) to (1);
    	\draw (3.center) to (1);
    \end{pgfonlayer}
    \end{tikzpicture}
    \begin{tikzpicture}
    \begin{pgfonlayer}{nodelayer}
    	\node [style=construct] (1) at (0, 0) {$i_2$};
    	\node [style=none] (2) at (0, 1) {};
    	\node [style=none] (3) at (0, -1) {};
        \node [style=none] (5) at (0.75, 0) {,};
        \node [style=none] (5) at (1, 0) {}; 
    \end{pgfonlayer}
    \begin{pgfonlayer}{edgelayer}
    	\draw (2.center) to (1);
    	\draw (3.center) to (1);
    \end{pgfonlayer}
    \end{tikzpicture}
    \begin{tikzpicture}
    \begin{pgfonlayer}{nodelayer}
    	\node [style=construct] (1) at (0, 0) {$\phi$};
    	\node [style=none] (2) at (0, 1) {};
    	\node [style=none] (3) at (0, -1) {};
    \end{pgfonlayer}
    \begin{pgfonlayer}{edgelayer}
    	\draw (2.center) to (1);
    	\draw (3.center) to (1);
    \end{pgfonlayer}
    \end{tikzpicture}
    \]
    We define $\sqcup \maps \GMV(m) \times \GMV(n) \to \GMV(m + n)$ as follows.
    \begin{equation}\label{def:disjoint}
    \begin{tikzpicture}[baseline=(current bounding  box.center)]
    \begin{pgfonlayer}{nodelayer}
    	\node [style=construct] (U) at (0, 0) {$\sqcup$};
    	\node [style=none] (1) at (-0.5, 1) {};
    	\node [style=none] (2) at (0.5, 1) {};
    	\node [style=none] (3) at (-0.5, 0.5) {};
    	\node [style=none] (4) at (0.5, 0.5) {};
    	\node [style=none] (5) at (0, -1) {};
        \node [style=none] () at (1.2, 0) {=};
        \node [style=none] () at (1.8, 0) {}; 
        \node [style=none] () at (0, -1.6) {}; 
    \end{pgfonlayer}
    \begin{pgfonlayer}{edgelayer}
    	\draw [bend right] (3.center) to (U);
    	\draw [bend left] (4.center) to (U);
    	\draw [] (3.center) to (1);
    	\draw [] (4.center) to (2);
    	\draw (5.center) to (U);
    \end{pgfonlayer}
    \end{tikzpicture}
    \begin{tikzpicture}[baseline=(current bounding  box.center)]
    \begin{pgfonlayer}{nodelayer}
    	\node [style=none] (4) at (0, -1) {};
    	\node [style=construct] (6) at (-0.5, 1) {$\phi$};
    	\node [style=construct] (7) at (0.5, 1) {$\phi$};
    	\node [style=construct] (15) at (0, 0) {$\cup$};
    	\node [style=none] (16) at (-0.5, 3) {};
    	\node [style=none] (17) at (0.5, 3) {};
    	\node [style=construct] (18) at (-0.5, 2) {$i_1$};
    	\node [style=construct] (19) at (0.5, 2) {$i_2$};
    \end{pgfonlayer}
    \begin{pgfonlayer}{edgelayer}
    	\draw (15) to (4.center);
    	\draw [bend right] (6) to (15);
    	\draw [bend left] (7) to (15);
    	\draw (18) to (6);
    	\draw (19) to (7);
    	\draw (16.center) to (18);
    	\draw (17.center) to (19);
    \end{pgfonlayer}
    \end{tikzpicture}
    \end{equation}
    \cref{prop:knesercomm} gives the following equation.
    \begin{equation}\label{prop13}
    \begin{tikzpicture}[baseline=(current  bounding  box.center)]
    \begin{pgfonlayer}{nodelayer}
    	\node [style=none] (4) at (0, -1) {};
    	\node [style=construct] (6) at (-0.5, 1) {$\phi$};
    	\node [style=construct] (7) at (0.5, 1) {$\phi$};
    	\node [style=construct] (15) at (0, 0) {$\cup$};
    	\node [style=none] (16) at (-0.5, 3) {};
    	\node [style=none] (17) at (0.5, 3) {};
    	\node [style=construct] (18) at (-0.5, 2) {$i_2$};
    	\node [style=construct] (19) at (0.5, 2) {$i_1$};
    	\node [style=none] (1) at (1.8, 1) {=};
    \end{pgfonlayer}
    \begin{pgfonlayer}{edgelayer}
    	\draw (15) to (4.center);
    	\draw [bend right] (6) to (15);
    	\draw [bend left] (7) to (15);
    	\draw (18) to (6);
    	\draw (19) to (7);
    	\draw (16.center) to (18);
    	\draw (17.center) to (19);
    \end{pgfonlayer}
    \end{tikzpicture}
    \begin{tikzpicture}[baseline=(current  bounding  box.center)]
    \begin{pgfonlayer}{nodelayer}
        \node [style=none] () at (-1.8, 1) {}; 
    	\node [style=construct] (f1) at (-0.5, 1) {$\phi$};
    	\node [style=construct] (f2) at (0.5, 1) {$\phi$};
    	\node [style=construct] (U) at (0, 0) {$\cup$};
    	\node [style=construct] (i1) at (-0.5, 2) {$i_1$};
    	\node [style=construct] (i2) at (0.5, 2) {$i_2$};
    	\node [style=none] (1) at (0, -1) {};
    	\node [style=none] (2) at (-0.5, 2.5) {};
    	\node [style=none] (3) at (0.5, 2.5) {};
    	\node [style=none] (4) at (0, 3) {};
    	\node [style=none] (5) at (-0.5, 3.5) {};
    	\node [style=none] (6) at (0.5, 3.5) {};
    \end{pgfonlayer}
    \begin{pgfonlayer}{edgelayer}
    	\draw [bend right] (5) to (4.center);
    	\draw [bend left] (6) to (4.center);
    	\draw [bend left] (4.center) to (3.center);
    	\draw [bend right] (4.center) to (2.center);
    	\draw (2.center) to (i1);
    	\draw (3.center) to (i2);
    	\draw (i1) to (f1);
    	\draw (i2) to (f2);
    	\draw [bend right] (f1) to (U);
    	\draw [bend left] (f2) to (U);
    	\draw (U) to (1);
    \end{pgfonlayer}
    \end{tikzpicture}
    \end{equation}
    Since $\phi$ is a homomorphism, we get the following equation.
    \begin{equation}\label{phiishom}
    \begin{tikzpicture}[baseline=(current  bounding  box.center)]
    \begin{pgfonlayer}{nodelayer}
    	\node [style=construct] (1) at (0, 0) {$\cup$};
    	\node [style=construct] (2) at (-0.5, 1) {$\phi$};
    	\node [style=construct] (3) at (0.5, 1) {$\phi$};
    	\node [style=none] (4) at (-0.5, 2) {};
    	\node [style=none] (5) at (0.5, 2) {};
    	\node [style=none] (6) at (0, -1) {};
    	\node [style=none] (7) at (1.8, 0.5) {=};
    	\node [style=none] (8) at (2.8, 0) {};
    \end{pgfonlayer}
    \begin{pgfonlayer}{edgelayer}
    	\draw (4) to (2);
    	\draw (5) to (3);
    	\draw [bend right] (2) to (1);
    	\draw [bend left] (3) to (1);
    	\draw (1) to (6);
    \end{pgfonlayer}
    \end{tikzpicture}
    \begin{tikzpicture}[baseline=(current  bounding  box.center)]
    \begin{pgfonlayer}{nodelayer}
    	\node [style=construct] (g) at (0, 1) {$\gamma$};
    	\node [style=construct] (f) at (0, 0) {$\phi$};
    	\node [style=none] (1) at (-0.5, 2) {};
    	\node [style=none] (2) at (0.5, 2) {};
    	\node [style=none] (3) at (-0.5, 1.5) {};
    	\node [style=none] (4) at (0.5, 1.5) {};
    	\node [style=none] (5) at (0, -1) {};
    \end{pgfonlayer}
    \begin{pgfonlayer}{edgelayer}
        \draw (1) to (3.center);
        \draw (2) to (4.center);
    	\draw [bend right] (3.center) to (g);
    	\draw [bend left] (4.center) to (g);
    	\draw (g) to (f);
    	\draw (f) to (5);
    \end{pgfonlayer}
    \end{tikzpicture}
    \end{equation}
    Since $i_1$ and $i_2$ are homomorphisms, we get the following equations.
    \begin{equation}\label{isarehoms}
    \begin{tikzpicture}[baseline=(current  bounding  box.center)]
    \begin{pgfonlayer}{nodelayer}
    	\node [style=construct] (1) at (0, 0) {$\gamma$};
    	\node [style=construct] (2) at (-0.5, 1) {$i_j$};
    	\node [style=construct] (3) at (0.5, 1) {$i_j$};
    	\node [style=none] (4) at (-0.5, 2) {};
    	\node [style=none] (5) at (0.5, 2) {};
    	\node [style=none] (6) at (0, -1) {};
    	\node [style=none] (7) at (1.8, 0.5) {=};
    	\node [style=none] (8) at (2.8, 0) {};
    \end{pgfonlayer}
    \begin{pgfonlayer}{edgelayer}
    	\draw (4) to (2);
    	\draw (5) to (3);
    	\draw [bend right] (2) to (1);
    	\draw [bend left] (3) to (1);
    	\draw (1) to (6);
    \end{pgfonlayer}
    \end{tikzpicture}
    \begin{tikzpicture}[baseline=(current  bounding  box.center)]
    \begin{pgfonlayer}{nodelayer}
    	\node [style=construct] (g) at (0, 1) {$\gamma$};
    	\node [style=construct] (f) at (0, 0) {$i_j$};
    	\node [style=none] (1) at (-0.5, 2) {};
    	\node [style=none] (2) at (0.5, 2) {};
    	\node [style=none] (3) at (-0.5, 1.5) {};
    	\node [style=none] (4) at (0.5, 1.5) {};
    	\node [style=none] (5) at (0, -1) {};
    \end{pgfonlayer}
    \begin{pgfonlayer}{edgelayer}
        \draw (1) to (3.center);
        \draw (2) to (4.center);
    	\draw [bend right] (3.center) to (g);
    	\draw [bend left] (4.center) to (g);
    	\draw (g) to (f);
    	\draw (f) to (5);
    \end{pgfonlayer}
    \end{tikzpicture}
    \end{equation}
    
    We want to show that $(g \sqcup h) \cup (g' \sqcup h') = (g \cup g') \sqcup (h \cup h')$. We compute:
    
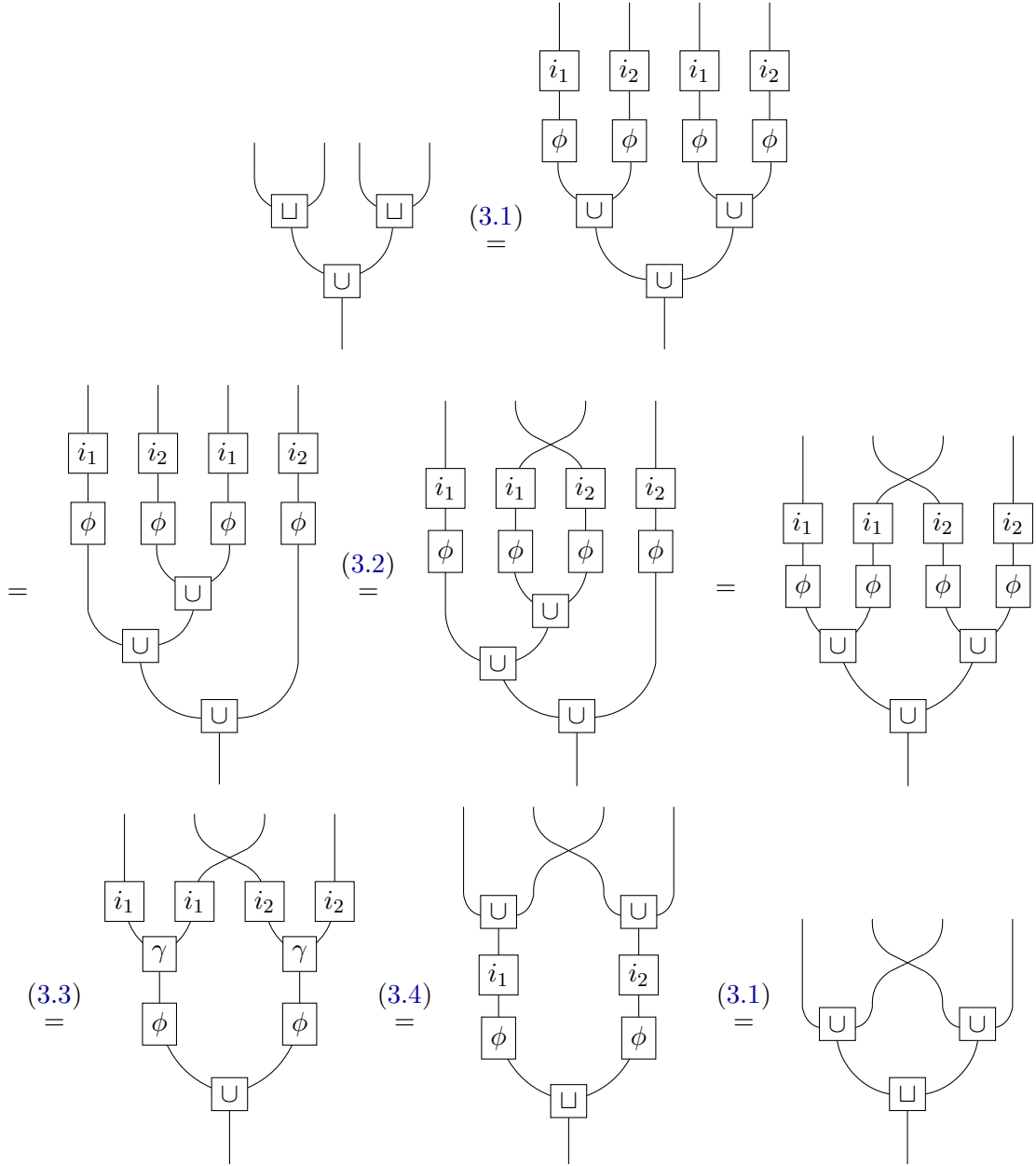
\[
    \begin{tikzpicture}
    \begin{pgfonlayer}{nodelayer}
    	\node [style=construct] (s1) at (-0.75, 1) {$\sqcup$};
    	\node [style=construct] (s2) at (0.75, 1) {$\sqcup$};
    	\node [style=construct] (U) at (0, 0) {$\cup$};
    	\node [style=none] (1) at (-1.25, 1.5) {};
    	\node [style=none] (2) at (0.25, 1.5) {};
    	\node [style=none] (3) at (-0.25, 1.5) {};
    	\node [style=none] (4) at (1.25, 1.5) {};
    	\node [style=none] (5) at (0, -1) {};
    	\node [style=none] (6) at (-1.25, 2) {};
    	\node [style=none] (7) at (0.25, 2) {};
    	\node [style=none] (8) at (-0.25, 2) {};
    	\node [style=none] (9) at (1.25, 2) {};
    	\node [style=none] () at (2.2, 0.5) {=};
    	\node [style=none] () at (2.2, 0.9) {(\ref{def:disjoint})};
    	\node [style=none] () at (2.8, 1) {}; 
    \end{pgfonlayer}
    \begin{pgfonlayer}{edgelayer}
        \draw (6) to (1.center);
        \draw (8) to (3.center);
        \draw (7) to (2.center);
        \draw (9) to (4.center);
        \draw [bend right] (1.center) to (s1);
        \draw [bend left] (3.center) to (s1);
        \draw [bend right] (2.center) to (s2);
        \draw [bend left] (4.center) to (s2);
        \draw [bend right] (s1) to (U);
        \draw [bend left] (s2) to (U);
        \draw (U) to (5);
    \end{pgfonlayer}
    \end{tikzpicture}
    \begin{tikzpicture}
    \begin{pgfonlayer}{nodelayer}
    	\node [style=construct] (U1) at (-1, -1) {$\cup$};
    	\node [style=construct] (U2) at (1, -1) {$\cup$};
    	\node [style=construct] (U3) at (0, -2) {$\cup$};
    	\node [style=construct] (i1) at (-1.5, 1) {$i_1$};
    	\node [style=construct] (i2) at (-0.5, 1) {$i_2$};
    	\node [style=construct] (i1') at (0.5, 1) {$i_1$};
    	\node [style=construct] (i2') at (1.5, 1) {$i_2$};
    	\node [style=construct] (f1) at (-1.5, 0) {$\phi$};
    	\node [style=construct] (f2) at (-0.5, 0) {$\phi$};
    	\node [style=construct] (f3) at (0.5, 0) {$\phi$};
    	\node [style=construct] (f4) at (1.5, 0) {$\phi$};
    	\node [style=none] (1) at (-1.5, 2) {};
    	\node [style=none] (2) at (0.5, 2) {};
    	\node [style=none] (3) at (-0.5, 2) {};
    	\node [style=none] (4) at (1.5, 2) {};
    	\node [style=none] (5) at (0, -3) {};
    \end{pgfonlayer}
    \begin{pgfonlayer}{edgelayer}
    	\draw (1) to (i1);
    	\draw (3) to (i2);
    	\draw (2) to (i1');
    	\draw (4) to (i2');
    	\draw (i1) to (f1);
    	\draw (i2) to (f2);
    	\draw (i1') to (f3);
    	\draw (i2') to (f4);
    	\draw [bend right] (f1) to (U1);
    	\draw [bend left] (f2) to (U1);
    	\draw [bend right] (f3) to (U2);
    	\draw [bend left] (f4) to (U2);
    	\draw [bend right = 40] (U1) to (U3);
    	\draw [bend left = 40] (U2) to (U3);
    	\draw (U3) to (5);
    \end{pgfonlayer}
    \end{tikzpicture}
    \]\[
    \begin{tikzpicture}
    \begin{pgfonlayer}{nodelayer}
    	\node [style=construct] (i1) at (-1.5, 2) {$i_1$};
    	\node [style=construct] (i2) at (-0.5, 2) {$i_2$};
    	\node [style=construct] (i1') at (0.5, 2) {$i_1$};
    	\node [style=construct] (i2') at (1.5, 2) {$i_2$};
    	\node [style=construct] (f1) at (-1.5, 1) {$\phi$};
    	\node [style=construct] (f2) at (-0.5, 1) {$\phi$};
    	\node [style=construct] (f3) at (0.5, 1) {$\phi$};
    	\node [style=construct] (f4) at (1.5, 1) {$\phi$};
    	\node [style=construct] (U1) at (0, 0) {$\cup$};
    	\node [style=construct] (U2) at (-0.75, -0.75) {$\cup$};
    	\node [style=construct] (U3) at (0.375, -1.75) {$\cup$};
    	\node [style=none] (1) at (-1.5, 3) {};
    	\node [style=none] (2) at (-0.5, 3) {};
    	\node [style=none] (3) at (0.5, 3) {};
    	\node [style=none] (4) at (1.5, 3) {};
    	\node [style=none] (5) at (0.375, -2.75) {};
    	\node [style=none] (6) at (-1.5, -0.25) {};
    	\node [style=none] (7) at (0, -0.25) {};
    	\node [style=none] (8) at (1.5, -1) {};
    	\node [style=none] (9) at (-0.75, -1) {};
    	\node [style=none] () at (2.5, 0) {=};
    	\node [style=none] () at (2.5, 0.4) {(\ref{prop13})};
    	\node [style=none] () at (-2.5, 0) {=};
    	\node [style=none] () at (3.3, 0) {}; 
    \end{pgfonlayer}
    \begin{pgfonlayer}{edgelayer}
        \draw (1) to (i1);
        \draw (2) to (i2);
        \draw (3) to (i1');
        \draw (4) to (i2');
        \draw (i1) to (f1);
        \draw (i2) to (f2);
        \draw (i1') to (f3);
        \draw (i2') to (f4);
        \draw (f1) to (6.center);
        \draw [bend right] (f2) to (U1);
        \draw [bend left] (f3) to (U1);
        \draw [bend right] (6.center) to (U2);
        \draw (U1) to (7.center);
        \draw [bend left] (7.center) to (U2);
        \draw (U2) to (9.center);
        \draw [bend right = 40] (9.center) to (U3);
        \draw (f4) to (8.center);
        \draw [bend left = 40] (8.center) to (U3);
        \draw (U3) to (5);
    \end{pgfonlayer}
    \end{tikzpicture}
    \begin{tikzpicture}
	\begin{pgfonlayer}{nodelayer}
		\node [style=construct] (3) at (-1.5, 0.75) {$i_1$};
		\node [style=construct] (4) at (0.5, 0.75) {$i_2$};
		\node [style=construct] (5) at (-0.5, 0.75) {$i_1$};
		\node [style=construct] (6) at (1.5, 0.75) {$i_2$};
		\node [style=construct] (17) at (-1.5, -0.15) {$\phi$};
		\node [style=construct] (18) at (-0.5, -0.15) {$\phi$};
		\node [style=construct] (19) at (0.5, -0.15) {$\phi$};
		\node [style=construct] (20) at (1.5, -0.15) {$\phi$};
		\node [style=construct] (21) at (0, -1) {$\cup$};
		\node [style=construct] (22) at (-0.75, -1.75) {$\cup$};
		\node [style=construct] (23) at (0.375, -2.5) {$\cup$};
		\node [style=none] (8) at (-1.5, 2) {};
		\node [style=none] (9) at (0.25, 1.25) {};
		\node [style=none] (11) at (-0.25, 1.25) {};
		\node [style=none] (12) at (1.5, 2) {};
		\node [style=none] (13) at (-0.25, 1.5) {};
		\node [style=none] (14) at (0.25, 1.5) {};
		\node [style=none] (15) at (-0.5, 2) {};
		\node [style=none] (16) at (0.5, 2) {};
		\node [style=none] (24) at (0.375, -3.5) {};
		\node [style=none] (25) at (1.5, -1.75) {};
		\node [style=none] (26) at (-1.5, -1) {};
		\node [style=none] () at (2.5, -0.7) {=};
		\node [style=none] () at (3.3, -0.7) {}; 
	\end{pgfonlayer}
	\begin{pgfonlayer}{edgelayer}
		\draw (8.center) to (3.center);
		\draw [bend left, looseness=1] (9.center) to (4.center);
		\draw [bend right, looseness=1] (11.center) to (5.center);
		\draw (12.center) to (6.center);
		\draw (11.center) to (14.center);
		\draw [bend right, looseness=1] (14.center) to (16.center);
		\draw [bend right, looseness=1] (15.center) to (13.center);
		\draw (13.center) to (9.center);
		\draw (3.center) to (17.center);
		\draw (5.center) to (18.center);
		\draw (4.center) to (19.center);
		\draw (6.center) to (20.center);
		\draw (17.center) to (26.center);
		\draw [bend right=45] (26.center) to (22.center);
	    \draw [bend right=45] (18.center) to (21.center);
	    \draw [bend left=45] (19.center) to (21.center);
	    \draw (20.center) to (25.center);
	    \draw [bend left=45] (21.center) to (22.center);
	    \draw [bend left=45] (25.center) to (23.center);
	    \draw [bend right=45] (22.center) to (23.center);
		\draw (23.center) to (24.center);
	\end{pgfonlayer}
    \end{tikzpicture}
    \begin{tikzpicture}
	\begin{pgfonlayer}{nodelayer}
		\node [style=construct] (3) at (-1.5, 0.75) {$i_1$};
		\node [style=construct] (4) at (0.5, 0.75) {$i_2$};
		\node [style=construct] (5) at (-0.5, 0.75) {$i_1$};
		\node [style=construct] (6) at (1.5, 0.75) {$i_2$};
		\node [style=construct] (17) at (-1.5, -0.15) {$\phi$};
		\node [style=construct] (18) at (-0.5, -0.15) {$\phi$};
		\node [style=construct] (19) at (0.5, -0.15) {$\phi$};
		\node [style=construct] (20) at (1.5, -0.15) {$\phi$};
		\node [style=construct] (21) at (-1, -1) {$\cup$};
		\node [style=construct] (22) at (1, -1) {$\cup$};
		\node [style=construct] (23) at (0, -2) {$\cup$};
		\node [style=none] (8) at (-1.5, 2) {};
		\node [style=none] (9) at (0.25, 1.25) {};
		\node [style=none] (11) at (-0.25, 1.25) {};
		\node [style=none] (12) at (1.5, 2) {};
		\node [style=none] (13) at (-0.25, 1.5) {};
		\node [style=none] (14) at (0.25, 1.5) {};
		\node [style=none] (15) at (-0.5, 2) {};
		\node [style=none] (16) at (0.5, 2) {};
		\node [style=none] (24) at (0, -3) {};
	\end{pgfonlayer}
	\begin{pgfonlayer}{edgelayer}
		\draw (8.center) to (3.center);
		\draw [bend left, looseness=1] (9.center) to (4.center);
		\draw [bend right, looseness=1] (11.center) to (5.center);
		\draw (12.center) to (6.center);
		\draw (11.center) to (14.center);
		\draw [bend right, looseness=1] (14.center) to (16.center);
		\draw [bend right, looseness=1] (15.center) to (13.center);
		\draw (13.center) to (9.center);
		\draw (3.center) to (17.center);
		\draw (5.center) to (18.center);
		\draw (4.center) to (19.center);
		\draw (6.center) to (20.center);
		\draw [bend right] (17.center) to (21.center);
		\draw [bend left] (18.center) to (21.center);
		\draw [bend right] (19.center) to (22.center);
		\draw [bend left] (20.center) to (22.center);
		\draw [bend right] (21.center) to (23.center);
		\draw [bend left] (22.center) to (23.center);
		\draw (23.center) to (24.center);
	\end{pgfonlayer}
    \end{tikzpicture}
    \]\[
    \begin{tikzpicture}
	\begin{pgfonlayer}{nodelayer}
		\node [style=construct] (3) at (-1.5, 0.75) {$i_1$};
		\node [style=construct] (4) at (0.5, 0.75) {$i_2$};
		\node [style=construct] (5) at (-0.5, 0.75) {$i_1$};
		\node [style=construct] (6) at (1.5, 0.75) {$i_2$};
		\node [style=construct] (17) at (-1, 0) {$\gamma$};
		\node [style=construct] (18) at (1, 0) {$\gamma$};
		\node [style=construct] (19) at (-1, -1) {$\phi$};
		\node [style=construct] (20) at (1, -1) {$\phi$};
		\node [style=construct] (21) at (0, -2) {$\cup$};
		\node [style=none] (22) at (0, -3) {};
		\node [style=none] (8) at (-1.5, 2) {};
		\node [style=none] (9) at (0.25, 1.25) {};
		\node [style=none] (11) at (-0.25, 1.25) {};
		\node [style=none] (12) at (1.5, 2) {};
		\node [style=none] (13) at (-0.25, 1.5) {};
		\node [style=none] (14) at (0.25, 1.5) {};
		\node [style=none] (15) at (-0.5, 2) {};
		\node [style=none] (16) at (0.5, 2) {};
		\node [style=none] () at (-2.5, -1) {=};
		\node [style=none] () at (-2.5, -0.6) {(\ref{phiishom})};
		\node [style=none] () at (2.5, -1) {=};
		\node [style=none] () at (2.5, -0.6) {(\ref{isarehoms})};
		\node [style=none] () at (3.3, -1) {};
	\end{pgfonlayer}
	\begin{pgfonlayer}{edgelayer}
		\draw (8.center) to (3.center);
		\draw [bend left, looseness=1] (9.center) to (4.center);
		\draw [bend right, looseness=1] (11.center) to (5.center);
		\draw (12.center) to (6.center);
		\draw (11.center) to (14.center);
		\draw [bend right, looseness=1] (14.center) to (16.center);
		\draw [bend right, looseness=1] (15.center) to (13.center);
		\draw (13.center) to (9.center);
		\draw [bend right] (3.center) to (17.center);
		\draw [bend left] (5.center) to (17.center);
		\draw [bend right] (4.center) to (18.center);
		\draw [bend left] (6.center) to (18.center);
		\draw (18.center) to (20.center);
		\draw [bend left](20.center) to (21.center);
		\draw (17.center) to (19.center);
		\draw [bend right](19.center) to (21.center);
		\draw (21.center) to (22.center);
	\end{pgfonlayer}
    \end{tikzpicture}
    \begin{tikzpicture}
	\begin{pgfonlayer}{nodelayer}
		\node [style=construct] (0) at (0, -2.7) {$\sqcup$};
		\node [style=construct] (1) at (-1, 0) {$\cup$};
		\node [style=construct] (2) at (1, 0) {$\cup$};
		\node [style=construct] (17) at (-1, -0.9) {$i_1$};
		\node [style=construct] (18) at (-1, -1.8) {$\phi$};
		\node [style=construct] (19) at (1, -0.9) {$i_2$};
		\node [style=construct] (20) at (1, -1.8) {$\phi$};
		\node [style=none] (3) at (-1.5, 0.25) {};
		\node [style=none] (4) at (0.5, 0.25) {};
		\node [style=none] (5) at (-0.5, 0.25) {};
		\node [style=none] (6) at (1.5, 0.25) {};
		\node [style=none] (7) at (0, -3.6) {};
		\node [style=none] (8) at (-1.5, 1.5) {};
		\node [style=none] (9) at (0.25, 0.75) {};
		\node [style=none] (11) at (-0.25, 0.75) {};
		\node [style=none] (12) at (1.5, 1.5) {};
		\node [style=none] (13) at (-0.25, 1) {};
		\node [style=none] (14) at (0.25, 1) {};
		\node [style=none] (15) at (-0.5, 1.5) {};
		\node [style=none] (16) at (0.5, 1.5) {};
		\node [style=none] () at (2.5, -1.6) {=};
		\node [style=none] () at (2.5, -1.2) {(\ref{def:disjoint})};
		\node [style=none] () at (3.3, 0) {};
	\end{pgfonlayer}
	\begin{pgfonlayer}{edgelayer}
		\draw [bend right=55, looseness=1.25] (3.center) to (1.center);
		\draw [bend right=55, looseness=1.25] (4.center) to (2.center);
		\draw [bend left=55, looseness=1.25] (5.center) to (1.center);
		\draw [bend left=55, looseness=1.25] (6.center) to (2.center);
		\draw (0.center) to (7.center);
		\draw (8.center) to (3.center);
		\draw [bend left, looseness=1] (9.center) to (4.center);
		\draw [bend right, looseness=1] (11.center) to (5.center);
		\draw (12.center) to (6.center);
		\draw (11.center) to (14.center);
		\draw [bend right, looseness=1] (14.center) to (16.center);
		\draw [bend right, looseness=1] (15.center) to (13.center);
		\draw (13.center) to (9.center);
		\draw (1.center) to (17.center);
		\draw (2.center) to (19.center);
		\draw (17.center) to (18.center);
		\draw (19.center) to (20.center);
		\draw [bend right] (18.center) to (0.center);
		\draw [bend left] (20.center) to (0.center);
	\end{pgfonlayer}
    \end{tikzpicture}
    \begin{tikzpicture}
	\begin{pgfonlayer}{nodelayer}
		\node [style=construct] (0) at (0, -1) {$\sqcup$};
		\node [style=construct] (1) at (-1, 0) {$\cup$};
		\node [style=construct] (2) at (1, 0) {$\cup$};
		\node [style=none] (3) at (-1.5, 0.25) {};
		\node [style=none] (4) at (0.5, 0.25) {};
		\node [style=none] (5) at (-0.5, 0.25) {};
		\node [style=none] (6) at (1.5, 0.25) {};
		\node [style=none] (7) at (0, -2) {};
		\node [style=none] (8) at (-1.5, 1.5) {};
		\node [style=none] (9) at (0.25, 0.75) {};
		\node [style=none] (11) at (-0.25, 0.75) {};
		\node [style=none] (12) at (1.5, 1.5) {};
		\node [style=none] (13) at (-0.25, 1) {};
		\node [style=none] (14) at (0.25, 1) {};
		\node [style=none] (15) at (-0.5, 1.5) {};
		\node [style=none] (16) at (0.5, 1.5) {};
	\end{pgfonlayer}
	\begin{pgfonlayer}{edgelayer}
		\draw [bend right=55, looseness=1.25] (3.center) to (1.center);
		\draw [bend right=55, looseness=1.25] (4.center) to (2.center);
		\draw [bend left=55, looseness=1.25] (5.center) to (1.center);
		\draw [bend left=55, looseness=1.25] (6.center) to (2.center);
		\draw [bend right = 55] (1.center) to (0.center);
		\draw [bend left = 55] (2.center) to (0.center);
		\draw (0.center) to (7.center);
		\draw (8.center) to (3.center);
		\draw [bend left, looseness=1] (9.center) to (4.center);
		\draw [bend right, looseness=1] (11.center) to (5.center);
		\draw (12.center) to (6.center);
		\draw (11.center) to (14.center);
		\draw [bend right, looseness=1] (14.center) to (16.center);
		\draw [bend right, looseness=1] (15.center) to (13.center);
		\draw (13.center) to (9.center);
	\end{pgfonlayer}
\end{tikzpicture}
    \]
    
    Let $\sigma \in S_m$ and $\tau \in S_n$. Then 
    \begin{align*}
        \GMV(\sigma + \tau)(g \sqcup h)
        &= \GMV(\sigma + \tau)\phi(i_1(g) \cup i_2(h))
        \\&= \GMV(\sigma)\phi(i_1(g)) \cup \GMV(\tau)\phi(i_2(h))
        \\&= \GMV(\sigma(g)) \sqcup \GMV(\tau(h)),
    \end{align*}
    so the following diagram commutes.
    \[
    \begin{tikzcd}
        \GMV(m) \times \GMV(n)
        \arrow[r, "\sqcup"]
        \arrow[d, "\GMV(\sigma) \times \GMV(\tau)", swap]
        &
        \GMV(m+n)
        \arrow[d, "\GMV(\sigma + \tau)"]
        \\
        \GMV(m) \times \GMV(n)
        \arrow[r, "\sqcup", swap]
        &
        \GMV(m+n)
    \end{tikzcd}
    \]
    Thus $\sqcup$ is a natural transformation.
\end{proof}

\begin{proof}[Proof of Theorem \ref{thm:main}]
    Checking the coherence conditions for $\sqcup$ to be a laxator is a straightforward computation.
    Let $f \maps M \to N$. Then define the natural transformation $f_\V \maps \GMV \to \Gamma_{N,\V}$ with components $(f_\V)_n \maps \GMV(n) \to \Gamma_{N,\V}(n)$ given by the universal property. Composition is clearly preserved.
\end{proof}

\begin{thm}
    The functor $\Gamma_{-,\V}$ is left adjoint to $E \maps \NetMod_\V \to \V$ where $E(F) = F(\mathbf 2)$ for $(F,\Phi) \maps (\S, +) \to (\V, \times)$ a $\V$-network model.
\end{thm}
Because of this, we call $\Gamma_{M,\V}$ the \define{free $\V$-valued network model on the monoid $M$} or the \define{free $\V$ network model on $M$}.

\begin{proof}
    By construction, $\Gamma_{M,\V}(\2) = M$, so let the unit $\eta = 1_{1_\V} \maps 1_\V \to \Gamma_{-,\V}(\2)$. 
    
    We use the universal property of $\GMV$ to construct the counit. We define a map $F(\2) \to F(n)$ for each vertex in $\KG_{n,2}$, and a map $F(\2) \times F(\2) \to F(n)$ for each edge in $\KG_{n,2}$. 
    
    If $i,j \leq n$, then $F((1\; i)(2\; j)) \maps F(n) \to F(n)$. If $e$ is the unit of the monoid $F (\mathbf{n - 2} )$, and $m \in F(\2)$, then $\Phi_{\2, n - \2} (m, e) \in F(n)$. Define maps $c_{i,j} \maps F(\2) \to F(n)$ by 
    \[
        c_{i, j} = F((1\; i)(2\; j)) (\Phi_{\2, n - \2}(m, e)).
    \]
    The intuition here is that $m$ is a value on one edge of the graph, and $e$ is a graph with $n-2$ vertices and no edges. Then $\Phi(m,e)$ is the graph with $n$ vertices, and just one $m$-valued edge between vertices $1$ and $2$. Then the permutation $(1\; i)(2\; j)$ permutes this one-edge graph to put $m$ between vertex $i$ and vertex $j$. So the map $c_{i,j}$ places the one-edge monoid $M$ at the $i,j$-position in the $n$-vertex monoid.
    
    Define maps $c_{i,j,p,q} \maps F(\2) \times F(\2) \to F(n)$ by $c_{i,j,p,q} (m, m') = c_{i,j} (m) c_{p,q} (m')$. The second gives a monoid homomorphism precisely because $(F,\Phi)$ is a network model. 
    
    Then we get a map $(\epsilon_F)_n \maps \Gamma_{F(\2),\V}(n) \to F(n)$ by universal property, which gives a monoidal natural transformation automatically. That these maps form the components of a natural transformation can be seen by a routine computation.
    
    Notice that 
    \begin{align*}
        (\epsilon \Gamma_{-,\V})_M = \epsilon_{\Gamma_{M,\V}} = 1_{\Gamma_{M,\V}},\\
        (\Gamma_{-,\V} \eta)_M = \Gamma_{1_M,\V} = 1_{\Gamma_{M,\V}},\\
        (E \epsilon)_F = E(\epsilon_F) = (\epsilon_F)_2 = 1_{F(2)},\\
        (\eta E)_F = \eta_{F(2)} = 1_{F(2)}.
    \end{align*}
    Thus, checking that the snake equations hold is routine. 
\end{proof}

\begin{expl}
\label{ex:oldnew}
    In $\CMon$, products and coproducts are isomorphic. In particular, for a commutative monoid $M$, $\Gamma_{M,\CMon} \cong \Gamma_M$.
\end{expl}

Note that this does not indicate that varietal network models completely encompass ordinary network models. If $M$ is a noncommutative monoid,then $\Gamma_{M, \CMon}$ is not defined, but $\Gamma_M$ is.

\section{Commitment Networks}
\label{sec:commitment}

The motivating example of network models in general is $\SG$, the network model of simple graphs. By Example \ref{ex:oldnew}, this network model is an example of the main construction of this chapter, $\SG = \Gamma_{\B,\CMon}$. The boolean monoid is not only an object in $\CMon$, it is also an object in $\GMon$, the variety of graphic monoids. Then we can consider the network models $\Gamma_{\B,\Mon}$ and $\Gamma_{\B,\GMon}$.

\begin{expl}
    Elements of the monoid $\Gamma_{\B,\Mon}(n)$ are words $e_{p_1,q_1} \dots e_{p_k,q_k}$. These words are interpreted as graphs with edges that look like they were built with popsicle sticks, and if two edges lie directly on top of each other, they are identified.
    Besides that relation, you can stack edges as high as you want by placing them between different pairs of vertices, but sharing one vertex.  
\end{expl}

There are networks one could imagine building with this popsicle stick intuition which are not allowed by this formalism. For instance, consider a network with three nodes and an edge for each pair of nodes, each overlapping exactly one of its neighbors, forming an Escher-esque ever-ascending staircase. This sort of network is not allowed by the formalism, since networks are actually equivalence classes of words, where letters have a definite position relative to each other. This is an important feature for this network model as it is necessary to guarantee that the procedure in the following example is well-defined, giving an algebra of the related network operad. What this means in terms of popsicle stick intuition is that allowed networks are built by placing popsicle sticks one at a time. 

\begin{expl}
    Elements of the $\Gamma_{\B,\GMon}(n)$ are similar to those in the previous example, except that they must obey the graphic identity, $xyx = xy$ for all $x,y \in \Gamma_{\B,\GMon}(n)$. What this means in the graphical interpretation is that all edges can be identified with the lowest occuring instance of an edge on the same vertex pair. This means that these networks in reduced form have at most as many edges as the complete simple graph with the same number of edges. Essentially these networks are simple graphs with a partial order on the edges which respects disjointness of edges.
\end{expl}

The networks in the previous example have exactly what we need in a network model to realize networks of bounded degree as an algebra of a network operad.

\begin{expl}[\bf Networks of bounded degree, revisited]
    The \define{degree} of a vertex in a simple graph is the number of edges in the graph which contain that vertex. For $k \in \N$, we say that a simple graph is \define{$k$-bounded} if all vertices have degree less than or equal to $k$. Then we can consider the set $B_k(n)$ of $k$-bounded simple graphs. We can define an action of $\Gamma_{\B, \GMon}(n)$ on $B_k(n)$ in the following way. Let $g = e_1 \dots e_l \in \Gamma_{\B, \GMon}(n)$ and $h \in B_k(n)$. Choose a graph $h' \in \Gamma_{\B, \GMon}(n)$ which has the same edges as $h$. Define $h_0 = h'$, then define $h_i = h_{i-1}e_i$ if that is $k$-bounded, else $h_i= h_{i-1}$. Let $hg$ denote $h_l$, which is a $k$-bounded element of $\Gamma_{\B, \GMon}(n)$. Let $\Gamma^k_{\B, \GMon}(n)$ denote the set of $k$-bounded elements of $\Gamma_{\B, \GMon}(n)$. There is a function $s \maps \Gamma^k_{\B, \GMon}(n) \to B_k(n)$. So we define $h g$ to be $s(h_l)$. This is independent of the choice of $h'$ and defines an action of $\Gamma_{\B,\GMon}$ on $B_k(n)$.
\end{expl}

The networks in the question in \cref{sec:NNMIntro} can be represented by simple graphs with vertex degrees bounded by $k$. Then $B_k(n)$ gives an algebra of the operad $\O_{\B,\GMon}$. This resolves the conflict encountered in the question in \cref{sec:NNMIntro}. Ordinary network models could not record the order in which edges were added to a network, which was necessary to define a systematic way of attempting to add new connections to a network which has degree limitations on each vertex.

}{\ssp
\setcounter{chapter}{3}
\chapter{Petri Nets}
\label{ch:PetriNets}

\section{Introduction} 

Petri nets are a widely studied formalism for describing collections of entities of different types, and how they turn into other entities \cite{GiraultValk, Peterson}. In this chapter, we combine Petri nets with network models. This is worthwhile because while both formalisms involve networks, they serve different functions, and are in some sense complementary.

A Petri net can be drawn as a bipartite directed graph with vertices of two kinds: \emph{places}, drawn as circles below, and \emph{transitions} drawn as squares:
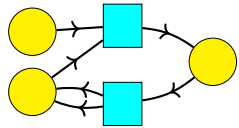
\[
\vcenter{\hbox{\scalebox{0.8}{
\begin{tikzpicture}
	\begin{pgfonlayer}{nodelayer}
        \node [style=species] (C) at (-1, 0) {$\quad\;$};
        \node [style=species] (B) at (-4, -0.5) {$\quad\;$};
        \node [style=species] (A) at (-4, 0.5) {$\quad\;$};
        \node [style=transition] (tau1) at (-2.5, 0.6) {$\big.\;\;\;\, $};
        \node [style=transition] (tau2) at (-2.5, -0.7) {$\big.\;\;\;\, $};
	\end{pgfonlayer}
	\begin{pgfonlayer}{edgelayer}
        \draw [style=inarrow] (A) to (tau1);
        \draw [style=inarrow] (B) to (tau1);
        \draw [style=inarrow, bend left=15, looseness=1.00] (tau1) to (C);
        \draw [style=inarrow, bend left=15, looseness=1.00] (C) to (tau2);
        \draw [style=inarrow, bend left=15, looseness=1.00] (tau2) to (B); 
        \draw [style=inarrow, bend right =15, looseness=1.00] (tau2) to (B); 
	\end{pgfonlayer}
\end{tikzpicture}
}}}
\]
In applications to chemistry, places are also called \emph{species}. When we run a Petri net, we start by placing a finite number of \emph{tokens} in each place:
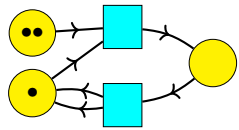
\[
\vcenter{\hbox{\scalebox{0.8}{
\begin{tikzpicture}
	\begin{pgfonlayer}{nodelayer}
		\node [style=species] (C) at (-1, 0) {$\quad\;$};
		\node [style=species] (B) at (-4, -0.5) {$\;\bullet\;$};
		\node [style=species] (A) at (-4, 0.5) {$\bullet\bullet$};
		\node [style=transition] (tau1) at (-2.5, 0.6) {$\big.\;\;\;\, $};
        \node [style=transition] (tau2) at (-2.5, -0.7) {$\big.\;\;\;\, $};
	\end{pgfonlayer}
	\begin{pgfonlayer}{edgelayer}
		\draw [style=inarrow] (A) to (tau1);
		\draw [style=inarrow] (B) to (tau1);
		\draw [style=inarrow, bend left=15, looseness=1.00] (tau1) to (C);
        \draw [style=inarrow, bend left=15, looseness=1.00] (C) to (tau2);
        \draw [style=inarrow, bend left=15, looseness=1.00] (tau2) to (B); 
        \draw [style=inarrow, bend right =15, looseness=1.00] (tau2) to (B); 
	\end{pgfonlayer}
\end{tikzpicture}
}}}
\]
This is called a \emph{marking}. Then we repeatedly change the marking using the transitions. For example, the above marking can change to this:
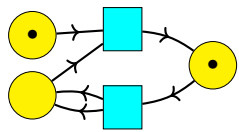
\[
\vcenter{\hbox{\scalebox{0.8}{
\begin{tikzpicture}
	\begin{pgfonlayer}{nodelayer}
        \node [style=species] (C) at (-1, 0) {$\;\bullet\;$};
        \node [style=species] (B) at (-4, -0.5) {$\quad\;$};
        \node [style=species] (A) at (-4, 0.5) {$\;\bullet\;$};
        \node [style=transition] (tau1) at (-2.5, 0.6) {$\big.\;\;\;\, $};
        \node [style=transition] (tau2) at (-2.5, -0.7) {$\big.\;\;\;\, $};
	\end{pgfonlayer}
	\begin{pgfonlayer}{edgelayer}
		\draw [style=inarrow] (A) to (tau1);
		\draw [style=inarrow] (B) to (tau1);
		\draw [style=inarrow, bend left=15, looseness=1.00] (tau1) to (C);
        \draw [style=inarrow, bend left=15, looseness=1.00] (C) to (tau2);
        \draw [style=inarrow, bend left=15, looseness=1.00] (tau2) to (B); 
        \draw [style=inarrow, bend right =15, looseness=1.00] (tau2) to (B); 
	\end{pgfonlayer}
\end{tikzpicture}
}}}
\]
and then this:
\[
\vcenter{\hbox{\scalebox{0.8}
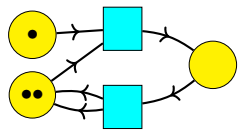
{
\begin{tikzpicture}
	\begin{pgfonlayer}{nodelayer}
		\node [style=species] (C) at (-1, 0) {$\quad\;$};
		\node [style=species] (B) at (-4, -0.5) {$\bullet\bullet$};
		\node [style=species] (A) at (-4, 0.5) {$\;\bullet\;$};
		\node [style=transition] (tau1) at (-2.5, 0.6) {$\big.\;\;\;\, $};
        \node [style=transition] (tau2) at (-2.5, -0.7) {$\big.\;\;\;\, $};
	\end{pgfonlayer}
	\begin{pgfonlayer}{edgelayer}
		\draw [style=inarrow] (A) to (tau1);
		\draw [style=inarrow] (B) to (tau1);
		\draw [style=inarrow, bend left=15, looseness=1.00] (tau1) to (C);
        \draw [style=inarrow, bend left=15, looseness=1.00] (C) to (tau2);
        \draw [style=inarrow, bend left=15, looseness=1.00] (tau2) to (B); 
        \draw [style=inarrow, bend right =15, looseness=1.00] (tau2) to (B); 
	\end{pgfonlayer}
\end{tikzpicture}
}}}
\]
Thus, the places represent different \emph{types} of entity, and the transitions describe ways that one collection of entities of specified types can turn into another such collection. 

Network models serve a different function than Petri nets: they are a general tool for working with networks of many kinds. A network model is a lax symmetric monoidal functor $G \maps \S(C) \to \Cat$, where $\S(C)$ is the free strict symmetric monoidal category on a set $C$. Elements of $C$ represent different kinds of ``agents''. Unlike in a Petri net, we do not usually consider processes where these agents turn into other agents. Instead, we wish to study everything that can be done with a fixed collection of agents. Any object $x \in \S(C)$ is of the form $c_1 \otimes \cdots \otimes c_n$ for some $c_i \in C$; thus, it describes a collection of agents of various kinds. The functor $G$ maps this object to a category $G(x)$ that describes everything that can be done with this collection of agents. 

In many examples considered so far, $G(x)$ is a category whose morphisms are graphs whose nodes are agents of types $c_1, \dots, c_n$. Composing these morphisms corresponds to \emph{overlaying} graphs. Network models of this sort let us design networks where the nodes are agents and the edges are communication channels or shared commitments. In \cref{ch:NetworkModels}, the operation of overlaying graphs was always commutative. In \cref{ch:NNM} we introduced more general noncommutative overlay operations. This lets us design networks where each agent has a limit on how many communication channels or commitments it can handle; the noncommutativity allows us to take a first come, first served approach to resolving conflicting commitments.

Here we take a different tack: we instead take $G(x)$ to be a category whose morphisms are \emph{processes that the given collection of agents, $x$, can carry out}. Composition of morphisms corresponds to carrying out first one process and then another.

This idea meshes well with Petri net theory, because any Petri net $P$ determines a symmetric monoidal category $FP$ whose morphisms are processes that can be carried out using this Petri net. More precisely, the objects in $FP$ are markings of $P$, and the morphisms are sequences of ways to change these markings using transitions, e.g.:
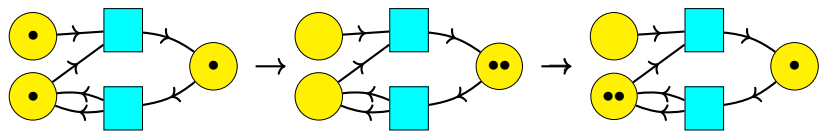
\[
\vcenter{\hbox{\scalebox{0.8}{
\begin{tikzpicture}
	\begin{pgfonlayer}{nodelayer}
		\node [style=species] (C) at (-1, 0) {$\;\bullet\;$};
		\node [style=species] (B) at (-4, -0.5) {$\;\bullet\;$};
		\node [style=species] (A) at (-4, 0.5) {$\;\bullet\;$};
		\node [style=transition] (tau1) at (-2.5, 0.6) {$\big.\;\;\;\, $};
        \node [style=transition] (tau2) at (-2.5, -0.7) {$\big.\;\;\;\, $};
	\end{pgfonlayer}
	\begin{pgfonlayer}{edgelayer}
		\draw[->, line width=1.00] (-0.3, 0) to (0.2, 0);
		\draw [style=inarrow] (A) to (tau1);
		\draw [style=inarrow] (B) to (tau1);
        \draw [style=inarrow, bend left=15, looseness=1.00] (tau1) to (C);
        \draw [style=inarrow, bend left=15, looseness=1.00] (C) to (tau2);
        \draw [style=inarrow, bend left=15, looseness=1.00] (tau2) to (B); 
        \draw [style=inarrow, bend right =15, looseness=1.00] (tau2) to (B); 
	\end{pgfonlayer}
\end{tikzpicture}
\begin{tikzpicture}
	\begin{pgfonlayer}{nodelayer}
        \node [style=species] (C) at (-1, 0) {$\bullet\bullet$};
        \node [style=species] (B) at (-4, -0.5) {$\quad\;$};
        \node [style=species] (A) at (-4, 0.5) {$\quad\;$};
        \node [style=transition] (tau1) at (-2.5, 0.6) {$\big.\;\;\;\, $};
        \node [style=transition] (tau2) at (-2.5, -0.7) {$\big.\;\;\;\, $};
        \node [style=empty] at (0, 0) {{$\to$}};
	\end{pgfonlayer}
	\begin{pgfonlayer}{edgelayer}
        \draw[->, line width=1.00] (-0.3, 0) to (0.2, 0);
        \draw [style=inarrow] (A) to (tau1);
        \draw [style=inarrow] (B) to (tau1);
        \draw [style=inarrow, bend left=15, looseness=1.00] (tau1) to (C);
        \draw [style=inarrow, bend left=15, looseness=1.00] (C) to (tau2);
        \draw [style=inarrow, bend left=15, looseness=1.00] (tau2) to (B); 
        \draw [style=inarrow, bend right =15, looseness=1.00] (tau2) to (B); 
	\end{pgfonlayer}
\end{tikzpicture}
\begin{tikzpicture}
	\begin{pgfonlayer}{nodelayer}
		\node [style=species] (C) at (-1, 0) {$\;\bullet\;$};
		\node [style=species] (B) at (-4, -0.5) {$\bullet\bullet$};
		\node [style=species] (A) at (-4, 0.5) {$\quad\;$};
		\node [style=transition] (tau1) at (-2.5, 0.6) {$\big.\;\;\;\, $};
        \node [style=transition] (tau2) at (-2.5, -0.7) {$\big.\;\;\;\, $};
	\end{pgfonlayer}
	\begin{pgfonlayer}{edgelayer}
		\draw [style=inarrow] (A) to (tau1);
		\draw [style=inarrow] (B) to (tau1);
        \draw [style=inarrow, bend left=15, looseness=1.00] (tau1) to (C);
        \draw [style=inarrow, bend left=15, looseness=1.00] (C) to (tau2);
        \draw [style=inarrow, bend left=15, looseness=1.00] (tau2) to (B); 
        \draw [style=inarrow, bend right =15, looseness=1.00] (tau2) to (B); 
	\end{pgfonlayer}
\end{tikzpicture}
}}}
\]

Given a Petri net, then, how do we construct a network model $G \maps \S(C) \to \Cat$, and in particular, what is the set $C$? In a network model the elements of $C$ represent different kinds of agents. In the simplest scenario, these agents persist in time. Thus, it is natural to take $C$ to be some set of ``catalysts''. In chemistry, a reaction may require a catalyst to proceed, but it neither increases nor decrease the amount of this catalyst present. For a Petri net, \emph{catalysts} are species that are neither increased nor decreased in number by any transition. For example, species $a$ is a catalyst in the following Petri net, so we outline it in red:
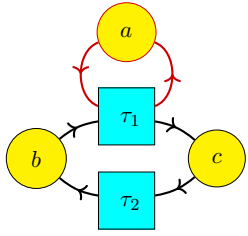
\[
\vcenter{\hbox{\scalebox{0.8}{
\begin{tikzpicture}
	\begin{pgfonlayer}{nodelayer}
		\node [style=species] (C) at (-1, -0.1) {$\;\;c\;\;$};
		\node [style=species] (B) at (-4, -0.1) {$\;\;b\;\;$};
		\node [style=catalyst] (A) at (-2.5, 2) {$\;\;a\;\;$};
		\node [style=transition] (tau1) at (-2.5, 0.6) {$\;\phantom{\Big{|}}\tau_1\;$};
        \node [style=transition] (tau2) at (-2.5, -0.8){$\;\phantom{\Big{|}}\tau_2\;$};
	\end{pgfonlayer}
	\begin{pgfonlayer}{edgelayer}
		\draw [style=inarrow, bend right=70, looseness=1.00, red] (A) to (tau1);
		\draw [style=inarrow, bend left=15, looseness=1.00] (B) to (tau1);
		\draw [style=inarrow, bend right=70, looseness=1.00, red] (tau1) to (A);
		\draw [style=inarrow, bend left=15, looseness=1.00] (tau1) to (C);
	    \draw [style=inarrow, bend left=15, looseness=1.00] (C) to (tau2);
        \draw [style=inarrow, bend left=15, looseness=1.00] (tau2) to (B); 
	\end{pgfonlayer}
\end{tikzpicture}
}}}
\]
but neither $b$ nor $c$ is a catalyst. The transition $\tau_1$ requires one token of type $a$ as input to proceed, but it also outputs one token of this type, so the total number of such tokens is unchanged. Similarly, the transition $\tau_2$ requires no tokens of type $a$ as input to proceed, and it also outputs no tokens of this type, so the total number of such tokens is unchanged. 

In Theorem \ref{thm:petri_network_model} we prove that given any Petri net $P$, and any subset $C$ of the catalysts of $P$, there is a network model $G \maps \S(C) \to \Cat$. An object $x \in \S(C)$ says how many tokens of each catalyst are present; $G(x)$ is then the subcategory of $FP$ where the objects are markings that have this specified amount of each catalyst, and morphisms are processes going between these. 

From the functor $G \maps \S(C) \to \Cat$ we can construct a category $\int G$ by the Grothendieck construction. Because $G$ is symmetric monoidal we can make $\int G$ into a symmetric monoidal category by the monoidal Grothendieck construction of \cref{ch:MonGroth}. The tensor product in $\int G$ describes doing processes in parallel. The category $\int G$ is similar to $FP$, but it is better suited to applications where agents each have their own individuality, because $FP$ is actually a \emph{commutative} monoidal category, where permuting agents has no effect at all, while $\int G$ is not so degenerate. In Theorem \ref{thm:grothendieck} we make this precise by more concretely describing $\int G$ as a symmetric monoidal category, and clarifying its relation to $FP$.

There are no morphisms between an object of $G(x)$ and an object of $G(x')$ unless $x \cong x'$, since no transitions can change the amount of catalysts present. The category $FP$ is thus a disjoint union, or more precisely a coproduct, of subcategories $FP_i$ where $i$, an element of free commutative monoid on $C$, specifies the amount of each catalyst present. The tensor product on $FP$ has the property that tensoring an object in $FP_i$ with one in $FP_j$ gives an object in $FP_{i+j}$, and similarly for morphisms. 

However, in Prop.\ \ref{prop:monoidal} we show that each subcategory $FP_i$ also has its own tensor product, which describes doing one process and then another while reusing catalyst tokens.  This tensor product makes $FP_i$ into a \emph{premonoidal} category---an interesting generalization of a monoidal category which we recall.  Finally, in Theorem \ref{thm:lift} we show that these monoidal structures define a lift of the functor $G \maps \S(C) \to \Cat$ to a functor $\hat{G} \maps \S(C) \to \PreMonCat$, where $\PreMonCat$ is the category of strict premonoidal categories.

\section{Petri Nets}
\label{sec:petri}

A Petri net generates a symmetric monoidal category whose objects are tensor products of species and whose morphisms are built from the transitions by repeatedly taking composites and tensor products. There is a long line of work on this topic starting with the papers of Meseguer--Montanari \cite{Petrinetsaremonoids} and Engberg--Winskel \cite{EW}, both dating to roughly 1990. It continues to this day, because the issues involved are surprisingly subtle \cite{DMM, Sassone, SassoneCategory, SassoneAxiomatization, Congruence, GeneralizedPetriNets}. In particular, there are various kinds of symmetric monoidal categories to choose from. Following the work of Master and Baez \cite{OpenPetriNets} we use `commutative' monoidal categories. These are just commutative monoid objects in $\Cat$, so their associator: 
\[  
    \alpha_{a, b, c} \colon (a \otimes b) \otimes c \stackrel{\sim}{\longrightarrow} a \otimes (b \otimes c), 
\]
their left and right unitor:
\[  
    \lambda_a \maps I \otimes a \stackrel{\sim}{\longrightarrow} a , \qquad
    \rho_a \maps a \otimes I \stackrel{\sim}{\longrightarrow} a , 
\]
and even their braiding:
\[   
    \sigma_{a, b} \maps a \otimes b \stackrel{\sim}{\longrightarrow} b \otimes a 
\]
are all identity morphisms. While every symmetric monoidal category is equivalent to one with trivial associator and unitors, this ceases to be true if we also require the braiding to be trivial. However, it seems that  Petri nets most naturally serve to present symmetric monoidal categories of this very strict sort. Thus, we shall describe a functor from the category of Petri nets to the category of commutative monoidal categories, which we call $\CMon\Cat$:
\[ 
    F \colon \Petri \to \CMon\Cat .
\]

To begin, let $\CMon$ be the category of commutative monoids and monoid homomorphisms. There is a forgetful functor from $\CMon$ to $\Set$ that sends commutative monoids to their underlying sets and monoid homomorphisms to their underlying functions. It has a left adjoint $\N \maps \Set \to \CMon$ sending any set $X$ to the free commutative monoid on $X$. An element $a \in \N[X]$ is formal linear combination of elements of $X$:
\[    a = \sum_{x \in X} a_x \, x ,\]
where the coefficients $a_x$ are natural numbers and all but finitely many are zero. The set $X$ naturally includes in $\N[X]$, and for any function $f \maps X \to Y$, $\N[f] \maps \N[X] \to \N[Y]$ is the unique monoid homomorphism that extends $f$. We often abuse language and use $\N[X]$ to mean the underlying set of the free commutative monoid on $X$. 

\begin{defn} 
    A \define{Petri net} is a pair of functions of the following form:
    \[\begin{tikzcd}
        T
        \arrow[r, shift left = 1, "s"]
        \arrow[r, shift right = 1, "t", swap]
        &
        \N[S].
    \end{tikzcd}\]
    We call $T$ the set of \define{transitions}, $S$ the set of \define{places} or \define{species}, $s$ the \define{source} function, and $t$ the \define{target} function. We call an element of $\N[S]$ a \define{marking} of the Petri net.
\end{defn}
For example, in this Petri net:
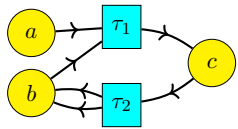
\[
\vcenter{\hbox{\scalebox{0.8}{
\begin{tikzpicture}
	\begin{pgfonlayer}{nodelayer}
        \node [style=species] (C) at (-1, 0) {$\quad\;$};
        \node [style=species] (B) at (-4, -0.5) {$\quad\;$};
        \node [style=species] (A) at (-4, 0.5) {$\quad\;$};
        \node [style=transition] (tau1) at (-2.5, 0.6) {$\big.\;\;\;\, $};
        \node [style=transition] (tau2) at (-2.5, -0.7) {$\big.\;\;\;\, $};
        \node [style = none] () at (-2.5, 0.6) {$\tau_1$};
        \node [style = none] () at (-2.5, -0.69) {$\tau_2$};
        \node [style = none] () at (-4, 0.5) {$a$};
        \node [style = none] () at (-4, -0.5) {$b$};
        \node [style = none] () at (-1, 0) {$c$};
	\end{pgfonlayer}
	\begin{pgfonlayer}{edgelayer}
        \draw [style=inarrow] (A) to (tau1);
        \draw [style=inarrow] (B) to (tau1);
        \draw [style=inarrow, bend left=15, looseness=1.00] (tau1) to (C);
        \draw [style=inarrow, bend left=15, looseness=1.00] (C) to (tau2);
        \draw [style=inarrow, bend left=15, looseness=1.00] (tau2) to (B); 
        \draw [style=inarrow, bend right =15, looseness=1.00] (tau2) to (B); 
	\end{pgfonlayer}
\end{tikzpicture}
}}}
\]
we have $S = \{a,b,c\}$, $T = \{\tau_1, \tau_2\}$, and 
\[
\begin{array}{ll} s(\tau_1) = a+b & t(\tau_1) = c \\
  s(\tau_2) = c & t(\tau_2) = 2b.
\end{array}
\]
The term `species' is used in applications of Petri nets to chemistry. Since the concept of `catalyst' also arose in chemistry, we henceforth use the term `species' rather than `places'. 

\begin{defn} 
    A \define{Petri net morphism} from the Petri net $P$ to the Petri net $P'$ is a pair of functions ($f \maps T \to T'$, $g \maps S \to S'$) such that the following diagrams commute:
    \[
    \begin{tikzcd}
        T
        \arrow[r, "s"]
        \arrow[d, "f", swap]
        &
        \N[S]
        \arrow[d, "\N\lbrack g \rbrack"]
        \\
        T'
        \arrow[r, swap, "s'"]
        &
        \N[S']
    \end{tikzcd}
    \quad
    \begin{tikzcd}
        T
        \arrow[r, "t"]
        \arrow[d, "f", swap]
        &
        \N[S]
        \arrow[d, "\N\lbrack g \rbrack"]
        \\
        T'
        \arrow[r, swap, "t'"]
        &
        \N[S']
    \end{tikzcd}
    \]
    Let $\Petri$ denote the category of Petri nets and Petri net morphisms with composition defined by \[(f, g) \circ (f', g') = (f \circ f', g \circ g').\]
\end{defn}

\begin{defn} 
  A \define{commutative monoidal category} is a commutative monoid object in $(\Cat, \times)$. Let $\CMon\Cat$ denote the category of commutative monoid objects in $(\Cat,\times)$.
\end{defn}

More concretely, a commutative monoidal category is a strict monoidal category for which $a \otimes b = b \otimes a$ for all pairs of objects and all pairs of morphisms, and the braid isomorphism $a \otimes b \to b \otimes a$ is the identity map.

Every Petri net $P = \left( s, t \maps T \to \N[S] \right)$ gives rise to a commutative monoidal category $FP$ as follows. We take the commutative monoid of objects $\Ob(FP)$ to be the free commutative monoid on $S$. We construct the commutative monoid of morphisms $\Mor(FP)$ as follows. First we generate morphisms recursively:
\begin{itemize}
    \item for every transition $\tau \in T$ we include a morphism $\tau \maps s(\tau) \to t(\tau)$;
    \item for any object $a$ we include a morphism $1_a \maps a \to a$;
    \item for any morphisms $f \maps a \to b$ and $g \maps a' \to b'$ we include a morphism denoted $f+g \maps a +a' \to b +b'$ to serve as their tensor product;
    \item for any morphisms $f \maps a \to b$ and $g \maps b \to c$ we include a morphism $g\circ f \maps a \to c$ to serve as their composite.
\end{itemize}
Then we quotient by an equivalence relation on morphisms that imposes the laws of a commutative monoidal category, obtaining the commutative monoid $\Mor(FP)$.	

Similarly, morphisms between Petri nets give morphisms between their commutative monoidal categories. Given a Petri net morphism 
\[
\begin{tikzcd}
    T
    \arrow[r, shift left = 1]
    \arrow[r, shift right = 1]
    \arrow[d, "f", swap]
    &
    \N[S]
    \arrow[d, "\N\lbrack g\rbrack"]
    \\
    T'
    \arrow[r, shift left = 1]
    \arrow[r, shift right = 1]
    &
    \N[S']
\end{tikzcd}
\]
we define the functor $F(f, g) \maps FP \to FP'$ to be $\N[g]$ on objects, and on morphisms to be the unique map extending $f$ that preserves identities, composition, and the tensor product. This functor is strict symmetric monoidal. 

\begin{prop}
    There is a functor $F \maps \Petri \to \CMon\Cat$ defined as above.
\end{prop}

\begin{proof}
    This is straightforward; the proof that $F$ is a left adjoint is harder \cite{GeneralizedPetriNets}, but we do not need this here.
\end{proof}

\section{Catalysts}
\label{sec:catalysts}

One thinks of a transition $\tau$ of a Petri net as a process that consumes the source species $s(\tau)$ and produces the target species $t(\tau)$. An example of something that can be represented by a Petri net is a chemical reaction network \cite{BaezBiamonte, RxNet}. Indeed, this is why Carl Petri originally invented them. A `catalyst' in a chemical reaction is a species that is necessary for the reaction to occur, or helps lower the activation energy for reaction, but is neither increased nor depleted by the reaction. We use a modest generalization of this notion, defining a \emph{catalyst} in a Petri net to be a species that is neither increased nor depleted by \emph{any} transition in the Petri net. 

Given a Petri net $s, t \maps T \to \N[S]$, recall that for any marking $a \in \N[S]$ we have
\[    a = \sum_{x \in S} a_x x \]
for certain coefficients $a_x \in \N$. Thus, for any transition $\tau$ of a Petri net, $s(\tau)_x$ is the coefficient of the place $x$ in the source of $\tau$, while $t(\tau)_x$ is its coefficient in the target of $\tau$.

\begin{defn}
    A species $x \in S$ in a Petri net $P = (s, t \maps T \to \N[S])$ is called a \define{catalyst} if $s(\tau)_x = t(\tau)_x$ for every transition $\tau \in T$. Let $S_{\mathrm{cat}} \subseteq S$ denote the set of catalysts in $P$.
\end{defn}

\begin{defn} 
    A \define{Petri net with catalysts} is a Petri net $P = (s, t \maps T \to \N[S])$ with a chosen subset $C \subseteq S_{\mathrm{cat}}$. We denote a Petri net $P$ with catalysts $C$ as $(P, C)$.
\end{defn}

Suppose we have a Petri net with catalysts $(P, C)$. Recall that the set of objects of $FP$ is the free commutative monoid $\N[C]$. We have a natural isomorphism 
\[
    \N[S] \cong \N[C] \times \N[S \setminus C]. 
\]
We write
\[
    \pi_C \maps \N[S] \to \N[C] 
\]
for the projection. Given any object $a \in FP$, $\pi_C(a)$ says how many catalysts of each species in $C$ occur in $a$.

\begin{defn} 
    Given a Petri net with catalysts $(P, C)$ and any $i \in \N[C]$, let $FP_i$ be the full subcategory of $FP$ whose objects are objects $a \in FP$ with $\pi_C (a) = i$.
\end{defn}

Morphisms in $FP_i$ describe processes that the Petri net can carry out with a specific fixed amount of every catalyst. Since no transition in $P$ creates or destroys any catalyst, if $f \maps a \to b$ is a morphism in $FP$ then 
\[
    \pi_C(a) = \pi_C(b) .
\]
Thus, $FP$ is the coproduct of all the 
subcategories $FP_i$:
\[
    FP \cong \coprod_{i \in \N[C]} FP_i
\]
as categories. The subcategories $FP_i$ are not generally monoidal subcategories because if $a, b \in FP$ and $a+b$ is their tensor product then 
\[
    \pi_C(a+b) = \pi_C(a) + \pi_C(b) 
\]
so for any $i, j \in \N[C]$ we have
\[
    a \in FP_i, \; b \in FP_j \Rightarrow a + b \in FP_{i + j}
\]
and similarly for morphisms.
Thus, we can think of $FP$ as a commutative monoidal category `graded' by $\N[C]$. But note we are free to reinterpret any process as using a \emph{greater} amount of various catalysts, by tensoring it with identity morphism on this \emph{additional} amount of 
catalysts. That is, given any morphism in $FP_i$, we can always tensor it with the identity on $j$ to get a morphism in $FP_{i+j}$.

Since $\N[C]$ is a commutative monoid we can think of it as a commutative monoidal category with only identity morphisms, and we freely do this in what follows. Network models rely on a similar but less trivial way of constructing a symmetric monoidal category from a set $C$. Namely, for any set $C$ there is a category $\S(C)$ for which:
\begin{itemize}
    \item Objects are formal expressions of the form
    \[
           c_1 \otimes \cdots \otimes c_n 
    \]
    for $n \in \N$ and $c_1, \dots, c_n \in C$. When $n = 0$ we write this expression as $I$.
    \item There exist morphisms
    \[
           f \maps c_1 \otimes \cdots \otimes c_m \to c'_1 \otimes \cdots \otimes c'_n 
    \]
    only if $m = n$, and in that case a morphism is a permutation $\sigma \in S_n$ such that $c'_{\sigma(i)} = c_i$ for all $i=1, \dots , n$.
    \item Composition is the usual composition of permutations. 
\end{itemize}
In short, an object of $\S(C)$ is a list of catalysts, possibly empty, and allowing repetitions. A morphism is a permutation that maps one list to another list. 

As shown in \cref{prop:free}, $\S(C)$ is the free strict symmetric monoidal category on the set $C$. There is thus a strict symmetric monoidal functor
\[
    p \maps \S(C) \to \N[C] 
\]
sending each object $c_1 \otimes \cdots \otimes c_n$ to the object $c_1 + \cdots + c_n$, and sending every morphism to an identity morphism. This can also be seen directly. In what follows, we use this functor $p$ to construct a lax symmetric monoidal functor $G \maps \S(C) \to \Cat$, where $\Cat$ is made symmetric monoidal using its cartesian product.

\begin{prop} 
    Given a Petri net with catalysts $(P, C)$, there exists a unique functor $G \maps \S(C) \to \Cat$ sending each object $x \in \S(C)$ to the category $FP_{p(x)}$ and each morphism in $\S(C)$ to an identity functor.
\end{prop}

\begin{proof} 
    The uniqueness is clear. For existence, note that since $\N[C]$ has only identity morphisms there is a functor $H \maps \N[C] \to \Cat$ sending each object $x \in \N[C]$ to the category $FP_{p(x)}$. If we compose $H$ with the functor $p \maps \S(C) \to \N[C]$ described above we obtain the functor $G$. 
\end{proof}

\begin{thm} 
\label{thm:petri_network_model}
    The functor $G \maps \S(C) \to \Cat$ becomes lax symmetric monoidal with the lax structure map
    \[
        \Phi_{x, y} \maps FP_{p(x)} \times FP_{p(y)} \to FP_{p(x \otimes y)}
    \]
    given by the tensor product in $FP$, and the map
    \[
        \phi \maps 1 \to FP_0 
    \]
    sending the unique object of the terminal category $1 \in \Cat$ to the unit for the tensor product in $FP$, which is the object $0 \in FP_0$.
\end{thm}

\begin{proof} 
    Recall that $G$ is the composite of $p \maps \S(C) \to \N[C]$ and $H \maps \N[C] \to \Cat$. The functor $p$ is strict symmetric monoidal. The functor $p$ is strict symmetric monoidal. One can check that the functor $H$ becomes lax symmetric monoidal if we equip it with the lax structure map
    \[
        FP_i \times FP_j \to FP_{i+j} 
    \]
    given by the tensor product in $FP$, and the map 
    \[
        1 \to FP_0 
    \]
    sending the unique object of $1 \in \Cat$ to the unit for the tensor product in $FP$, namely $0 \in \N[S] = \Ob(FP)$. Composing the lax symmetric monoidal functor $H$ and with the strict symmetric monoidal functor $p$, we obtain the lax symmetric monoidal functor $G$ described in the theorem statement. 
\end{proof}

We defined \define{$C$-colored network model} in \cref{ch:NetworkModels} to be a lax symmetric monoidal functor from $\S(C)$ to $\Cat$. 

\begin{defn}
    We call the $C$-colored network model $G \maps \S(C) \to \Cat$ of Theorem \ref{thm:petri_network_model} the \define{Petri network model} associated to the Petri net with catalysts $(P, C)$.
\end{defn}

\begin{expl}
\label{ex:1}
    The following Petri net $P$ has species $S = \{a, b, c, d, e\}$ and transitions $T = \{\tau_1, \tau_2\}$:
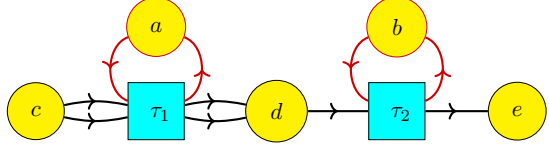
    \[
    \vcenter{\hbox{\scalebox{0.8}{
    \begin{tikzpicture}
    	\begin{pgfonlayer}{nodelayer}
    	    \node [style=catalyst] (A) at (-2, 2) {$\;\;a\;\;$};
    	    \node [style=catalyst] (B) at (2, 2) {$\;\;b\;\;$};
    		\node [style=species] (C) at (-4, 0.6) {$\;\;c\;\;$};
    		\node [style=species] (D) at (0, 0.6) {$\;\;d\;\;$};
    		\node [style=species] (E) at (4, 0.6) {$\;\;e\;\;$};
    		\node [style=transition] (tau1) at (-2, 0.6) {$\;\phantom{\Big{|}}\tau_1\;$};
            \node [style=transition] (tau2) at (2, 0.6) {$\;\phantom{\Big{|}}\tau_2\;$};
    	\end{pgfonlayer}
    	\begin{pgfonlayer}{edgelayer}
    		\draw [style=inarrow, bend right=70, looseness=1.00, red] (A) to (tau1);
    		\draw [style=inarrow, bend right=70, looseness=1.00, red] (tau1) to (A);
    		\draw [style=inarrow, bend right=70, looseness=1.00, red] (B) to (tau2);
    		\draw [style=inarrow, bend right=70, looseness=1.00, red] (tau2) to (B);
    		\draw [style=inarrow, bend left=12, looseness=1] (C) to (tau1);
    		\draw [style=inarrow, bend right=12, looseness=1] (C) to (tau1);
    	    \draw [style=inarrow, bend left=12, looseness=1] (tau1) to (D);
    	    \draw [style=inarrow, bend right=12, looseness=1] (tau1) to (D);
    	    \draw [style=inarrow] (D) to (tau2);
    	    \draw [style=inarrow] (tau2) to (E);
    	\end{pgfonlayer}
    \end{tikzpicture}
    }}}
    \]
    Species $a$ and $b$ are catalysts, and the rest are not. We thus can take $C = \{a, b\}$ and obtain a Petri net with catalysts $(P, C)$, which in turn gives a Petri network model $G \maps \S(C) \to \Cat$.  We outline catalyst species in red, and also draw the edges connecting them to transitions in red.
    
    Here is one possible interpretation of this Petri net. Tokens in $\, c\, $ represent people at a base on land, tokens in $\, d\, $ are people at the shore, and tokens in $\, e\, $ are people on a nearby island. Tokens in $\, a\, $ represent jeeps, each of which can carry two people at a time from the base to the shore and then return to the base. Tokens in $\, b\, $ represent boats that carry one person at a time from the shore to the island and then return. 
    
    Let us examine the effect of the functor $G \maps \S(C) \to \Cat$ on various objects of $\S(C)$. The object $a \in \S(C)$ describes a situation where there is one jeep present but no boats. The category $G(a)$ is isomorphic to $FX$, where $X$ is this Petri net:
    \[
    \vcenter{\hbox{\scalebox{0.8}
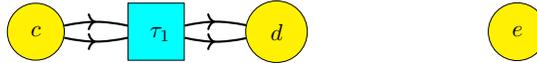
{
    \begin{tikzpicture}
    	\begin{pgfonlayer}{nodelayer}
    		\node [style=species] (C) at (-4, 0.6) {$\;\;c\;\;$};
    		\node [style=species] (D) at (0, 0.6) {$\;\;d\;\;$};
    		\node [style=species] (E) at (4, 0.6) {$\;\;e\;\;$};
    		\node [style=transition] (tau1) at (-2, 0.6) {$\;\phantom{\Big{|}}\tau_1\;$};
    	\end{pgfonlayer}
    	\begin{pgfonlayer}{edgelayer}
    		\draw [style=inarrow, bend left=12, looseness=1] (C) to (tau1);
    		\draw [style=inarrow, bend right=12, looseness=1] (C) to (tau1);
    	    \draw [style=inarrow, bend left=12, looseness=1] (tau1) to (D);
    	    \draw [style=inarrow, bend right=12, looseness=1] (tau1) to (D);
    	\end{pgfonlayer}
    \end{tikzpicture}
    }}}
    \]
    That is, people can go from the base to the shore in pairs, but they cannot go to the island. Similarly, the object $b$ describes a situation with one boat present but no jeeps, and the category $G(b)$ is isomorphic to $FY$, where $Y$ is this Petri net:
    \[
    \vcenter{\hbox{\scalebox{0.8}
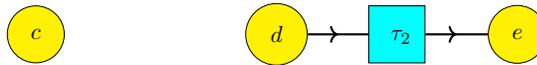
{
    \begin{tikzpicture}
    	\begin{pgfonlayer}{nodelayer}
    		\node [style=species] (C) at (-4, 0.6) {$\;\;c\;\;$};
    		\node [style=species] (D) at (0, 0.6) {$\;\;d\;\;$};
    		\node [style=species] (E) at (4, 0.6) {$\;\;e\;\;$};
            \node [style=transition] (tau2) at (2, 0.6) {$\;\phantom{\Big{|}}\tau_2\;$};
    	\end{pgfonlayer}
    	\begin{pgfonlayer}{edgelayer}
    	    \draw [style=inarrow] (D) to (tau2);
    	    \draw [style=inarrow] (tau2) to (E);
    	\end{pgfonlayer}
    \end{tikzpicture}
    }}}
    \]
    Now people can only go from the shore to the island, one at a time.
        
    The object $a \otimes b \in \S(C)$ describes a situation with one jeep and one boat. The category $G(a \otimes b)$ is isomorphic to $FZ$ for this Petri net $Z$:
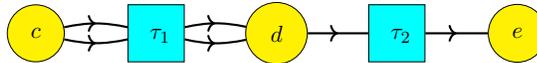
    \[
    \vcenter{\hbox{\scalebox{0.8}{
    \begin{tikzpicture}
    	\begin{pgfonlayer}{nodelayer}
    		\node [style=species] (C) at (-4, 0.6) {$\;\;c\;\;$};
    		\node [style=species] (D) at (0, 0.6) {$\;\;d\;\;$};
    		\node [style=species] (E) at (4, 0.6) {$\;\;e\;\;$};
    		\node [style=transition] (tau1) at (-2, 0.6) {$\;\phantom{\Big{|}}\tau_1\;$};
            \node [style=transition] (tau2) at (2, 0.6) {$\;\phantom{\Big{|}}\tau_2\;$};
    	\end{pgfonlayer}
    	\begin{pgfonlayer}{edgelayer}
    		\draw [style=inarrow, bend left=12, looseness=1] (C) to (tau1);
    		\draw [style=inarrow, bend right=12, looseness=1] (C) to (tau1);
    	    \draw [style=inarrow, bend left=12, looseness=1] (tau1) to (D);
    	    \draw [style=inarrow, bend right=12, looseness=1] (tau1) to (D);
    	    \draw [style=inarrow] (D) to (tau2);
    	    \draw [style=inarrow] (tau2) to (E);
    	\end{pgfonlayer}
    \end{tikzpicture}
    }}}
    \]
    Now people can go from the base to the shore in pairs and also go from the shore to the island one at a time.
    
    Surprisingly, an object $x \in \S(C)$ with additional jeeps and/or boats always produces a category $G(x)$ that is isomorphic to one of the three just shown: $G(a), G(b)$ and $G(a \otimes b)$. For example, consider the object $b \otimes b \in \S(C)$, where there are two boats present but no jeeps. There is an isomorphism of categories
    \[  
         - + b \maps G(b)  \to  G(b \otimes b) 
    \]
    defined as follows. Recall that $G(b) = FP_b$ and $G(b \otimes b) = FP_{b+b}$, where $FP_b$ and $FP_{b+b}$ are subcategories of $FP$. The functor
    \[
        - + b \maps FP_b \to FP_{b+b}
    \]
    sends each object $x \in FP_b$ to the object $ x + b$, and sends each morphism $f \maps x \to y$ in $FP_b$ to the morphism $1_b + f \maps b + x \to b + y$. That this defines a functor is clear; the surprising part is that it is an isomorphism. One might have thought that the presence of a second boat would enable one to carry out a given task in more different ways.
        
    Indeed, while this is true in real life, the category $FP$ is \emph{commutative} monoidal, so tokens of the same species have no `individuality': permuting them has no effect. There is thus, for example, no difference between the following two morphisms in $FP_{b+b}$:
    \begin{itemize}
        \item using one boat to transport one person from the base to shore and another boat to transport another person, and
        \item using one boat to transport first one person and then another.
    \end{itemize} 
    
    It is useful to draw morphisms in $FP$ as string diagrams, since such diagrams serve as a general notation for morphisms in monoidal categories \cite{JoyalStreet}.   For expository treatments, see \cite{BaezStay,Selinger}.   The rough idea is that objects of a monoidal category are drawn as labelled wires, and a morphism $f \maps x_1 \otimes \cdots \otimes x_m \to y_1 \otimes \cdots \otimes y_n$ is drawn as a box with $m$ wires coming in on top and $n$ wires coming out at the bottom.  Composites of morphisms are drawn by attaching output wires of one morphism to input wires of another, while tensor products of morphisms are drawn by setting pictures side by side.  In symmetric monoidal categories, the braiding is drawn as a crossing of wires.   The rules governing string diagrams let us manipulate them while not changing the morphisms they denote.  In the case of symmetric monoidal categories, these rules are well known \cite{JoyalStreet,Selinger}.   For \emph{commutative} monoidal categories there is one additional rule:
    \[
    \vcenter{\hbox{\scalebox{0.8}{
    \begin{tikzpicture}
    	\begin{pgfonlayer}{nodelayer}
    		\node [style=empty] (a) at (0, 1) {$x$};
    		\node [style=empty] (b) at (1, 1) {$y$};
            \node [style=empty] (c) at (0, -1) {$y$};
            \node [style=empty] (d) at (1, -1) {$x$};
    	\end{pgfonlayer}
    	\begin{pgfonlayer}{edgelayer}
    		\draw [line width=1.5 pt] (a) to (d);
    		\draw [line width=1.5 pt] (b) to (c);
        	\end{pgfonlayer}
    \end{tikzpicture}
    }}}
    =
    \vcenter{\hbox{\scalebox{0.8}{
    \begin{tikzpicture}
    	\begin{pgfonlayer}{nodelayer}
    		\node [style=empty] (a) at (0, 1) {$x$};
    		\node [style=empty] (b) at (1, 1) {$y$};
            \node [style=empty] (c) at (1, -1) {$y$};
            \node [style=empty] (d) at (0, -1) {$x$};
    	\end{pgfonlayer}
    	\begin{pgfonlayer}{edgelayer}
    		\draw [line width=1.5 pt] (a) to (d);
    		\draw [line width=1.5 pt] (b) to (c);
        	\end{pgfonlayer}
    \end{tikzpicture}
    }}}
    \]
    This says both that $x \otimes y = y \otimes x$ and that the braiding $\sigma_{x,y} \maps x \otimes y \to y \otimes x$ is the identity.
    
    Here is the string diagram notation for the equation we mentioned between two morphisms in $FP$:
    \[
    \vcenter{\hbox{\scalebox{0.8}
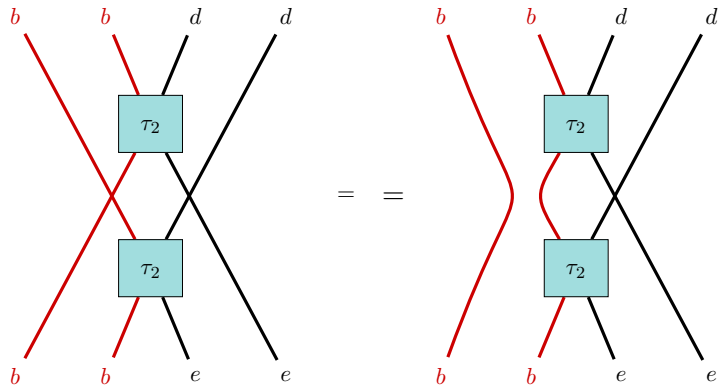
{
    \begin{tikzpicture}
    	\begin{pgfonlayer}{nodelayer}
    		\node [style=empty, red] (b) at (0, 4) {$b$};
    		\node [style=empty, red] (b') at (1.5, 4) {$b$};
    		\node [style=empty] (d) at (3, 4) {$d$};
    		\node [style=empty] (d') at (4.5, 4) {$d$};
    		\node [style=morphism] (tau1) at (2.25, 2.2) {$\;\phantom{\Big|}\tau_2\phantom{\Big|}\;$};
    		\node [style=empty] (equals) at (5.5, 1) {$=$};
            \node [style=morphism] (tau2) at (2.25, -0.2) {$\;\phantom{\Big|}\tau_2\phantom{\Big|\;}$};
            \node [style=empty, red] (b'') at (0, -2) {$b$};
            \node [style=empty, red] (b''') at (1.5, -2) {$b$};
            \node [style=empty] (e) at (3, -2) {$e$};
            \node [style=empty] (e') at (4.5, -2) {$e$};
    	\end{pgfonlayer}
    	\begin{pgfonlayer}{edgelayer}
    		\draw [line width=1.5 pt, red] (b) to (tau2);
    		\draw [line width=1.5 pt, red] (b') to (tau1);
    		\draw [line width=1.5 pt] (d) to (tau1);
    		\draw [line width=1.5 pt] (d') to (tau2);
    		\draw [line width=1.5 pt, red] (tau1) to (b'');
    		\draw [line width=1.5 pt] (tau1) to (e');
    		\draw [line width=1.5 pt, red] (tau2) to (b''');
    		\draw [line width=1.5 pt] (tau2) to (e);
        	\end{pgfonlayer}
    \end{tikzpicture}
    }}}
    =
    \vcenter{\hbox{\scalebox{0.8}{
    \begin{tikzpicture}
    	\begin{pgfonlayer}{nodelayer}
    		\node [style=empty, red] (b) at (0, 4) {$b$};
    		\node [style=empty, red] (b') at (1.5, 4) {$b$};
    		\node [style=empty] (d) at (3, 4) {$d$};
    		\node [style=empty] (d') at (4.5, 4) {$d$};
    		\node [style=morphism] (tau1) at (2.25, 2.2) {$\;\phantom{\Big|}\tau_2\phantom{\Big|}\;$};
            \node [style=morphism] (tau2) at (2.25, -0.2) {$\;\phantom{\Big|}\tau_2\phantom{\Big|\;}$};
            \node [style=empty, red] (b'') at (0, -2) {$b$};
            \node [style=empty, red] (b''') at (1.5, -2) {$b$};
            \node [style=empty] (e) at (3, -2) {$e$};
            \node [style=empty] (e') at (4.5, -2) {$e$};
    	\end{pgfonlayer}
    	\begin{pgfonlayer}{edgelayer}
    		\draw [line width=1.5 pt, bend left=20, looseness=2, red] (b) to (b'');
    		\draw [line width=1.5 pt, red] (b') to (tau1);
    		\draw [line width=1.5 pt] (d) to (tau1);
    		\draw [line width=1.5 pt] (d') to (tau2);
    		\draw [line width=1.5 pt, bend right =30, looseness=1.5, red] (tau1) to (tau2);
    		\draw [line width=1.5 pt] (tau1) to (e');
    		\draw [line width=1.5 pt, red] (tau2) to (b''');
    		\draw [line width=1.5 pt] (tau2) to (e);
        	\end{pgfonlayer}
        \end{tikzpicture}
        }}}
    \]
    We draw the object $b$ (standing for a boat) in red to emphasize that it serves as a catalyst.  At left we are first using one boat to transport one person from the base to shore, and then using another boat to transport another person. At right we are using the same boat to transport first one person and then another, while another boat stands by and does nothing.  These morphisms are equal because they differ only by the presence of the braiding $\sigma_{b,b} \maps b + b \to b + b$ in the left hand side, and this is an identity morphism.
\end{expl}

The above example illustrates an important point: in the commutative monoidal category $FP$, permuting catalyst tokens has no effect. Next we construct a symmetric monoidal category $\int G$ in which permuting such tokens has a nontrivial effect.  One reason for wanting this is that in applications, the catalyst tokens may represent agents with their own individuality. For example, when directing a boat to transport a person from base to shore, we need to say \emph{which} boat should do this.  For this we need a symmetric monoidal category that gives the catalyst tokens a nontrivial braiding.

To create this category, we use the symmetric monoidal Grothendieck construction of \cref{ch:MonGroth}. Given any symmetric monoidal category $X$ and any lax symmetric monoidal functor $F \maps X \to \Cat$, this construction gives a symmetric monoidal category $\int F$ equipped with a functor (indeed an opfibration) $\int F \to X$. In \cref{ch:NetworkModels}, we used this construction to build an operad from any network model, whose operations are ways to assemble larger networks from smaller ones. Now this construction has a new significance.

Starting from a Petri network model $G \maps \S(C) \to \Cat$, the symmetric monoidal Grothendieck construction gives a symmetric monoidal category $\int G$ in which:
\begin{itemize}
    \item an object is a pair $(x, a)$ where $x \in \S(C)$ and $a \in FP_{p(x)}$.
    \item a morphism from $(x, a)$ to $(x', a')$ is a pair $(\sigma, f)$ where $\sigma \maps x \to x'$ is a morphism in $\S(C)$ and $f \maps a \to a'$ is a morphism in $FP$.
    \item morphisms are composed componentwise.
    \item the tensor product is computed componentwise: in particular, the tensor product of objects $(x, a)$ and $(x', a')$ is $(x \otimes x', a + a')$.
    \item the associators, unitors and braiding are also computed componentwise (and hence are trivial in the second component, since $FP$ is a commutative monoidal category).
\end{itemize}
The functor $\int G \to \S(C)$ simply sends each pair to its first component.

This is simpler than one typically expects from the Grothendieck construction. There are two main reasons: first, $G$ maps every morphism in $\S(C)$ to an identity morphism in $\Cat$, and second, the lax structure map for $G$ is simply the tensor product in $FP$. However, this construction still has an important effect: it makes the process of switching two tokens of the same catalyst species into a nontrivial morphism in $\int G$.
More formally, we have:

\begin{thm}
\label{thm:grothendieck}
    If $G \maps \S(C) \to \Cat$ is the Petri network model associated to the Petri net with catalysts $(P, C)$, then $\int G$ is equivalent, as a symmetric monoidal category, to the full subcategory of $\S(C) \times FP$ whose objects are those of the form $(x, a)$ with $x \in \S(C)$ and $a \in FP_{p(x)}$. 
\end{thm}

\begin{proof} 
    One can read this off from the description of $\int G$ given above.
\end{proof}

The difference between $\int G$ and $FP$ is that the former category keeps track of processes where catalyst tokens are permuted, while the latter category treats them as identity morphisms.  In the terminology of Glabbeek and Plotkin, $\int G$  implements the `individual token philosophy' on catalysts, in which permuting tokens of the same catalyst is regarded as having a nontrivial effect \cite{GlabbeekPlotkin}.  By contrast, $FP$ implements the `collective token philosophy', where all that matters is the \emph{number} of tokens of each catalyst, and permuting them has no effect.  

There is a map from $\int G$ to $FP$ that forgets the individuality of the catalyst tokens. A morphism in $\int G$ is a pair $(\sigma, f)$ where $\sigma \maps x \to x'$ is a morphism in $\S(C)$ and $f \maps a \to a'$ is a morphism in $FP$ with $a \in G(x), a' \in G(x')$. There is a symmetric monoidal functor
\[
    \textstyle{\int} G \to FP 
\]
that discards this extra information, mapping $(\sigma, f)$ to $f$. The symmetric monoidal Grothendieck construction also gives a symmetric monoidal opfibration 
\[ \textstyle{\int} G \to \S(C)  \]
which maps $(\sigma, f)$ to $\sigma$, by \cref{ch:MonGroth}.

\begin{expl}
\label{ex:2}
Let $(P, C)$ be the Petri net with catalysts in Ex.\ \ref{ex:1}, and $G \maps \S(C) \to \Cat$ the resulting Petri network model. In $\int G$ the following two morphisms are not equal:
    \[
    \vcenter{\hbox{\scalebox{0.8}
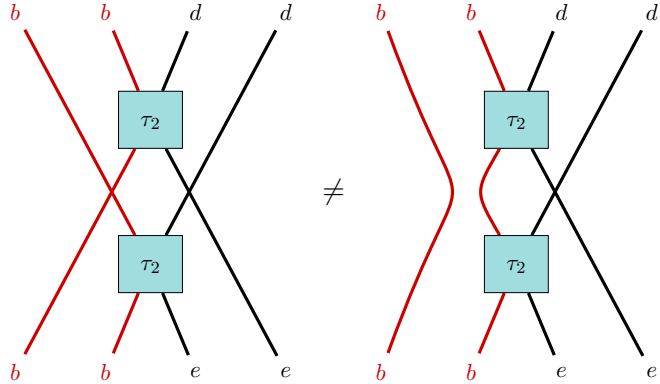
{
    \begin{tikzpicture}
    	\begin{pgfonlayer}{nodelayer}
    		\node [style=empty, red] (b) at (0, 4) {$b$};
    		\node [style=empty, red] (b') at (1.5, 4) {$b$};
    		\node [style=empty] (d) at (3, 4) {$d$};
    		\node [style=empty] (d') at (4.5, 4) {$d$};
    		\node [style=morphism] (tau1) at (2.25, 2.2) {$\;\phantom{\Big|}\tau_2\phantom{\Big|}\;$};
            \node [style=morphism] (tau2) at (2.25, -0.2) {$\;\phantom{\Big|}\tau_2\phantom{\Big|\;}$};
            \node [style=empty, red] (b'') at (0, -2) {$b$};
            \node [style=empty, red] (b''') at (1.5, -2) {$b$};
            \node [style=empty] (e) at (3, -2) {$e$};
            \node [style=empty] (e') at (4.5, -2) {$e$};
    	\end{pgfonlayer}
    	\begin{pgfonlayer}{edgelayer}
    		\draw [line width=1.5 pt, red] (b) to (tau2);
    		\draw [line width=1.5 pt, red] (b') to (tau1);
    		\draw [line width=1.5 pt] (d) to (tau1);
    		\draw [line width=1.5 pt] (d') to (tau2);
    		\draw [line width=1.5 pt, red] (tau1) to (b'');
    		\draw [line width=1.5 pt] (tau1) to (e');
    		\draw [line width=1.5 pt, red] (tau2) to (b''');
    		\draw [line width=1.5 pt] (tau2) to (e); 
    	\end{pgfonlayer}
    \end{tikzpicture}
    }}}
    \ne
    \vcenter{\hbox{\scalebox{0.8}{
    \begin{tikzpicture}
    	\begin{pgfonlayer}{nodelayer}
    		\node [style=empty, red] (b) at (0, 4) {$b$};
    		\node [style=empty, red] (b') at (1.5, 4) {$b$};
    		\node [style=empty] (d) at (3, 4) {$d$};
    		\node [style=empty] (d') at (4.5, 4) {$d$};
    		\node [style=morphism] (tau1) at (2.25, 2.2) {$\;\phantom{\Big|}\tau_2\phantom{\Big|}\;$};
            \node [style=morphism] (tau2) at (2.25, -0.2) {$\;\phantom{\Big|}\tau_2\phantom{\Big|\;}$};
            \node [style=empty, red] (b'') at (0, -2) {$b$};
            \node [style=empty, red] (b''') at (1.5, -2) {$b$};
            \node [style=empty] (e) at (3, -2) {$e$};
            \node [style=empty] (e') at (4.5, -2) {$e$};
    	\end{pgfonlayer}
    	\begin{pgfonlayer}{edgelayer}
    		\draw [line width=1.5 pt, bend left=20, looseness=2, red] (b) to (b'');
    		\draw [line width=1.5 pt, red] (b') to (tau1);
    		\draw [line width=1.5 pt] (d) to (tau1);
    		\draw [line width=1.5 pt] (d') to (tau2);
    		\draw [line width=1.5 pt, bend right =30, looseness=1.5, red] (tau1) to (tau2);
    		\draw [line width=1.5 pt] (tau1) to (e');
    		\draw [line width=1.5 pt, red] (tau2) to (b''');
    		\draw [line width=1.5 pt] (tau2) to (e); (a) to (f);
    	\end{pgfonlayer}
    \end{tikzpicture}
    }}}
    \]
because the braiding of catalyst species in $\int G$ is nontrivial.  This says that in $\int G$ we consider these two processes as different:
\begin{itemize}
    \item using one boat to transport one person from the base to shore and another boat to transport another person, and
    \item using one boat to transport first one person and then another.
\end{itemize} 

On the other hand, in $\int G$ we have
    \[
    \vcenter{\hbox{\scalebox{0.8}
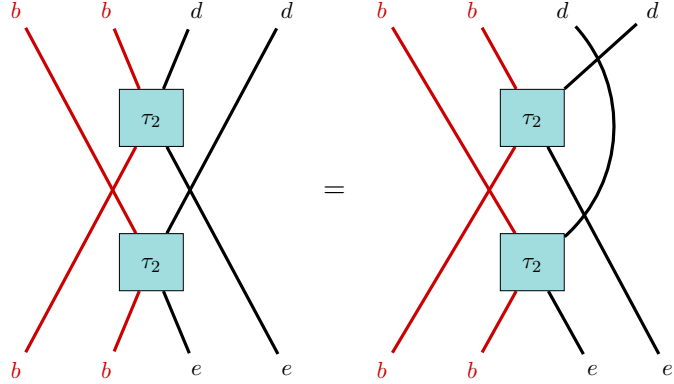
{
    \begin{tikzpicture}
    	\begin{pgfonlayer}{nodelayer}
    		\node [style=empty, red] (b) at (0, 4) {$b$};
    		\node [style=empty, red] (b') at (1.5, 4) {$b$};
    		\node [style=empty] (d) at (3, 4) {$d$};
    		\node [style=empty] (d') at (4.5, 4) {$d$};
    		\node [style=morphism] (tau1) at (2.25, 2.2) {$\;\phantom{\Big|}\tau_2\phantom{\Big|}\;$};
            \node [style=morphism] (tau2) at (2.25, -0.2) {$\;\phantom{\Big|}\tau_2\phantom{\Big|\;}$};
            \node [style=empty, red] (b'') at (0, -2) {$b$};
            \node [style=empty, red] (b''') at (1.5, -2) {$b$};
            \node [style=empty] (e) at (3, -2) {$e$};
            \node [style=empty] (e') at (4.5, -2) {$e$};
    	\end{pgfonlayer}
    	\begin{pgfonlayer}{edgelayer}
    		\draw [line width=1.5 pt, red] (b) to (tau2);
    		\draw [line width=1.5 pt, red] (b') to (tau1);
    		\draw [line width=1.5 pt] (d) to (tau1);
    		\draw [line width=1.5 pt] (d') to (tau2);
    		\draw [line width=1.5 pt, red] (tau1) to (b'');
    		\draw [line width=1.5 pt] (tau1) to (e');
    		\draw [line width=1.5 pt, red] (tau2) to (b''');
    		\draw [line width=1.5 pt] (tau2) to (e);
    	\end{pgfonlayer}
    \end{tikzpicture}
    }}}
    =
    \vcenter{\hbox{\scalebox{0.8}{
    \begin{tikzpicture}
    	\begin{pgfonlayer}{nodelayer}
    		\node [style=empty, red] (b) at (0, 4) {$b$};
    		\node [style=empty, red] (b') at (1.5, 4) {$b$};
    		\node [style=empty] (d) at (3, 4) {$d$};
    		\node [style=empty] (d') at (4.5, 4) {$d$};
    		\node [style=morphism] (tau1) at (2.5, 2.2) {$\;\phantom{\Big|}\tau_2\phantom{\Big|}\;$};
            \node [style=morphism] (tau2) at (2.5, -0.2) {$\;\phantom{\Big|}\tau_2\phantom{\Big|\;}$};
            \node [style=empty, red] (b'') at (0, -2) {$b$};
            \node [style=empty, red] (b''') at (1.5, -2) {$b$};
            \node [style=empty] (e) at (3.5, -2) {$e$};
            \node [style=empty] (e') at (4.75, -2) {$e$};
    	\end{pgfonlayer}
    	\begin{pgfonlayer}{edgelayer}
    		\draw [line width=1.5 pt, red] (b) to (tau2);
    		\draw [line width=1.5 pt, red] (b') to (tau1);
    		\draw [line width=1.5 pt, bend left=45, looseness=1] (d) to (tau2);
    		\draw [line width=1.5 pt] (d') to (tau1);
    		\draw [line width=1.5 pt, red] (tau1) to (b'');
    		\draw [line width=1.5 pt] (tau1) to (e');
    		\draw [line width=1.5 pt, red] (tau2) to (b''');
    		\draw [line width=1.5 pt] (tau2) to (e);
    	\end{pgfonlayer}
    \end{tikzpicture}
    }}}
    \]
because these morphisms differ only by two people on the shore switching place before they board the boats, and the braiding of non-catalyst species is the identity.  In short, the $\int G$ construction implements the individual token philosophy only for catalyst tokens; tokens of other species are governed by the collective token philosophy.
\end{expl}

\section{Premonoidal Categories}
\label{sec:premonoidal}

We have seen that for a Petri net $P$, a choice of catalysts $C$ lets us write the category $FP$ as a coproduct of subcategories $FP_i$, one for each possible amount $i \in \N[C]$ of the catalysts. The subcategory $FP_i$ is only a \emph{monoidal} subcategory when $i = 0$. Indeed, only $FP_0$ contains the monoidal unit of $FP$.  However, we shall see that each subcategory $FP_i$ can be given the structure of a \emph{premonoidal} category, as defined by Power and Robinson \cite{PowerRobinson}.  We motivate our use of this structure by describing two failed attempts to make $FP_i$ into a monoidal category.

Given two morphisms in $FP_i$ we typically cannot carry out these two processes simultaneously, because of the limited availability of catalysts. But we can do first one and then the other.  For example, imagine that two people are trying to walk through a doorway, but the door is only wide enough for one person to walk through. The door is a resource that is not depleted by its use, and thus a catalyst. Both people can use the door, but not at the same time: they must make an arbitrary choice of who goes first. 

We can attempt to define a tensor product on $FP_i$ using this idea.  Fix some amount of catalysts $i \in \N[C]$. Objects of $FP_i$ are of the form $i + a$ with $a \in \N[S - C]$.  On objects we define
\[(i + a) \otimes_i (i + a' ) = i +  a + a'.\]
The unit object for $\otimes_i$ is therefore $i + 0$, or simply $i$. For morphisms
\begin{align*} 
    f &\maps i + a \to i + b \\ 
    f' &\maps i +  a' \to i + b'
\end{align*}
we define 
\[
    f \otimes_i f' = (f + 1_{b'}) \circ (1_{a} + f').
\]

The tensor product
$f \otimes_i f' = (f + 1_{b'}) \circ (1_{a} + f') $
of morphisms in $FP_i$ involves an arbitrary choice: namely, the choice to do $f'$ first. This is perhaps clearer if we draw this morphism as a string diagram in $FP$.
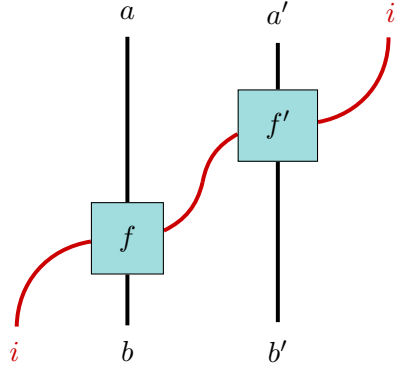
\[
\vcenter{\hbox{\scalebox{1}{
\begin{tikzpicture}
	\begin{pgfonlayer}{nodelayer}
		\node [style=empty, red] (i) at (0, -3) {$i$};
		\node [style=empty] (a) at (1.5, 1.5) {$a$};
		\node [style=empty] (a') at (3.5, 1.5) {$a'$};
		\node [style=empty] (m) at (2.5, -0.75) {};
		\node [style=morphism] (f) at (1.5, -1.5) {$\;\phantom{\Big|}f\phantom{\Big|}\;$};
        \node [style=morphism] (f') at (3.5, 0) {$\;\phantom{\Big|}f'\phantom{\Big|\;}$};
        \node [style=empty] (b) at (1.5, -3) {$b$};
        \node [style=empty] (b') at (3.5, -3) {$b'$};
        \node [style=empty, red] (i') at (5, 1.5) {$i$};
	\end{pgfonlayer}
	\begin{pgfonlayer}{edgelayer}
		\draw [line width=1.5 pt, bend left = 40, looseness=1, red] (i') to (f');
		\draw [line width=1.5 pt] (a') to (f');
		\draw [line width=1.5 pt] (a) to (f);
		\draw [line width=1.5 pt, bend right = 25, red] (f) to (m.center);
		\draw [line width=1.5 pt, bend left = 25, red] (m.center) to (f');
		\draw [line width=1.5 pt] (f') to (b');
		\draw [line width=1.5 pt, bend right = 40, looseness=1, red] (f) to (i);
		\draw [line width=1.5 pt] (f) to (b);
	\end{pgfonlayer}
\end{tikzpicture}
}}}
\]
If instead we choose to do $f$ first, we can define a tensor product ${}_i \otimes$ which is the same on objects but given on morphisms by
\[     
    f \, {}_i \!\otimes f' = (1_b + f') \circ (f + 1_{a'})  .
\]
It looks like this:
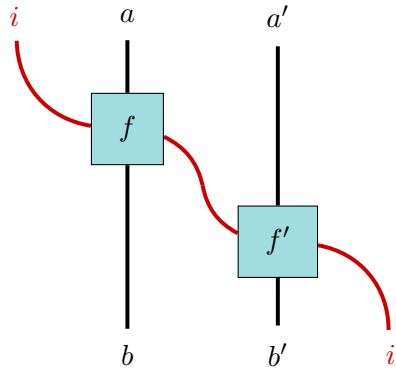
\[
\vcenter{\hbox{\scalebox{1}{
\begin{tikzpicture}
	\begin{pgfonlayer}{nodelayer}
		\node [style=empty, red] (i) at (0, 1.5) {$i$};
		\node [style=empty] (a) at (1.5, 1.5) {$a$};
		\node [style=empty] (a') at (3.5, 1.5) {$a'$};\node [style=empty] (m) at (2.5, -0.75) {};
		\node [style=morphism] (f) at (1.5, 0) {$\;\phantom{\Big|}f\phantom{\Big|}\;$};
        \node [style=morphism] (f') at (3.5, -1.5) {$\;\phantom{\Big|}f'\phantom{\Big|\;}$};
        \node [style=empty] (b) at (1.5, -3) {$b$};
        \node [style=empty] (b') at (3.5, -3) {$b'$};
        \node [style=empty, red] (i') at (5, -3) {$i$};
	\end{pgfonlayer}
	\begin{pgfonlayer}{edgelayer}
		\draw [line width=1.5 pt, bend right = 40, looseness=1, red] (i) to (f);
		\draw [line width=1.5 pt] (a') to (f');
		\draw [line width=1.5 pt] (a) to (f);
		\draw [line width=1.5 pt, bend left = 25, red] (f) to (m.center);
		\draw [line width=1.5 pt, bend right = 25, red] (m.center) to (f');
		\draw [line width=1.5 pt] (f') to (b');
		\draw [line width=1.5 pt, bend left = 40, looseness=1, red] (f') to (i');
		\draw [line width=1.5 pt] (f) to (b);
	\end{pgfonlayer}
\end{tikzpicture}
}}}
\]
Unfortunately, neither of these tensor products makes $FP_i$ into a monoidal category! Each makes the set of objects $\Ob(FP_i)$ and the set of morphisms $\Mor(FP_i)$ into a monoid in such a way that the source and target maps $s,t \maps \Mor(FP_i) \to \Ob(FP_i)$, as well as the identity-assigning map $i \maps \Ob(FP_i) \to \Mor(FP_i)$, are monoid homomorphisms.  The problem is that neither obeys the interchange law, so neither of these tensor products defines a functor from $FP_i \times FP_i$ to $FP_i$. For example, 
\[     
    (1 \otimes_i f')\circ (f \otimes_i 1) \ne (f \otimes_i 1) \circ (1 \otimes_i f') .
\]
The other tensor product suffers from the same problem.

What is going on here?  It turns out that $FP_i$ is a `strict premonoidal category'.   While these structures first arose in computer science \cite{PowerRobinson}, they are also mathematically natural, for the following reason.  There are only two symmetric monoidal closed structures on $\Cat$, up to isomorphism \cite{FKL}.  One is the the cartesian product.  The other is the `funny tensor product' \cite{Weber}.   A monoid in $\Cat$ with its cartesian product is a strict monoidal category, but a monoid in $\Cat$ with its funny tensor product is a strict premonoidal category. The \define{funny tensor product} $\C \square \D$ of categories $\C$ and $\D$ is defined as the following pushout in $\Cat$:
\[\begin{tikzcd}
    \C_0 \times \D_0
    \arrow[d, "i \times 1", swap]
    \arrow[r, "1 \times j"]
    &
    \C_0 \times \D
    \arrow[d]
    \\
    \C \times \D_0 
    \arrow[r]
    &
    \C \square \D
\end{tikzcd}\]
Here $\C_0$ is the subcategory of $\C$ consisting of all the objects and only identity
morphisms, $i \maps \C_0 \to \C$ is the inclusion, and similarly for $j \maps \D_0 \to \D$.
Thus, given morphisms $f \maps x \to y$ in $\C$ and $f' \maps x' \to y'$ in $\C$, the category $C \square D$ in contains a square of the form
\[    \begin{tikzcd}
        x \square x'
        \arrow[d, swap, "f \square 1"]   \arrow[r, "1 \square f' "]
        &
        x \square y'
        \arrow[d,  "f \square 1"]
        \\
        x' \square y
        \arrow[r, swap, "1 \square f' "]
        &
        x' \square y',
  \end{tikzcd}\]
but in general this square does not commute, unlike the corresponding square in
$\C \times \D$.

\begin{defn}
A \define{strict premonoidal category} is a category $\C$ equipped with a functor
$\boxtimes \maps \C \square \C \to \C$ that obeys the associative law and an object $I \in \C$ that serves as a left and right unit for $\boxtimes$. 
\end{defn}

Given two morphisms $f \maps x \to y$, $f' \maps x' \to y'$ in a strict premonoidal category $\C$ we obtain a square
 \[\begin{tikzcd}
    x \boxtimes x'
    \arrow[d, swap, "f \boxtimes 1"]   \arrow[r,  "1 \boxtimes f' "]
    &
    x \boxtimes y'
    \arrow[d,  "f \boxtimes 1"]
    \\
    x' \boxtimes y
    \arrow[r, swap, "1 \boxtimes f' "]
    &
    x' \boxtimes y',
\end{tikzcd}\]
but this square may not commute.  There are thus two candidates for a morphism from
$x \boxtimes x'$ to $y \boxtimes y'$.  When these always agree, we can make $\C$
monoidal by setting $f \boxtimes f'$ equal to either (and thus both) of these candidates.   We shall give $FP_i$ a strict premonoidal structure where these two candidates do not agree: one is $f \otimes_i f'$ while the other is $f \, {}_i \! \otimes f'$.  This explains the meaning of these two failed attempts to give $FP_i$ a monoidal structure.

Thanks to the description of $\C \square \C$ as a pushout, to know the tensor product $\boxtimes$ in a strict premonoidal category $\C$ it suffices to know $x \boxtimes y$, $x \boxtimes f$ and $f \boxtimes y$ for all objects $x,y$ and morphisms $f$ of $\C$.  (Here we find it useful to write
$x \boxtimes f$ for $1_x \boxtimes f$ and $f \boxtimes y$ for $f \boxtimes 1_y$.)  In the case at hand, we define
\[    \boxtimes_i \maps FP_i \square FP_i \to FP_i \]
on objects by setting
\[(i + a) \boxtimes_i (i + a' ) = i +  a + a'\]
for all $a,a' \in \N[S-C]$, while for morphisms
\begin{align*} 
    f &\maps i + a \to i + b \\ 
    f' &\maps i +  a' \to i + b'
\end{align*}
we set
\[
	a \boxtimes f' = f' + 1_a, \qquad
     f \boxtimes a' = f + 1_{a'}.
\]

\begin{prop}
\label{prop:monoidal}
    The tensor product $\boxtimes_i$ makes $FP_i$ into a strict premonoidal category.
\end{prop}

\begin{proof}
This can be checked directly, but this is also a special case of a construction in Power and Robinson's paper on premonoidal categories \cite[Ex.\ 3.4]{PowerRobinson}.  They describe a construction, sometimes called `linear state passing' \cite{MogelbergStaton}, that takes any object $i$ in any symmetric monoidal category $C$ and yields a premonoidal category $C_i$ where objects are of the form $i \otimes c$ for $c \in C$ and morphisms are morphisms in $C$ of the form $f \maps i \otimes c \to i \otimes c'$.  We are considering the special case where $C = FP$, and because $FP$ is commutative monoidal the resulting premonoidal category is strict: all the coherence isomorphisms are identities.
\end{proof}

Finally, we show that the tensor products $\boxtimes_i$ on the categories $FP_i$ let us lift our network model $G$ from $\Cat$ to the category of strict premonoidal categories.

\begin{defn}
Let $\PreMonCat$ be the category of strict premonoidal categories and \define{strict premonoidal functors}, meaning functors between strict premonoidal categories that strictly preserve the tensor product.  Let $U \maps \PreMonCat \to \Cat$ denote the forgetful functor which sends a strict premonoidal category to its underlying category.
\end{defn}

\begin{thm}
\label{thm:lift}
    The network model $G \maps \S(C) \to \Cat$ lifts to a functor $\hat G \maps \S(C) \to \PreMonCat$:
    \[\begin{tikzcd}
        &
        \PreMonCat
        \arrow[d, "U"]
        \\
        \S(C)
        \arrow[r, "G", swap]
        \arrow[ur, "\hat G"]
        &
        \Cat
    \end{tikzcd}\]
    where $\hat G(x) = FP_{p(x)}$ with the strict premonoidal structure described in Prop.~\ref{prop:monoidal}.
\end{thm}

\begin{proof}
    Since $G$ sends each morphism in $\S(C)$ to an identity functor, so must $\hat G$. 
\end{proof}
}

{\ssp 
\setcounter{chapter}{4}
\chapter{Monoidal Grothendieck Construction}
\label{ch:MonGroth}

\section{Introduction}

The Grothendieck construction \cite{SGAI} exhibits one of the most fundamental relations in category theory, namely the equivalence between contravariant pseudofunctors into $\Cat$ and fibrations. In previous chapters, we have described how the to construct a \emph{total category}, denoted $\int F$, from a functor of the form $F \maps \X\op \to \Cat$. Actually, we really could have been using pseudofunctors, since $\Cat$ is more naturally thought of as a 2-category. We refer to pseudofunctors of the form $F \maps \X\op \to \Cat$ as \emph{indexed categories}. The construction of $\int F$ from a given indexed category essentially forgets the distinction between the categories $Fx$ for $x \in \X$, and incorporates the functors $Ff \maps Fy \to Fx$ as maps between the objects of $Fy$ and $Fx$. The distinction between these categories could be remembered via a \emph{fibration}, a special sort of functor $P \maps \int F \to \X$, which tells you how to take preimage categories of the objects, $P\inv(x)$, and turn certain maps in $\int F$ into functors between the preimage categories. For a general fibration $P \maps \A \to \X$, the category $\X$ is called the \emph{base category} and the category $\A$ is called the \emph{total category}. For an object $x \in \X$, the preimage category $P\inv(x)$ is called the \emph{fibre} of $P$ over $x$. A fibration is precisely what is needed to reconstruct all the data in the indexed category from its total category. Indeed, the Grothendieck construction gives an equivalence between the 2-categories $\ICat$ of indexed categories, and $\Fib$ of fibrations. This equivalence allows us to move between the worlds of indexed categories and fibred categories, providing access to tools and results from both. We recall the basic theory of fibrations, indexed categories, and the Grothendieck construction in \cref{sec:fibrations}. 

Due to the importance of the Grothendieck construction, it is only natural that one would be interested in extra structure these objects may have, and how the correspondence extends. In particular, a version which handles monoidal structures on the various categories in play could potentially be very useful, as monoidal categories are of central interest in both pure and applied category theory. There are several categories to consider as equipped with monoidal structures in this scenario: fibers $P\inv(x)$ of a fibration $P \maps \A \to \X$, its base category $\X$, its total category $\A$, the indexing category $\X$ of an indexed category $F \maps \X\op \to \Cat$, and the categories $Fx$ indexed by $F$. Of course these options are not really all distinct. The base of the fibrations correspond to the indexing category under the equivalence, and the fibres correspond to the individual categories selected by the indexed category. This boils our options down to two monoidal variants: fibre-wise, and global.

In the first variant---the fibre-wise approach---the fibres are equipped with a monoidal structure, and the reindexing functors are equipped with a strict monoidal structure. The additional structure this gives to the corresponding indexed categories turns them into pseudofunctors into $\Mon\Cat$, which were called \emph{indexed monoidal categories} by Hofstra and de Marchi \cite{DescentForMonads}. In the second variant---the global approach---the total category and the base category of the fibration are each equipped with the structure of a monoidal category, and the fibration is equipped with a strict monoidal structure. The corresponding structure equipped to the related indexed category is a little less obvious. The indexing category is equipped with a monoidal structure as in the fibration side of the picture, and the pseudofunctor is now equipped with the structure of a lax monoidal structure into $\Cat$ with its cartesian structure. We call these \emph{monoidally indexed categories} or just \emph{monoidal indexed categories}.

Both of these variants can be seen as special cases of a much more general phenomenon. Pseudomonoids are a categorification of monoid objects internal to a monoidal category. It would be reasonable to call it ``monoidal category internal to a monoidal 2-category''. We can see both of the monoidal variants of both fibrations and indexed categories described above as examples of pseudomonoids in certain 2-categories of fibrations or indexed categories.

The 2-category $\ICat(\X)$ of indexed categories over a fixed base category has finite products, and thus a cartesian monoidal structure. Pseudomonoids taken with respect to this monoidal structure are precisely pseudofunctors $\X\op \to \Mon\Cat$, i.e.\ the fibre-wise monoidal indexed categories described above. Similarly, the 2-category $\Fib(\X)$ of fibrations over a fixed base category has a cartesian monoidal structure, for which pseudomonoids are precisely the fibre-wise monoidal fibrations described above.

The 2-category $\ICat$ of indexed categories over different base categories has finite products, and thus a cartesian monoidal structure. Pseudomonoids taken with respect to this monoidal structure are precisely lax monoidal pseudofunctors $(\X\op, \otimes) \to (\Cat, \times)$, i.e.\ the global monoidal indexed categories described above. Similarly, the 2-category $\Fib$ of fibrations over different base categories has a cartesian monoidal structure, for which pseudomonoids are precisely the global monoidal fibrations described above.

An immediate consequence of this perspective on these objects is that the Grothendieck construction lifts naturally into both settings. The 2-category of fibre-wise monoidal fibrations is equivalent to the 2-category of fibre-wise monoidal indexed categories, c.f.\ \cref{thm:mainthm}. Similarly, the 2-category of global monoidal fibrations is equivalent to the 2-category of global monoidal indexed categories, c.f.\ \cref{thm:fibrmonGr}.

When $\X$ is cartesian monoidal, a global monoidal structure can be constructed from fibre-wise monoidal data, and vice versa, c.f.\ \cref{thm:fibrewise=global}. We use our high-level perspective to give a new proof of the result of Shulman giving an equivalence between fibre-wise monoidal indexed categories and global monoidal fibrations over cartesian base categories \cite{FramedBicats}.

The fact that the monoidal Grothendieck construction naturally arises in diverse settings is what motivated the theoretical clarification presented here. We gather a few examples in the last section of the chapter to exhibit the various constructions concretely, and we are convinced that many more exist and would benefit from such a viewpoint. The examples include standard (op)fibrations such as the (co)domain (op)fibration and families classified in their monoidal contexts, as well as certain special algebraic cases of interest such as monoid-(co)algebras as objects in monoidal Grothendieck categories. Moreover, global categories of (co)modules for (co)monoids in any monoidal category, as well as (co)modules for (co)monads in monoidal double categories also naturally fit in this context. Finally, certain categorical approaches to systems theory employ algebras for monoidal categories, namely monoidal indexed categories, as their basic compositional tool for nesting of systems; clearly these also fall into place, giving rise to total monoidal categories of systems with new potential to be explored.

In \cref{sec:monfib} and \cref{sec:monicat}, we give the fibre-wise and global monoidal versions of fibrations and indexed categories. In the first version, the fibers are equipped with a monoidal structure. In the second, the base and total categories are monoidal, and the fibration is (strict) monoidal. In \cref{sec:monequiv}, we lift the Grothendieck construction into these monoidal settings as well. In \cref{sec:summary}, we give a detailed description of the monoidal structures given by the correspondences. 

\section{Monoidal Fibres and Monoidal Fibrations}\label{sec:monfib}

We begin by describing the two monoidal variants of fibrations. This requires familiarity with notions such as monoidal 2-categories, pseudomonoids, and the 2-categories $\Fib$ and $\Fib(\X)$. The 2-categories $\Fib$ and $\OpFib$ of (op)fibrations over arbitrary bases, explained in \cref{sec:fibrations}, have a cartesian monoidal structure inherited from $\Cat^\2$. For two fibrations $P$ and $Q$, their product in $\Cat^\2$
\begin{equation}\label{Fib_cart}
    P \times Q \maps \A \times \B \to \X \times \Y
\end{equation}
is also a fibration, where a cartesian lifting is a pair consisting of a $P$-lifting and a $Q$-lifting; similarly for opfibrations. The monoidal unit is the trivial (op)fibration $1_\1 \maps \1 \to \1$. Since the monoidal structure is cartesian, they are both symmetric monoidal 2-categories. We refer to a pseudomonoid in $(\Fib, \times, 1_\1)$ as a \define{monoidal fibration}. By the following result, this aligns with the common notion of monoidal fibration \cite{FramedBicats}.

\begin{prop}\label{def:monoidal_fibration}
    A monoidal fibration $P \maps \A \to \X$ is a fibration for which both the total $\A$ and base category $\X$ are monoidal, $P$ is a strict monoidal functor and the tensor product $\otimes_\A$ of $\A$ preserves cartesian liftings.
\end{prop}
\begin{proof}
    The multiplication and unit are fibred 1-cells $\mlt = (\otimes_\A, \otimes_\X) \maps P \times  P \to  P$ and $\uni = (I_\A, I_\X) \maps \1 \to  P$
    displayed as follows.
    \begin{equation} \label{multunitmonoidalfibr}
    \begin{tikzcd}
        \A \times \A
        \arrow[r, "\otimes_\A"]
        \arrow[d, swap, "P \times P"]
        &
        \A
        \arrow[d, "P"]
        &\text{and}&
        \1
        \arrow[r, "I_\A"]
        \arrow[d, swap, "1_\1"]
        &
        \A
        \arrow[d, "P"]
        \\
        \X \times \X 
        \arrow[r, swap, "\otimes_\X"]
        &
        \X
        &&
        \1
        \arrow[r, swap, "I_\X"]
        &
        \X
    \end{tikzcd}
    \end{equation}
    A morphism $(\phi_1, \phi_2)$ in $\A \times \A$ is $P\times P$-cartesian if and only if $\phi_1$ and $\phi_2$ are both $P$-cartesian. The condition of $(\otimes_\A, \otimes_\X)$ forming a fibred 1-cell tells us precisely that $\phi_1 \otimes_\A \phi_2$ is $P$-cartesian. The pieces of associativity and unitality 2-cells corresponding to $\A$ and to $\X$ give precisely the associativity and unitality structures for each to be given the structure of a monoidal category. The functor $P$ is strict with respect to these monoidal structures on $\A$ and $\X$ due to the fact that the diagrams above commute.
\end{proof}

A \define{monoidal fibred 1-cell} between two monoidal fibrations $P \maps \A \to \X$ and $Q \maps \B \to \Y$ is a (strong) morphism of pseudomonoids between them, as defined in \cref{app:Monoidal2cats}. \begin{prop}\label{prop:monoidalfibred1cell}
    A monoidal fibred 1-cell between two monoidal fibrations $P$ and $Q$ is a fibred 1-cell $(H, F)$ where both functors are monoidal, $(H, \phi, \phi_0)$ and $(F, \psi, \psi_0)$, such that $Q (\phi_{a, b})=\psi_{ P a, P b}$ and $Q \phi_0=\psi_0$.
\end{prop}
\begin{proof}
    A monoidal fibred 1-cell amounts to a fibred 1-cell, i.e.\ a commutative square
    \begin{equation}\label{eq:HF}
    \begin{tikzcd}
        \A    
        \arrow[r, "H"]    
        \arrow[d, "P"'] 
        &  
        \B    
        \arrow[d, "Q"]    
        \\
        \X  
        \arrow[r, "F"'] 
        &  
        \Y  
    \end{tikzcd}
    \end{equation}
    where $H$ preserves cartesian liftings, 
    along with invertible 2-cells \cref{eq:laxmorphism} in $\Fib$ satisfying \cref{eq:axiomslaxmorphism}. By \cref{eq:fibred2cell}, these are fibred 2-cells
    \begin{displaymath}
    \begin{tikzcd}
        &&
        \B \times \B
        \arrow[dr, bend left, "\otimes_\B"]
        \arrow[ddd, "Q"]
        \arrow[ddl, Rightarrow, swap, "\phi"]
        \\
        \A \times \A
        \arrow[ddd, swap, "P \times P"]
        \arrow[urr, bend left, "H \times H"]
        \arrow[dr, bend right, "\otimes_\A"]
        &&&
        \B
        \arrow[ddd, "Q"]
        \\&
        \A
        \arrow[urr, bend right, -, white, line width = 5]
        \arrow[urr, swap, bend right, "H", pos = 0.6]
        \\&&
        \Y \times \Y
        \arrow[dr, bend left, swap, "\otimes_\Y"]
        \arrow[ddl, Rightarrow, "\psi"]
        \\
        \X \times \X
        \arrow[urr, bend left, "F \times F"]
        \arrow[dr, bend right, swap, "\otimes_\X"]
        &&&
        \Y
        \\&
        \X
        \arrow[urr, bend right, swap, "F"]
        \arrow[uuu, -, white, line width = 5]
        \arrow[uuu, leftarrow, "P"]
    \end{tikzcd}
    \qquad\qquad
    \begin{tikzcd}
        &&
        \\
        1
        \arrow[dd, swap, equal]
        \arrow[rrr, bend left, "I_\B"]
        \arrow[rrr, bend left, swap, phantom, ""{name = IB}, pos = 0.7]
        \arrow[dr, bend right, "I_\A"]
        &&&
        \B
        \arrow[dd, "Q"]
        \\&
        \A
        \arrow[urr, bend right, -, white, line width = 5]
        \arrow[urr, swap, bend right, "H", pos = 0.6]
        \arrow[to = IB, Leftarrow, "\phi_0"]
        \\
        1
        \arrow[rrr, bend left, swap, "I_\Y", pos = 0.6]
        \arrow[rrr, bend left, phantom, swap, ""{name = IY}, pos = 0.7]
        \arrow[dr, bend right, swap, "I_\X"]
        &&&
        \Y
        \\&
        \X
        \arrow[urr, bend right, swap, "F"]
        \arrow[uu, -, white, line width = 5]
        \arrow[uu, leftarrow, "P"]
        \arrow[to = IY, Leftarrow, "\psi_0", swap]
    \end{tikzcd}
    \end{displaymath}
    where $\phi$ and $\psi$ are natural isomorphisms with components
    \begin{displaymath}
        \phi_{a, b} \maps Ha \otimes Hb \xrightarrow{\sim} H(a \otimes b), \quad \psi_{x, y} \maps Fx \otimes Fy \xrightarrow{\sim} F(x \otimes y)
    \end{displaymath}
    such that $\phi$ is above $\psi$, i.e.\ the following diagram commutes:
    \begin{displaymath}
    \begin{tikzcd}[column sep=.6in]
        Q (Ha \otimes Hb) 
        \arrow[r, "{ Q \phi_{a, b}}"]
        \arrow[d, equal, "{\cref{multunitmonoidalfibr}}"'] 
        &  
        Q H(a \otimes b) 
        \arrow[d, equal, "{\cref{eq:HF}}"] 
        \\
        Q H a \otimes Q H b 
        \arrow[d, equal, "{\cref{eq:HF}}"' ] 
        & 
        F P(a \otimes b) 
        \arrow[d, equal, "{\cref{multunitmonoidalfibr}}"] 
        \\
        F P a \otimes F P b 
        \arrow[r, "{\psi_{ P a, P b}}"']
        & 
        F( P a \otimes  P b)
    \end{tikzcd}
    \end{displaymath}
    Similarly, $\phi_0$ and $\psi_0$ have single components $\phi_0 \maps I_\B \xrightarrow{\sim} H(I_\A)$ and $ \psi_0 \maps I_\Y \xrightarrow{\sim} F(I_\X)$ such that $Q (\phi_0) = \psi_0$.
    These two conditions in fact say that the identity transformation, a.k.a.\ commutative square \cref{eq:HF} is a monoidal one, as expressed in \cite[12.5]{FramedBicats}.
    The relevant axioms dictate that $(\phi, \phi_0)$ and $(\psi, \psi_0)$ give $H$ and $F$ the structure of strong monoidal functors.
\end{proof}

For lax or oplax morphisms of pseudomonoids in $\Fib$, we obtain appropriate notions of monoidal fibred 1-cells, where the top and bottom functors of \cref{eq:HF} are lax or oplax monoidal respectively.

Finally, a \define{monoidal fibred 2-cell} is a 2-cell between morphisms $(H, F)$ and $(K, G)$ of pseudomonoids $P$, $Q$ in $\Fib$. 

\begin{prop}\label{prop:monfib2cell}
    A monoidal fibred 2-cell between two monoidal fibred 1-cells is an ordinary fibred 2-cell $(\alpha, \beta)$ where both natural transformations are monoidal.
\end{prop}
\begin{proof}
    Unpacking the definition, we see that a monoidal fibred 2-cell is a fibred 2-cell as described in \cref{sec:fibrations}
    \begin{displaymath}
    \begin{tikzcd}[row sep=.45in, column sep=.7in]
        \A
        \arrow[d, " P"']
        \arrow[r, bend left = 20, "H"]
        \arrow[r, bend right = 20, "K"']
        \arrow[r, phantom, "\Downarrow {\scriptstyle\beta}" description] 
        & \B \arrow[d, " Q "] \\
        \X    \arrow[r, bend left = 20, "F"]
        \arrow[r, bend right = 20, "G"']
        \arrow[r, phantom, "\Downarrow{\scriptstyle\alpha}" description] &  \Y
    \end{tikzcd}
    \end{displaymath}
    satisfying the axioms \cref{monoidal2cell}. These amount to the fact that both $\beta$ and $\alpha$ are monoidal natural transformations between the respective lax monoidal functors.
\end{proof}

We denote by $\PsMon(\Fib)= \MonFib$ the 2-category of monoidal fibrations, monoidal fibred 1-cells and monoidal fibred 2-cells.
By changing the notion of morphisms between pseudomonoids to lax or oplax, we obtain 2-categories $\MonFib_\lax$ and $\MonFib_\opl$.
There are also 2-categories $\BrMonFib$ and $\SymMonFib$ of \define{braided} (resp. \define{symmetric}) \define{monoidal fibrations}, \define{braided} (resp. \define{symmetric}) \define{monoidal fibred 1-cells} and monoidal fibred 2-cells, defined to be $\BrPsMon(\Fib)$ and $\SymPsMon(\Fib)$ respectively; see \cref{prop:BrPsMon}.

Dually, we have appropriate 2-categories of \define{monoidal opfibrations}, \define{monoidal opfibred 1-cells} and \define{monoidal opfibred 2-cells} and their braided and symmetric variations, $\MonOpFib$, $\BrMonOpFib$ and $\SymMonOpFib$. All the structures are constructed dually, where a monoidal opfibration, namely
a pseudomonoid in the cartesian monoidal $(\OpFib, \times, 1_\1)$, 
is a strict monoidal functor such that the tensor product of the total category preserves cocartesian liftings.

All the above 2-categories have sub-2-categories of monoidal (op)fibrations over a fixed monoidal base $(\X, \otimes, I)$, e.g.\ $\MonFib(\X)$ and $\MonOpFib(\X)$. The morphisms are \define{monoid\-al (op)fibred functors}, i.e.\ fibred 1-cells of the form $(H, 1_\X)$ with $H$ monoidal, and the 2-cells are \define{monoidal (op)fibred natural transformations}, i.e.\ fibred 2-cells of the form $(\beta, 1_{1_\X})$ with $\beta$ monoidal. These 2-categories correspond to the `global' monoidal part of the story.

Moreover, the above constructions can be adjusted accordingly to the context of split fibrations. Explicitly, the 2-category $\PsMon(\Fib_\spl)=\MonFib_\spl$ has as objects \define{monoidal split fibrations}, namely split fibrations $P \maps \A\to\X$ between monoidal categories which are strict monoidal functors and $\otimes_\A$ strictly preserves cartesian liftings (compare to \cref{def:monoidal_fibration}). Furthermore, the hom-categories $\MonFib_\spl(P, Q)$ between monoidal split fibrations are full subcategories of $\MonFib(P, Q)$ spanned by the monoidal fibred 1-cells which are split as fibred 1-cells, namely $(H, F)$ as in \cref{prop:monoidalfibred1cell} where $H$ strictly preserves cartesian liftings.

We end this section by considering a different monoidal object in the context of (op)fibra\-tions, starting over from the usual 2-categories of (op)fibrations over a fixed base $\X$, (op)fibred functor and (op)fibred natural transformations $\Fib(\X)$ and $\OpFib(\X)$. Notice that contrary to the earlier devopment, there is no monoidal structure on $\X$. Both these 2-categories are also cartesian monoidal, but in a different manner than $\Fib$ and $\OpFib$, due to the cartesian monoidal structure of $\Cat/\X$; see for example \cite[1.7.4]{Jacobs}. Explicitly, for fibrations $P \maps \A\to\X$ and $Q \maps \B \to \X$, their tensor product $P \boxtimes Q$ is given by any of the two equal functors to $\X$ from the following pullback
\begin{equation}\label{Fib_X_cart}
\begin{tikzcd}[column sep=.5in, row sep=.5in]   \A \times_\X \B
        \arrow[r]
        \arrow[d]
        \arrow[dr, phantom, very near start, "\lrcorner"]\arrow[dr, dashed, bend left, "P\boxtimes Q"description]
        & 
        \A
        \arrow[d, "P"] 
        \\
        \B
        \arrow[r, "Q"'] 
        & 
        \X
    \end{tikzcd}
\end{equation}
since fibrations are closed under pullbacks and of course composition. The monoidal unit is $1_\X \maps \X \to \X$.

A pseudomonoid in $(\Fib(\X), \boxtimes, 1_\X)$ is an ordinary fibration $P \maps \A\to\X$ equip\-ped with two fibred functors $(\mlt, 1_\X) \maps P\boxtimes P\to P$ and $(\uni, 1_\X) \maps 1_\X\to P$ displayed as
\begin{equation}\label{eq:fibrewisetensor}
    \begin{tikzcd}
        \A \times_\X \A 
        \arrow[dr, "P\boxtimes P"']
        \arrow[rr, "\mlt"] 
        && 
        \A
        \arrow[dl, "P"] 
        \\
        & 
        \X 
        &
    \end{tikzcd}\qquad
    \begin{tikzcd}
        \X
        \arrow[rr, "\uni"]
        \arrow[dr, "1_X"'] 
        && 
        \A
        \arrow[dl, "P"] 
        \\
        & 
        \X 
        &
        \end{tikzcd}
    \end{equation}
along with invertible fibred 2-cells satisfying the usual axioms. 
In more detail, the pullback $\A \times_\X \A$  consists of pairs of objects of $\A$ which are in the same fibre of $P$, and $P\boxtimes P$ sends such a pair to their underlying object defining their fibre. 
The functor $\mlt$ maps any $(a, b)\in\A_x$ to some $m(a, b):=a\otimes_x b\in\A_x$ and the map $\uni$ sends an object $x \in \X$ to a chosen one, $I_x$, in its fibre. The invertible 2-cells and the axioms guarantee that these maps define a monoidal structure on each fibre $\A_x$, providing the associativity, left and right unitors. The fact that $\mlt$ and $\uni$ preserve cartesian liftings translate into a strong monoidal structure on the reindexing functors: for any $f \maps x\to y$ and $a, b\in\A_y$, $f^*a\otimes_x f^*b\cong f^*(a\otimes_y b)$ and $I_y\cong f^*(I_x)$.

A (lax) morphism between two such fibrations is a fibred functor \cref{eq:fibredfunctor} such that the induced functors $H_x \maps \A_x \to \B_x$ between the fibres as in \cref{eq:functorbetweenfibres} are (lax) monoidal, whereas a 2-cell between them is a fibred natural transformation $\beta \maps H\Rightarrow K$ \cref{eq:fibrednaturaltrans} which is monoidal when restricted to the fibers, $\beta_x|_{\A_x} \maps H_x\Rightarrow K_x$. In this way, we obtain the 2-category $\PsMon(\Fib(\X))$ and dually $\PsMon(\OpFib(\X))$. These 2-categories correspond to the `fibrewise' monoidal part of the story.

Finally, taking pseudomonoids in the 2-category of split fibrations over a fixed base, we obtain the 2-category $\PsMon(\Fib_\spl(\X))$ with objects split fibrations equipped with a fibrewise tensor product and unit as above, but now the reindexing functors strictly preserve that monoidal structure since the top functors of \cref{eq:fibrewisetensor} strictly preserve cartesian liftings: $f^*a \otimes_x f^*b = f^* (a \otimes_y b)$ and $I_y = f^* (I_x)$. Moreover, $\PsMon (\Fib_s(\X)) (P, Q)$ is the full subcategory of $\PsMon(\Fib(\X))(P, Q)$ spanned by split fibred functors, namely $H \maps \A\to\B$ which strictly preserve cartesian liftings but still $H_x$ are monoidal functors between the monoidal fibres as before.

As is evident from the above descriptions, the 2-categories $\MonFib(\X)$ and $\PsMon(\Fib(\X))$ are  different in general. A monoidal fibration over $\X$ is a strict monoidal functor, whereas a pseudomonoid in fixed-base fibrations is a fibration with monoidal fibres in a coherent way: none of the base or the total category need to be monoidal. 

\section{Indexed Categories and Monoidal Structures}\label{sec:monicat}

The 2-categories of indexed and opindexed categories $\ICat$ and $\OpICat$, explained in \cref{sec:fibrations}, are both monoidal. Explicitly, given two indexed categories $\M \maps \X\op \to \Cat$ and $\psN \maps \Y\op \to \Cat$, their tensor product $\M \otimes \psN \maps (\X \times \Y)\op \to \Cat$ is the composite 
\begin{equation}\label{ICat_cart}
    (\X \times \Y)\op \cong \X\op \times \Y\op \xrightarrow{\M \times \psN} \Cat \times \Cat \xrightarrow{\times} \Cat
\end{equation}
i.e.\ $(\M \otimes \psN)(x, y) = \M(x) \times \psN(y)$ using the cartesian monoidal structure of $\Cat$. The monoidal unit is the indexed category $\Delta \1 \maps \1\op \to \Cat$ that picks out the terminal category $\1$ in $\Cat$, and similarly for opindexed categories. Notice that this monoidal 2-structure, formed pointwise in $\Cat$, is also cartesian.

We call a pseudomonoid in $(\ICat, \otimes, \Delta \1)$ a \define{monoidal indexed category}. 

\begin{prop}\label{def:monoidal_indexedcat}
    A monoidal indexed category is a lax monoidal pseudofunctor \[(\M, \mu, \mu_0) \maps (\X\op, \otimes\op, I) \to (\Cat, \times, \1),\] where $(\X, \otimes, I)$ is an (ordinary) monoidal category.
\end{prop}
\begin{proof}
    Unpacking the definition, we see that a monoidal indexed category is an indexed category $\M \maps \X\op \to \Cat$ equipped with multiplication and unit indexed 1-cells $(\otimes_\X, \mu) \maps \M \otimes \M \to \M$, $(\eta, \mu_0) \maps \Delta \mathbf 1 \to \M$ which by \cref{eq:indexed1cell} are as follows.
    \[
    \begin{tikzcd}[column sep=.7in, row sep=.2in]
        \X\op \times \X\op
        \arrow[dr, "\M \otimes \M"]
        \arrow[dd, "\otimes\op"'] 
        \\ 
        \arrow[r, phantom, "\Downarrow{\scriptstyle\mu}"description] 
        & 
        \Cat 
        \\
        \X\op 
        \arrow[ur, "\M"']
    \end{tikzcd}\qquad 
    \begin{tikzcd}[column sep=.7in, row sep=.2in]
        \1\op
        \arrow[dr, "\Delta\1"]
        \arrow[dd, "I\op"'] 
        \\ 
        \arrow[r, phantom, "\Downarrow{\scriptstyle\mu_0}"description] 
        & 
        \Cat 
        \\
        \X\op 
        \arrow[ur, "\M"']
    \end{tikzcd}
    \]
    These come equipped with invertible indexed 2-cells 
    as in \cref{alphalambdarho}; the axioms this data is required to satisfy, on the one hand, render $\X$ a monoidal category with $\otimes \maps \X \times \X \to \X$ its tensor product functor and $I \maps \1 \to \X$ its unit. On the other hand, 
    the resulting axioms for the components
    \begin{equation}\label{eq:laxatorunitor}
        \mu_{x, y} \maps \M x \times \M y \to \M (x \otimes y), \qquad
        \mu_0 \maps \1 \to \M (I)
    \end{equation}
    of the above pseudonatural transformations precisely give $\M$ the structure of a lax monoidal pseudofunctor, recalled in \cref{app:Monoidal2cats}.
\end{proof}

We then define a \define{monoidal indexed 1-cell} to be a (strong) morphism between pseudomonoids in $(\ICat, \otimes, \Delta\1)$. 

\begin{prop}\label{prop:moni1cell}
    A monoidal indexed 1-cell between two monoidal indexed categories $\M$ and $\psN$ is an indexed 1-cell $(F, \tau)$, where the functor $F$ is (strong) monoidal and the pseudonatural transformation $\tau$ is monoidal. 
\end{prop}
\begin{proof}
    Unpacking the definition, we see that a monoidal indexed 1-cell is an indexed 1-cell $(F, \tau) \maps \M \to \psN$
    \begin{displaymath}
    \begin{tikzcd}[column sep=.7in, row sep=.2in]
        \X\op
        \arrow[dr, "\M"]
        \arrow[dd, "F\op"'] 
        \\ 
        \arrow[r, phantom, "\Downarrow{\scriptstyle\tau}"description] 
        & 
        \Cat 
        \\
        \Y\op 
        \arrow[ur, "\psN"']
    \end{tikzcd}
    \end{displaymath}
    between two monoidal indexed categories
    $(\M, \mu, \mu_0)$ and $(\psN, \nu, \nu_0)$ equipped with two invertible indexed 2-cells $(\psi, m)$ and $(\psi_0, m_0)$ as in \cref{eq:laxmorphism}, which explicitly consist of natural isomorphisms $\psi$, $\psi_0$ and invertible modifications
    \[
    \begin{tikzcd}[row sep=.3in, column sep=.25in]
        \X\op \times \X\op
        \arrow[d, "\otimes\op"']
        \arrow[dr, "F\op \times F\op" description]
        \arrow[drrr, "\M \otimes \M", bend left=20]
        \arrow[drrr, phantom, bend left=5, "\Downarrow{\scriptstyle \tau \times \tau}"]
        &&&&
        \X\op \times \X\op
        \arrow[drrr, "\M \otimes \M", bend left]
        \arrow[d, "\otimes\op"']\arrow[drrr, phantom, "\Downarrow{\scriptstyle\mu}"description]
        \\
        \X\op
        \arrow[d, "F\op"']
        \arrow[r, phantom, "\Downarrow{\scriptstyle\psi}"description, bend right]
        &
        \Y\op \times \Y\op
        \arrow[rr, "\psN \otimes \psN"description]
        \arrow[dr, "\otimes\op" description]
        \arrow[drr, phantom, "\Downarrow{\scriptstyle\nu}"description]
        &&
        \Cat
        \arrow[r, phantom, "\stackrel{m}{\Rrightarrow}"]
        &
        \X\op
        \arrow[rrr, "\M"description]\arrow[d, "F\op"']\arrow[drrr, phantom, "\Downarrow{\scriptstyle\tau}"description]
        &&&
        \Cat
        \\
        \Y\op\arrow[rr, "\mathrm{id}"']
        &&
        \Y\op\arrow[ur, "\psN"', bend right]
        & \phantom{A} &
        \Y\op\arrow[urrr, "\psN"', bend right]
        &&& \phantom{A}
    \end{tikzcd}
    \]
    
    \[
    \begin{tikzcd}[row sep=.25in, column sep=.3in]
        \1\op\arrow[d, "I\op"']\arrow[drrr, "\Delta\1", bend left=20]
        \arrow[drrr, phantom, "\Downarrow{\scriptstyle\nu_0}"]\arrow[ddrr, "I\op"description]
        &&&&&
        \1\op\arrow[drrr, "\Delta\1", bend left]\arrow[d, "I\op"']\arrow[drrr, phantom, "\Downarrow{\scriptstyle\mu_0}"description]
        \\
        \X\op\arrow[d, "F\op"']\arrow[r, phantom, "\Downarrow{\scriptstyle\psi_0}"description, bend right]
        & \phantom{A} &&
        \Cat\arrow[rr, phantom, "\stackrel{m_0}{\Rrightarrow}"]
        &&
        \X\op\arrow[rrr, "\M"description]\arrow[d, "F\op"']\arrow[drrr, phantom, "\Downarrow{\scriptstyle\tau}"description]
        &&&
        \Cat
        \\
        \Y\op\arrow[rr, "\mathrm{id}"']
        &&
        \Y\op\arrow[ur, "\psN"', bend right]
        & \phantom{A} &&
        \Y\op\arrow[urrr, "\psN"', bend right]
        &&& \phantom{A}
    \end{tikzcd}
    \]
    as dictated by the general form \cref{eq:indexed2cell} of indexed 2-cells. 
    The natural isomorphisms $\psi$ and $\psi_0$ have components
    \begin{displaymath}
        \psi_{x, z} \maps Fx\otimes Fy\xrightarrow{\sim} F(x\otimes y), \quad
        \psi_0 \maps I\xrightarrow{\sim} F(I)\quad \textrm{ in $\Y\op$}
    \end{displaymath}
    whereas the modifications $m$ and $m_0$ are given by families of invertible natural transformations 
    \begin{displaymath}
    \begin{tikzcd}[column sep=.3in, row sep=.2in]
        & \psN Fx \times \psN Fy
        \arrow[r, "\nu_{Fx, Fy}"]
        \arrow[dr, dashed] 
        & 
        \psN (Fx \otimes Fy)
        \arrow[d, "\psN\psi_{x, y}"] 
        \\
        \M x{\times}\M y
        \arrow[ur, "\tau_x\times\tau_y"]
        \arrow[dr, "\mu_{x, y}"']
        \arrow[rr, phantom, "\Downarrow{\scriptstyle m_{x, y}}"description] 
        && 
        \psN F (x\otimes y) 
        \\& 
        \M (x\otimes y)
        \arrow[ur, "\tau_{x\otimes y}"'] 
        &
    \end{tikzcd}\quad
    \begin{tikzcd}[column sep=.2in, row sep=.2in]
        & 
        \psN(I)
        \arrow[dr, "\psN \psi_0"] 
        & \\
        \1\arrow[ur, "\nu_0"]
        \arrow[dr, "\mu_0"']
        \arrow[rr, phantom, "\Downarrow{\scriptstyle m_0}"description] 
        &&
        \psN(FI) 
        \\& 
        \M(I)\arrow[ur, "\tau_I"'] 
    \end{tikzcd}
    \end{displaymath}
    The appropriate coherence axioms ensure that the functor $F \maps \X \to \Y$ has a strong monoidal structure $(F, \psi, \psi_0)$, and that the pseudonatural transformation $\tau \maps \M \Rightarrow \psN \circ F\op$ is monoidal with $m_{x, y}$, $m_0$ as in \cref{eq:monpseudocomponents}. Notice that $F\op$ being monoidal makes $F$ monoidal with inverse structure isomorphisms. 
\end{proof}

Finally, a \define{monoidal indexed 2-cell} is a 2-cell between morphisms of pseudomonoids in $(\ICat, \otimes, \Delta\1)$. 

\begin{prop}\label{prop:moni2cell}
    A monoidal indexed 2-cell between two monoidal indexed 1-cells $(F, \tau)$ and $(G, \sigma)$ is an indexed 2-cell  $(\alpha, m)$ such that $\alpha$ is an ordinary monoidal natural transformation and $m$ is a monoidal modification.
\end{prop} 
\begin{proof}
    Following the definition of \cref{app:Monoidal2cats}, an indexed 2-cell $(a, m) \maps (F, \tau)\Rightarrow(G, \sigma) \maps \M\to\psN$ as in \cref{eq:indexed2cell}, which consists of a natural transformation $\alpha \maps F\Rightarrow G$ and a modification $m$ with components
    \begin{displaymath}
    \begin{tikzcd}[column sep=.5in, row sep=.15in]
        \M x\arrow[dr, bend right=10, "\sigma_x"']\arrow[rr, bend left, "\tau_x"] \arrow[rr, phantom, "\Downarrow{\scriptstyle m_x}"description] && \psN Fx \\
        & \psN Gx\arrow[ur, bend right=10, "\psN\alpha_x"'] & 
    \end{tikzcd}
    \end{displaymath}
    is monoidal, exactly when $\alpha \maps F\Rightarrow G$ is compatible with the strong monoidal structures of $F$ and $G$, and the modification $m \maps \tau \Rrightarrow \psN \alpha\op \circ \sigma$ satisfies \cref{eq:monoidalmodaxioms} for the induced monoidal structures on its domain and target pseudonatural transformations.
\end{proof}

We write $\PsMon(\ICat) = \Mon\ICat$, the 2-category of monoidal indexed categories, monoidal indexed 1-cells and monoidal indexed 2-cells. Moreover, their braided and symmetric counterparts form $\Br\Mon\ICat$ and $\Sym\Mon\ICat$ respectively, as the 2-categories of braided and symmetric pseudomonoids in $(\ICat, \otimes, \Delta\1)$ formally discussed in \cref{app:Monoidal2cats}.
Similarly, we have 2-categories of \define{(braided} or \define{ symmetric) monoidal opindexed} categories, 1-cells and 2-cells $\Mon\OpICat$, $\Br\Mon\OpICat$ and $\Sym\Mon\OpICat$.

All these 2-categories have sub-2-categories of monoidal (op)indexed categories with a fixed monoidal domain $(\X, \otimes, I)$, and specifically
\begin{gather}\label{eq:monicatX}
    \MonICat(\X)=\MonTCat_\pse(\X\op, \Cat) \\ \Mon\OpICat(\X)=\MonTCat_\pse(\X, \Cat)\nonumber
\end{gather}
the functor 2-categories of lax monoidal pseudofunctors, monoidal pseudonatural transformations and monoidal modifications.

Moreover, we can consider pseudomonoids in the strict context. Explicitly, the 2-category $\PsMon(\ICat_\spl)=\MonICat_\spl$ has as objects \define{monoidal strict indexed categories} namely (2-)functors $\M \maps \X\op\to\Cat$ from an ordinary monoidal category $\X$ which are lax monoidal as before, but the laxator and unitor \cref{eq:laxatorunitor} are strictly natural rather than pseudonatural transformations. The hom-categories $\PsMon(\ICat_\spl)(\M, \psN)$ between monoidal strict indexed categories are full subcategories of $\MonICat(\M, \psN)$ spanned by strict natural transformations---which are however still lax monoidal, i.e.\ equipped with isomorphisms \cref{eq:monpseudocomponents}.

Similarly to the previous \cref{sec:monfib} on fibrations, we end this section with the study of pseudomonoids in a different but related monoidal 2-category, namely $\ICat(\X)=\TCat_\pse(\X\op, \Cat)$ of indexed categories with a fixed domain $\X$. Working in this 2-category, or in $\OpICat(\X)$, there is no assumed monoidal structure on $\X$. Their monoidal structure is again cartesian: for two $\X$-indexed categories $\M, \psN \maps \X\op \to \Cat$, their product is
\begin{equation}\label{eq:icatxprod}
    \M\boxtimes\psN \maps \X\op \xrightarrow{\Delta} \X\op \times \X\op \xrightarrow{\M \times \psN} \Cat \times \Cat \xrightarrow{\times} \Cat
\end{equation}
with pointwise components $(\M \boxtimes \psN) (x) = \M (x) \times \psN (x)$ in $\Cat$. The monoidal unit is just $\X\op \xrightarrow{!} \1 \xrightarrow{\Delta \1} \Cat$, which we will also call $\Delta \1$.

A pseudomonoid in $(\ICat(\X), \boxtimes, \Delta\1)$ is a pseudofunctor $\M \maps \X\op \to \Cat$ equipped with indexed functors \cref{eq:ifun} $\mlt \maps \M\boxtimes\M\to\M$ and $\uni \maps \Delta\1\to\M$ namely
\begin{displaymath}
\begin{tikzcd}[row sep=.1in]
    & 
    \X\op \times \X\op
    \arrow[r, "\M\times\M"] 
    & 
    \Cat \times \Cat
    \arrow[dr, "\times"] 
    &&& 
    \1
    \arrow[dr, "\Delta\1"] 
    \\
    \X\op
    \arrow[ur, "\Delta"]
    \arrow[rrr, phantom, "\Downarrow{\scriptstyle\mlt}"description]
    \arrow[rrr, bend right=20, "\M"'] 
    &&&
    \Cat 
    & 
    \X\op
    \arrow[ur, "!"]
    \arrow[rr, bend right, "\M"']
    \arrow[rr, phantom, "\Downarrow{\scriptstyle\uni}"description] 
    && 
    \Cat
\end{tikzcd}
\end{displaymath}
with components $\mlt_x \maps \M x \times \M x \to \M x$ and $\uni_x \maps \1 \to \M x$ which are pseudonatural via
\begin{equation}\label{eq:pseudonaturalmult}
\begin{tikzcd}[column sep=.6in]
    \M x
    \times \M x
    \arrow[d, "\mlt_x"']
    \arrow[r, "\M f\times\M f"]
    \arrow[dr, phantom, "\cong"description] 
    & 
    \M y \times \M y
    \arrow[d, "\mlt_y"] 
    \\
    \M x 
    \arrow[r, "\M f"'] 
    & 
    \M y
\end{tikzcd}\quad
\begin{tikzcd}
    \1
    \arrow[r, "="]
    \arrow[d, "\uni_x"']
    \arrow[dr, phantom, "\cong"description] 
    & 
    \1
    \arrow[d, "\uni_y"] 
    \\
    \M x
    \arrow[r, "\M f"'] 
    & 
    \M y
\end{tikzcd}
\end{equation}
If we denote $\mlt_x=\otimes_x$ and $\uni_x=I_x$, the pseudomonoid invertible 2-cells \cref{alphalambdarho} and the axioms these data satisfy make each $\M x$ into a monoidal category $(\M x, \otimes_x, I_x)$, and each $\M f$ into a strong monoidal functor: the above isomorphisms have components $\M f(a)\otimes_y\M f(b)\cong\M f(a\otimes_x b)$ and $I_y\cong\M f(I_x)$ for any $a, b\in\M x$. 

Such a structure, namely a pseudofunctor $\M \maps \X\op \to \Mon\Cat$ into the 2-category of monoidal categories, strong monoidal functors and monoidal natural transformations, was directly defined as an \emph{indexed strong monoidal category} in \cite{DescentForMonads}, and as \emph{indexed monoidal category} in \cite{PontoShulman}. We will avoid this notation in order to not create confusion with the term \emph{monoidal indexed categories}.

A strong morphism of pseudomonoids \cref{eq:laxmorphism} in $(\ICat(\X), \boxtimes, \Delta\1)$ ends up being a pseudonatural trasformation $\tau \maps \M\Rightarrow\psN \maps \X\op\to\Cat$ (indexed functor) whose components $\tau_x \maps \M x\to\psN x$ are strong monoidal functors, whereas a 2-cell between strong morphisms of pseudomonoids is an ordinary modification
\begin{displaymath}
\begin{tikzcd}
    \X\op
    \arrow[rr, bend left=40, "\M", ""'{name = F}]
    \arrow[rr, bend right=40, "\psN"', ""{name = G}] 
    \arrow[rr, phantom, "\stackrel{m} {\Rrightarrow}"description] 
    && 
    \Cat
    \arrow[from = F, to = G, Rightarrow, "\tau"', bend right=50]
    \arrow[from = F, to = G, Rightarrow, "\sigma", bend left=50]
\end{tikzcd}
\end{displaymath}
whose components $m_x \maps \tau_x\Rightarrow\sigma_x$ are monoidal natural transformations.

We thus obtain the 2-categories $\PsMon (\ICat(\X))$ as well as $\PsMon(\OpICat(\X))$; from the above descriptions, it is clear that  
\begin{gather}\label{eq:imoncats}
    \PsMon(\ICat(\X))=\TCat_\pse(\X\op, \Mon\Cat) \\
    \PsMon(\OpICat(\X))=\TCat_\pse(\X, \Mon\Cat)\nonumber
\end{gather}
which will also be rediscovered by \cref{prop:imoncat_formally}.

Finally, taking pseudomonoids in strict $\X$-indexed categories $\ICat_\spl(\X)=[\X\op, \Cat]$ produces the 2-category $\PsMon(\ICat_\spl(\X))$ with objects functors $\M \maps \X\op\to\Mon\Cat_\mathrm{st}$ into monoidal categories with strict monoidal functors: the isomorphisms \cref{eq:pseudonaturalmult} are now equalities due to strict naturality of the multiplication and unit. Then the hom-categories $\PsMon(\ICat_\spl(\X))(\M, \psN)$ are full subcategories of $\PsMon(\ICat(\X))(\M, \psN)$ spanned by strictly natural transformations $\tau \maps \M\Rightarrow\psN$, still with strong monoidal components $\tau_x$. For example, it would not be correct to write $\PsMon(\ICat_\spl(\X))=[\X\op, \Mon\Cat_{(\mathrm{st})}]$.

It is evident that $\MonICat(\X)$ and $\PsMon(\ICat(\X))$ are in principle different. A monoidal indexed category with base $\X$ is a lax monoidal pseudofunctor into $\Cat$ (and $\X$ is required to be monoidal already), whereas a pseudomonoid in $\X$-indexed categories is a pseudofunctor from an ordinary category $\X$ into $\Mon\Cat$. 
    
\section{Two Monoidal Grothendieck Constructions}
\label{sec:monequiv}

In \cref{sec:fibrations}, we recalled 
the standard equivalence between fibrations and indexed categories via the Grothendieck construction. We will now lift this correspondence to their monoidal versions studied in Sections \ref{sec:monfib} and \ref{sec:monicat}, using general results about pseudomonoids in arbitrary monoidal 2-categories described in \cref{app:Monoidal2cats}.

Since both $\Fib$ and $\ICat$ are cartesian monoidal 2-categories, via \cref{Fib_cart} and \cref{ICat_cart} respectively, our first task is to ensure that they are \emph{monoidally} equivalent.

\begin{lem}\label{lem:monoidalGrfunctor}
    The 2-equivalence $\Fib\simeq\ICat$ between the cartesian monoidal 2-categories of fibrations and indexed categories is (symmetric) monoidal.
\end{lem}

\begin{proof}
    Since they form an equivalence, both 2-functors from \cref{thm:Grothendieck} preserve limits, therefore are monoidal 2-functors. Moreover, it can be verified that the natural isomorphisms with components $\F\cong\F_{P_\F}$ and $P\cong P_{\F_P}$ are monoidal with respect to the cartesian structure, due to universal properties of products.
\end{proof}

\begin{thm}
\label{thm:mainthm}
    There are 2-equivalences 
    \begin{gather*}
        \MonFib\simeq \MonICat \\
        \BrMonFib\simeq\Br\Mon\ICat \\
        \SymMonFib\simeq\Sym\Mon\ICat
    \end{gather*}
    between the 2-categories of monoidal fibrations and monoidal indexed categories, as well as their braided and symmetric versions. Dually, there is a 2-equivalence $\MonOpFib\simeq\MonOpICat$ between the 2-categories of monoidal opfibrations and monoidal opindexed categories, as well as their braided and symmetric versions.
\end{thm}
\begin{proof}
    Since $\MonFib=\PsMon(\Fib)$ and $\MonICat=\PsMon(\ICat)$, we obtain the equivalence as a special case of \cref{prop:2equivpseudomon}; similar for $\OpFib\simeq\OpICat$.
\end{proof}

\begin{cor}
\label{cor:fixedbasemonoidalGr}
    The above 2-equivalences restrict to the sub-2-categories of fixed bases or domains, which by \cref{eq:monicatX} are
    \begin{gather*}
        \MonFib(\X) \simeq\MonTCat_\pse(\X\op,\Cat) \\
        \MonOpFib(\X) \simeq\MonTCat_\pse(\X\op,\Cat)
    \end{gather*}
\end{cor}

These results correspond to the \emph{global} monoidal structure of fibrations and indexed categories. Even though they were directly derived via abstract reasoning, for exposition purposes we briefly describe this equivalence on the level of objects; some relevant details can also be found in \cite[Sec 6]{NetworkModels}. Independently and much earlier, in his thesis \cite{ShulmanPhD} Shulman explores such a fixed-base equivalence on the level of double categories (of monoidal fibrations and monoidal pseudofunctors over the same base).

Suppose that $(\M, \mu, \mu_0) \maps (\X\op, \otimes, I) \to (\Cat, \times, \1)$ is a monoidal indexed category, i.e.\ a lax monoidal pseudofunctor with structure maps \cref{eq:laxatorunitor}. The induced monoidal product $\otimes_\mu \maps \inta \M \times \inta \M \to \inta \M$ on the Grothendieck category is defined on objects by 
\begin{equation}
\label{eq:globalmonstr}
    (x,a) \otimes_\mu (y,b) = (x \otimes y, \mu_{x,y}(a,b))
\end{equation}
and $I_\mu=(I, \mu_0(*))$ is the unit object. Clearly, the induced fibration $\inta \M \to \X$ which maps each pair to the underlying $\X$-object strictly preserves the monoidal structure. Moreover, pseudonaturality of $\mu$ implies that $\otimes_\mu$ preserves cartesian liftings, so all clauses of \cref{def:monoidal_fibration} are satisfied. For a more detailed exposition of the structure, as well as the braided and symmetric version, we refer the reader to the \cref{sec:monoidal}.

We can also restrict to the context of split fibrations and strict indexed categories. 

\begin{thm}
\label{thm:mainthmsplit}
    There are 2-equivalences
    \begin{gather*}
        \MonFib_\spl \simeq \MonICat_\spl \\
        \MonOpFib_\spl \simeq \MonOpICat_\spl
    \end{gather*}
    between monoidal split (op)fibrations and monoidal strict (op)indexed categories, as well as for the fixed-base case.
\end{thm}
\begin{proof}
    Again by applying $\PsMon(\textrm{-})$ to the 2-equivalence $\ICat_\spl\simeq\Fib_\spl$, we obtain equivalences between the respective structures discussed in \cref{sec:monfib,sec:monicat}, as the strict counterparts of \cref{thm:mainthm} and \cref{cor:fixedbasemonoidalGr}. Recall that a monoidal strict indexed category is a lax monoidal 2-functor $\X\op\to\Cat$ whose structure maps $(\phi,\phi_0)$ are strictly natural transformations, and corresponds to a split fibration which is monoidal like before, only the tensor product of the total category strictly preserves cartesian liftings.
\end{proof}

We close this section in a similar manner to Sections \ref{sec:monfib} and \ref{sec:monicat}, namely by working in the cartesian monoidal 2-categories $(\Fib(\X), \boxtimes, 1_\X)$ and $(\ICat(\X), \boxtimes, \Delta\1)$ of fibrations and indexed categories with a fixed base category. 

\begin{thm}
\label{thm:fibrmonGr}
    There are 2-equivalences between (op)fibrations with monoidal fibres and strong monoidal reindexing functors, and pseudofunctors into $\Mon\Cat$
    \begin{align*}
        \PsMon(\Fib(\X)) &\simeq \TCat_\pse(\X\op,\Mon\Cat)\quad 
        \\
        \PsMon(\OpFib(\X)) &\simeq \TCat_\pse(\X\op,\Mon\Cat) 
    \end{align*}
    Moreover, these restrict to 2-equivalences between split (op)fibrations with monoidal fibres and strict monoidal reindexing functors, and ordinary functors into $\Mon\Cat_\mathrm{st}$.
\end{thm}
\begin{proof}
    Since $\Fib(\X) \simeq \ICat(\X)$ is also a monoidal 2-equivalence, \cref{prop:2equivpseudomon} applies once more -- recall \cref{eq:imoncats}.
\end{proof}

These equivalences correspond to the \emph{fibrewise} monoidal structure on fibrations and indexed categories. In more detail, a pseudofunctor $\M \maps \X\op \to \Mon\Cat$ maps every object $x$ to a monoidal category $\M x$ and every morphism $f \maps x \to y$ to a strong monoidal functor $\M f \maps \M y \to \M x$; under the usual Grothendieck construction, these are precisely the fibre categories and the reindexing functors between them for the induced fibration, as described at the end of \cref{sec:monfib}. Notice how, in particular, $\X$ is \emph{not} a monoidal category, as was the case in \cref{cor:fixedbasemonoidalGr}.

A very similar, relaxed version of the fibrewise monoidal correspondence seems to connect the concepts of an \emph{indexed monoidal category}, defined in \cite{DescentForMonads} as a pseudofunctor $\M\maps\X\op\to\Mon\Cat_\lax$, and that of of a \emph{lax monoidal fibration}, defined in \cite{LaxMonFibs}. Notice that these terms are misleading with respect to ours: an indexed monoidal category is \emph{not} a monoidal indexed category, and also a lax monoidal fibration is \emph{not} a functor with a lax monoidal stucture.

Briefly, there is a full sub-2-category $\Fib_\opl(\X)\subseteq\Cat/\X$ of fibrations, namely fibred 1-cells \cref{commutativefibredcell} which are not required to have a cartesian functor on top. As discussed in \cite[Prop.3.6]{FramedBicats}, this is 2-equivalent to $\TCat_{ps,opl}(\X\op,\Cat)$, the 2-category of pseudofunctors, oplax natural transformations and modifications. Describing pseudomonoids therein appears to give rise to a fibration with monoidal fibres and \emph{lax} monoidal reindexing functors between them, or equivalently a pseudofunctor into $\Mon\Cat_\lax$. We omit the details so as to not digress from our main development.

\section{Summary of Structures}\label{sec:summary}

The bulk of this chapter is dedicated to proving various monoidal variations of the equivalence between fibrations and indexed categories, using general results in monoidal 2-category theory. In this section, we detail the descriptions of the (braided/symmetric) monoidal structures on the total category of the Grothendieck construction, assuming the appropriate data is present. We also provide a hands-on correspondence that underlies the proof of \cref{thm:fibrewise=global} regarding the transfer of monoidal structure from a functor to its target and vice versa. We hope this section can serve as a quick and clear reference on some fundamental constructions of this chapter.

\subsection{Monoidal Structures}\label{sec:monoidal}

As sketched under \cref{cor:fixedbasemonoidalGr}, let $(\X, \otimes, I)$ be a monoidal category, and \[(\M, \mu, \mu_0) \maps (\X\op, \otimes\op, I) \to (\Cat, \times, \1)\] a monoidal indexed category, a.k.a. lax monoidal pseudofunctor.  Recall that $\mu$ is pseudonatural transformation consisting of functors $\mu_{x,y} \maps \M x \times \M y \to \M(x \otimes y)$ for any objects $x$ and $y$ of $\X$, and natural isomorphisms
\[
\begin{tikzcd}[column sep = 70]
    \M z \times \M w 
    \arrow[r, "\M f \times \M g"]
    \arrow[d, "\mu_{z,w}", swap]
    &
    \M x \times \M y
    \arrow[d, "\mu_{x,y}"]
    \arrow[dl, phantom, "{\scriptstyle\stackrel{\mu_{f,g}}{\cong}}"]
    \\
    \M(z \otimes w)
    \arrow[r, "\M(f \otimes g)", swap]
    &
    \M(x \otimes y)
\end{tikzcd}\]
for any arrows $f \maps x \to z$ and $g \maps y \to w$ in $\X$. Also the unique component of $\mu_0$ is the functor $\mu_0\maps\1\to\M(I)$.

The induced tensor product functor on the total category, denoted as $\otimes_\mu \maps \inta \M \times \inta \M \to \inta \M$, is given on objects by 
\begin{displaymath}
    (x,a) \otimes_\mu (y,b) = (x \otimes y, \mu_{x,y}(a,b))
\end{displaymath}
On morphisms $\left(f\maps x\to z, k\maps a\to(\M f)c\right)$ and $\left(g\maps y\to w,\ell\maps b\to(\M g)d\right)$, we get
\[
    (f, k) \otimes_\mu (g, \ell) = (x\otimes y\xrightarrow{f \otimes g} z\otimes w, \mu_{f,g}(\mu_{x, y}(k, \ell)))
\]
where the latter is the composite morphism
\begin{displaymath}
    \mu_{x,y}(a,b)\xrightarrow{\mu_{x,y}(k,\ell)}\mu_{x,y}\left((\M f)(c),(\M g)(d)\right)\xrightarrow{\sim}\M(f\otimes g)(\mu_{z.w}(c,d))\textrm{ in }\M(x\otimes y).
\end{displaymath}
The monoidal unit is $I_\mu=(I, \mu_0)$.

If $a_{x,y,z} \maps (x \otimes y) \otimes z \to x \otimes (y \otimes z)$ denotes the associator in $\X$, the associator for $(\inta \M, \otimes_\mu, I_\mu)$ is given by
\begin{displaymath}
    \alpha_{(x,b), (y,c), (z,d)} = (\alpha_{x,y,z}, \omega_{x,y,z} (b,c,d))
\end{displaymath}
where $\omega$ is the invertible modification \cref{eq:omega}. 

If $l_x \maps I \otimes x \to x$ and $r_x \maps x \otimes I \to x$ are the left and right unitors in $\X$, the unitors in $\inta \M$ are defined as
\begin{gather*}
    \lambda_x = (l_x, \xi_x^{\text{-}1}(a))\maps (I,\mu_0)\otimes_\mu(x,a)\to(x,a) \\
    \rho_x = (r_x, \zeta_x(a))\maps(x,a)\otimes_\mu(I,\mu_0)\to(x,a)
\end{gather*}
where $\zeta$ and $\xi$ are invertible modifications as in \cref{eq:omega}.

We now turn to the correspondence between 1-cells of \cref{thm:mainthm}:
given a monoidal indexed 1-cell
\[
\begin{tikzcd}
    (\X, \otimes, I)\op
    \arrow[dr, "{(\M, \mu, \mu_0)}"]
    \arrow[dd, "{(F, \psi, \psi_0)\op}", swap]
    \\
    \arrow[r, phantom, "\Downarrow{\scriptstyle\tau}"]
    &
    (\Cat, \times, \1)
    \\
    (\Y, \otimes, I)\op
    \arrow[ur, "{(\psN, \nu, \nu_0)}", swap]
\end{tikzcd}\]
where $\M$ and $\psN$ are lax monoidal pseudofunctors and $F$ is a monoidal functor, as in \cref{prop:moni1cell}, we first of all obtain an ordinary fibred 1-cell $(P_\tau,F)\maps P_\M\to P_\psN$ as explained above \cref{eq:inducedfibred1cell}
\begin{displaymath}
\begin{tikzcd}
    \inta \M
    \arrow[r,"P_\tau"]
    \arrow[d,"P_\M"'] 
    & 
    \inta
    \psN
    \arrow[d,"P_\psN"] 
    \\
    \X
    \arrow[r,"F"'] 
    & 
    \Y
\end{tikzcd}
\end{displaymath}
with $P_\tau (x, a) = (F x, \tau_x (a))$. The functor $F$ is already monoidal, and $P_\tau$ obtains a monoidal structure too: for example, there are isomorphisms
\begin{displaymath}
    P_\tau(x,a) \otimes_\nu P_\tau (y, b) \xrightarrow{\sim} P_\tau ((x, a) \otimes_\mu (y, b)) \quad \textrm{ in } \inta \psN
\end{displaymath}
between the objects
\begin{align*}
    P_\tau (x, a) \otimes_\nu P_\tau (y, b)
    &= (Fx, \tau_x (a)) \otimes_\nu (Fy, \tau_y (b) 
    = (Fx \otimes Fy, \nu_{Fx, Fy} (\tau_x (a), \tau_y (b)) 
    \\
    P_\tau((x,a)\otimes_\mu(y,b))
    &= P_\tau(x\otimes y,\mu_{x,y}(a,b))
    = (F(x\otimes y),\tau_{x\otimes y}(\mu_{x,y}(a,b)))
\end{align*}
given by $\psi_{x,y} \maps Fx \otimes Fy \xrightarrow{\sim} F (x \otimes y)$ and by 
\[
    \nu_{Fx, Fy} (\tau_x (a), \tau_y (b)) \cong \psN (\psi_{x, y}) (\tau_{x \otimes y} (\mu_{x, y} (a, b)))
\]
essentially given by the monoidal pseudonatural isomorphism \cref{eq:monpseudocomponents} for $\tau \maps \M \Rightarrow \psN F\op$. As a result, $(P_\tau,F)$ is indeed a monoidal fibred 1-cell as in \cref{prop:monoidalfibred1cell}.

Finally, it can be verified that starting with a monoidal indexed 2-cell as in \cref{prop:moni2cell}, the induced fibred 2-cell \cref{eq:inducedfibred2cell} is monoidal, i.e.\ $P_m$ satisfies the conditions of a monoidal natural transformation.

Regarding the induced braided and symmetric monoidal structures, suppose that $(\X,\otimes,I)$ is a braided monoidal category, with braiding $b$ with components
\[
    \braid_{x,y} \maps x \otimes y \xrightarrow{\sim} y \otimes x;
\]
then $\X\op$ is braided monoidal with the inverse braiding, namely $(\X\op,\otimes\op,I,\braid^{-1})$.
Now if $(\M,\mu,\mu_0)\maps\X\op\to\Cat$ is a \emph{braided} lax monoidal pseudofunctor, i.e.\ a braided monoidal indexed category,
by \cref{thm:mainthm} we have an induced braided monoidal structure on $(\inta\M, \otimes_\mu, I_\mu)$, namely
\[
    B_{(x, a), (y, b)} \maps (x, a) \otimes_\mu (y,b)=(x\otimes y,\mu_{x,y}(a,b)) \to (y,b) \otimes_\mu (x, a)=(y\otimes x,\mu_{y,x}(b,a))
\] 
are given by $\braid_{x,y} \maps x \otimes y \cong y \otimes x$ in $\X$ and $(v_{x,y})_{(a,b)} \maps \mu_{x,y} (a,b) \cong \M (\braid^{-1}_{x,y}) (\mu_{y,x} (b,a))$, where $v$ is as in \cref{eq:brweakmonpseudo}. 

If $\M$ is a symmetric lax monoidal pseudofunctor, it can be verified that 
\[
    B_{(y, b),(x, a)} \circ B_{(x, a),(y, b)} = 1_{(x,a) \otimes_\mu (y,b)}
\] 
therefore $\inta \M$ is also symmetric monoidal, as is the monoidal fibration $P_\M \maps \inta \M \to \X$.

\subsection{Monoidal Indexed Categories as Ordinary Pseudofunctors}\label{monicat=imoncat}

Here we detail the correspondence between monoidal opindexed categories and a pseudofunctors into $\Mon\Cat$ when the domain is a cocartesian monoidal category, as established by \cref{thm:fibrewise=global}; the one for indexed categories is of course similar. We denote by $\nabla_x \maps x+x \to x$ the induced natural components due to the universal property of coproduct, and $\iota_x \maps x \to x+y$ the inclusion into a coproduct.

Start with a lax monoidal pseudofunctor $\M \maps (\X, +, 0) \to (\Cat, \times, \1)$ equipped with $\mu_{x,y}\maps \M(x) \times \M(y) \to \M(x + y)$ and $\mu_0 \maps \1 \to \M(0)$, which gives the global monoidal structure \cref{eq:globalmonstr} of the corresponding opfibration. There exists an induced monoidal structure on each fibre $\M(x)$ as follows:
\begin{gather}
\label{eq:explicitstructure1}
    \otimes_x \maps \M(x) \times \M(x) \xrightarrow{\mu_{x,x}} \M(x + x) \xrightarrow{\M(\nabla)} \M(x)
    \\
    I_x \maps \1 \xrightarrow{\mu_0} \M(0) \xrightarrow{\M(!)} \M(x) \nonumber
\end{gather}
Moreover, each $\M f \maps \M x \to \M y$ is a strong monoidal functor, with $\phi_{a,b} \maps (\M f)(a) \otimes_y (\M f)(b) \xrightarrow{\sim} \M f(a \otimes_xb)$ and $\phi_0 \maps I_y \xrightarrow{\sim} (\M f) I_x$ essentially given by the following isomorphisms
\begin{equation}\label{eq:strongmonreindex}
\begin{tikzcd}[row sep=.3in,column sep=.8in]
    \M x \times \M x
    \arrow[r,"\M f\times\M f"]
    \arrow[d,"\mu_{x,x}"']
    \arrow[dr,phantom,"{\scriptstyle\stackrel{\mu^{f,f}}{\cong}}"description] 
    & 
    \M y \times \M y
    \arrow[d,"\mu_{y,y}"] 
    \\
    \M(x+x)
    \arrow[d,"\M(\nabla_x)"']
    \arrow[r,"\M(f+f)"description]
    \arrow[dr,phantom,"{\scriptstyle\cong}"description] 
    & \M(y+y)\arrow[d,"\M(\nabla_y)"] 
    \\
    \M x \arrow[r,"\M f"'] 
    & 
    \M y
\end{tikzcd}
\qquad
\begin{tikzcd}
    \1
    \arrow[r,"\mu_0"]
    \arrow[d,"\mu_0"']
    \arrow[ddr,phantom,"{\scriptstyle\cong}"description] 
    & 
    \M(0)
    \arrow[dd,"\M(!)"] 
    \\
    \M(0)
    \arrow[d,"\M(!)"'] 
    &\\
    \M x 
    \arrow[r,"\M f"']
    & 
    \M y
\end{tikzcd}
\end{equation}
since $\nabla$ and $!$ are natural and $\M$ is a pseudofunctor.

In the opposite direction, take an ordinary pseudofunctor $\M\maps\X\to\Mon\Cat$ into the 2-category of monoidal categories, strong monoidal functors and monoidal natural transformations, with $\otimes_x\maps\M(x)\times\M(x)\to\M(x)$ and $I_x$ the fibrewise monoidal structures in every $\M x$. We can use those to endow $\M$ with a lax monoidal structure via 
\begin{gather*}
    \mu_{x,y} \maps \M(x) \times \M(y) \xrightarrow{\M(\iota_x) \times \M(\iota_y)} \M(x+y) \times \M(x+y) \xrightarrow{\otimes_{x+y}} \M(x+y) 
    \\
    \mu_0 \maps \1 \xrightarrow{I_0} \M(0)
\end{gather*}
The fact that all $\M f$ are strong monoidal imply that the above components form pseudonatural transformations, and all appropriate conditions are satisfied. 

In the strict context, a lax monoidal 2-functor $\M\maps(\X,+,0)\to(\Cat,\times,\1)$ with natural laxator and unitor bijectively corresponds to a functor $\X\to\Mon\Cat_\textrm{st}$ since \cref{eq:strongmonreindex} are in fact strictly commutative, by naturality of $\mu,\mu_0$ and functoriality of $\M$.
    
In the even more special case of an ordinary lax monoidal functor $\M\maps(\X,+,0)\to(\Cat,\times,\1)$, the fibres $\M(x)$ turn out to be \emph{strict} monoidal. For example, strict associativity of the tensor is established by
\begin{displaymath}
\begin{tikzcd}[column sep=.15in]
    \M x \times \M x \times \M x
    \arrow[rr, "1 \times \otimes_x"]
    \arrow[dr, "1 \times \mu_{x, x}"']
    \arrow[dddd] 
    &
    \phantom{A}
    & 
    \M x\times \M x
    \arrow[ddrr, bend left, "\otimes_x"]
    \arrow[dr, "\mu_{x, x}"description] 
    \\&
    \M x\times \M (x+x)
    \arrow[dd, phantom, "(*)"]
    \arrow[ur, "1\times\M(\nabla)"description]
    \arrow[dr, "\mu_{x, x+x}"description] 
    && 
    \M (x+x)
    \arrow[dr, "\M\nabla"description]
    \\&& 
    \M(x+x+x)
    \arrow[ur, "\M (1+\nabla)"description]
    \arrow[dr, "\M(\nabla+1)"description] 
    && 
    \M x 
    \\& 
    \M(x+x)\times\M x
    \arrow[ur, "\mu_{x+x, x}"description]
    \arrow[dr, "\M(\nabla)\times1"description] 
    && 
    \M(x+x)
    \arrow[ur, "\M\nabla"description] 
    \\
    \M x\times\M x\times\M x
    \arrow[ur, "\mu_{x, x}\times1"]
    \arrow[rr, "\otimes_x\times1"'] 
    &\phantom{A}&
    \M x\times\M x
    \arrow[ur, "\mu_x"description]
    \arrow[uurr, bend right, "\otimes_x"'] 
    &&
\end{tikzcd}
\end{displaymath}
where the three diamond-shaped diagrams on the right commute due to naturality of $\mu$ as well as associativity of $\nabla$ and functoriality of $\M$ already in the monoidal strict opindexed case, whereas $(*)$ is in general $\omega$ from \cref{eq:omega} which in this case is an identity, and the four triangular diagrams commute due to \cref{eq:explicitstructure1}.

\subsection{Comparison with Higher-Dimensional Grothendieck Constructions}

Monoidal categories are precisely bicategories with one object. As recalled in \cref{sec:fibrations}, there is a theory of fibred bicategories and indexed bicategories, and a corresponding Grothendieck construction. It is natural to consider the possibility that the monoidal Grothendieck constructions presented here are special cases of this bicategorical version. However, it is easy to see that this cannot be the case. When one restricts their view to just the objects, the bicategorical Grothendieck construction is just taking the disjoint union of the object sets of the fibres. If you consider an indexed monoidal category as a special case of an indexed bicategory, where each fibre has one object, then generally you would not expect the total bicategory to have one object. It would have as many objects as the base category. Thus, the result would not be a monoidal category. The construction given here always produces a monoidal category.

\section{The (Co)cartesian Case}\label{sec:cartesiancase}

In the previous section, we obtain two different equivalences between fixed-base fibrations and fixed-domain indexed categories of monoidal flavor: \cref{cor:fixedbasemonoidalGr} where both total and base categories are monoidal, and \cref{thm:fibrmonGr} where only the fibres are monoidal. Clearly, neither of these two cases implies the other in general. The global monoidal structure as defined in \cref{eq:globalmonstr} sends two objects in arbitrary fibres to a new object lying in the fibre of the tensor of their underlying objects in the base, whereas having a fibre-wise tensor products does not give a way of multiplying objects in different fibres of the total category.

In \cite{FramedBicats}, Shulman introduces monoidal fibrations (\cref{def:monoidal_fibration}) as a building block for fibrant double categories. Due to the nature of the examples, the results restrict to the case where the base of the monoidal fibration $P \maps \A \to \X$ is equipped with specifically a cartesian or cocartesian monoidal structure; the main idea is that these fibrations form a ``parameterized family of monoidal categories''. Formally, a central result therein lifts the Grothendieck construction to the monoidal setting, by showing an equivalence between monoidal fibrations over a fixed (co)cartesian base and ordinary pseudofunctors into $\Mon\Cat$.

\begin{thm}[\cite{FramedBicats}]
\label{thm:Shulman}
    If $\X$ is cartesian monoidal, 
    \begin{equation}\label{eq:Shulmanequiv}
        \MonFib(\X) \simeq \TCat_\pse(\X\op, \Mon\Cat)
    \end{equation}
    Dually, if $\X$ is cocartesian monoidal, $\MonOpFib(\X) \simeq \TCat_\pse(\X, \Mon\Cat)$.
\end{thm}

Bringing all these structures together, we obtain the following.

\begin{thm}\label{thm:fibrewise=global}
    If $\X$ is a cartesian monoidal category, 
    \begin{displaymath}
    \begin{tikzcd}[ampersand replacement=\&, sep=.25in]
        \MonFib(\X)
        \arrow[r, "\simeq"]
        \arrow[d, "\simeq"'{anchor=south, rotate=90, inner sep=.5mm}]
        \&
        \MonTCat_\pse(\X\op, \Cat)
        \arrow[d, "\simeq"{anchor=south, rotate=270, inner sep=.5mm}] 
        \\
        \PsMon(\Fib(\X))
        \arrow[r, "\simeq"] 
        \& 
        \TCat_\pse(\X\op, \Mon\Cat)
    \end{tikzcd}
    \end{displaymath}
    Dually, if $\X$ is a cocartesian monoidal category, 
    \begin{displaymath}
    \begin{tikzcd}[ampersand replacement=\&, sep=.25in]
        \MonOpFib(\X)
        \arrow[r, "\simeq"]
        \arrow[d, "\simeq"'{anchor=south, rotate=90, inner sep=.5mm}]
        \&
        \MonTCat_\pse(\X, \Cat)
        \arrow[d, "\simeq"{anchor=south, rotate=270, inner sep=.5mm}] 
        \\
        \PsMon(\OpFib(\X))
        \arrow[r, "\simeq"] 
        \& 
        \TCat_\pse(\X, \Mon\Cat)
    \end{tikzcd}
    \end{displaymath}
    In the strict context, the restricted equivalences give a correspondence between monoidal split fibrations over $\X$ and functors $\X^{\op} \to \Mon\Cat_\mathrm{st}$, and between monoidal split opfibrations over $\X$ and functors $\X \to \Mon\Cat_\mathrm{st}$.
\end{thm}

The original proof of \cref{thm:Shulman} is an explicit, piece-by-piece construction of an equivalence, and employs the reindexing functors $\Delta^*$ and $\pi^*$ induced by the diagonal and projections in order to move between the appropriate fibres and build the required structures. The global monoidal structure is therein called \emph{external} and the fibre-wise \emph{internal}.

Here we present a different argument that does not focus on the fibrations side. The equivalence between lax monoidal pseudofunctors $\X\op \to \Cat$ and ordinary pseudofunctors $\X\op \to \Mon\Cat$, which essentially provides a way of transferring the monoidal structure from the target category to the functor itself and vice versa, brings a new perspective on the behavior of such objects.

\begin{lem}\label{lem:helplemma}
For any two monoidal 2-categories $\K$ and $\L$, the following are true.
\begin{enumerate}
    \item For an arbitrary 2-category $\A$, 
    \begin{equation}\label{eq:equiv1}
        \TCat_\pse(\A, \MonTCat_\pse(\K, \L))\simeq\MonTCat_\pse(\K, \TCat_\pse(\A, \L))
    \end{equation}
    \item For a cocartesian 2-category $\A$, 
    \begin{equation}\label{eq:equiv2}
        \TCat_\pse(\A, \MonTCat_\pse(\K, \L))\simeq\MonTCat_\pse(\A\times\K, \L)
    \end{equation}
\end{enumerate}
\end{lem}

\begin{proof}
    First of all, recall \cite[1.34]{FibrationsinBicats} that there are equivalences
    \begin{displaymath}
        \TCat_\pse(\A, \TCat_\pse(\K, \L))\simeq
        \TCat_\pse(\A\times\K, \L)\simeq
        \TCat_\pse(\K, \TCat_\pse(\A, \L))
    \end{displaymath}
    which underlie \cref{eq:equiv1} and \cref{eq:equiv2} for the respective pseudofunctors; so the only part needed is the correspondence between the respective monoidal structures. Notice that $\A \times \K$ is a monoidal 2-category since both $\A$ and $\K$ are, and also $\TCat_\pse (\A, \L)$ is monoidal since $\L$ is: define $\otimes_{[]}$ and $I_{[]}$ by $(\F \otimes_{[]} \G)(a) = \F a \otimes_\L \G a$ (similarly to \cref{eq:icatxprod}) and $I_{[]} \maps \A \xrightarrow{!} \1 \xrightarrow{I_\L} \L$.
    
    First, we prove 1. Take a pseudofunctor $\F \maps \A \to \MonTCat_\pse (\K, \L)$. For every $a \in \A$, its image pseudofunctor $\F a$ is lax monoidal, i.e.\ comes equipped with maps in $\L$:
    \begin{equation}
    \label{eq:Faweakmon}
        \phi_{x, y}^a \maps (\F a) (x) \otimes_\L (\F a) (y) \to (\F a) (x \otimes_\K y), \quad \phi_0^a \maps I_\L \to (\F a) I_\K
    \end{equation}
    for every $x, y \in \K$, satisfying coherence axioms.
    
    Now define the pseudofunctor $\overline{\F} \maps \K \to \TCat_\pse(\A, \L)$, with $(\overline{\F} x) (a) := (\F a) (x)$. It has a lax monoidal structure, given by pseudonatural transformations
    \begin{displaymath}
        \overline{\F} x \otimes_{[]} \overline{\F} y \Rightarrow \overline{\F} (x \otimes_\K y), \quad I_{[]} \Rightarrow\overline{\F}(I_\K) 
    \end{displaymath}
    whose components evaluated on some $a \in \A$ are defined to be \cref{eq:Faweakmon}. Pseudonaturality and lax monoidal axioms follow, and in a similar way we can establish the opposite direction and verify the equivalence.
    
    Now, we prove 2. If $\A$ is a cocartesian monoidal 2-category, a lax monoidal pseudofunctor $\F \maps \A \to \MonTCat_\pse (\K, \L)$ induces a pseudofunctor $\tilde{\F} \maps \A \times \K \to \L$ by $\tilde{\F} (a, x) := (\F a) (x)$. Its lax monoidal structure is given by the composite
    \begin{displaymath}
    \begin{tikzcd}[row sep=.1in, column sep=.2in]
        \tilde{\F}(a, x) \otimes_\L \tilde{\F}(b, y)
        \arrow[d, equal]
        \arrow[rr, dashed, "\psi_{(a, x), (b, y)}"] 
        && 
        \tilde{\F}(a + b, x \otimes_\K y)
        \arrow[d, equal] 
        \\
        (\F a)(x) \otimes_\L (\F b)(y)
        \arrow[rdd, "{(\F \iota_a)_x \otimes (\F\iota_b)_y}"'] 
        && 
        (\F(a+b))(x \otimes_\K y) 
        \\\\& 
        (\F(a+b))(x) \otimes_\L(\F(a+b))(y)
        \arrow[uur, "\phi^{a+b}_{x, y}"'] 
        & 
    \end{tikzcd}
    \end{displaymath}
    where $a \xrightarrow{\iota_a} a + b \xleftarrow{\iota_b} b$ are the inclusions, and $\psi_0 \maps I_\L \xrightarrow{\phi_0^0} \tilde{\F} (0, I_\K)$; the respective axioms follow.
    
    In the opposite direction, starting with some pseudofunctor $\G \maps \A \times \K \to \L$ equipped with a lax monoidal structure $\psi_{(a, x), (b, y)}$ and $\psi_0$, we can build $\hat{\G} \maps \A \to \MonTCat_\pse(\K, \L)$ for which every $\hat{\G}a$ is a lax monoidal pseudofunctor, via
    \begin{gather*}
    \begin{tikzcd}[row sep=.1in, column sep=.5in, ampersand replacement=\&]
        (\hat{\G}a)(x) \otimes_\L (\hat{\G}b)(y)
        \arrow[d, equal]
        \arrow[rr, dashed, "\phi^a_{(x, y)}"] 
        \&\& 
        (\hat{\G}a)(x\otimes_\K y)
        \arrow[d, equal] 
        \\
        \G(a, x) \otimes_\L \G(a, y)
        \arrow[r, "{\psi_{(a, x), (a, y)}}"'] 
        \&
        \G(a + a, x \otimes_\K y)
        \arrow[r, "{G(\nabla, 1)}"'] 
        \& 
        \G(a, x\otimes_\K y)
    \end{tikzcd}\\
    \phi_0^a \maps I_\L \xrightarrow{\psi_0} G(0, I_\K) \xrightarrow{G(!, 1)} G(a, I_\K)
    \end{gather*}
    The equivalence follows, using the universal properties of coproducts and initial object.
\end{proof} 

\begin{proof}[Proof of \cref{thm:fibrewise=global}]
The top and bottom right 2-categories of the first square are equivalent as follows, where $\X\op$ is cocartesian.
\begin{align*}
    \TCat_\pse(\X\op, \Mon\Cat)
    & \simeq \TCat_\pse (\X\op, \PsMon(\Cat)) & \text{\cref{eq:PsMon}}\\
    & \simeq \TCat_\pse (\X\op, \MonTCat_\pse(\1, \Cat)) & \text{\cref{eq:equiv2}} \\
    & \simeq \MonTCat_\pse (\X\op \times \1, \Cat)  \\
    & \simeq \MonTCat_\pse(\X\op, \Cat)
\end{align*}
The strict context equivalence can be explicitly verified as a special case of the above, where the corresponding 1-cells and 2-cells are as described in \cref{sec:monfib} and \cref{sec:monicat}.
\end{proof}

The decisive step in the above proof is
the much broader \cref{lem:helplemma}; for a grounded explanation of the correspondence of the relevant structures, see \cref{monicat=imoncat}.
In simpler words, a lax monoidal structure of a pseudofunctor $F\maps(\A, +, 0)\to(\Cat, \times, \1)$ gives a pseudofunctor $F\maps\A\to\Mon\Cat$ and vice versa: in a sense, `monoidality' can move between the functor and its target.

As another corollary of \cref{lem:helplemma}, we can formally deduce that pseudomonoids in $(\ICat(\X), \boxtimes, \Delta\1)$ are functors into $\Mon\Cat$, as described at the end of \cref{sec:monicat}.

\begin{prop}\label{prop:imoncat_formally}
    For any $\X$, $\PsMon(\ICat(\X))\simeq\TCat_\pse(\X\op, \Mon\Cat)$.
\end{prop}

\begin{proof} There are equivalences
\begin{align*}
    \PsMon(\ICat(\X))
    & = \PsMon(\TCat_\pse(\X\op, \Cat)) 
    \\& \simeq \MonTCat_\pse (\1, \TCat_\pse(\X\op, \Cat)) & \text{\cref{eq:equiv1}}
    \\& \simeq \TCat_\pse(\X\op, \MonTCat_\pse(\1, \Cat)) & \text{\cref{eq:PsMon}} 
    \\& \simeq \TCat_\pse(\X\op, \PsMon(\Cat))
    \\& \simeq \TCat_\pse(\X\op, \Mon\Cat) 
\end{align*}
as desired.
\end{proof}

As a first and meaningful example of \cref{thm:fibrewise=global}, recall that the categories $\Fib$ and $\ICat$ are themselves fibred over $\Cat$, with fibres $\Fib(\X)$ and $\ICat(\X)$ respectively. The base category in both cases is the cartesian monoidal category $(\Cat, \times, 1)$, therefore \cref{thm:fibrewise=global} applies. The following proposition shows that the monoidal structures of $\Fib$, $\ICat$ and $\Fib(\X)$, $\ICat(\X)$, instrumental for the study of global and fibre-wise monoidal structures, follow the very same abstract pattern.  

\begin{prop}\label{prop:FibICatglobalfibrewise}
    The fibrations $\Fib\to\Cat$ and $\ICat\to\Cat$ are monoidal, and moreover their fibres $\Fib(\X)$ and $\ICat(\X)$ are monoidal and the reindexing functors are strong monoidal.
\end{prop}

\begin{proof}
    The pseudofunctors inducing $\Fib\to\Cat$ and $\ICat\to\Cat$ are
    \begin{displaymath}
    \begin{tikzcd}[row sep=.05in]
        \Cat\op
        \arrow[rr] 
        && 
        \mathsf{CAT} 
        && 
        \Cat\op
        \arrow[rr] 
        && 
        \mathsf{CAT} 
        \\
        \X
        \arrow[mapsto, rr]
        \arrow[dd, "F"'] 
        && 
        \Fib(\X) 
        && 
        \X
        \arrow[mapsto, rr]
        \arrow[dd, "F"'] 
        && 
        \ICat(\X) 
        \\\\
        \Y
        \arrow[mapsto, rr] 
        && 
        \Fib(\Y)
        \arrow[uu, "F^*"'] 
        &&
        \Y
        \arrow[mapsto, rr] 
        && 
        \ICat(\Y)
        \arrow[uu, "-\circ F\op"']
    \end{tikzcd}
    \end{displaymath}
    where $\mathsf{CAT}$ is the 2-category of possibly large categories, $F^*$ takes pullbacks along $F$ and $-\circ F\op$ precomposes with the opposite of $F$. These are both lax monoidal, with the respective structures essentially being \cref{Fib_cart} and \cref{ICat_cart} giving the global monoidal structure on the fibrations.
    
    Since the base of both monoidal fibrations is cartesian, the global monoidal structure is equivalent to a fibre-wise monoidal structure, as per the theme of this whole section. The induced monoidal structure on each $\Fib(\X)$ is given by \cref{Fib_X_cart} and on each $\ICat(\X)$ by \cref{eq:icatxprod}, and $F^*$, $-\circ F\op$ are strong monoidal functors accordingly. 
\end{proof}

The above essentially lifts the global and fibre-wise monoidal structure development one level up, exhibiting fibrations and indexed categories as examples of the monoidal Grothendieck construction themselves.

Concluding this investigation on monoidal structures of fibrations and indexed categories, we consider the (co)cartesian monoidal (op)fibration case; for example, a monoidal fibration $P\maps(\A, \times, 1)\to(\X, \times, 1)$ as in \cref{def:monoidal_fibration} where $P$ preserves products (or coproducts for opfibrations) on the nose. As remarked in \cite[12.9]{FramedBicats}, the equivalence \cref{eq:Shulmanequiv} restricts to one between pseudofunctors which land to cartesian monoidal categories, and monoidal fibrations where the total category is cartesian monoidal. With the appropriate 1-cells and 2-cells that preserve the structure, we can write the respective equivalences as
\begin{gather}
    \TCat_\pse(\X\op, \Cart) \simeq \cMonFib(\X) \textrm{ for cartesian $\X$}\label{eq:cocartspecialcase} \\
    \TCat_\pse(\X, \Cocart) \simeq \cocMonOpFib(\X) \textrm{ for cocartesian $\X$}\nonumber
\end{gather}
where the prefixes $\mathsf{c}$ and $\mathsf{coc}$ correspond to the respective (co)cartesian structures. Explicitly, in order for the total category to specifically be endowed with (co)cartesian monoidal structure, it is required not only that the base category is but also the fibres are and the reindexing functors preserve finite (co)products.

This special case of the monoidal Grothendieck construction that connects the existence of (co)products and initial/terminal object in the fibres and in the total category, is reminiscent (and also an example of) the general theory of \emph{fibred limits} originated from \cite{Grayfibredandcofibred}. Explicitly, \cite[Cor.~4.9]{HermidaFib} deduces that if the base of a fibration $P \maps \A \to \X$ has $\J$-limits for any small category $\J$, then the fibres have and the reindexing functors preserve $\J$-limits if and only if $\A$ has $\J$-limits and $P$ strictly preserves them, and dually for opfibrations and colimits. Hence for finite (co)products in (op)fibrations, \cref{eq:cocartspecialcase} re-discovers that result using the monoidal Grothendieck correspondence. 

Moreover, since the squares of \cref{thm:fibrewise=global} reduce to their (co)cartesian variants, we would like to identify the conditions that the corresponding lax monoidal pseudofunctor into $\Cat$ needs to satisfy in order to give rise to a (co)cartesian monoidal (op)fibration. We employ \cref{prop:cocarthasmonoids} to tackle the opfibration case: if, in a symmetric monoidal category $\X$, there exist monoidal natural transformations with components
\begin{displaymath}
    \nabla_x \maps x \otimes x \to x, \quad u_x \maps I \to x
\end{displaymath}
satisfying the commutativity of
\begin{equation}\label{eq:nabla}
\begin{tikzcd}
    I \otimes x
    \arrow[r, "u_x\otimes1"]
    \arrow[dr, "\sim"{rotate=-30}, "\ell_x"'] 
    & 
    x\otimes x
    \arrow[d, "\nabla_x"] 
    & 
    x\otimes I
    \arrow[dr, "\sim"{rotate=-30}, "r_x"']
    \arrow[r, "1\otimes u_x"] 
    & 
    x\otimes x
    \arrow[d, "\nabla_x"] 
    \\& 
    x
    && 
    x
\end{tikzcd}
\end{equation}
then $\X$ is cocartesian monoidal. In fact, it is the case that a symmetric monoidal category is cocartesian if and only if $\Mon(\X)\cong\X$.

Suppose $(\M, \mu, \mu_0) \maps \X\to\Cat$ is a (symmetric) lax monoidal pseudofunctor, such that the corresponding Grothendieck category $(\inta \M, \otimes_\mu, I_\mu)$ described in \cref{sec:monequiv} is cocartesian monoidal. This means there are monoidal natural transformations with components
\begin{equation*}
    \nabla_{(x, a)} 
    \maps (x, a) \otimes_{\mu} (x, a) 
    \to (x, a)
    \quad\textrm{and}\quad  u_{(x, a)} 
    \maps (I, \mu_0(*)) 
    \to (x, a)
\end{equation*}
making the diagrams \cref{eq:nabla} commute.
Explicitly, by \cref{eq:globalmonstr}, $\nabla_{(x, a)}$ consists of morphisms
$f_x \maps x \otimes x \to x$ in $\X$ and
$\kappa_a\maps (\M f_x)(\mu_{x, x}(a, a)) \to a$ in $\M x$, 
whereas $u_{(x, a)}$ consists of 
$i_x \maps I \to x$ in $\X$ and
$\lambda_a \maps (\M i_x)\mu_0 \to a$
in $\M x$.

The conditions \cref{eq:nabla} say that the composites
\[
    (I, \mu_0) \otimes_\mu (x, a) \xrightarrow{u_{(x, a)} \otimes_\mu 1_{(x, a)}} (x, a) \otimes_\mu (x, a) \xrightarrow{\nabla_{(x, a)}} (x, a)
\]
\[
    (x, a) \otimes_\mu (I, \mu_0) \xrightarrow{1_{(x, a)} \otimes_\mu u_{(x, a)}} (x, a) \otimes_\mu (x, a) \xrightarrow{\nabla_{(x, a)}} (x, a)
\]
are equal to the left and right unitor on $x$, where all respective structures are detailed in \cref{sec:monoidal}. Using the composition inside $\inta \M$ analogously to \cref{eq:comp_intM}, these conditions translate, on the one hand, to the base being cocartesian monoidal $(\X, +, 0)$ with $f_x=\nabla_x$ and $i_x=u_x$. On the other hand, $\kappa_a$ and $\lambda_a$ form natural transformations 
\begin{equation}\label{kappalambda}
\begin{tikzcd}[row sep=.1in, column sep=.2in]
    & 
    \M x \times \M x 
    \arrow[r, "\mu_{x, x}"] 
    & 
    \M (x+x)
    \arrow[dr, "\M(\nabla_x)"] 
    \\
    \M x 
    \arrow[ur, "\Delta"] 
    \arrow[rrr, bend right=20, "1"'] 
    \arrow[rrr, phantom, "\Downarrow {\scriptstyle \kappa^x}"] 
    &&& 
    \M x
\end{tikzcd}\quad
\begin{tikzcd}[row sep=.1in, column sep=.2in]
    & \1
    \arrow[r, "\mu_0"] 
    & 
    \M(0)
    \arrow[dr, "\M(u_x)"]
    \\ 
    \M x
    \arrow[ur, "!"]
    \arrow[rrr, bend right=20, "1"']
    \arrow[rrr, phantom, "\Downarrow{\scriptstyle \lambda^x}"] 
    &&& 
    \M x
\end{tikzcd}
\end{equation}
satisfying the commutativity of
\begin{displaymath}
\begin{tikzcd}[column sep=.2in, row sep=.2in]
    \M (\nabla_x \circ (u_x + 1)) (\mu_{0, x} (\mu_0(*), a))
    \arrow[ddd, "\mathrm{id}"']
    \arrow[rr, "\sim"', "\delta"] 
    && 
    (\M (\nabla_x) \circ \M(u_x + 1)) ((\mu_{0, x}(\mu_0(*), a))
    \arrow[d, "\sim"{anchor=south, rotate=90, inner sep=.5mm}, "{\M(\nabla_x)(\mu_{u_x, 1})}"] 
    \\&& 
    \M (\nabla_x) (\mu_{x, x} (\M(u_x) (\mu_0(*), a)))
    \arrow[d, "{\M (\nabla_x) \left(\mu_{x, x} (\lambda^x_a, \gamma)\right)}"] 
    \\&& 
    \M (\nabla_x) (\mu_{x, x} (a, a))
    \arrow[d, "{\kappa^x_a}"] 
    \\
    \M (\ell_x) (\mu_{0, x} (\mu_0(*), a))
    \arrow[rr, "\xi", "\sim"'] 
    && 
    a
\end{tikzcd}
\end{displaymath}
and a similar one with $\mu_0$ on second arguments. The above greatly simplifies if $\M$ is just a lax monoidal functor:
the first condition becomes $1_a \cong \kappa^x_a \circ \M(\nabla_x) (\mu_{x, x} (\lambda_a^x, 1))$, and the second one $1_a \cong \kappa^x_a \circ \M(\nabla_x) (\mu_{x, x} (1_a, \lambda_a^x))$.

\begin{cor}\label{cor:kappalambda}
    A lax monoidal pseudofunctor $\M\maps(\X, +, 0)\to(\Cat, \times, \1)$ equipped with natural transformations $\kappa$ and $\lambda$ as in \cref{kappalambda} corresponds to an ordinary pseudofunctor $\M\maps\X\to\Cocart$, or equivalently \cref{eq:cocartspecialcase} to a cocartesian monoidal opfibration.
\end{cor}

\section{Examples}\label{sec:applications}

In this section, we explore certain settings where the equivalence between monoidal fibrations and monoidal indexed categories naturally arises. Instead of going into details that would result in a much longer text, we mostly sketch the appropriate example cases up to the point of exhibition of the monoidal Grothendieck correspondence, providing indications of further work and references for the interested reader.

\subsection{Fundamental Bifibration}\label{sec:fundamentalfib}

For any category $\X$, the \emph{codomain} or \emph{fundamental} opfibration is the usual functor from its arrow category
\[\cod \maps \X^2 \longrightarrow \X\] mapping every morphism to its codomain and every commutative square to its right-hand side leg. It uniquely corresponds to the strict opindexed category, i.e.\ mere functor
\begin{equation}\label{eq:fundamentalindexedcat}
\begin{tikzcd}[row sep=.05in]
    \X
    \arrow[r] 
    & 
    \Cat 
    \\
    x
    \arrow[r, mapsto]
    \arrow[dd, "f"'] 
    &     
    \X/x
    \arrow[dd, "f_!"] 
    \\\\
    y
    \arrow[mapsto, r] 
    & 
    \X/y
\end{tikzcd}
\end{equation}
that maps an object to the slice category over it and a morphism to the post-composition functor $f_!=f\circ-$ induced by it.

If the category has a monoidal structure $(\X, \otimes, I)$, this (2-)functor naturally becomes lax monoidal with structure maps
\begin{equation}\label{eq:slicelaxator}
    \X/x\times\X/y\xrightarrow{\otimes}\X/(x\otimes y), \quad \1\xrightarrow{1_I}\X/I.
\end{equation}
These components form strictly natural transformations, and for example the invertible modification $\omega$ \cref{eq:omega} has components the evident isomorphisms, for $(f, g, h)\in\X/x\times\X/y\times\X/z$, between 
\begin{align}
    a \otimes (b \otimes c)& \xrightarrow{f \otimes (g \otimes h)} x \otimes (y \otimes z) \cong (x \otimes y) \otimes z \label{eq:pseudoassociativity}\\
    (a \otimes b) \otimes c & \xrightarrow{(f \otimes g) \otimes h} (x \otimes y) \otimes z\nonumber
\end{align}
By \cref{thm:mainthmsplit}, this monoidal strict opindexed category correspondes to a monoidal split fibration, i.e.\ $(\X^\2, \otimes, 1_I)$ is monoidal and $\cod$ strict monoidal, where $\otimes_{\X^\2}$ strictly preserves cartesian liftings via $f_!k \otimes g_! \ell = (f \otimes g)_! (k \otimes \ell)$ -- which can of course be independently verified. However in general, the slice categories $\X/x$ do not inherit the monoidal structure: there is no way to restrict the global monoidal structure to a fibrewise one.

According to \cref{thm:fibrewise=global}, there is an induced monoidal structure on the categories $\X/x$ and a strict monoidal structure on all $f_!$ only when the monoidal structure on $\X$ is given by binary coproducts and an initial object (i.e.\ cocartesian). In that case, for each $k\maps a\to x$ and $\ell\maps b\to x$ in the same fibre $\X/x$, their tensor product in $\X/x$ is given by
\begin{displaymath}
    a+b\xrightarrow{\;k+\ell\;}x+x\xrightarrow{\nabla_x}x
\end{displaymath}
as a simple example of \cref{eq:explicitstructure1}.
In fact, this is precisely the coproduct of two objects in $\X/x$, and $0\xrightarrow{!}x$ the initial object, due to the way colimits in the slice categories are constructed. Therefore this falls under the cocartesian-fibres special case \cref{eq:cocartspecialcase}, bijectively corresponding to the cocartesian structure on $\X^\2$ inherited from $\X$.

Now suppose an ordinary category $\X$ has pullbacks. This endows the codomain functor also with a fibration structure, corresponding to the indexed category
\begin{displaymath}
\begin{tikzcd}[row sep=.05in]
    \X\op
    \arrow[r] 
    & 
    \Cat 
    \\
    x
    \arrow[r, mapsto]
    \arrow[dd, "f"'] 
    & 
    \X/x
    \\\\
    y
    \arrow[mapsto, r] 
    & 
    \X/y
    \arrow[uu, "f^*"']
\end{tikzcd}
\end{displaymath}
with the same mapping on objects as \cref{eq:fundamentalindexedcat} but by taking pullbacks rather than post-composing along morphisms, a pseudofunctorial assignment. This gives $\cod\maps \X^2\to\X$ a bifibration structure, also by that classic fact that $f_!\dashv f^*$.

In this case, if $\X$ has a general monoidal structure, there is no naturally induced lax monoidal structure of that pseudofunctor as before: there is no reason for the pullback of a tensor to be isomorphic to the tensor of two pullbacks.  
However, if $\X$ is cartesian monoidal (hence has all finite limits), the components
\begin{displaymath}
    \X/x\times\X/y\xrightarrow{\times}\X/(x\times y), \qquad \1\xrightarrow{\Delta_!}\X/1
\end{displaymath}
are pseudonatural since pullbacks commute with products. Moreover, this bijectively corresponds to monoidal fibres and strong monoidal reindexing functors, in fact also cartesian ones: for morphisms $k \maps a\to x$ and $\ell \maps b\to x$ in $\X/x$, their induced product is given by
\begin{displaymath}
\begin{tikzcd}[sep=.5in]
    \bullet
    \arrow[r]
    \arrow[d, "\delta^*(k\times\ell)"']
    \arrow[dr, phantom, very near start, "\lrcorner"] 
    & 
    a \times b
    \arrow[d, "k\times\ell"] 
    \\
    x
    \arrow[r, "\delta"] 
    & 
    x \times x
\end{tikzcd}
\end{displaymath}
and $1_x \maps x\to x$ is the unit of each slice $\X/x$, this indexed monoidal category also described in \cite{DescentForMonads}.
The monoidal fibration structure on $\cod \maps (\X^2, \times, 1_1) \to (X, \times, 1)$ is the evident one, so it again falls in the special case \cref{eq:cocartspecialcase} now for cartesian fibres, by construction of products in slice categories.

As a final remark, analogous constructions hold for the domain functor which is again a bifibration: its fibration structure comes from pre-composing along morphisms, whereas its opfibration structure comes from taking pushouts along morphisms. 

\subsection{Family Fibration: Zunino and Turaev Categories}\label{sec:familyfib}

Recall that for any category $\C$, the standard \emph{family fibration} is induced by the (strict) functor
\begin{equation}\label{eq:functV}
    [-, \C]\maps \Set\op\to\Cat
\end{equation}
which maps every discrete category $X$ to the functor category $[X, \C]$ and every function $f \maps X \to Y$ to the functor $f^* = [f, 1]$, i.e.\ pre-composition with $f$.
The total category of the induced fibration $\Fam(\C)\to\C$ has as objects pairs $(X, M\maps X\to\C)$ essentially given by a family of $X$-indexed objects in $\C$, written $\{M_x\}_{x\in X}$, whereas the morphisms are \begin{displaymath}
\begin{tikzcd}[column sep=.7in, row sep=.2in]
    X
    \arrow[dr, "M"]
    \arrow[dd, "f"'] 
    \\ 
    \arrow[r, phantom, "\Downarrow {\scriptstyle\alpha}"description] 
    & 
    \C 
    \\
    Y 
    \arrow[ur, "N"']
\end{tikzcd}
\end{displaymath}
namely a function $f\maps X\to Y$ together with families of morphisms $\alpha_x\maps M_x\to N_{fx}$ in $\C$. Notice the similarity of this description with \cref{eq:indexed1cell}, which for the strict indexed categories case looks like a non-discrete version of the family fibration, for $\C=\Cat$. Moreover, it is a folklore fact that $\Fam(\C)$ is the free coproduct cocompletion on the category $\C$. 

On the other hand, we could consider the opfibration induced by the very same functor \cref{eq:functV}, denoted by $\Maf(\C)\to\Set\op$. The objects of $\Maf(\C)$ are the same as $\Fam(\C)$, but morphisms $\{M_x\}_{x\in X}\to\{N_y\}_{y\in Y}$ between them are functions $g\maps Y\to X$ (i.e.\ $X\to Y$ in $\Set\op$) together with families of arrows $\beta_y\maps M_{gy}\to N_y$ in $\C$. Notice that these are now indexed over the set $Y$ rather than $X$ like before, and in fact $\Maf (\X) = \Fam (\X\op)\op$.

In the case that the category is monoidal $(\C, \otimes, I)$, the (2-)functor $[-, \C]$ has a canonical lax monoidal structure. Explicitly, by taking its domain $\Set\op$ to be cocartesian by the usual cartesian monoidal structure $(\Set, \times, 1)$, the structure maps are
\begin{displaymath}
\phi_{X, Y} \maps [X, \C] \times [Y, \C] \to [X \times Y, \C], \qquad \phi_0 \maps \1 \xrightarrow{I_\C} [\1, \C] \cong \C
\end{displaymath}
where $\phi_{X, Y}$ corresponds, under the tensor-hom adjunction in $\Cat$, to
\begin{displaymath}
    [X, \C] \times [Y, \C] \times X \times Y \xrightarrow{\sim} [X, \C] \times X \times [Y, \C] \times Y \xrightarrow{\textrm{ev}_X \times \textrm{ev}_Y} \C \times \C \xrightarrow{\otimes} \C.
\end{displaymath}
These are again natural components, and for example \cref{eq:omega} has components the natural isomorphisms between the assignments $Mx\otimes (Ny\otimes Uz)$ and $(Mx\otimes Ny)\otimes Uz$. 
By \cref{thm:mainthmsplit}, this monoidal strict indexed category endows the corresponding split fibration $\Fam(\X) \to \Set$ with a monoidal structure via $\{M_x\} \otimes \{N_y\} := \{M_x \otimes N_y \}_{X \times Y}$. On the other hand, we could use the dual part of the same theorem, and instead consider the induced monoidal split opfibration $\Maf(\X) \to \Set\op$ corresponding to the same $([-, \C], \phi, \phi_0)$.

Moreover, since $\Set$ is cartesian, \cref{thm:fibrewise=global} also applies in both cases, giving a monoidal structure to the fibres as well: for $M\maps X\to\C$ and $N\maps X\to\C$, their fibrewise tensor product and unit are given by
\begin{displaymath}
X\xrightarrow{\Delta}X\times X\xrightarrow{M\times N}\C\times\C\xrightarrow{\otimes}\C, \qquad
X\xrightarrow{!}1\xrightarrow{I}\C
\end{displaymath}
which are precisely constructed as in \cref{eq:explicitstructure1}. Once again, notice the direct similary with \cref{eq:icatxprod}, the fibrewise monoidal structure on $\ICat(\X)$.

As an interesting example, consider $\C=\Mod_R$ for a commutative ring $R$, with its usual tensor product $\otimes_R$. In \cite{TuraevZunino}, the authors introduce a category $\mathcal{T}$ of \emph{Turaev} $R$-modules, as well as a category $\mathcal{Z}$ of \emph{Zunino} $R$-modules, which serve as symmetric monoidal categories where group-(co)algebras and Hopf group-(co)algebras, \cite{Turaev}, live as (co)monoids and Hopf monoids respectively. 

In more detail, the objects of both $\mathcal{T}$ and $\mathcal{Z}$ are defined to be pairs $(X, M)$ where $X$ is a set and $\{M_x\}_{x\in X}$ is an $X$-indexed family of $R$-modules, and their morphisms are respectively
\begin{displaymath}
    (\mathcal{T})
    \begin{cases}
        s \maps M_{g(y)} \to N_y \textrm{ in } \Mod_R \\
        g \maps Y \to X \textrm{ in } \Set
    \end{cases}
    \quad
    (\mathcal{Z})
    \begin{cases}
        t\maps M_x\to N_{f(x)}\textrm{ in }\Mod_R \\
        f\maps X\to Y\textrm{ in }\Set
    \end{cases}
\end{displaymath}
There is a symmetric pointwise monoidal structure, $\{M_x\otimes_R N_y\}_{X\times Y}$, and there are strict monoidal forgetful functors $\mathcal{T}\to\Set\op$, $\mathcal{Z}\to\Set$.
It is therein shown that comonoids in $\mathcal{T}$ are \emph{monoid-coalgebras} and monoids in $\mathcal{Z}$ are \emph{monoid-algebras}, i.e.\ families of $R$-modules indexed over a monoid, together with respective families of linear maps
\begin{gather*}
    (\mathcal{T})\quad C_{g*h}\to C_g\otimes C_h \qquad (\mathcal{Z})\quad
    A_g\otimes A_h\to A_{g*h} \\
    C_e\to R \qquad \phantom{ZZZZZZZ}R\to A_e
\end{gather*}
satisfying appropriate axioms. Based on the above, it is clear that $\mathcal{T}=\Maf(\Mod_R)$ and $\mathcal{Z}=\Fam(\Mod_R)$, which clarifies the origin of these categories and can be directly used to further generalize the notions of Hopf group-(co)monoids in arbitrary monoidal categories.

\subsection{Global Categories of Modules and Comodules}\label{sec:ModComod}

For any monoidal category $\mathcal{V}$, 
there exist \emph{global} categories of modules and comodules, denoted by $\Mod$ and $\Comod$ \cite[6.2]{VasilakopoulouThesis}.
Their objects are all (co)modules over (co)monoids in $\V$, whereas a morphism between an $A$-module $M$ and a $B$-module $N$ is given by a monoid map $f\maps A\to B$ together with a morphism $k\maps M\to N$ in $\V$ satisfying the commutativity of 
\begin{displaymath}
\begin{tikzcd}
    A\otimes M\arrow[rr, "\mu"]\arrow[d, "1\otimes k"'] && M\arrow[d, "k"] \\
    A\otimes N\arrow[r, "f\otimes1"'] & B\otimes N\arrow[r, "\mu"'] & N
\end{tikzcd}
\end{displaymath}
where $\mu$ denotes the respective action, and dually for comodules.
Both these categories arise as the total categories induced by the Grothendieck construction on the functors 
\begin{equation}\label{eq:functors}
\begin{tikzcd}[row sep = tiny]
    \Mon(\V)\op
    \arrow[r]
    &
    \Cat
    &
    \Comon(\V)
    \arrow[r]
    &
    \Cat
    \\
    A
    \arrow[r, dashed, mapsto]
    \arrow[dd, swap, "f"]
    &
    \Mod_\V(A)
    &
    C
    \arrow[r, dashed, mapsto]
    \arrow[dd, swap, "g"]
    &
    \Comod_\V(C)
    \arrow[dd, "g_!"]
    \\{}\\
    B
    \arrow[r, dashed, mapsto]
    &
    \Mod_\V(B)
    \arrow[uu, swap, "f^*"]
    &
    D
    \arrow[r, dashed, mapsto]
    &
    \Comod_\V(D)
\end{tikzcd}
\end{equation}
where $f^*$ and $g_!$ are (co)restriction of scalars: if $M$ is a $B$-module, $f^*(M)$ is an $A$-module via the action
\begin{displaymath}
    A \otimes M \xrightarrow{f \otimes 1} B \otimes M \xrightarrow{\mu} M.
\end{displaymath}
The induced split fibration and opfibration, $\Mod \to \Mon(\V)$ and $\Comod \to \Comon(\V)$, map a (co)module to its respective (co)monoid.

Recall that when $(\V, \otimes, I, \sigma)$ is braided monoidal, its categories of monoids and como\-noids inherit the monoidal structure: if $A$ and $B$ are monoids, then $A\otimes B$ has also a monoid structure via
\begin{displaymath}
    A \otimes B \otimes A \otimes B \xrightarrow{1 \otimes \sigma \otimes 1} A \otimes A \otimes B \otimes B \xrightarrow{m \otimes m} A \otimes B, \qquad
    I \cong I \otimes I \xrightarrow{j \otimes j} A \otimes B
\end{displaymath}
where $m$ and $j$ give the respective monoid structures. 
In that case, the induced split fibration and opfibration are both monoidal. This can be deduced by directly checking the conditions of \cref{def:monoidal_fibration}, as was the case in the relevant references, or in our setting by using \cref{thm:mainthmsplit} since both (2-)functors \cref{eq:functors} are lax monoidal. For example, for any $A, B \in \Mon(\V)$ there are natural maps
\begin{displaymath}
    \phi_{A, B} \maps \Mod_\V (A) \times \Mod_\V (B) \to \Mod_\V (A \otimes B)\qquad  \phi_0 \maps \1 \to \Mod_\V (I)
\end{displaymath}
with $\phi_{A, B} (M, N) = M \otimes N$, with the $A \otimes B$-module structure being
\begin{displaymath}
    A \otimes B \otimes M \otimes N \xrightarrow{1 \otimes \sigma \otimes 1} A \otimes M \otimes B \otimes N \xrightarrow{\mu \otimes \mu} M \otimes N
\end{displaymath}
and $\phi_0(*)=I$, which are pseudoassociative and pseudounital in the sense that e.g. for any $M, N, P\in\Mod_\mathcal{V}(A)\times\Mod_\mathcal{V}(B)\times\Mod_\mathcal{V}(C)$, $M\otimes(N\otimes P)$ is only isomorphic to $(M\otimes N)\otimes P$ as $(A\otimes B)\otimes C$-modules.

Notice that in general, the monoidal bases $\Mon(\V)$ and $\Comon(\V)$ are not (co)\-ca\-rte\-sian, since they have the same tensor as $(\V, \otimes, I, \sigma)$. Therefore this case does not fall under \cref{thm:fibrewise=global}, hence the fibre categories are not monoidal. For example in $(\mathsf{Vect}_k, \otimes_k, k)$, the $k$-tensor product of two $A$-modules for a $k$-algebra $A$ is not an $A$-module as well.

We remark that the induced monoidal opfibration $\Comod \to \Comon(\V)$ in fact serves as the monoidal base of an \emph{enriched fibration} structure on $\Mod \to \Mon (\V)$ as explained in \cite{OnEnrichedFibrations}, built upon an enrichment between the monoidal bases $\Mon(\V)$ in $\Comon(\V)$ established in \cite{Measuringcomonoid}. Moreover, analogous monoidal structures are induced on the (op)fibrations of monads and comonads in any fibrant monoidal double category, see \cite[Prop. 3.18]{VCocats}.

\subsection{Systems as Monoidal Indexed Categories}\label{sec:systemsasmonicats}

In \cite{DynamicalSystemsSheaves} as well as in earlier works e.g.\ \cite{Vagner.Spivak.Lerman:2015a}, the authors investigate a categorical framework for modeling systems of systems using algebras for a monoidal category. In more detail, systems in a broad sense are perceived as lax monoidal pseudofunctors
\begin{displaymath}
    \mathcal{W}_\mathcal{C}\to\Cat
\end{displaymath}
where $\mathcal{W}_\C$ is the monoidal category of $\C$-\emph{labeled boxes} and \emph{wiring diagrams} with types in a finite product category $\mathcal{C}$. Briefly, the objects in $\mathcal{W}_\mathcal{C}$ are pairs $X=(X^\mathrm{in}, X^\mathrm{out})$ of finite sets equipped with functions to $\ob\mathcal{C}$, thought of as boxes
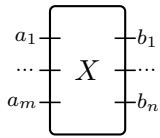
\begin{displaymath}
\begin{tikzpicture}[oriented WD, bbx=.1cm, bby =.1cm, bb port sep=.15cm]
	\node [bb={3}{3}] (X) {$X$};
	\draw[label]
    	node[left=.1 of X_in1]  {$a_1$}
    	node[left=.1 of X_in2]  {$\dotso$}
    	node[left=.1 of X_in3]  {$a_m$}
    	node[right=.1 of X_out1] {$b_1$}
    	node[right=.1 of X_out2]  {$\dotso$}
    	node[right=.1 of X_out3] {$b_n$};
\end{tikzpicture}
\end{displaymath}
where $X^\mathrm{in}=\{a_1, \ldots, a_m\}$ are the input ports, $X^\mathrm{out}=\{b_1, \ldots, b_n\}$
the output ones and all wires are associated to a $\mathcal{C}$-object expressing the type of information that can go through them. A morphism $\phi\maps X\to Y$ in this category consists
of a pair of functions
\begin{displaymath}
    \left\{\begin{array}{l}
    \inp{\phi}\maps \inp{X}\to\out{X}+\inp{Y} \\
    \out{\phi}\maps \out{Y}\to\out{X}\end{array}\right.
\end{displaymath}
that respect the $\mathcal{C}$-types, 
which roughly express which port is `fed information' by which.
Graphically, we can picture it as

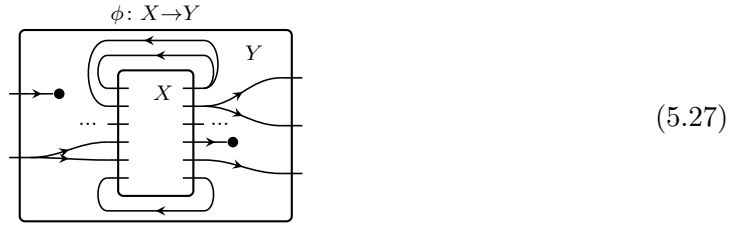
\begin{equation}\label{eq:wiringdiagpic}
\begin{tikzpicture}[oriented WD, baseline=(Y.center), bbx=2em, bby=1.2ex, bb port sep=1.2]
    \node[bb={6}{6}] (X) {};
    \node[bb={2}{3}, fit={($(X.north east)+(0.7, 1.7)$) ($(X.south west)-(.7, .7)$)}] (Y) {};
    \node [circle, minimum size=4pt, inner sep=0, fill] (dot1) at ($(Y_in1')+(.5, 0)$) {};
    \node [circle, minimum size=4pt, inner sep=0, fill] (dot2) at ($(X_out4)+(.5, 0)$) {};
    \draw[ar] (Y_in1') to (dot1);
    \draw[ar] (X_out4) to (dot2);
    \draw[ar] (Y_in2') to (X_in5);
    \draw[ar] (Y_in2') to (X_in4);
    \draw[ar] (X_out5) to (Y_out3');
    \draw[ar] (X_out2) to (Y_out1');
    \draw[ar] (X_out2) to (Y_out2');
    \draw[ar] let \p1=(X.north west), \p2=(X.north east), \n1={\y1+\bby}, \n2=\bbportlen in
    	(X_out1) to[in=0] (\x2+\n2, \n1) -- (\x1-\n2, \n1) to[out=180] (X_in1);
    \draw[ar] let \p1=(X.north west), \p2=(X.north east), \n1={\y1+2*\bby}, \n2=\bbportlen in
    	(X_out1) to[in=0] (\x2+\n2, \n1) -- (\x1-\n2, \n1) to[out=180] (X_in2);
    \draw[ar] let \p1=(X.south west), \p2=(X.south east), \n1={\y1-\bby}, \n2=\bbportlen in
    	(X_out6) to[in=0] (\x2+\n2, \n1) -- (\x1-\n2, \n1) to[out=180] (X_in6);	
    \draw [label] 
        node at ($(Y.north east)-(.5cm, .3cm)$) {$Y$}
        node at ($(X.north east)-(.4cm, .3cm)$) {$X$}
        node[left=.1 of X_in3]  {$\dotso$}
        node[right=.1 of X_out3] {$\dotso$}
        node[above=of Y.north] {$\phi\maps X\to Y$}
        ;
\end{tikzpicture}
\end{equation}
Composition of morphisms can be thought of a zoomed-in picture of three boxes, and the monoidal structure amounts to parallel placement of boxes as in
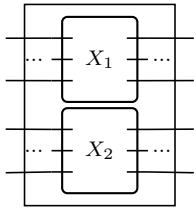
\begin{displaymath}
\begin{tikzpicture}[oriented WD, baseline=(Y.center), bbx=1.3em, bby=1ex, bb port sep=.06cm]
    \node[bb={3}{3}] (X1) {};
    \node[bb={3}{3}, below =.5 of X1] (X2) {};
    \node[fit=(X1)(X2), draw] {};
    \draw[label] 
    node at ($(X1.west)+(1, 0)$) {$X_1$}
    node at ($(X2.west)+(1, 0)$) {$X_2$}
    node[left=.1 of X1_in2]  {$\dotso$}
    node[right=.1 of X1_out2]  {$\dotso$}
    node[left=.1 of X2_in2]  {$\dotso$}
    node[right=.1 of X2_out2]  {$\dotso$};
    \draw (X1_in1) -- (-2.5, 1.7);
    \draw (X1_out1) -- (2.5, 1.7);
    \draw (X1_in3) -- (-2.5, -1.7);
    \draw (X1_out3) -- (2.5, -1.7);
    \draw (X2_in1) -- (-2.5, -5.6);
    \draw (X2_out1) -- (2.5, -5.6);
    \draw (X2_in3) -- (-2.5, -9.1);
    \draw (X2_out3) -- (2.5, -9.1);
 \end{tikzpicture}
\end{displaymath}
There is a close connection between the definition of $\mathcal{W}_\mathcal{C}$ and that of \emph{Dialectica} categories as well as \emph{lenses}; such considerations are the topic of work in progress \cite{EverythingisDialectica}.

The systems-as-algebras formalism uses lax monoidal pseudofunctors from this category
$\mathcal{W}_\mathcal{C}$ to $\Cat$ that essentially receive a general picture such as 
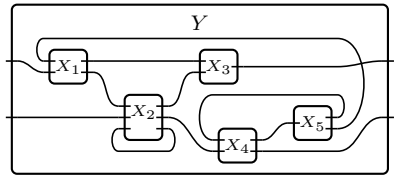
\begin{displaymath}
\begin{tikzpicture}[oriented WD, bb min width =.5cm, bbx=.5cm, bb port sep =1, bb port length=.08cm, bby=.15cm]
\node[bb={2}{2}, bb name = {\tiny$X_1$}] (X11) {};
\node[bb={3}{3}, below right=of X11, bb name = {\tiny$X_2$}] (X12) {};
\node[bb={2}{1}, above right=of X12, bb name = {\tiny$X_3$}] (X13) {};
\draw (X11_out1) to (X13_in1);
\draw (X11_out2) to (X12_in1);
\draw (X12_out1) to (X13_in2);
\node[bb={2}{2}, below right = -1 and 1.5 of X12, bb name = {\tiny$X_4$}] (X21) {};
\node[bb={1}{2}, above right=-1 and 1 of X21, bb name = {\tiny$X_5$}] (X22) {};
\draw (X21_out1) to (X22_in1);
\draw let \p1=(X22.north east), \p2=(X21.north west), \n1={\y1+\bby}, \n2=\bbportlen in
         (X22_out1) to[in=0] (\x1+\n2, \n1) -- (\x2-\n2, \n1) to[out=180] (X21_in1);
\node[bb={2}{2}, fit = {($(X11.north east)+(-1, 3)$) (X12) (X13) ($(X21.south)$) ($(X22.east)+(.5, 0)$)}, bb name ={\scriptsize $Y$}] (Z) {};
\draw (Z_in1') to (X11_in2);
\draw (Z_in2') to (X12_in2);
\draw (X12_out2) to (X21_in2);
\draw let \p1=(X22.south east), \n1={\y1-\bby}, \n2=\bbportlen in
  (X21_out2) to (\x1+\n2, \n1) to (Z_out2');
 \draw let \p1=(X12.south east), \p2=(X12.south west), \n1={\y1-\bby}, \n2=\bbportlen in
  (X12_out3) to[in=0] (\x1+\n2, \n1) -- (\x2-\n2, \n1) to[out=180] (X12_in3);
\draw let \p1=(X22.north east), \p2=(X11.north west), \n1={\y2+\bby}, \n2=\bbportlen in
  (X22_out2) to[in=0] (\x1+\n2, \n1) -- (\x2-\n2, \n1) to[out=180] (X11_in1);
\draw let \p1=(X13_out1), \p2=(X22.north east), \n2=\bbportlen in
 (X13_out1) to (\x1+\n2, \y1) -- (\x2+\n2, \y1) to (Z_out1');
\end{tikzpicture}
\end{displaymath}
(which really takes place in the underlying operad of $\mathcal{W}_\mathcal{C}$) and assign systems of a certain kind to all inner boxes; the lax monoidal and pseudo\-functorial structure of this assignment formally produce a system of the same kind for the outer box. 

Examples of such systems are discrete dynamical systems (Moore machines in the finite case), continuous dynamical systems but also more general systems with deterministic or total conditions; details can be found in the provided references. Since all these systems are lax monoidal pseudofunctors from the non-cocartesian monoidal category of wiring diagrams to $\Cat$, i.e.\ monoidal indexed categories, the monoidal Grothendieck construction \cref{thm:mainthm} induces a corresponding monoidal fibration in each system case, and this global structure does not reduce to a fibrewise one. 

For example, the algebra for discrete dynamical systems
\cite[Sec.\ 2.3]{DynamicalSystemsSheaves}
\begin{equation}\label{eq:DDS}
\mathrm{DDS}\maps \mathcal{W}_\Set\to\Cat  
\end{equation}
assigns to each box $X=(X^\mathrm{in}, X^\mathrm{out})$
the category of all discrete dynamical systems with fixed input and output sets being $\prod_{x\in X^\mathrm{in}}x$ and $\prod_{y\in X^\mathrm{out}}y$ respectively. There exist morphisms between systems of the same input and output set, but not between those with different ones. To each morphism, i.e.\ wiring diagram as in \cref{eq:wiringdiagpic}, $\mathrm{DDS}$
produces a functor that maps an inner discrete dynamical system to a new outer one, with changed input and output sets accordingly.
(Pseudo)functoriality of this assignment 
allows the coherent zoom-in and zoom-out on dynamical systems built out of smaller dynamical systems, and monoidality allows the creation of new dynamical systems on parallel boxes.

Being a monoidal indexed category, \cref{eq:DDS} gives rise to a monoidal opfibration over $\mathcal{W}_\Set$. Its  total category $\inta \mathrm{DDS}$ has objects all dynamical systems with arbitrary input and output sets, morphisms that can now go between systems of different inputs/outputs, and also a natural tensor product inherited from that in $\mathcal{W}_\Set$ and the laxator of $\inta\mathrm{DDS}$. In a sense, this category has all the required flexibility for the direct communication (via morphisms in the total category) between any discrete dynamical system, or any composite of systems or parallel placement of them, whereas the wiring diagram algebra \cref{eq:DDS} focuses on the machinery of building new discrete dynamical systems systems from old.

This classic change of point of view also transfers over to maps of algebras, i.e.\ indexed monoidal 1-cells. As an example, see \cite[Sec.\ 5.1]{DynamicalSystemsSheaves}, discrete dynamical systems can naturally be viewed as general \emph{total} and \emph{deterministic} machines denoted by $\mathrm{Mch}^\mathrm{td}$, 
via a monoidal pseudonatural transformation
\[
\begin{tikzcd}
    \mathcal{W}_\Set
    \arrow[dr, "\mathrm{DDS}"]
    \arrow[dd]
    \\
    \arrow[r, phantom, "\Downarrow"]
    &
    \Cat
    \\
    \mathcal{W}_{\widetilde{\mathrm{Int}_N}}
    \arrow[ur, "\mathrm{Mch}^\mathrm{td}", swap]
\end{tikzcd}\]
which also changes the type of input and output wires from sets to \emph{discrete interval sheaves} $\widetilde{\mathrm{Int}_N}$. This gives rise to a monoidal opfibred 1-cell
\begin{displaymath}
\begin{tikzcd}
    \inta\mathrm{DDS}
    \arrow[r]
    \arrow[d] 
    & 
    \inta\mathrm{Mch}^\mathrm{td}
    \arrow[d] 
    \\
    \mathcal{W}_\Set
    \arrow[r] 
    & 
\mathcal{W}_{\widetilde{\mathrm{Int}}_N}
\end{tikzcd}
\end{displaymath}
which provides a direct functorial translation between the one sort of system to the other in a way compatible with the monoidal structure.

As a final note, this method of modeling certain objects as algebras for a monoidal category (a.k.a.\ strict or general monoidal indexed categories) carries over to further contexts than systems and the wiring diagram category. Examples include hypergraph categories as algebras on cospans  \cite{HypergraphCats} and traced monoidal categories as algebras on cobordisms  \cite{TracedMonCatsAlg}. In all these cases, the monoidal Grothendieck construction gives a potentially fruitful change of perspective that should be further investigated.

\subsection{Graphs}

As we show in \cref{expl:grphsoverset}, the category of (directed, multi) graphs, is bifibred over set, where the bifibration $\mathsf V \maps \Grph \to \Set$ is given by sending a graph to its vertex set. 

Since $\mathsf V \maps \Grph \to \Set$ preserves products, then it can be given the structure of a strict monoidal monoidal functor with respect to the cartesian monoidal structures on $\Grph$ and $\Set$. Since the cartesian morphisms are those that form pullback squares, and products in $\Grph$ are given pointwise, then the monoidal structure in $\Grph$ preserves cartesian morphisms. We can then apply \cref{cor:fixedbasemonoidalGr} to obtain a symmetric lax monoidal structure for the pseudofunctor $\Grph^* \maps \Set\op \to \Cat$. The lax structure map $\gamma_{X,Y} \maps \Grph_X \times \Grph_Y \to \Grph_{X \times Y}$ is given by taking the product of the two graphs within $\Grph$. Notice the product has vertex set given by $X \times Y$. Since the base category is cartesian monoidal, we can apply \cref{thm:fibrewise=global}, granting a symmetric monoidal structure to the fibres $\Grph_X$. The monoidal product is given by the following composite.
\[
    \Grph_X \times \Grph_X \xrightarrow{\gamma_{X,X}} \Grph_{X\times X} \xrightarrow{\Delta^*} \Grph_X
\]
Simply put, this operation is given by taking the product of the two graphs on $X$, and then restricting to the vertices on the diagonal. Indeed, this is the cartesian monoidal structure on $\Grph_X$.

Since the category $\Set$ also has all finite colimits, we obtain a symmetric lax monoidal structure for the pseudofunctor $\Grph_* \maps \Set\op \to \Cat$. The lax structure map $\phi_{X,Y} \maps \Grph_X \times \Grph_Y \to \Grph_{X + Y}$ is given by taking the disjoint union of the two graphs. Notice the disjoint union has vertex set given by $X + Y$. Since the base category is cocartesian monoidal, we can apply \cref{thm:fibrewise=global}, granting a symmetric monoidal structure to the fibres $\Grph_X$. The monoidal product is given by the following composite. 
\[
    \Grph_X \times \Grph_X \xrightarrow{\phi_{X,X}} \Grph_{X + X} \xrightarrow{\Delta_*} \Grph_X
\]
Simply put, this operation is given by taking the disjoint union of the edges. Indeed, this is the cocartesian monoidal structure on $\Grph_X$. This is also the overlay operation for the network model of directed multi graphs.
}

{\ssp\appendix
\setcounter{chapter}{0}
\chapter{Monoidal Categories}
\label{app:monoidalcats}

Monoidal categories lie at the center of applied category theory. This section is included mainly to establish notation and terminology used throughout this thesis. Some standard references are \cite{MacLane} and \cite{TensorCategories}.

\section{Definitions}

\subsection{Monoidal, Braided, and Symmetric Categories}

\begin{defn}
\label{def:monoidalcategory}
    A \define{monoidal category} $(\C, \otimes, I, \assoc, \lambda, \rho)$ consists of 
    \begin{itemize}
        \item a category $\C$
        \item a functor $\otimes \maps \C \times \C \to \C$ called the \define{tensor}
        \item a functor $I \maps 1 \to \C$ called the \define{unit}
        \item a natural transformation $\assoc$ with components of the form $\assoc_{x,y,z} \maps (x \otimes y) \otimes z \to x \otimes (y \otimes z)$ called the \define{associator}
        \item a natural transformation $\lambda$ with components of the form $\lambda_x \maps I \otimes x \to x$ called the \define{left unitor}
        \item a natural transformation $\rho$ with components of the form $\rho_x \maps x \otimes I \to x$ called the \define{right unitor}
    \end{itemize}
    such that the following diagrams commute.
    
    \noindent \define{Pentagon identity:}
    \begin{equation}
    \label{monoidalpentagon}
    \begin{tikzcd}[column sep = large]
        &
        (w \otimes x) \otimes (y \otimes z)
        \arrow[ddr, "\alpha_{w,x,y \otimes z}"]
        \\
        ((w \otimes x) \otimes y) \otimes z
        \arrow[ur, "\alpha_{w \otimes x,y,z}"]
        \arrow[dd, swap, "\alpha_{w,x,y} \otimes 1_z"]
        \\&&
        w \otimes (x \otimes (y \otimes z))
        \\
        (w \otimes (x \otimes y)) \otimes z
        \arrow[dr, swap, "\alpha_{w,x \otimes y,z}"]
        \\&
        w \otimes ((x \otimes y) \otimes z)
        \arrow[uur, swap, "1_w \otimes \alpha_{x,y,z}"]
    \end{tikzcd}
    \end{equation}
    
    \noindent \define{Triangle identity:}
    \begin{equation}
    \label{leftorrightor}
    \begin{tikzcd}
        (x \otimes I) \otimes y
        \arrow[rr, "\assoc_{x, I, y}"]
        \arrow[dr, swap, "\rho_x \otimes 1_y"]
        &&
        x \otimes (I \otimes y)
        \arrow[dl, "1_x \otimes \lambda_y"]
        \\&
        x \otimes y
    \end{tikzcd}
    \end{equation}
    A \define{strict monoidal category} is one where the associator, left unitor, and right unitor are all identity.
    
    A \define{braided monoidal category} \cite{BraidedTensorCats} is a monoidal category equipped with a natural transformation $\braid$ called the \define{braiding} with components $\braid_{x,y} \maps x \otimes y \to y \otimes x$, such that the following diagrams commute.
    \begin{equation}
    \begin{tikzcd}
        (x \otimes y) \otimes z
        \arrow[r, "\assoc_{x,y,z}"]
        \arrow[d, swap, "\braid_{x,y} \otimes 1_z"]
        &
        x \otimes (y \otimes z)
        \arrow[d, "\braid_{x, y \otimes z}"]
        \\
        (y \otimes x) \otimes z
        \arrow[d, swap, "\assoc_{y,x,z}"]
        &
        (y \otimes z) \otimes x
        \arrow[d, "\assoc_{y,z,x}"]
        \\
        y \otimes (x \otimes z)
        \arrow[r, swap, "1_y \otimes \braid_{x,z}"]
        &
        y \otimes (z \otimes x)
    \end{tikzcd}
    \qquad
    \begin{tikzcd}
        x \otimes (y \otimes z)
        \arrow[r, "\assoc_{x,y,z}\inv"]
        \arrow[d, swap, "1_x \otimes \braid_{y,z}"]
        &
        (x \otimes y) \otimes z
        \arrow[d, "\braid_{x \otimes y, z}"]
        \\
        x \otimes (z \otimes y)
        \arrow[d, swap, "\assoc_{x,z,y}\inv"]
        &
        z \otimes (x \otimes y)
        \arrow[d, "\assoc_{z,x,y}\inv"]
        \\
        (x \otimes z) \otimes y
        \arrow[r, swap, "\braid_{x,z} \otimes 1_y"]
        &
        (z \otimes x) \otimes y
    \end{tikzcd}
    \end{equation}
    A \define{symmetric monoidal category} is a braided monoidal category where the braiding satisfies the equation $\braid_{y,x} \circ \braid_{x,y} = 1_{x \otimes y}$ for all objects $x,y \in \C$. A \define{commutative monoidal category} is a symmetric monoidal category where the braiding is identity.
\end{defn}

For general (braided/symmetric) monoidal categories, we write $\C$, $\D$, or $\E$.

\subsection{Monoidal, Braided, and Symmetric Functors}

\begin{defn}
\label{def:monoidalfunctor}
    Let $(\C, \otimes_\C, I_\C, \assoc^\C, \leftor^\C, \rightor^\C)$ and $(\D, \otimes_\D, I_\D, \assoc^\D, \leftor^\D, \rightor^\D)$ be monoidal categories. A \define{lax monoidal functor} from $\C$ to $\D$ consists of 
    \begin{itemize}
        \item a functor $F \maps \C \to \D$
        \item a natural transformation with components $\phi_{x,y} \maps Fx \otimes_\D Fy \to F(x \otimes_\C y)$ called the \define{laxator}
        \item a natural transformation with unique component $\phi_0 \maps I_\D \to F I_\C$ called the \define{unit laxator}
    \end{itemize}
    such that the following diagrams commute.
    \begin{equation}
    \begin{tikzcd}[column sep = large]
        (Fx \otimes_\D Fy) \otimes_\D Fz
        \arrow[r, "\assoc^\D_{Fx,Fy,Fz}"]
        \arrow[d, swap, "\phi_{x,y} \otimes_\D 1_{Fz}"]
        &
        Fx \otimes_\D (Fy \otimes_\D Fz)
        \arrow[d, "1_{Fx} \otimes_\D \phi_{y,z}"]
        \\
        F(x \otimes_\C y) \otimes_\D Fz
        \arrow[d, swap, "\phi_{x \otimes_\C y, z}"]
        &
        Fx \otimes_\D F(y \otimes_\C z)
        \arrow[d, "\phi_{x, y \otimes_\C z}"]
        \\
        F((x \otimes_\C y) \otimes_\C z)
        \arrow[r, swap, "F(\assoc^\C_{x,y,z})"]
        &
        F(x \otimes_\C (y \otimes_\C z))
    \end{tikzcd}
    \end{equation}
    \begin{equation}
    \begin{tikzcd}
        I_\D \otimes_\D Fx
        \arrow[r, "\leftor^\D_{Fx}"]
        \arrow[d, swap, "\phi_0 \otimes_\D 1_{Fx}"]
        &
        Fx
        \\
        FI_\C \otimes_\D Fx
        \arrow[r, swap, "\phi_{I_\C, x}"]
        &
        F(I_\C \otimes_\C x)
        \arrow[u, swap, "F(\leftor^\C_x)"]
    \end{tikzcd}
    \qquad
    \begin{tikzcd}
        Fx \otimes_\D I_\D
        \arrow[r, "\rightor^\D_x"]
        \arrow[d, swap, "1_{Fx} \otimes_\D \phi_0"]
        &
        Fx
        \\
        Fx \otimes_\D FI_\C
        \arrow[r, swap, "\phi_{x,I_\C}"]
        &
        F(x \otimes_\C I_\C)
        \arrow[u, swap, "F(\rightor^\C_x)"]
    \end{tikzcd}
    \end{equation}
    We say that $F$ is simply a \define{monoidal functor} when $\phi$ and $\phi_0$ are natural isomorphisms. It is worth noting that there exists a notion of ``oplax'' monoidal functors, where the structure map is reversed: $\phi_{x,y} \maps F(x \otimes y) \to Fx \otimes Fy$. However, oplax monoidal functors do not appear in this thesis, so we spend no further time on them.
    
    A \define{lax braided monoidal functor} is a lax monoidal functor $(F, \phi, \phi_0) \maps (\C, \otimes_\C, I_\C) \to (\D, \otimes_\D, I_\D)$ where $\C$ and $\D$ are braided monoidal categories, with $\braid^\C$ and $\braid^\D$ being the respective braidings, such that the following diagram commutes.
    \begin{equation}
    \begin{tikzcd}
        Fx \otimes_\D Fy
        \arrow[r, "\braid^\D_{Fx,Fy}"]
        \arrow[d, swap, "\phi_{x,y}"]
        &
        Fy \otimes_\D Fx
        \arrow[d, "\phi_{y,x}"]
        \\
        F(x \otimes_\C y)
        \arrow[r, swap, "F\braid^\C_{x,y}"]
        &
        F(y \otimes_\C x)
    \end{tikzcd}
    \end{equation}
    A (lax) braided monoidal functor between symmetric monoidal categories is called a \define{(lax) symmetric monoidal functor} with no further requirements. 
\end{defn}

\begin{lem}
    Composition of lax monoidal functors is strictly associative.
\end{lem}

We get categories $\Mon\Cat_\ell$, $\Mon\Cat$, $\Br\Mon\Cat_\ell$, $\Br\Mon\Cat$, $\Sym\Mon\Cat_\ell$, and $\Sym\Mon\Cat$ where the objects are monoidal categories, the functors are monoidal categories, the prefix $\Br$ (resp.\ $\Sym$) indicates the objects and morphisms are braided (resp.\ symmetric), and the subscript $\ell$ indicated the morphisms are lax monoidal.

\subsection{Monoidal Natural Transformations}

\begin{defn}
    Let $(F, \phi, \phi_0)$ and $(G, \gamma, \gamma_0)$ be lax monoidal functors. A \define{monoidal natural transformation} is a natural transformation $\theta \maps F \To G$ such that the following diagrams commute.
    \begin{equation}
    \begin{tikzcd}
        Fx \otimes_\D Fy
        \arrow[r, "\theta_x \otimes_\D \theta_y"]
        \arrow[d, swap, "\phi_{x,y}"]
        &
        Gx \otimes_\D Gy
        \arrow[d, "\gamma_{x,y}"]
        \\
        F(x \otimes_\C y)
        \arrow[r, swap, "\theta_{x \otimes_C y}"]
        &
        G(x \otimes_\C y)
    \end{tikzcd}
    \qquad
    \begin{tikzcd}
        &
        I_\D
        \arrow[dr, "\phi_0"]
        \arrow[dl, swap, "\gamma_0"]
        \\
        FI_\C
        \arrow[rr, swap, "\theta_{I_\C}"]
        &&
        GI_\C
    \end{tikzcd}
    \end{equation}
    There are no new laws which can be imposed on a monoidal natural transformation between braided or symmetric monoidal functors. So we do not specialize this concept any further.
\end{defn}

\section{Examples}

\begin{expl}
    Let $(M, \cdot, e)$ be a monoid. If we can consider $M$ as a discrete category, then it can be given a strict monoidal structure where the tensor is given by $\cdot$ and the unit is $e$. The functor $\Mon \hookrightarrow \Mon\Cat$ which realizes a monoid as a discrete monoidal category is full and faithful. If we think of this as ``forgetting discreteness'', then discreteness is a property.
\end{expl}

\begin{expl}
\label{expl:reversemonoidal}
    Given a monoidal category $(\C, \otimes, I)$, we can define $ \otimes\rev \maps \C \times \C \to \C$ by 
    \[
    \begin{tikzcd}
        \C \times \C
        \arrow[rr, "\otimes\rev"]
        \arrow[dr, swap, "\braid^\Cat_{\C,\C}"]
        &&
        \C
        \\&
        \C \times \C
        \arrow[ur, swap, "\otimes"]
    \end{tikzcd}\]
    This defines an idempotent automorphism on $\Mon\Cat$.
\end{expl}

\begin{expl}
\label{expl:opmonoidal}
    Given a monoidal category $(\C, \otimes, I)$, the category $\C$ can be equipped with a monoidal structure given by $\otimes\op \maps \C\op \times \C\op \to \C\op$ and the same unit object. This defines an idempotent automorphism on $\Mon\Cat$.
\end{expl}

\begin{expl}
\label{expl:cartesianmonoidal}
    Any category $C$ with finite products can be equipped with a symmetric monoidal structure as follows. For every pair of objects $c, d$, choose some object satisfying the universal property of the product of $c$ and $d$, call it $c \times d$. Given a pair of morphisms $f \maps a \to b$ and $g \maps c \to d$, the universal property gives a morphism $f \times g \maps a \times b \to c \times d$ as follows.
    \[
    \begin{tikzcd}[row sep = tiny]
        &
        a \times b 
        \arrow[dl, "\pi_a", swap]
        \arrow[dr, "\pi_b"]
        \arrow[dd, dashed, "\exists!"]
        \\
        a
        \arrow[dd, "f", swap]
        &&
        b
        \arrow[dd, "g"]
        \\&
        c \times d
        \arrow[dl, "\pi_c"]
        \arrow[dr, "\pi_d", swap]
        \\
        c
        &&
        d
    \end{tikzcd}\]
    We claim that this defines a functor $\times \maps C \times C \to C$. Consider a pair of morphisms $(f_1, f_2) \maps (a_1, a_2) \to (b_1, b_2)$ and $(g_1, g_2) \maps (b_1, b_2) \to (c_1, c_2)$. 
    Since $(g_1 \circ f_1) \times (g_2 \circ g_2)$ and $(g_1 \times g_2) \circ (f_1 \times f_2)$ both make the following diagram commute, they must be equal.
    \[
    \begin{tikzcd}[row sep = small]
        &
        a_1 \times a_2 
        \arrow[dl]
        \arrow[dr]
        \arrow[dd, dashed, ""]
        \\
        a_1
        \arrow[d, "f_1", swap]
        &&
        a_2
        \arrow[d, "f_2"]
        \\
        b_1
        \arrow[d, "g_1", swap]
        &
        c_1 \times c_2
        \arrow[dl]
        \arrow[dr]
        &
        b_2
        \arrow[d, "g_2"]
        \\
        c_1
        &&
        c_2
    \end{tikzcd}\]
    Identity maps are preserved because the identity map on $a \times b$ makes the diagram below commute.
    \[
    \begin{tikzcd}[row sep = tiny]
        &
        a \times b 
        \arrow[dl, "\pi_a", swap]
        \arrow[dr, "\pi_b"]
        \arrow[dd, dashed]
        \\
        a
        \arrow[dd, "1_a", swap]
        &&
        b
        \arrow[dd, "1_b"]
        \\&
        a \times b
        \arrow[dl, "\pi_a"]
        \arrow[dr, "\pi_b", swap]
        \\
        a
        &&
        b
    \end{tikzcd}\]

    We define the unit object to be some chosen terminal object, call it $1$. The associator, unitors, pentagon, hexagon, braiding, hexagon law, and symmetric law can all be derived from the universal property of products. This gives $\C$ the structure of a symmetric monoidal category. This is called the \define{cartesian monoidal structure}, and $(\C, \times, 1)$ is called a \define{cartesian monoidal category}.
\end{expl}

\begin{expl}
\label{expl:cocartesianmonoidal}
    Any category with finite coproducts can be equipped with a symmetric monoidal structure by \cref{expl:cartesianmonoidal} and \cref{expl:opmonoidal}.
\end{expl}

\begin{expl}
\label{expl:pointwisemonoidal}
    If $\C$ is monoidal and $\D$ is a category, the functor category $\C^\D$ can be given a \define{pointwise monoidal structure} as follows. Define $\otimes_{pt} \maps \C^\D \times \C^\D \to \C^\D$ by $\otimes_{pt} = \otimes (F \times G) \circ \Delta$. The unit object $1 \to \C^\D$ is given by currying the composite $D \xrightarrow{!} 1 \xrightarrow{I} \C$. The rest of the structures and the necessary properties all carry over from their counterparts in $\C$. Similarly, if $\C$ is braided or symmetric, then $\C^\D$ can be given a pointwise braided or symmetric monoidal structure respectively.
\end{expl}

\begin{expl}
\label{Dayconvolution}
    Let $C$ be a small monoidal category. Then the \define{Day convolution} tensor product \cite{Day}
    \[\otimes_\Day \maps \Set^{\C\op} \times \Set^{\C\op} \to \Set^{\C\op}\]
    is the following left Kan extension.
    \[\begin{tikzcd}
        C\op \times C\op
        \arrow[r, "{(X,Y)}"]
        \arrow[d, "\otimes", swap]
        &
        \Set
        \\
        C\op
        \arrow[ur, "X \otimes_\Day Y", swap, dashed]
    \end{tikzcd}\]
    This can be given by the following coend formula \cite{Coend}.  \[X \otimes_\Day Y \maps c \mapsto \int^{c_1,c_2 \in C} C(c_1 \otimes c_2, c) \times X(c_1) \times Y(c_2)\]
    Similarly, we can define the unit via left Kan extension.
    \[\begin{tikzcd}
        1
        \arrow[r, "\Delta1"]
        \arrow[d, "I", swap]
        &
        \Set
        \\
        C\op
        \arrow[ur, "I_\Day", swap, dashed]
    \end{tikzcd}\]
    Day convolution gives the functor category $\Set^{\C\op}$ a monoidal structure. Many nice properties of this structure can be found in the literature, e.g. \cite{Coend}. However, these properties are not heavily used in this thesis, so we choose to leave them out.
\end{expl}

\section{Monoid Objects}

A monoidal structure is exactly what a category needs to have if we want to define monoid objects in this category.
\begin{defn}
\label{def:monoidobject}
    Let $(\C, \otimes, I)$ be a monoidal category. A \define{monoid object} internal to $\C$ consist of 
    \begin{itemize}
        \item an object $x \in \C$
        \item a morphism $\mu \maps x \otimes x \to x$
        \item a morphism $\varepsilon \maps I \to x$
    \end{itemize}
    such that the following diagrams commute.
    \begin{equation}
    \begin{tikzcd}
        (x \otimes x) \otimes x
        \arrow[rr, "\assoc_{x,x,x}"]
        \arrow[d, swap, "\mu \otimes 1_x"]
        &&
        x \otimes (x \otimes x)
        \arrow[d, "1_x \otimes \mu"]
        \\
        x \otimes x
        \arrow[dr, swap, "\mu"]
        &&
        x \otimes x
        \arrow[dl, "\mu"]
        \\&
        x
    \end{tikzcd}
    \end{equation}
    \begin{equation} 
    \begin{tikzcd}
        I \otimes x
        \arrow[r, "\varepsilon \otimes x"]
        \arrow[dr, swap, "\leftor_x"]
        &
        x \otimes x
        \arrow[d, "\mu"]
        &
        x \otimes I
        \arrow[l, swap, "x \otimes \varepsilon"]
        \arrow[dl, "\rightor_x"]
        \\&
        x
    \end{tikzcd}
    \end{equation}
    Alternatively, we can express these structures with string diagrams as follows.
    \begin{itemize}
        \item multiplication
        \begin{tikzpicture}
        \begin{pgfonlayer}{nodelayer}
        	\node [style=blackdot] (mu) at (0, 0) {};
        	\node [style=none] (2) at (-0.5, 0.5) {};
        	\node [style=none] (3) at (0.5, 0.5) {};
        	\node [style=none] (4) at (0, -0.5) {};
            \node [style=none] () at (1, 0) {}; 
        \end{pgfonlayer}
        \begin{pgfonlayer}{edgelayer}
        	\draw [bend right] (2.center) to (mu);
        	\draw [bend left] (3.center) to (mu);
        	\draw (4.center) to (mu);
        \end{pgfonlayer}
        \end{tikzpicture}
        \item neutral element
        \begin{tikzpicture}
        \begin{pgfonlayer}{nodelayer}
        	\node [style=blackdot] (mu) at (0, 0) {};
        	\node [style=none] (4) at (0, -0.5) {};
            \node [style=none] () at (1, 0) {}; 
        \end{pgfonlayer}
        \begin{pgfonlayer}{edgelayer}
        	\draw (4.center) to (mu);
        \end{pgfonlayer}
        \end{tikzpicture}
    \end{itemize}
    such that
    \begin{equation}
    \begin{tikzpicture}
    \begin{pgfonlayer}{nodelayer}
    	\node [style=none] (1) at (0, 1) {};
    	\node [style=none] (2) at (0.5, 1) {};
    	\node [style=none] (3) at (1, 1) {};
    	\node [style=none] (4) at (1, 0.5) {};
    	\node [style=none] (5) at (0.625, -0.5) {};
        \node [style=blackdot] (b1) at (0.25, 0.5) {};
        \node [style=blackdot] (b2) at (0.625, 0) {};
    	\node [style=none] () at (1.5, 0.25) {=};
    	\node [style=none] () at (2, 0.5) {}; 
    \end{pgfonlayer}
    \begin{pgfonlayer}{edgelayer}
    	\draw [bend right] (1.center) to (b1.center);
    	\draw [bend left] (2.center) to (b1.center);
    	\draw (3.center) to (4.center);
    	\draw [bend left] (4.center) to (b2.center);
    	\draw [bend right] (b1.center) to (b2.center);
    	\draw (b2.center) to (5.center);
    \end{pgfonlayer}
    \end{tikzpicture}
    \begin{tikzpicture}
    \begin{pgfonlayer}{nodelayer}
    	\node [style=none] (1) at (0, 1) {};
    	\node [style=none] (2) at (0.5, 1) {};
    	\node [style=none] (3) at (1, 1) {};
    	\node [style=none] (4) at (0, 0.5) {};
    	\node [style=none] (5) at (0.375, -0.5) {};
        \node [style=blackdot] (b1) at (0.75, 0.5) {};
        \node [style=blackdot] (b2) at (0.375, 0) {};
    \end{pgfonlayer}
    \begin{pgfonlayer}{edgelayer}
    	\draw (1.center) to (4.center);
    	\draw [bend right] (2.center) to (b1.center);
    	\draw [bend left] (3.center) to (b1.center);
    	\draw [bend right] (4.center) to (b2.center);
    	\draw [bend left] (b1.center) to (b2.center);
    	\draw (b2.center) to (5.center);
    \end{pgfonlayer}
    \end{tikzpicture}
    \end{equation}
    \begin{equation}
    \begin{tikzpicture}
    \begin{pgfonlayer}{nodelayer}
    	\node [style=blackdot] (mu) at (0, 0) {};
    	\node [style=blackdot] (ep) at (0.5, 1) {};
        \node [style=none] (1) at (-0.5, 1.5) {};
    	\node [style=none] (2) at (0.5, 0.5) {};
    	\node [style=none] (4) at (-0.5, 0.5) {};
    	\node [style=none] (5) at (0, -0.5) {};
    	\node [style=none] () at (1.25, 0.5) {=};
    	\node [style=none] () at (1.75, 1) {}; 
    \end{pgfonlayer}
    \begin{pgfonlayer}{edgelayer}
    	\draw (1.center) to (4.center);
    	\draw (ep.center) to (2.center);
    	\draw [bend left] (2.center) to (mu.center);
    	\draw [bend right] (4.center) to (mu.center);
    	\draw (mu.center) to (5.center);
    \end{pgfonlayer}
    \end{tikzpicture}
    \begin{tikzpicture}
    \begin{pgfonlayer}{nodelayer}
        \node [style=none] (1) at (0, 1.5) {};
    	\node [style=none] (2) at (0, -0.5) {};
    	\node [style=none] () at (0.5, 0.5) {=};
    	\node [style=none] () at (1, 1) {}; 
    \end{pgfonlayer}
    \begin{pgfonlayer}{edgelayer}
    	\draw (1.center) to (2.center);
    \end{pgfonlayer}
    \end{tikzpicture}
    \begin{tikzpicture}
    \begin{pgfonlayer}{nodelayer}
    	\node [style=blackdot] (mu) at (0.5, 0) {};
    	\node [style=blackdot] (ep) at (0, 1) {};
        \node [style=none] (1) at (1, 1.5) {};
    	\node [style=none] (2) at (0, 0.5) {};
    	\node [style=none] (4) at (1, 0.5) {};
    	\node [style=none] (5) at (0.5, -0.5) {};
    \end{pgfonlayer}
    \begin{pgfonlayer}{edgelayer}
    	\draw (1.center) to (4.center);
    	\draw (ep.center) to (2.center);
    	\draw [bend right] (2.center) to (mu.center);
    	\draw [bend left] (4.center) to (mu.center);
    	\draw (mu.center) to (5.center);
    \end{pgfonlayer}
    \end{tikzpicture}
    \end{equation}
    
    Let $(x, \mu, \varepsilon)$ and $(y, \nu, \delta)$ be monoids in $\C$. A morphism $f \maps x \to y$ is called a \define{monoid homomorphism} if the following diagrams commute.
    \[
    \begin{tikzcd}
        x \otimes x
        \arrow[r, "f \otimes f"]
        \arrow[d, swap, "\mu"]
        &
        y \otimes y
        \arrow[d, "\nu"]
        \\
        x
        \arrow[r, swap, "f"]
        &
        y
    \end{tikzcd}
    \qquad
    \begin{tikzcd}
        &
        1
        \arrow[dl, swap, "\varepsilon"]
        \arrow[dr, "\delta"]
        \\
        x
        \arrow[rr, swap, "f"]
        &&
        y
    \end{tikzcd}
    \]
    In strings, these equations are depicted as follows.
    \[
    \begin{tikzpicture}
    \begin{pgfonlayer}{nodelayer}
    	\node [style=none] (x) at (-0.5, 0.75) {};
        \node [style=none] (y) at (0.5, 0.75) {};
        \node [style=construct] (f1) at (-0.5, 0) {$f$};
        \node [style=construct] (f2) at (0.5, 0) {$f$};
        \node [style=blackdot] (m) at (0, -0.75) {};
        \node [style=none] (q) at (0, -1.25) {};
        \node [style=none] () at (1.3, 0) {=};
        \node [style=none] () at (1.7, 0) {};
    \end{pgfonlayer}
    \begin{pgfonlayer}{edgelayer}
    	\draw (x.center) to (f1.center);
    	\draw (y.center) to (f2.center);
    	\draw [bend right = 50] (f1.center) to (m.center);
    	\draw [bend left = 50](f2.center) to (m.center);
    	\draw (m.center) to (q.center);
    \end{pgfonlayer}
    \end{tikzpicture}
    \begin{tikzpicture}
    \begin{pgfonlayer}{nodelayer}
    	\node [style=none] (x) at (-0.5, 0.75) {};
        \node [style=none] (y) at (0.5, 0.75) {};
        \node [style=blackdot] (m) at (0, 0) {};
        \node [style=construct] (f) at (0, -0.75) {$f$};
        \node [style=none] (q) at (0, -1.25) {};
    \end{pgfonlayer}
    \begin{pgfonlayer}{edgelayer}
    	\draw [bend right = 50] (x.center) to (m.center);
    	\draw [bend left = 50] (y.center) to (m.center);
    	\draw (m.center) to (f.center);
    	\draw (f.center) to (q.center);
    \end{pgfonlayer}
    \end{tikzpicture}
    \qquad\qquad\qquad
    \begin{tikzpicture}
    \begin{pgfonlayer}{nodelayer}
        \node [style=blackdot] (m) at (0, 0) {};
        \node [style=construct] (f) at (0, -0.75) {$f$};
        \node [style=none] (q) at (0, -1.25) {};
        \node [style=none] () at (0.8, -0.5) {=};
        \node [style=none] () at (1.3, 0) {};
    \end{pgfonlayer}
    \begin{pgfonlayer}{edgelayer}
    	\draw (m.center) to (f.center);
    	\draw (f.center) to (q.center);
    \end{pgfonlayer}
    \end{tikzpicture}
    \begin{tikzpicture}
    \begin{pgfonlayer}{nodelayer}
        \node [style=blackdot] (m) at (0, 0) {};
        \node [style=none] (q) at (0, -1.25) {};
    \end{pgfonlayer}
    \begin{pgfonlayer}{edgelayer}
    	\draw (m.center) to (q.center);
    \end{pgfonlayer}
    \end{tikzpicture}
    \]
    Let $\Mon(\C, \otimes)$ denote the category of monoid objects in $\C$ and their homomorphisms. If $\C$ is $\Set$ with its cartesian monoidal structure, we simply denote the category of monoids by $\Mon$.
\end{defn}

\begin{defn}
\label{def:commmon}
    Let $(\C, \otimes, I)$ be a braided monoidal category. A \define{commutative monoid} in $\C$ is a monoid object in $\C$ where the following equation holds.
    \begin{equation}
    \begin{tikzpicture}
    \begin{pgfonlayer}{nodelayer}
        \node [style=none] (3) at (0, -0.5) {};
    	\node [style=blackdot] (r) at (0, 0) {};
        \node [style=none] (1) at (-0.25, 0.5) {};
        \node [style=none] (2) at (0.25, 0.5) {};
    	\node [style=none] () at (1, 0) {=};
    	\node [style=none] () at (1.5, 0) {}; 
    \end{pgfonlayer}
    \begin{pgfonlayer}{edgelayer}
    	\draw [bend right] (1.center) to (r.center);
    	\draw [bend left] (2.center) to (r.center);
    	\draw (r.center) to (3.center);
    \end{pgfonlayer}
    \end{tikzpicture}
    \begin{tikzpicture}
    \begin{pgfonlayer}{nodelayer}
        \node [style=none] (1) at (0, -0.5) {};
    	\node [style=blackdot] (m) at (0, 0) {};
        \node [style=none] (2) at (0, 0.55) {};
        \node [style=none] (3) at (-0.25, 1) {};
        \node [style=none] (4) at (0.25, 1) {};
    \end{pgfonlayer}
    \begin{pgfonlayer}{edgelayer}
    	\draw (1.center) to (m.center);
    	\draw [bend right = 90, looseness = 1.5] (m.center) to (2.center);
    	\draw [bend left = 45] (2.center) to (3.center);
    	\draw [bend left = 90, white, line width = 3, looseness = 1.5] (m.center) to (2.center);
    	\draw [bend right = 45, white, line width = 3] (2.center) to (4.center);
    	\draw [bend left = 90, looseness = 1.5] (m.center) to (2.center);
    	\draw [bend right = 45] (2.center) to (4.center);
    \end{pgfonlayer}
    \end{tikzpicture}
    \end{equation}
    Let $\CMon(\C, \otimes)$ denote the category of commutative monoid objects in $\C$ and their homomorphisms. If $\C$ is $\Set$ with its cartesian monoidal structure, we simply denote the category of commutative monoids by $\CMon$.
\end{defn}

\section{The Eckmann--Hilton Argument}

\begin{thm}
\label{EckmannHilton}
    Let $C$ be a braided monoidal category, and let $x$ be an object equipped with two distinct monoid structures $(x, \mu, \varepsilon)$ and $(x, \nu, \eta)$ such that $\mu$ and $\nu$ are related by the following equation.
    \begin{equation}
    \label{interchange}
        \mu \circ (\nu \otimes \nu) = \nu(\mu \otimes \mu) \circ (1 \otimes \braid \otimes 1)
    \end{equation}
    Then $\varepsilon = \eta$, $\mu = \nu$, and $(x, \mu, \varepsilon)$ is commutative.
\end{thm}

It is important to note that if $\C$ is $\Set$ or some other concrete category, and the operations are instead denoted by $\circ$ and $\star$, \cref{interchange} becomes $(a \circ b) \star (c \circ d) = (a \star c) \circ (b \star d)$. Due to this formulation, this relation as it appears in many contexts is called the \define{middle-four interchange law}.

\begin{proof}
    We prove it using string diagrams, just for fun. Let the following string diagram components represent $\varepsilon$, $\mu$, $\eta$, and $\nu$, respectively.
    
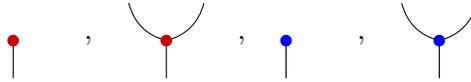
\[
    \begin{tikzpicture}
    \begin{pgfonlayer}{nodelayer}
    	\node [style=reddot] (ru) at (0, 0.5) {};
        \node [style=none] (1) at (0, 0) {};
    	\node [style=none] () at (1, 0.5) {,};
    	\node [style=none] () at (1.5, 0.5) {}; 
    \end{pgfonlayer}
    \begin{pgfonlayer}{edgelayer}
    	\draw (ru.center) to (1.center);
    \end{pgfonlayer}
    \end{tikzpicture}
    \begin{tikzpicture}
    \begin{pgfonlayer}{nodelayer}
    	\node [style=reddot] (rm) at (0, 0.5) {};
        \node [style=none] (1) at (0, 0) {};
        \node [style=none] (2) at (-0.5, 1) {};
        \node [style=none] (3) at (0.5, 1) {};
    	\node [style=none] () at (1, 0.5) {,};
    	\node [style=none] () at (1.5, 0.5) {}; 
    \end{pgfonlayer}
    \begin{pgfonlayer}{edgelayer}
    	\draw [bend left] (3.center) to (rm.center);
    	\draw [bend right] (2.center) to (rm.center);
    	\draw (rm.center) to (1.center);
    \end{pgfonlayer}
    \end{tikzpicture}
    \begin{tikzpicture}
    \begin{pgfonlayer}{nodelayer}
    	\node [style=bluedot] (bu) at (0, 0.5) {};
        \node [style=none] (1) at (0, 0) {};
    	\node [style=none] () at (1, 0.5) {,};
    	\node [style=none] () at (1.5, 0.5) {}; 
    \end{pgfonlayer}
    \begin{pgfonlayer}{edgelayer}
    	\draw (bu.center) to (1.center);
    \end{pgfonlayer}
    \end{tikzpicture}
    \begin{tikzpicture}
    \begin{pgfonlayer}{nodelayer}
    	\node [style=bluedot] (bm) at (0, 0.5) {};
        \node [style=none] (1) at (0, 0) {};
        \node [style=none] (2) at (-0.5, 1) {};
        \node [style=none] (3) at (0.5, 1) {};
    \end{pgfonlayer}
    \begin{pgfonlayer}{edgelayer}
    	\draw [bend left] (3.center) to (bm.center);
    	\draw [bend right] (2.center) to (bm.center);
    	\draw (bm.center) to (1.center);
    \end{pgfonlayer}
    \end{tikzpicture}
    \]
    Then we can draw \cref{interchange} as follows.
    
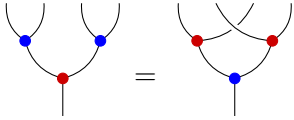
\[
    \begin{tikzpicture}
    \begin{pgfonlayer}{nodelayer}
    	\node [style=none] (r1) at (-0.75, 1) {};
    	\node [style=none] (b1) at (-0.25, 1) {};
    	\node [style=none] (b2) at (0.25, 1) {};
    	\node [style=none] (r2) at (0.75, 1) {};
    	\node [style=bluedot] (b3) at (-0.5, 0.5) {};
    	\node [style=bluedot] (b4) at (0.5, 0.5) {};
    	\node [style=reddot] (r3) at (0, 0) {};
        \node [style=none] (1) at (0, -0.5) {};
    	\node [style=none] () at (1.1, 0) {=};
    	\node [style=none] () at (1.5, 0) {}; 
    \end{pgfonlayer}
    \begin{pgfonlayer}{edgelayer}
    	\draw [bend right] (r1.center) to (b3.center);
    	\draw [bend left] (b1.center) to (b3.center);
    	\draw [bend right] (b2.center) to (b4.center);
    	\draw [bend left] (r2.center) to (b4.center);
    	\draw [bend right] (b3.center) to (r3.center);
    	\draw [bend left] (b4.center) to (r3.center);
    	\draw (r3.center) to (1.center);
    \end{pgfonlayer}
    \end{tikzpicture}
    \begin{tikzpicture}
    \begin{pgfonlayer}{nodelayer}
    	\node [style=none] (r1) at (-0.75, 1) {};
    	\node [style=none] (b1) at (-0.25, 1) {};
    	\node [style=none] (b2) at (0.25, 1) {};
    	\node [style=none] (r2) at (0.75, 1) {};
    	\node [style=reddot] (r3) at (-0.5, 0.5) {};
    	\node [style=reddot] (r4) at (0.5, 0.5) {};
    	\node [style=bluedot] (b3) at (0, 0) {};
        \node [style=none] (1) at (0, -0.5) {};
    \end{pgfonlayer}
    \begin{pgfonlayer}{edgelayer}
    	\draw [bend right] (r1.center) to (r3.center);
    	\draw [bend left] (b2.center) to (r3.center);
    	\draw [bend right, line width = 3, white] (b1.center) to (r4.center);
    	\draw [bend right] (b1.center) to (r4.center);
    	\draw [bend left] (r2.center) to (r4.center);
    	\draw [bend right] (r3.center) to (b3.center);
    	\draw [bend left] (r4.center) to (b3.center);
    	\draw (b3.center) to (1.center);
    \end{pgfonlayer}
    \end{tikzpicture}
    \]
    First, we show that the units coincide.
    
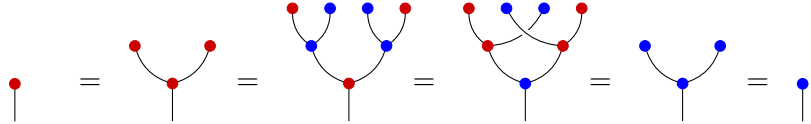
\[
    \begin{tikzpicture}
    \begin{pgfonlayer}{nodelayer}
    	\node [style=reddot] (ru) at (0, 0.5) {};
        \node [style=none] (1) at (0, 0) {};
    	\node [style=none] () at (1, 0.5) {=};
    	\node [style=none] () at (1.5, 0.5) {}; 
    \end{pgfonlayer}
    \begin{pgfonlayer}{edgelayer}
    	\draw (ru.center) to (1.center);
    \end{pgfonlayer}
    \end{tikzpicture}
    \begin{tikzpicture}
    \begin{pgfonlayer}{nodelayer}
    	\node [style=reddot] (r1) at (-0.5, 0.5) {};
    	\node [style=reddot] (r2) at (0.5, 0.5) {};
    	\node [style=reddot] (r3) at (0, 0) {};
        \node [style=none] (1) at (0, -0.5) {};
    	\node [style=none] () at (1, 0) {=};
    	\node [style=none] () at (1.5, 0) {}; 
    \end{pgfonlayer}
    \begin{pgfonlayer}{edgelayer}
    	\draw [bend right] (r1.center) to (r3.center);
    	\draw [bend left] (r2.center) to (r3.center);
    	\draw (r3.center) to (1.center);
    \end{pgfonlayer}
    \end{tikzpicture}
    \begin{tikzpicture}
    \begin{pgfonlayer}{nodelayer}
    	\node [style=reddot] (r1) at (-0.75, 1) {};
    	\node [style=bluedot] (b1) at (-0.25, 1) {};
    	\node [style=bluedot] (b2) at (0.25, 1) {};
    	\node [style=reddot] (r2) at (0.75, 1) {};
    	\node [style=bluedot] (b3) at (-0.5, 0.5) {};
    	\node [style=bluedot] (b4) at (0.5, 0.5) {};
    	\node [style=reddot] (r3) at (0, 0) {};
        \node [style=none] (1) at (0, -0.5) {};
    	\node [style=none] () at (1, 0) {=};
    	\node [style=none] () at (1.5, 0) {}; 
    \end{pgfonlayer}
    \begin{pgfonlayer}{edgelayer}
    	\draw [bend right] (r1.center) to (b3.center);
    	\draw [bend left] (b1.center) to (b3.center);
    	\draw [bend right] (b2.center) to (b4.center);
    	\draw [bend left] (r2.center) to (b4.center);
    	\draw [bend right] (b3.center) to (r3.center);
    	\draw [bend left] (b4.center) to (r3.center);
    	\draw (r3.center) to (1.center);
    \end{pgfonlayer}
    \end{tikzpicture}
    \begin{tikzpicture}
    \begin{pgfonlayer}{nodelayer}
    	\node [style=reddot] (r1) at (-0.75, 1) {};
    	\node [style=bluedot] (b1) at (-0.25, 1) {};
    	\node [style=bluedot] (b2) at (0.25, 1) {};
    	\node [style=reddot] (r2) at (0.75, 1) {};
    	\node [style=reddot] (r3) at (-0.5, 0.5) {};
    	\node [style=reddot] (r4) at (0.5, 0.5) {};
    	\node [style=bluedot] (b3) at (0, 0) {};
        \node [style=none] (1) at (0, -0.5) {};
    	\node [style=none] () at (1, 0) {=};
    	\node [style=none] () at (1.5, 0) {}; 
    \end{pgfonlayer}
    \begin{pgfonlayer}{edgelayer}
    	\draw [bend right] (r1.center) to (r3.center);
    	\draw [bend left] (b2.center) to (r3.center);
    	\draw [bend right, line width = 3, white] (b1.center) to (r4.center);
    	\draw [bend right] (b1.center) to (r4.center);
    	\draw [bend left] (r2.center) to (r4.center);
    	\draw [bend right] (r3.center) to (b3.center);
    	\draw [bend left] (r4.center) to (b3.center);
    	\draw (b3.center) to (1.center);
    \end{pgfonlayer}
    \end{tikzpicture}
    \begin{tikzpicture}
    \begin{pgfonlayer}{nodelayer}
    	\node [style=bluedot] (b1) at (-0.5, 0.5) {};
    	\node [style=bluedot] (b2) at (0.5, 0.5) {};
    	\node [style=bluedot] (b3) at (0, 0) {};
        \node [style=none] (1) at (0, -0.5) {};
    	\node [style=none] () at (1, 0) {=};
    	\node [style=none] () at (1.5, 0) {}; 
    \end{pgfonlayer}
    \begin{pgfonlayer}{edgelayer}
    	\draw [bend right] (b1.center) to (b3.center);
    	\draw [bend left] (b2.center) to (b3.center);
    	\draw (b3.center) to (1.center);
    \end{pgfonlayer}
    \end{tikzpicture}
    \begin{tikzpicture}
    \begin{pgfonlayer}{nodelayer}
    	\node [style=bluedot] (bu) at (0, 0.5) {};
        \node [style=none] (1) at (0, 0) {};
    \end{pgfonlayer}
    \begin{pgfonlayer}{edgelayer}
    	\draw (bu.center) to (1.center);
    \end{pgfonlayer}
    \end{tikzpicture}
    \]
    Since they are equal, we denote the unit with a black circle in the remainder of the proof. Next, we show in one calculation that the two operations are equal and commutative. 
    
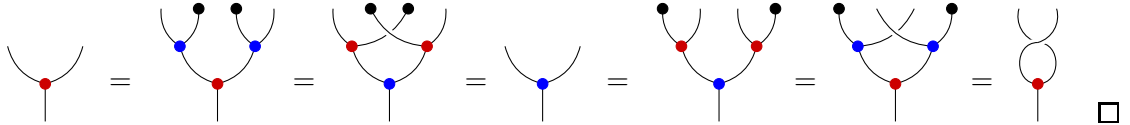
\[
    \begin{tikzpicture}
    \begin{pgfonlayer}{nodelayer}
    	\node [style=reddot] (r) at (0, 0) {};
        \node [style=none] (1) at (-0.5, 0.5) {};
        \node [style=none] (2) at (0.5, 0.5) {};
        \node [style=none] (3) at (0, -0.5) {};
    	\node [style=none] () at (1, 0) {=};
    	\node [style=none] () at (1.5, 0) {}; 
    \end{pgfonlayer}
    \begin{pgfonlayer}{edgelayer}
    	\draw [bend right] (1.center) to (r.center);
    	\draw [bend left] (2.center) to (r.center);
    	\draw (r.center) to (3.center);
    \end{pgfonlayer}
    \end{tikzpicture}
    \begin{tikzpicture}
    \begin{pgfonlayer}{nodelayer}
    	\node [style=blackdot] (b1) at (-0.25, 1) {};
    	\node [style=blackdot] (b2) at (0.25, 1) {};
    	\node [style=bluedot] (bl1) at (-0.5, 0.5) {};
    	\node [style=bluedot] (bl2) at (0.5, 0.5) {};
    	\node [style=reddot] (r) at (0, 0) {};
        \node [style=none] (1) at (-0.75, 1) {};
        \node [style=none] (2) at (0.75, 1) {};
        \node [style=none] (3) at (0, -0.5) {};
    	\node [style=none] () at (1.15, 0) {=};
    	\node [style=none] () at (1.5, 0) {}; 
    \end{pgfonlayer}
    \begin{pgfonlayer}{edgelayer}
    	\draw [bend right] (1.center) to (bl1.center);
    	\draw [bend left] (b1.center) to (bl1.center);
    	\draw [bend right] (b2.center) to (bl2.center);
    	\draw [bend left] (2.center) to (bl2.center);
    	\draw [bend right] (bl1.center) to (r.center);
    	\draw [bend left] (bl2.center) to (r.center);
    	\draw (r.center) to (3.center);
    \end{pgfonlayer}
    \end{tikzpicture}
    \begin{tikzpicture}
    \begin{pgfonlayer}{nodelayer}
    	\node [style=blackdot] (b1) at (-0.25, 1) {};
    	\node [style=blackdot] (b2) at (0.25, 1) {};
    	\node [style=reddot] (r1) at (-0.5, 0.5) {};
    	\node [style=reddot] (r2) at (0.5, 0.5) {};
    	\node [style=bluedot] (b) at (0, 0) {};
        \node [style=none] (1) at (-0.75, 1) {};
        \node [style=none] (2) at (0.75, 1) {};
        \node [style=none] (3) at (0, -0.5) {};
    	\node [style=none] () at (1.15, 0) {=};
    	\node [style=none] () at (1.5, 0) {}; 
    \end{pgfonlayer}
    \begin{pgfonlayer}{edgelayer}
    	\draw [bend right] (1.center) to (r1.center);
    	\draw [bend left] (b2.center) to (r1.center);
    	\draw [bend right, white, line width = 3] (b1.center) to (r2.center);
    	\draw [bend right] (b1.center) to (r2.center);
    	\draw [bend left] (2.center) to (r2.center);
    	\draw [bend right] (r1.center) to (b.center);
    	\draw [bend left] (r2.center) to (b.center);
    	\draw (b.center) to (3.center);
    \end{pgfonlayer}
    \end{tikzpicture}
    \begin{tikzpicture}
    \begin{pgfonlayer}{nodelayer}
    	\node [style=bluedot] (b) at (0, 0) {};
        \node [style=none] (1) at (-0.5, 0.5) {};
        \node [style=none] (2) at (0.5, 0.5) {};
        \node [style=none] (3) at (0, -0.5) {};
    	\node [style=none] () at (1, 0) {=};
    	\node [style=none] () at (1.5, 0) {}; 
    \end{pgfonlayer}
    \begin{pgfonlayer}{edgelayer}
    	\draw [bend right] (1.center) to (b.center);
    	\draw [bend left] (2.center) to (b.center);
    	\draw (b.center) to (3.center);
    \end{pgfonlayer}
    \end{tikzpicture}
    \begin{tikzpicture}
    \begin{pgfonlayer}{nodelayer}
        \node [style=blackdot] (b1) at (-0.75, 1) {};
        \node [style=blackdot] (b2) at (0.75, 1) {};
    	\node [style=reddot] (r1) at (-0.5, 0.5) {};
    	\node [style=reddot] (r2) at (0.5, 0.5) {};
    	\node [style=bluedot] (b) at (0, 0) {};
    	\node [style=none] (1) at (-0.25, 1) {};
    	\node [style=none] (2) at (0.25, 1) {};
        \node [style=none] (3) at (0, -0.5) {};
    	\node [style=none] () at (1.15, 0) {=};
    	\node [style=none] () at (1.5, 0) {}; 
    \end{pgfonlayer}
    \begin{pgfonlayer}{edgelayer}
    	\draw [bend right] (b1.center) to (bl1.center);
    	\draw [bend left] (1.center) to (bl1.center);
    	\draw [bend right] (2.center) to (bl2.center);
    	\draw [bend left] (b2.center) to (bl2.center);
    	\draw [bend right] (bl1.center) to (r.center);
    	\draw [bend left] (bl2.center) to (r.center);
    	\draw (r.center) to (3.center);
    \end{pgfonlayer}
    \end{tikzpicture}
    \begin{tikzpicture}
    \begin{pgfonlayer}{nodelayer}
    	\node [style=none] (b1) at (-0.25, 1) {};
    	\node [style=none] (b2) at (0.25, 1) {};
        \node [style=blackdot] (1) at (-0.75, 1) {};
        \node [style=blackdot] (2) at (0.75, 1) {};
    	\node [style=bluedot] (bl1) at (-0.5, 0.5) {};
    	\node [style=bluedot] (bl2) at (0.5, 0.5) {};
    	\node [style=reddot] (r) at (0, 0) {};
        \node [style=none] (3) at (0, -0.5) {};
    	\node [style=none] () at (1.15, 0) {=};
    	\node [style=none] () at (1.5, 0) {}; 
    \end{pgfonlayer}
    \begin{pgfonlayer}{edgelayer}
    	\draw [bend right] (1.center) to (bl1.center);
    	\draw [bend left] (b2.center) to (bl1.center);
    	\draw [bend right, white, line width = 3] (b1.center) to (bl2.center);
    	\draw [bend right] (b1.center) to (bl2.center);
    	\draw [bend left] (2.center) to (bl2.center);
    	\draw [bend right] (bl1.center) to (r.center);
    	\draw [bend left] (bl2.center) to (r.center);
    	\draw (r.center) to (3.center);
    \end{pgfonlayer}
    \end{tikzpicture}
    \begin{tikzpicture}
    \begin{pgfonlayer}{nodelayer}
        \node [style=none] (1) at (0, -0.5) {};
    	\node [style=reddot] (m) at (0, 0) {};
        \node [style=none] (2) at (0, 0.55) {};
        \node [style=none] (3) at (-0.25, 1) {};
        \node [style=none] (4) at (0.25, 1) {};
    \end{pgfonlayer}
    \begin{pgfonlayer}{edgelayer}
    	\draw (1.center) to (m.center);
    	\draw [bend right = 90, looseness = 1.5] (m.center) to (2.center);
    	\draw [bend left = 45] (2.center) to (3.center);
    	\draw [bend left = 90, white, line width = 3, looseness = 1.5] (m.center) to (2.center);
    	\draw [bend right = 45, white, line width = 3] (2.center) to (4.center);
    	\draw [bend left = 90, looseness = 1.5] (m.center) to (2.center);
    	\draw [bend right = 45] (2.center) to (4.center);
    \end{pgfonlayer}
    \end{tikzpicture}
    \qedhere
    \]
\end{proof}

\begin{cor}
    We have the following equivalences of categories.
    \[\Mon(\Mon, \times) \cong \Mon(\CMon, \times) \cong \CMon(\Mon, \times) \cong \CMon(\CMon, \times) \cong \CMon\]
\end{cor}

\section{Characterizing (co)cartesian monoidal categories}

In the previous section, we saw that a category with finite products can be equipped with a canonical symmetric monoidal structure, and dually so can a category with finite coproducts. In this section, we give conditions under which a symmetric monoidal category is monoidally equivalent to one given by a (co)cartesian structure \cite{Fox}.

\begin{expl}
\label{ex:monoidsincocart}
    Let $\C$ be a category with finite coproducts. By \cref{expl:cocartesianmonoidal}, $\C$ can be equipped with a cocartesian monoidal structure, with tensor denoted by $+$, and the unit (which is an initial object) denoted by $0$. Let $x$ be any object in $\C$. Universal property of coproducts gives a map $\nabla \maps x+x \to x$
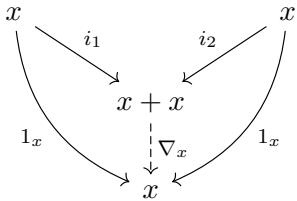
    \[
    \begin{tikzcd}
        x
        \arrow[dr, "i_1"]
        \arrow[ddr, swap, bend right, "1_x"]
        &&
        x
        \arrow[dl, swap, "i_2"]
        \arrow[ddl, bend left, "1_x"]
        \\&
        x+x
        \arrow[d, dashed, "\nabla_x"]
        \\&
        x
    \end{tikzcd}\]
    We draw string diagrams with respect to the cocartesian monoidal structure on $\C$. Then the map $\nabla_x$ is depicted as follows.
    
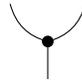
\[
    \begin{tikzpicture}
    \begin{pgfonlayer}{nodelayer}
    	\node [style=blackdot] (r) at (0, 0) {};
        \node [style=none] (1) at (-0.5, 0.5) {};
        \node [style=none] (2) at (0.5, 0.5) {};
        \node [style=none] (3) at (0, -0.5) {};
    \end{pgfonlayer}
    \begin{pgfonlayer}{edgelayer}
    	\draw [bend right] (1.center) to (r.center);
    	\draw [bend left] (2.center) to (r.center);
    	\draw (r.center) to (3.center);
    \end{pgfonlayer}
    \end{tikzpicture}
    \]
    Also, the universal property of an initial object gives a map $!_x \maps 0 \to x$, depicted as follows. 
    \[
    \begin{tikzpicture}
    \begin{pgfonlayer}{nodelayer}
    	\node [style=blackdot] (bu) at (0, 0.5) {};
        \node [style=none] (1) at (0, 0) {};
    \end{pgfonlayer}
    \begin{pgfonlayer}{edgelayer}
    	\draw (bu.center) to (1.center);
    \end{pgfonlayer}
    \end{tikzpicture}
    \]
    
    We show that this gives $x$ the structure of a commutative monoid. We begin by finding a formula for the left unitor of the cocartesian monoidal structure on $\C$. Notice that the left unitor makes the following diagram commute (by definition)
    
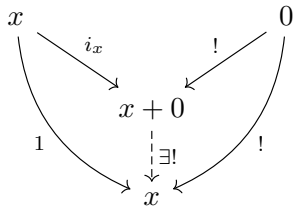
\[
    \begin{tikzcd}
        x
        \arrow[dr, "i_x"]
        \arrow[ddr, swap, "1", bend right]
        &&
        0
        \arrow[dl, swap, "!"]
        \arrow[ddl, "!", bend left]
        \\&
        x+0
        \arrow[d, dashed, "\exists !"]
        \\&
        x
    \end{tikzcd}
    \]
    and thus does so uniquely. Compare this to the diagram
    
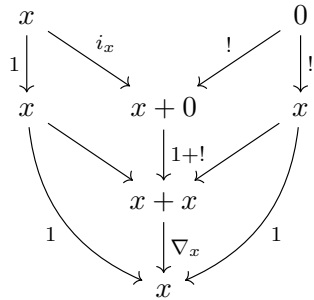
\[
    \begin{tikzcd}
        x
        \arrow[dr, "i_x"]
        \arrow[d, swap, "1"]
        &&
        0
        \arrow[dl, swap, "!"]
        \arrow[d, "!"]
        \\
        x
        \arrow[dr]
        \arrow[ddr, bend right, swap, "1"]
        &
        x+0
        \arrow[d, "1+!"]
        &
        x
        \arrow[dl]
        \arrow[ddl, bend left, "1"]
        \\&
        x+x
        \arrow[d, "\nabla_x"]
        \\&
        x
    \end{tikzcd}
    \]
    whose frame is equal to that of the previous. Thus we get $\nabla_x \circ (1_x + !_x) \circ i_x= 1_x$ and similarly $\nabla_x \circ (!_x + 1_x) \circ i'_x= 1_x$, which we draw as follows.
    \[
    \begin{tikzpicture}
    \begin{pgfonlayer}{nodelayer}
        \node [style=none] (1) at (0.25, 0.75) {};
    	\node [style=blackdot] (2) at (1, 0.5) {};
        \node [style=blackdot] (3) at (0.625, 0) {};
    	\node [style=none] (4) at (0.625, -0.5) {};
    	\node [style=none] () at (1.5, 0) {=};
    	\node [style=none] () at (2, 0.5) {}; 
    \end{pgfonlayer}
    \begin{pgfonlayer}{edgelayer}
    	\draw [bend left] (2.center) to (3.center);
    	\draw [bend right] (1.center) to (3.center);
    	\draw (3.center) to (4.center);
    \end{pgfonlayer}
    \end{tikzpicture}
    \begin{tikzpicture}
    \begin{pgfonlayer}{nodelayer}
    	\node [style=none] (1) at (0, 1) {};
        \node [style=none] (2) at (0, 0) {};
        \node [style=none] () at (4, 0) {};
    \end{pgfonlayer}
    \begin{pgfonlayer}{edgelayer}
    	\draw (1.center) to (2.center);
    \end{pgfonlayer}
    \end{tikzpicture}
    \begin{tikzpicture}
    \begin{pgfonlayer}{nodelayer}
        \node [style=blackdot] (1) at (0.25, 0.5) {};
    	\node [style=none] (2) at (1, 0.75) {};
        \node [style=blackdot] (3) at (0.625, 0) {};
    	\node [style=none] (4) at (0.625, -0.5) {};
    	\node [style=none] () at (1.5, 0) {=};
    	\node [style=none] () at (2, 0.5) {}; 
    \end{pgfonlayer}
    \begin{pgfonlayer}{edgelayer}
    	\draw [bend left] (2.center) to (3.center);
    	\draw [bend right] (1.center) to (3.center);
    	\draw (3.center) to (4.center);
    \end{pgfonlayer}
    \end{tikzpicture}
    \begin{tikzpicture}
    \begin{pgfonlayer}{nodelayer}
    	\node [style=none] (1) at (0, 1) {};
        \node [style=none] (2) at (0, 0) {};
    \end{pgfonlayer}
    \begin{pgfonlayer}{edgelayer}
    	\draw (1.center) to (2.center);
    \end{pgfonlayer}
    \end{tikzpicture}\]
    
    To show $\nabla_x$ is associative, we want to show that the following equation holds.
    \[
    \begin{tikzpicture}
    \begin{pgfonlayer}{nodelayer}
    	\node [style=none] (1) at (0, 1) {};
    	\node [style=none] (2) at (0.5, 1) {};
    	\node [style=none] (3) at (1, 1) {};
    	\node [style=none] (4) at (1, 0.5) {};
    	\node [style=none] (5) at (0.625, -0.5) {};
        \node [style=blackdot] (b1) at (0.25, 0.5) {};
        \node [style=blackdot] (b2) at (0.625, 0) {};
    	\node [style=none] () at (1.5, 0.25) {=};
    	\node [style=none] () at (2, 0.5) {}; 
    \end{pgfonlayer}
    \begin{pgfonlayer}{edgelayer}
    	\draw [bend right] (1.center) to (b1.center);
    	\draw [bend left] (2.center) to (b1.center);
    	\draw (3.center) to (4.center);
    	\draw [bend left] (4.center) to (b2.center);
    	\draw [bend right] (b1.center) to (b2.center);
    	\draw (b2.center) to (5.center);
    \end{pgfonlayer}
    \end{tikzpicture}
    \begin{tikzpicture}
    \begin{pgfonlayer}{nodelayer}
    	\node [style=none] (1) at (0, 1) {};
    	\node [style=none] (2) at (0.5, 1) {};
    	\node [style=none] (3) at (1, 1) {};
    	\node [style=none] (4) at (0, 0.5) {};
    	\node [style=none] (5) at (0.375, -0.5) {};
        \node [style=blackdot] (b1) at (0.75, 0.5) {};
        \node [style=blackdot] (b2) at (0.375, 0) {};
    \end{pgfonlayer}
    \begin{pgfonlayer}{edgelayer}
    	\draw (1.center) to (4.center);
    	\draw [bend right] (2.center) to (b1.center);
    	\draw [bend left] (3.center) to (b1.center);
    	\draw [bend right] (4.center) to (b2.center);
    	\draw [bend left] (b1.center) to (b2.center);
    	\draw (b2.center) to (5.center);
    \end{pgfonlayer}
    \end{tikzpicture}
    \]
    We have three inclusion maps $i_0, i_1, i_2 \maps x \to x+x+x$, which are given in strings below. 
    \[
    \begin{tikzpicture}
    \begin{pgfonlayer}{nodelayer}
    	\node [style=none] (bu) at (0, 1) {};
        \node [style=none] (1) at (0, 0) {};
        \node [style=blackdot] (2) at (0.5, 0.5) {};
        \node [style=none] (3) at (0.5, 0) {};
        \node [style=blackdot] (4) at (1, 0.5) {};
        \node [style=none] (5) at (1, 0) {};
    	\node [style=none] () at (1.5, 0.5) {,};
    	\node [style=none] () at (2, 0.5) {}; 
    \end{pgfonlayer}
    \begin{pgfonlayer}{edgelayer}
    	\draw (bu.center) to (1.center);
    	\draw (2.center) to (3.center);
    	\draw (4.center) to (5.center);
    \end{pgfonlayer}
    \end{tikzpicture}
    \begin{tikzpicture}
    \begin{pgfonlayer}{nodelayer}
    	\node [style=blackdot] (bu) at (0, 0.5) {};
        \node [style=none] (1) at (0, 0) {};
        \node [style=none] (2) at (0.5, 1) {};
        \node [style=none] (3) at (0.5, 0) {};
        \node [style=blackdot] (4) at (1, 0.5) {};
        \node [style=none] (5) at (1, 0) {};
    	\node [style=none] () at (1.5, 0.5) {,};
    	\node [style=none] () at (2, 0.5) {}; 
    \end{pgfonlayer}
    \begin{pgfonlayer}{edgelayer}
    	\draw (bu.center) to (1.center);
    	\draw (2.center) to (3.center);
    	\draw (4.center) to (5.center);
    \end{pgfonlayer}
    \end{tikzpicture}
    \begin{tikzpicture}
    \begin{pgfonlayer}{nodelayer}
    	\node [style=blackdot] (bu) at (0, 0.5) {};
        \node [style=none] (1) at (0, 0) {};
        \node [style=blackdot] (2) at (0.5, 0.5) {};
        \node [style=none] (3) at (0.5, 0) {};
        \node [style=none] (4) at (1, 1) {};
        \node [style=none] (5) at (1, 0) {};
    \end{pgfonlayer}
    \begin{pgfonlayer}{edgelayer}
    	\draw (bu.center) to (1.center);
    	\draw (2.center) to (3.center);
    	\draw (4.center) to (5.center);
    \end{pgfonlayer}
    \end{tikzpicture}
    \]
    The universal property of coproducts says that if the composites of the morphisms on the left and right side of the associativity equation above with any of the three inclusions is always the identity morphism on $x$, then those two morphisms must be equal. So we compute:
    \[
    \begin{tikzpicture}
    \begin{pgfonlayer}{nodelayer}
    	\node [style=none] (1) at (0, 1.25) {};
    	\node [style=blackdot] (2) at (0.5, 1) {};
    	\node [style=blackdot] (3) at (1, 1) {};
    	\node [style=none] (4) at (1, 0.5) {};
    	\node [style=none] (5) at (0.625, -0.5) {};
        \node [style=blackdot] (b1) at (0.25, 0.5) {};
        \node [style=blackdot] (b2) at (0.625, 0) {};
    	\node [style=none] () at (1.5, 0) {=};
    	\node [style=none] () at (2, 0.5) {}; 
    \end{pgfonlayer}
    \begin{pgfonlayer}{edgelayer}
    	\draw [bend right] (1.center) to (b1.center);
    	\draw [bend left] (2.center) to (b1.center);
    	\draw (3.center) to (4.center);
    	\draw [bend left] (4.center) to (b2.center);
    	\draw [bend right] (b1.center) to (b2.center);
    	\draw (b2.center) to (5.center);
    \end{pgfonlayer}
    \end{tikzpicture}
    \begin{tikzpicture}
    \begin{pgfonlayer}{nodelayer}
        \node [style=none] (1) at (0.25, 0.75) {};
    	\node [style=blackdot] (2) at (1, 0.5) {};
        \node [style=blackdot] (3) at (0.625, 0) {};
    	\node [style=none] (4) at (0.625, -0.5) {};
    	\node [style=none] () at (1.5, 0) {=};
    	\node [style=none] () at (2, 0.5) {}; 
    \end{pgfonlayer}
    \begin{pgfonlayer}{edgelayer}
    	\draw [bend left] (2.center) to (3.center);
    	\draw [bend right] (1.center) to (3.center);
    	\draw (3.center) to (4.center);
    \end{pgfonlayer}
    \end{tikzpicture}
    \begin{tikzpicture}
    \begin{pgfonlayer}{nodelayer}
    	\node [style=none] (1) at (0, 1) {};
        \node [style=none] (2) at (0, 0) {};
        \node [style=none] () at (4, 0) {};
    \end{pgfonlayer}
    \begin{pgfonlayer}{edgelayer}
    	\draw (1.center) to (2.center);
    \end{pgfonlayer}
    \end{tikzpicture}
    \begin{tikzpicture}
    \begin{pgfonlayer}{nodelayer}
    	\node [style=blackdot] (1) at (0, 1) {};
    	\node [style=none] (2) at (0.5, 1.25) {};
    	\node [style=blackdot] (3) at (1, 1) {};
    	\node [style=none] (4) at (1, 0.5) {};
    	\node [style=none] (5) at (0.625, -0.5) {};
        \node [style=blackdot] (b1) at (0.25, 0.5) {};
        \node [style=blackdot] (b2) at (0.625, 0) {};
    	\node [style=none] () at (1.5, 0) {=};
    	\node [style=none] () at (2, 0.5) {}; 
    \end{pgfonlayer}
    \begin{pgfonlayer}{edgelayer}
    	\draw [bend right] (1.center) to (b1.center);
    	\draw [bend left] (2.center) to (b1.center);
    	\draw (3.center) to (4.center);
    	\draw [bend left] (4.center) to (b2.center);
    	\draw [bend right] (b1.center) to (b2.center);
    	\draw (b2.center) to (5.center);
    \end{pgfonlayer}
    \end{tikzpicture}
    \begin{tikzpicture}
    \begin{pgfonlayer}{nodelayer}
        \node [style=none] (1) at (0.25, 0.75) {};
    	\node [style=blackdot] (2) at (1, 0.5) {};
        \node [style=blackdot] (3) at (0.625, 0) {};
    	\node [style=none] (4) at (0.625, -0.5) {};
    	\node [style=none] () at (1.5, 0) {=};
    	\node [style=none] () at (2, 0.5) {}; 
    \end{pgfonlayer}
    \begin{pgfonlayer}{edgelayer}
    	\draw [bend left] (2.center) to (3.center);
    	\draw [bend right] (1.center) to (3.center);
    	\draw (3.center) to (4.center);
    \end{pgfonlayer}
    \end{tikzpicture}
    \begin{tikzpicture}
    \begin{pgfonlayer}{nodelayer}
    	\node [style=none] (1) at (0, 1) {};
        \node [style=none] (2) at (0, 0) {};
    \end{pgfonlayer}
    \begin{pgfonlayer}{edgelayer}
    	\draw (1.center) to (2.center);
    \end{pgfonlayer}
    \end{tikzpicture}
    \]
    \[
    \begin{tikzpicture}
    \begin{pgfonlayer}{nodelayer}
    	\node [style=blackdot] (1) at (0, 1) {};
    	\node [style=blackdot] (2) at (0.5, 1) {};
    	\node [style=none] (3) at (1, 1.25) {};
    	\node [style=none] (4) at (1, 0.5) {};
    	\node [style=none] (5) at (0.625, -0.5) {};
        \node [style=blackdot] (b1) at (0.25, 0.5) {};
        \node [style=blackdot] (b2) at (0.625, 0) {};
    	\node [style=none] () at (1.5, 0) {=};
    	\node [style=none] () at (2, 0.5) {}; 
    \end{pgfonlayer}
    \begin{pgfonlayer}{edgelayer}
    	\draw [bend right] (1.center) to (b1.center);
    	\draw [bend left] (2.center) to (b1.center);
    	\draw (3.center) to (4.center);
    	\draw [bend left] (4.center) to (b2.center);
    	\draw [bend right] (b1.center) to (b2.center);
    	\draw (b2.center) to (5.center);
    \end{pgfonlayer}
    \end{tikzpicture}
    \begin{tikzpicture}
    \begin{pgfonlayer}{nodelayer}
        \node [style=blackdot] (1) at (0.25, 0.5) {};
    	\node [style=none] (2) at (1, 0.75) {};
        \node [style=blackdot] (3) at (0.625, 0) {};
    	\node [style=none] (4) at (0.625, -0.5) {};
    	\node [style=none] () at (1.5, 0) {=};
    	\node [style=none] () at (2, 0.5) {}; 
    \end{pgfonlayer}
    \begin{pgfonlayer}{edgelayer}
    	\draw [bend left] (2.center) to (3.center);
    	\draw [bend right] (1.center) to (3.center);
    	\draw (3.center) to (4.center);
    \end{pgfonlayer}
    \end{tikzpicture}
    \begin{tikzpicture}
    \begin{pgfonlayer}{nodelayer}
    	\node [style=none] (1) at (0, 1) {};
        \node [style=none] (2) at (0, 0) {};
    \end{pgfonlayer}
    \begin{pgfonlayer}{edgelayer}
    	\draw (1.center) to (2.center);
    \end{pgfonlayer}
    \end{tikzpicture}
    \]
    \[
    \begin{tikzpicture}
    \begin{pgfonlayer}{nodelayer}
    	\node [style=none] (1) at (0, 1.25) {};
    	\node [style=blackdot] (2) at (0.5, 1) {};
    	\node [style=blackdot] (3) at (1, 1) {};
    	\node [style=none] (4) at (0, 0.5) {};
    	\node [style=none] (5) at (0.375, -0.5) {};
        \node [style=blackdot] (b1) at (0.75, 0.5) {};
        \node [style=blackdot] (b2) at (0.375, 0) {};
    	\node [style=none] () at (1.5, 0) {=};
    	\node [style=none] () at (2, 0.5) {}; 
    \end{pgfonlayer}
    \begin{pgfonlayer}{edgelayer}
    	\draw (1.center) to (4.center);
    	\draw [bend right] (2.center) to (b1.center);
    	\draw [bend left] (3.center) to (b1.center);
    	\draw [bend right] (4.center) to (b2.center);
    	\draw [bend left] (b1.center) to (b2.center);
    	\draw (b2.center) to (5.center);
    \end{pgfonlayer}
    \end{tikzpicture}
    \begin{tikzpicture}
    \begin{pgfonlayer}{nodelayer}
        \node [style=none] (1) at (0.25, 0.75) {};
    	\node [style=blackdot] (2) at (1, 0.5) {};
        \node [style=blackdot] (3) at (0.625, 0) {};
    	\node [style=none] (4) at (0.625, -0.5) {};
    	\node [style=none] () at (1.5, 0) {=};
    	\node [style=none] () at (2, 0.5) {}; 
    \end{pgfonlayer}
    \begin{pgfonlayer}{edgelayer}
    	\draw [bend left] (2.center) to (3.center);
    	\draw [bend right] (1.center) to (3.center);
    	\draw (3.center) to (4.center);
    \end{pgfonlayer}
    \end{tikzpicture}
    \begin{tikzpicture}
    \begin{pgfonlayer}{nodelayer}
    	\node [style=none] (1) at (0, 1) {};
        \node [style=none] (2) at (0, 0) {};
        \node [style=none] () at (4, 0) {};
    \end{pgfonlayer}
    \begin{pgfonlayer}{edgelayer}
    	\draw (1.center) to (2.center);
    \end{pgfonlayer}
    \end{tikzpicture}
    \begin{tikzpicture}
    \begin{pgfonlayer}{nodelayer}
    	\node [style=blackdot] (1) at (0, 1) {};
    	\node [style=none] (2) at (0.5, 1.25) {};
    	\node [style=blackdot] (3) at (1, 1) {};
    	\node [style=none] (4) at (0, 0.5) {};
    	\node [style=none] (5) at (0.375, -0.5) {};
        \node [style=blackdot] (b1) at (0.75, 0.5) {};
        \node [style=blackdot] (b2) at (0.375, 0) {};
    	\node [style=none] () at (1.5, 0) {=};
    	\node [style=none] () at (2, 0.5) {}; 
    \end{pgfonlayer}
    \begin{pgfonlayer}{edgelayer}
    	\draw (1.center) to (4.center);
    	\draw [bend right] (2.center) to (b1.center);
    	\draw [bend left] (3.center) to (b1.center);
    	\draw [bend right] (4.center) to (b2.center);
    	\draw [bend left] (b1.center) to (b2.center);
    	\draw (b2.center) to (5.center);
    \end{pgfonlayer}
    \end{tikzpicture}
    \begin{tikzpicture}
    \begin{pgfonlayer}{nodelayer}
        \node [style=blackdot] (1) at (0.25, 0.5) {};
    	\node [style=none] (2) at (1, 0.75) {};
        \node [style=blackdot] (3) at (0.625, 0) {};
    	\node [style=none] (4) at (0.625, -0.5) {};
    	\node [style=none] () at (1.5, 0) {=};
    	\node [style=none] () at (2, 0.5) {}; 
    \end{pgfonlayer}
    \begin{pgfonlayer}{edgelayer}
    	\draw [bend left] (2.center) to (3.center);
    	\draw [bend right] (1.center) to (3.center);
    	\draw (3.center) to (4.center);
    \end{pgfonlayer}
    \end{tikzpicture}
    \begin{tikzpicture}
    \begin{pgfonlayer}{nodelayer}
    	\node [style=none] (1) at (0, 1) {};
        \node [style=none] (2) at (0, 0) {};
    \end{pgfonlayer}
    \begin{pgfonlayer}{edgelayer}
    	\draw (1.center) to (2.center);
    \end{pgfonlayer}
    \end{tikzpicture}
    \]
    \[
    \begin{tikzpicture}
    \begin{pgfonlayer}{nodelayer}
    	\node [style=blackdot] (1) at (0, 1) {};
    	\node [style=blackdot] (2) at (0.5, 1) {};
    	\node [style=none] (3) at (1, 1.25) {};
    	\node [style=none] (4) at (0, 0.5) {};
    	\node [style=none] (5) at (0.375, -0.5) {};
        \node [style=blackdot] (b1) at (0.75, 0.5) {};
        \node [style=blackdot] (b2) at (0.375, 0) {};
    	\node [style=none] () at (1.5, 0) {=};
    	\node [style=none] () at (2, 0.5) {}; 
    \end{pgfonlayer}
    \begin{pgfonlayer}{edgelayer}
    	\draw (1.center) to (4.center);
    	\draw [bend right] (2.center) to (b1.center);
    	\draw [bend left] (3.center) to (b1.center);
    	\draw [bend right] (4.center) to (b2.center);
    	\draw [bend left] (b1.center) to (b2.center);
    	\draw (b2.center) to (5.center);
    \end{pgfonlayer}
    \end{tikzpicture}
    \begin{tikzpicture}
    \begin{pgfonlayer}{nodelayer}
        \node [style=blackdot] (1) at (0.25, 0.5) {};
    	\node [style=none] (2) at (1, 0.75) {};
        \node [style=blackdot] (3) at (0.625, 0) {};
    	\node [style=none] (4) at (0.625, -0.5) {};
    	\node [style=none] () at (1.5, 0) {=};
    	\node [style=none] () at (2, 0.5) {}; 
    \end{pgfonlayer}
    \begin{pgfonlayer}{edgelayer}
    	\draw [bend left] (2.center) to (3.center);
    	\draw [bend right] (1.center) to (3.center);
    	\draw (3.center) to (4.center);
    \end{pgfonlayer}
    \end{tikzpicture}
    \begin{tikzpicture}
    \begin{pgfonlayer}{nodelayer}
    	\node [style=none] (1) at (0, 1) {};
        \node [style=none] (2) at (0, 0) {};
    \end{pgfonlayer}
    \begin{pgfonlayer}{edgelayer}
    	\draw (1.center) to (2.center);
    \end{pgfonlayer}
    \end{tikzpicture}
    \]
    Thus we have that $\nabla_x$ is associative.

    Recall that the braiding in $\C$ is derived from the universal property in the following way. 
    \[
    \begin{tikzcd}
        x
        \arrow[dr, "i_1"]
        \arrow[ddr, swap, "i_2", bend right]
        &&
        x
        \arrow[dl, swap, "i_2"]
        \arrow[ddl, "i_1", bend left]
        \\&
        x+x
        \arrow[d, dashed, "\sigma"]
        \\&
        x+x
    \end{tikzcd}\]
    Then the commutative diagram
    
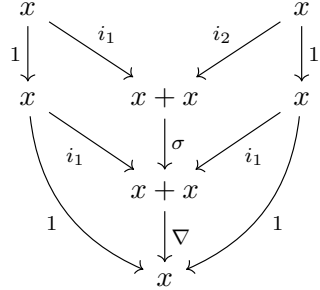
\[
    \begin{tikzcd}
        x
        \arrow[dr, "i_1"]
        \arrow[d, swap, "1"]
        &&
        x
        \arrow[dl, swap, "i_2"]
        \arrow[d, "1"]
        \\
        x
        \arrow[dr, swap, "i_1"]
        \arrow[ddr, bend right, swap, "1"]
        &
        x+x
        \arrow[d, "\sigma"]
        &
        x
        \arrow[dl, "i_1"]
        \arrow[ddl, bend left, "1"]
        \\&
        x+x
        \arrow[d, "\nabla"]
        \\&
        x
    \end{tikzcd}
    \]
    has precisely the frame for the universal construction of $\nabla_x$. Thus $\nabla_x \circ \sigma = \nabla_x$, displayed as follows.
    
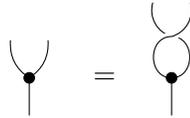
\[
    \begin{tikzpicture}
    \begin{pgfonlayer}{nodelayer}
        \node [style=none] (3) at (0, -0.5) {};
    	\node [style=blackdot] (r) at (0, 0) {};
        \node [style=none] (1) at (-0.25, 0.5) {};
        \node [style=none] (2) at (0.25, 0.5) {};
    	\node [style=none] () at (1, 0) {=};
    	\node [style=none] () at (1.5, 0) {}; 
    \end{pgfonlayer}
    \begin{pgfonlayer}{edgelayer}
    	\draw [bend right] (1.center) to (r.center);
    	\draw [bend left] (2.center) to (r.center);
    	\draw (r.center) to (3.center);
    \end{pgfonlayer}
    \end{tikzpicture}
    \begin{tikzpicture}
    \begin{pgfonlayer}{nodelayer}
        \node [style=none] (1) at (0, -0.5) {};
    	\node [style=blackdot] (m) at (0, 0) {};
        \node [style=none] (2) at (0, 0.55) {};
        \node [style=none] (3) at (-0.25, 1) {};
        \node [style=none] (4) at (0.25, 1) {};
    \end{pgfonlayer}
    \begin{pgfonlayer}{edgelayer}
    	\draw (1.center) to (m.center);
    	\draw [bend right = 90, looseness = 1.5] (m.center) to (2.center);
    	\draw [bend left = 45] (2.center) to (3.center);
    	\draw [bend left = 90, white, line width = 3, looseness = 1.5] (m.center) to (2.center);
    	\draw [bend right = 45, white, line width = 3] (2.center) to (4.center);
    	\draw [bend left = 90, looseness = 1.5] (m.center) to (2.center);
    	\draw [bend right = 45] (2.center) to (4.center);
    \end{pgfonlayer}
    \end{tikzpicture}
    \]
\end{expl}

\begin{prop}
\label{prop:cocarthasmonoids}
    A symmetric monoidal category is cocartesian if and only if each object has a natural commutative monoid structure.
\end{prop}
\begin{proof}
    Given an object $x$ in a cocartesian monoidal category $\C$, we constructed a commutative monoid structure on $x$ in \cref{ex:monoidsincocart}. 
    
    We have to show that the multiplication maps 
    \[\nabla_x \maps x+x \to x \]
    form the components of a natural transformation
    \[\nabla \maps + \circ \Delta \To 1_\C.\]
    For a given morphism $f \maps x \to y$, the naturality square is 
    \[
    \begin{tikzcd}
        x + x
        \arrow[r, "f+f"]
        \arrow[d, swap, "\nabla_x"]
        &
        y + y
        \arrow[d, "\nabla_y"]
        \\
        x
        \arrow[r, swap, "f"]
        &
        y
    \end{tikzcd}
    \]
    Recall that the map $f+f$ is derived from the universal property of coproducts by the following diagram.
    
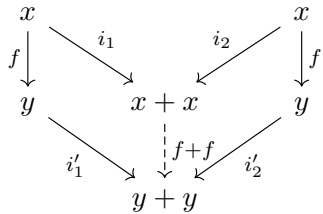
\[
    \begin{tikzcd}
        x
        \arrow[d, swap, "f"]
        \arrow[dr, "i_1"]
        &&
        x
        \arrow[dl, swap, "i_2"]
        \arrow[d, "f"]
        \\
        y
        \arrow[dr, swap, "i_1'"]
        &
        x+x
        \arrow[d, dashed, "f+f"]&
        y
        \arrow[dl, "i_2'"]
        \\&
        y + y
    \end{tikzcd}
    \]
    We want to show that $\nabla_y \circ (f+f) = f \circ \nabla_x$.
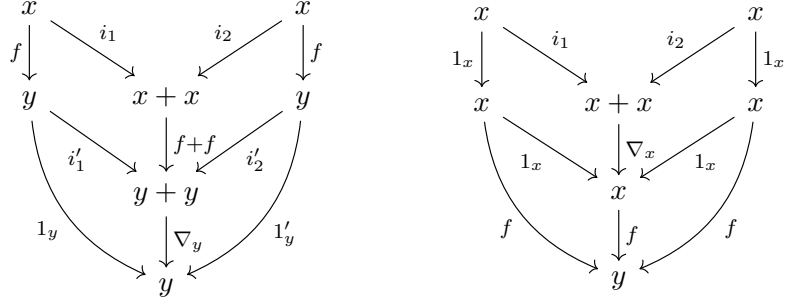
    \[
    \begin{tikzcd}
        x
        \arrow[d, swap, "f"]
        \arrow[dr, "i_1"]
        &&
        x
        \arrow[dl, swap, "i_2"]
        \arrow[d, "f"]
        \\
        y
        \arrow[dr, swap, "i_1'"]
        \arrow[ddr, swap, "1_y", bend right]
        &
        x+x
        \arrow[d, "f+f"]&
        y
        \arrow[dl, "i_2'"]
        \arrow[ddl, "1_y'", bend left]
        \\&
        y + y
        \arrow[d, "\nabla_y"]
        \\&y
    \end{tikzcd}
    \qquad\qquad
    \begin{tikzcd}
        x
        \arrow[d, swap, "1_x"]
        \arrow[dr, "i_1"]
        &&
        x
        \arrow[dl, swap, "i_2"]
        \arrow[d, "1_x"]
        \\
        x
        \arrow[dr, swap, "1_x"]
        \arrow[ddr, swap, "f", bend right]
        &
        x + x
        \arrow[d, "\nabla_x"]&
        x
        \arrow[dl, "1_x"]
        \arrow[ddl, "f", bend left]
        \\&
        x
        \arrow[d, "f"]
        \\&y
    \end{tikzcd}
    \]
    The frames of the above diagrams are identical, and they are equal to the frame which produces $\langle f , f \rangle$, the copairing of $f$ with itself. So by universal property, they are equal.
    
    The naturality square of the units collapses into the triangle below.
    
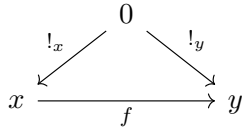
\[
    \begin{tikzcd}
        &
        0
        \arrow[dl, swap, "!_x"]
        \arrow[dr, "!_y"]
        \\
        x
        \arrow[rr, swap, "f"]
        &&
        y
    \end{tikzcd}
    \]
    which commutes by initiality of $0$.
    
    We have shown one direction: that if $\C$ is cocartesian monoidal, then each object has a natural commutative monoid structure. Now we must show the converse. Assume that $(\C, \otimes, I)$ is a symmetric monoidal category such that each object has a natural commutative monoid structure $m \maps \otimes \circ \Delta \To 1_\C$ and $\varepsilon \maps \Delta I \To 1_\C$. We represent the components of these structures with string diagrams as follows.
    
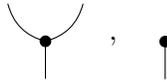
\[
    \begin{tikzpicture}
    \begin{pgfonlayer}{nodelayer}
    	\node [style=blackdot] (r) at (0, 0) {};
        \node [style=none] (1) at (-0.5, 0.5) {};
        \node [style=none] (2) at (0.5, 0.5) {};
        \node [style=none] (3) at (0, -0.5) {};
        \node [style=none] () at (0.9, 0) {,};
        \node [style=none] () at (1.5, 0) {};
    \end{pgfonlayer}
    \begin{pgfonlayer}{edgelayer}
    	\draw [bend right] (1.center) to (r.center);
    	\draw [bend left] (2.center) to (r.center);
    	\draw (r.center) to (3.center);
    \end{pgfonlayer}
    \end{tikzpicture}
    \begin{tikzpicture}
    \begin{pgfonlayer}{nodelayer}
    	\node [style=blackdot] (bu) at (0, 0.5) {};
        \node [style=none] (1) at (0, 0) {};
    \end{pgfonlayer}
    \begin{pgfonlayer}{edgelayer}
    	\draw (bu.center) to (1.center);
    \end{pgfonlayer}
    \end{tikzpicture}
    \]
    The unit object is initial by naturality of $\varepsilon$.
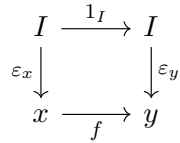
    \[
    \begin{tikzcd}
        I
        \arrow[d, swap, "\varepsilon_x"]
        \arrow[r, "1_I"]
        &
        I
        \arrow[d, "\varepsilon_y"]
        \\
        x
        \arrow[r, swap, "f"]
        &
        y
    \end{tikzcd}
    \]
    
    Now we must show that the monoidal structure on $\C$ is cocartesian, i.e.\ that the unit object is initial and tensor is coproduct. To show that $x \otimes y$ is actually the coproduct of $x$ and $y$, we first must provide inclusions, and then show that this cone satisfies the appropriate universal property. We propose that the inclusion maps $i_x \maps x \to x\otimes y$ and $i_y \maps y \to x \otimes y$ are given in string diagrams as follows.
    
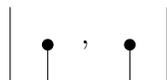
\[
    \begin{tikzpicture}
    \begin{pgfonlayer}{nodelayer}
    	\node [style=none] (bu) at (0, 1) {};
        \node [style=none] (1) at (0, 0) {};
        \node [style=blackdot] (2) at (0.5, 0.5) {};
        \node [style=none] (3) at (0.5, 0) {};
    	\node [style=none] () at (1, 0.5) {,};
    	\node [style=none] () at (1.5, 0.5) {}; 
    \end{pgfonlayer}
    \begin{pgfonlayer}{edgelayer}
    	\draw (bu.center) to (1.center);
    	\draw (2.center) to (3.center);
    \end{pgfonlayer}
    \end{tikzpicture}
    \begin{tikzpicture}
    \begin{pgfonlayer}{nodelayer}
    	\node [style=blackdot] (bu) at (0, 0.5) {};
        \node [style=none] (1) at (0, 0) {};
        \node [style=none] (2) at (0.5, 1) {};
        \node [style=none] (3) at (0.5, 0) {};
    \end{pgfonlayer}
    \begin{pgfonlayer}{edgelayer}
    	\draw (bu.center) to (1.center);
    	\draw (2.center) to (3.center);
    \end{pgfonlayer}
    \end{tikzpicture}
    \]
    Let $q$ be an object of $\C$, and $f \maps x \to q$ and $g \maps y \to q$ be maps in $\C$. Define the map $h \maps x \otimes y \to q$ to be the following composite.
    
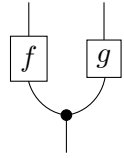
\[
    \begin{tikzpicture}
    \begin{pgfonlayer}{nodelayer}
    	\node [style=none] (x) at (-0.5, 0.75) {};
        \node [style=none] (y) at (0.5, 0.75) {};
        \node [style=construct] (f) at (-0.5, 0) {$f$};
        \node [style=construct] (g) at (0.5, 0) {$g$};
        \node [style=blackdot] (m) at (0, -0.75) {};
        \node [style=none] (q) at (0, -1.25) {};
    \end{pgfonlayer}
    \begin{pgfonlayer}{edgelayer}
    	\draw (x.center) to (f.center);
    	\draw (y.center) to (g.center);
    	\draw [bend right = 50] (f.center) to (m.center);
    	\draw [bend left = 50](g.center) to (m.center);
    	\draw (m.center) to (q.center);
    \end{pgfonlayer}
    \end{tikzpicture}
    \]
    Then we show the diagram
    
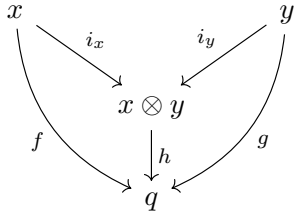
\[
    \begin{tikzcd}
        x
        \arrow[dr, "i_x"]
        \arrow[ddr, swap, "f", bend right]
        &&
        y
        \arrow[dl, swap, "i_y"]
        \arrow[ddl, "g", bend left]
        \\&
        x \otimes y
        \arrow[d, "h"]
        \\&q
    \end{tikzcd} 
    \]
    commutes by the following calculations.
    
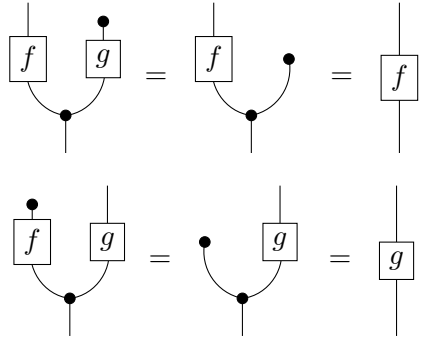
\[
    \begin{tikzpicture}
    \begin{pgfonlayer}{nodelayer}
    	\node [style=none] (x) at (-0.5, 0.75) {};
        \node [style=blackdot] (y) at (0.5, 0.5) {};
        \node [style=construct] (f) at (-0.5, 0) {$f$};
        \node [style=construct] (g) at (0.5, 0) {$g$};
        \node [style=blackdot] (m) at (0, -0.75) {};
        \node [style=none] (q) at (0, -1.25) {};
    	\node [style=none] () at (1.2, -0.25) {=};
    	\node [style=none] () at (1.7, 0) {}; 
    \end{pgfonlayer}
    \begin{pgfonlayer}{edgelayer}
    	\draw (x.center) to (f.center);
    	\draw (y.center) to (g.center);
    	\draw [bend right = 50] (f.center) to (m.center);
    	\draw [bend left = 50] (g.center) to (m.center);
    	\draw (m.center) to (q.center);
    \end{pgfonlayer}
    \end{tikzpicture}
    \begin{tikzpicture}
    \begin{pgfonlayer}{nodelayer}
    	\node [style=none] (x) at (-0.5, 0.75) {};
        \node [style=blackdot] (y) at (0.5, 0) {};
        \node [style=construct] (f) at (-0.5, 0) {$f$};
        \node [style=blackdot] (m) at (0, -0.75) {};
        \node [style=none] (q) at (0, -1.25) {};
    	\node [style=none] () at (1.2, -0.25) {=};
    	\node [style=none] () at (1.7, 0) {}; 
    \end{pgfonlayer}
    \begin{pgfonlayer}{edgelayer}
    	\draw (x.center) to (f.center);
    	\draw [bend right = 50] (f.center) to (m.center);
    	\draw [bend left = 45] (y.center) to (m.center);
    	\draw (m.center) to (q.center);
    \end{pgfonlayer}
    \end{tikzpicture}
    \begin{tikzpicture}
    \begin{pgfonlayer}{nodelayer}
    	\node [style=none] (x) at (0, 0.75) {};
        \node [style=construct] (f) at (0, -0.25) {$f$};
        \node [style=none] (q) at (0, -1.25) {};
    \end{pgfonlayer}
    \begin{pgfonlayer}{edgelayer}
    	\draw (x.center) to (f.center);
    	\draw (f.center) to (q.center);
    \end{pgfonlayer}
    \end{tikzpicture}
    \]
    \[
    \begin{tikzpicture}
    \begin{pgfonlayer}{nodelayer}
    	\node [style=blackdot] (x) at (-0.5, 0.5) {};
        \node [style=none] (y) at (0.5, 0.75) {};
        \node [style=construct] (f) at (-0.5, 0) {$f$};
        \node [style=construct] (g) at (0.5, 0) {$g$};
        \node [style=blackdot] (m) at (0, -0.75) {};
        \node [style=none] (q) at (0, -1.25) {};
    	\node [style=none] () at (1.2, -0.25) {=};
    	\node [style=none] () at (1.7, 0) {}; 
    \end{pgfonlayer}
    \begin{pgfonlayer}{edgelayer}
    	\draw (x.center) to (f.center);
    	\draw (y.center) to (g.center);
    	\draw [bend right = 50] (f.center) to (m.center);
    	\draw [bend left = 50] (g.center) to (m.center);
    	\draw (m.center) to (q.center);
    \end{pgfonlayer}
    \end{tikzpicture}
    \begin{tikzpicture}
    \begin{pgfonlayer}{nodelayer}
    	\node [style=blackdot] (x) at (-0.5, 0) {};
        \node [style=none] (y) at (0.5, 0.75) {};
        \node [style=construct] (g) at (0.5, 0) {$g$};
        \node [style=blackdot] (m) at (0, -0.75) {};
        \node [style=none] (q) at (0, -1.25) {};
    	\node [style=none] () at (1.3, -0.25) {=};
    	\node [style=none] () at (1.8, 0) {}; 
    \end{pgfonlayer}
    \begin{pgfonlayer}{edgelayer}
    	\draw [bend right = 45] (x.center) to (m.center);
    	\draw (y.center) to (g.center);
    	\draw [bend left = 50] (g.center) to (m.center);
    	\draw (m.center) to (q.center);
    \end{pgfonlayer}
    \end{tikzpicture}
    \begin{tikzpicture}
    \begin{pgfonlayer}{nodelayer}
    	\node [style=none] (x) at (0, 0.75) {};
        \node [style=construct] (g) at (0, -0.25) {$g$};
        \node [style=none] (q) at (0, -1.25) {};
    \end{pgfonlayer}
    \begin{pgfonlayer}{edgelayer}
    	\draw (x.center) to (g.center);
    	\draw (g.center) to (q.center);
    \end{pgfonlayer}
    \end{tikzpicture}
    \]
    Let $k \maps x \otimes y \to q$ be a map which makes that diagram commute. Then 
    
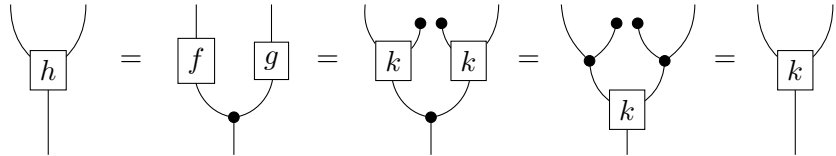
\[
    \begin{tikzpicture}
    \begin{pgfonlayer}{nodelayer}
    	\node [style=none] (x) at (-0.5, 0.75) {};
        \node [style=none] (y) at (0.5, 0.75) {};
        \node [style=construct] (k) at (0, -0.125) {$h$};
        \node [style=none] (q) at (0, -1.25) {};
        \node [style=none] () at (1.1, 0) {=};
        \node [style=none] () at (1.7, 0) {};
    \end{pgfonlayer}
    \begin{pgfonlayer}{edgelayer}
    	\draw [bend right] (x.center) to (k.center);
    	\draw [bend left] (y.center) to (k.center);
    	\draw (k.center) to (q.center);
    \end{pgfonlayer}
    \end{tikzpicture}
    \begin{tikzpicture}
    \begin{pgfonlayer}{nodelayer}
    	\node [style=none] (x) at (-0.5, 0.75) {};
        \node [style=none] (y) at (0.5, 0.75) {};
        \node [style=construct] (f) at (-0.5, 0) {$f$};
        \node [style=construct] (g) at (0.5, 0) {$g$};
        \node [style=blackdot] (m) at (0, -0.75) {};
        \node [style=none] (q) at (0, -1.25) {};
        \node [style=none] () at (1.25, 0) {=};
        \node [style=none] () at (1.7, 0) {};
    \end{pgfonlayer}
    \begin{pgfonlayer}{edgelayer}
    	\draw (x.center) to (f.center);
    	\draw (y.center) to (g.center);
    	\draw [bend right = 50] (f.center) to (m.center);
    	\draw [bend left = 50](g.center) to (m.center);
    	\draw (m.center) to (q.center);
    \end{pgfonlayer}
    \end{tikzpicture}
    \begin{tikzpicture}
    \begin{pgfonlayer}{nodelayer}
    	\node [style=none] (1) at (-0.875, 0.75) {};
        \node [style=blackdot] (2) at (-0.14, 0.5) {};
    	\node [style=blackdot] (3) at (0.14, 0.5) {};
        \node [style=none] (4) at (0.9, 0.75) {};
        \node [style=construct] (k1) at (-0.5, 0) {$k$};
        \node [style=construct] (k2) at (0.5, 0) {$k$};
        \node [style=blackdot] (m) at (0, -0.75) {};
        \node [style=none] (q) at (0, -1.25) {};
        \node [style=none] () at (1.3, 0) {=};
        \node [style=none] () at (1.7, 0) {};
    \end{pgfonlayer}
    \begin{pgfonlayer}{edgelayer}
    	\draw [bend right = 20] (1.center) to (k1);
    	\draw [bend left = 20] (2.center) to (k1);
    	\draw [bend right = 30] (3.center) to (k2.center);
    	\draw [bend left = 27] (4.center) to (k2.center);
    	\draw [bend right = 50] (k1.center) to (m.center);
    	\draw [bend left = 50] (k2.center) to (m.center);
    	\draw (m.center) to (q.center);
    \end{pgfonlayer}
    \end{tikzpicture}
    \begin{tikzpicture}
    \begin{pgfonlayer}{nodelayer}
    	\node [style=none] (1) at (-0.875, 0.75) {};
        \node [style=blackdot] (2) at (-0.14, 0.5) {};
    	\node [style=blackdot] (3) at (0.14, 0.5) {};
        \node [style=none] (4) at (0.9, 0.75) {};
        \node [style=blackdot] (k1) at (-0.5, 0) {};
        \node [style=blackdot] (k2) at (0.5, 0) {};
        \node [style=construct] (m) at (0, -0.65) {$k$};
        \node [style=none] (q) at (0, -1.25) {};
        \node [style=none] () at (1.3, 0) {=};
        \node [style=none] () at (1.7, 0) {};
    \end{pgfonlayer}
    \begin{pgfonlayer}{edgelayer}
    	\draw [bend right = 20] (1.center) to (k1);
    	\draw [bend left = 20] (2.center) to (k1);
    	\draw [bend right = 30] (3.center) to (k2.center);
    	\draw [bend left = 27] (4.center) to (k2.center);
    	\draw [bend right = 35] (k1.center) to (m.center);
    	\draw [bend left = 35] (k2.center) to (m.center);
    	\draw (m.center) to (q.center);
    \end{pgfonlayer}
    \end{tikzpicture}
    \begin{tikzpicture}
    \begin{pgfonlayer}{nodelayer}
    	\node [style=none] (x) at (-0.5, 0.75) {};
        \node [style=none] (y) at (0.5, 0.75) {};
        \node [style=construct] (k) at (0, -0.125) {$k$};
        \node [style=none] (q) at (0, -1.25) {};
    \end{pgfonlayer}
    \begin{pgfonlayer}{edgelayer}
    	\draw [bend right] (x.center) to (k.center);
    	\draw [bend left] (y.center) to (k.center);
    	\draw (k.center) to (q.center);
    \end{pgfonlayer}
    \end{tikzpicture}
    \]
    Thus $h$ is the unique such map. This demonstrates $x \otimes y$ as the coproduct of $x$ and $y$.
\end{proof}

There is a dual statement which characterizes cartesian monoidal categories, but in order to state it, we must first define \emph{comonoid}.

\begin{defn}
    A \define{comonoid object} in a monoidal category $\C$ is monoid in $\C\op$. Equivalently, a comonoid is an object $x \in \C$ equipped with a \define{comultiplication} map $\mu \maps x \to x \otimes x$ and a \define{counit} map $\varepsilon \maps x \to I$, which we express as 
    
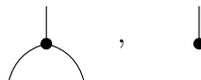
\[
    \begin{tikzpicture}
    \begin{pgfonlayer}{nodelayer}
        \node [style=none] (3) at (0, 0.5) {};
    	\node [style=blackdot] (r) at (0, 0) {};
        \node [style=none] (1) at (-0.5, -0.5) {};
        \node [style=none] (2) at (0.5, -0.5) {};
    	\node [style=none] () at (1, 0) {,};
    	\node [style=none] () at (1.5, 0) {}; 
    \end{pgfonlayer}
    \begin{pgfonlayer}{edgelayer}
    	\draw [bend left] (1.center) to (r.center);
    	\draw [bend right] (2.center) to (r.center);
    	\draw (r.center) to (3.center);
    \end{pgfonlayer}
    \end{tikzpicture}
    \begin{tikzpicture}
    \begin{pgfonlayer}{nodelayer}
    	\node [style=blackdot] (r) at (0, 0) {};
        \node [style=none] (1) at (-0.5, -0.5) {};
        \node [style=none] (2) at (0.5, -0.5) {};
        \node [style=none] (3) at (0, 0.5) {};
    \end{pgfonlayer}
    \begin{pgfonlayer}{edgelayer}
    	\draw (r.center) to (3.center);
    \end{pgfonlayer}
    \end{tikzpicture}
    \]
    satisfying the following equations.
    \[
    \begin{tikzpicture}
    \begin{pgfonlayer}{nodelayer}
    	\node [style=none] (1) at (0, -1) {};
    	\node [style=none] (2) at (0.5, -1) {};
    	\node [style=none] (3) at (1, -1) {};
    	\node [style=none] (4) at (1, -0.5) {};
    	\node [style=none] (5) at (0.625, 0.5) {};
        \node [style=blackdot] (b1) at (0.25, -0.5) {};
        \node [style=blackdot] (b2) at (0.625, 0) {};
    	\node [style=none] () at (1.5, -0.25) {=};
    	\node [style=none] () at (2, 0.5) {}; 
    \end{pgfonlayer}
    \begin{pgfonlayer}{edgelayer}
    	\draw [bend left] (1.center) to (b1.center);
    	\draw [bend right] (2.center) to (b1.center);
    	\draw (3.center) to (4.center);
    	\draw [bend right] (4.center) to (b2.center);
    	\draw [bend left] (b1.center) to (b2.center);
    	\draw (b2.center) to (5.center);
    \end{pgfonlayer}
    \end{tikzpicture}
    \begin{tikzpicture}
    \begin{pgfonlayer}{nodelayer}
    	\node [style=none] (1) at (0, -1) {};
    	\node [style=none] (2) at (0.5, -1) {};
    	\node [style=none] (3) at (1, -1) {};
    	\node [style=none] (4) at (0, -0.5) {};
    	\node [style=none] (5) at (0.375, 0.5) {};
        \node [style=blackdot] (b1) at (0.75, -0.5) {};
        \node [style=blackdot] (b2) at (0.375, 0) {};
    \end{pgfonlayer}
    \begin{pgfonlayer}{edgelayer}
    	\draw (1.center) to (4.center);
    	\draw [bend left] (2.center) to (b1.center);
    	\draw [bend right] (3.center) to (b1.center);
    	\draw [bend left] (4.center) to (b2.center);
    	\draw [bend right] (b1.center) to (b2.center);
    	\draw (b2.center) to (5.center);
    \end{pgfonlayer}
    \end{tikzpicture}
    \]
    \[
    \begin{tikzpicture}
    \begin{pgfonlayer}{nodelayer}
        \node [style=none] (1) at (0.25, -0.75) {};
    	\node [style=blackdot] (2) at (1, -0.5) {};
        \node [style=blackdot] (3) at (0.625, 0) {};
    	\node [style=none] (4) at (0.625, 0.5) {};
    	\node [style=none] () at (1.5, -0.25) {=};
    	\node [style=none] () at (2, 0.5) {}; 
    \end{pgfonlayer}
    \begin{pgfonlayer}{edgelayer}
    	\draw [bend right] (2.center) to (3.center);
    	\draw [bend left] (1.center) to (3.center);
    	\draw (3.center) to (4.center);
    \end{pgfonlayer}
    \end{tikzpicture}
    \begin{tikzpicture}
    \begin{pgfonlayer}{nodelayer}
    	\node [style=none] (1) at (0, 1) {};
        \node [style=none] (2) at (0, 0) {};
        \node [style=none] () at (4, 0) {};
    \end{pgfonlayer}
    \begin{pgfonlayer}{edgelayer}
    	\draw (1.center) to (2.center);
    \end{pgfonlayer}
    \end{tikzpicture}
    \begin{tikzpicture}
    \begin{pgfonlayer}{nodelayer}
        \node [style=blackdot] (1) at (0.25, -0.5) {};
    	\node [style=none] (2) at (1, -0.75) {};
        \node [style=blackdot] (3) at (0.625, 0) {};
    	\node [style=none] (4) at (0.625, 0.5) {};
    	\node [style=none] () at (1.5, -0.25) {=};
    	\node [style=none] () at (2, 0.5) {}; 
    \end{pgfonlayer}
    \begin{pgfonlayer}{edgelayer}
    	\draw [bend right] (2.center) to (3.center);
    	\draw [bend left] (1.center) to (3.center);
    	\draw (3.center) to (4.center);
    \end{pgfonlayer}
    \end{tikzpicture}
    \begin{tikzpicture}
    \begin{pgfonlayer}{nodelayer}
    	\node [style=none] (1) at (0, 1) {};
        \node [style=none] (2) at (0, 0) {};
    \end{pgfonlayer}
    \begin{pgfonlayer}{edgelayer}
    	\draw (1.center) to (2.center);
    \end{pgfonlayer}
    \end{tikzpicture}\]
    Let $\Comon(\C, \otimes)$ denote the category of comonoid objects in $\C$ and their homomorphisms.
\end{defn}

\begin{defn}
    A \define{cocommutative comonoid} is a comonoid for which the following equation holds.
    \[
    \begin{tikzpicture}
    \begin{pgfonlayer}{nodelayer}
        \node [style=none] (3) at (0, 0.5) {};
    	\node [style=blackdot] (r) at (0, 0) {};
        \node [style=none] (1) at (-0.25, -0.5) {};
        \node [style=none] (2) at (0.25, -0.5) {};
    	\node [style=none] () at (1, 0) {=};
    	\node [style=none] () at (1.5, 0) {}; 
    \end{pgfonlayer}
    \begin{pgfonlayer}{edgelayer}
    	\draw [bend left] (1.center) to (r.center);
    	\draw [bend right] (2.center) to (r.center);
    	\draw (r.center) to (3.center);
    \end{pgfonlayer}
    \end{tikzpicture}
    \begin{tikzpicture}
    \begin{pgfonlayer}{nodelayer}
        \node [style=none] (1) at (0, 0.5) {};
    	\node [style=blackdot] (m) at (0, 0) {};
        \node [style=none] (2) at (0, -0.55) {};
        \node [style=none] (3) at (-0.25, -1) {};
        \node [style=none] (4) at (0.25, -1) {};
    \end{pgfonlayer}
    \begin{pgfonlayer}{edgelayer}
    	\draw (1.center) to (m.center);
    	\draw [bend left = 90, looseness = 1.5] (m.center) to (2.center);
    	\draw [bend right = 45] (2.center) to (3.center);
    	\draw [bend right = 90, white, line width = 3, looseness = 1.5] (m.center) to (2.center);
    	\draw [bend left = 45, white, line width = 3] (2.center) to (4.center);
    	\draw [bend right = 90, looseness = 1.5] (m.center) to (2.center);
    	\draw [bend left = 45] (2.center) to (4.center);
    \end{pgfonlayer}
    \end{tikzpicture}
    \]
    Let $\CoComon(\C, \otimes)$ denote the category of cocommutative comonoids in $\C$ and their homomorphisms.
\end{defn}

\begin{prop}
\label{prop:carthascomonoids}
    A symmetric monoidal category is cartesian if and only if each object has a natural cocommutative comonoid structure.
\end{prop}
\begin{proof}
    This is dual to \cref{prop:cocarthasmonoids}.
\end{proof}

\begin{cor}
    Let $(\C, \otimes, I)$ be a symmetric monoidal category. Then $\CMon(\C, \otimes)$ has a cocartesian monoidal structure given by $\otimes$, and $\CoComon(\C, \otimes)$ has a cartesian monoidal structure given by $\otimes$. 
\end{cor}

\setcounter{chapter}{1}
\chapter{Monoidal 2-Categories and Pseudomonoids}
\label{app:Monoidal2cats}

There are many sources for the basic theory of 2-categories and bicategories \cite{Benabou, Review, 2-catcompanion, 2DCats}. Below we sketch some basic definitions and constructions regarding monoidal 2-categories, necessary for what follows; relevant references where explicit axioms can be found are \cite{Carmody, CoherenceTricats, Monoidalbicatshopfalgebroids, Mccruddencoalgebroids}.

\section{Monoidal 2-Categories}

A \define{monoidal 2-category} $\K$ is a 2-category equipped with a pseudofunctor $\otimes \maps \K \times \K \to \K$ and a unit object $I \maps 1 \to \K$ which are associative and unital up to coherent equivalence. A \define{lax monoidal pseudofunctor} $\F \maps \K \to \L$ between monoidal 2-categories is a pseudofunctor equipped with pseudonatural transformations
\begin{equation}\label{eq:weakmonpseudo}
    \begin{tikzcd}
        \K \times \K 
        \arrow[r, "\F \times \F"] 
        \arrow[d, "\otimes_\K"']
        & |[alias=doma]| 
        \L \times \L 
        \arrow[d, "\otimes_\L"] 
        \\
        |[alias=coda]| 
        \K 
        \arrow[r, "\F"'] 
        & \L
        \arrow[Rightarrow, from=doma, to=coda, "\mu"', shorten >=.25in, shorten <=.25in]
    \end{tikzcd}
    \quad\quad
    \begin{tikzcd}
        \1\arrow[d, "I_{\K}"']
        \arrow[dr, bend left, "I_{\L}"{name=doma}] 
        \\
        |[alias=coda]| 
        \K
        \arrow[r, "\F"'] 
        & \L
        \arrow[Rightarrow, from=doma, to=coda, "\mu_0"', shorten >=.1in, shorten <=.1in]
    \end{tikzcd} 
\end{equation}
with components $\mu_{a, b} \maps \F a \otimes \F b \to \F (a \otimes b)$, $\mu_0 \maps I \to \F I$, and invertible modifications 
\begin{equation}\label{eq:omega}
\begin{tikzcd}
    & 
    \L^3 
    \arrow[d, phantom, "\Downarrow{\scriptstyle\mu \times 1}"]
    \arrow[r, "\otimes _\L \times 1"] 
    &
    \L^2
    \arrow[dd, phantom, "\Downarrow{\scriptstyle\mu}"]
    \arrow[dr, "\otimes _\L"]
    &&&
    \L^3
    \arrow[dd, phantom, "\Downarrow{\scriptstyle 1 \times \mu}"]
    \arrow[r, "\otimes _\L \times 1"]
    \arrow[dr, "1 \times \otimes _\L"description]
    & 
    \L^2
    \arrow[d, phantom, "{\scriptstyle\cong}"]
    \arrow[dr, "\otimes _\L"]
    \\
    \K^3
    \arrow[ur, "\F \times \F \times \F"]
    \arrow[r, "\otimes _\K \times 1"]
    \arrow[dr, "1 \times \otimes _\K", swap]
    &
    \K^2
    \arrow[ur, "\F \times \F"description]
    \arrow[dr, "\otimes _\K"description]
    &&
    \L
    \arrow[r, phantom, "\stackrel{\omega}{\Rrightarrow}"]
    &
    \K^3
    \arrow[ur, "\F \times \F \times \F"]
    \arrow[dr, "1 \times \otimes _\K", swap]
    &&
    \L^2
    \arrow[d, phantom, "\Downarrow{\scriptstyle\mu}"]
    \arrow[r, "\otimes _\L"]
    &
    \L
    \\&
    \K^2
    \arrow[u, phantom, "{\scriptstyle\cong}"]
    \arrow[r, "\otimes _\K", swap]
    &
    \K 
    \arrow[ur, "\F", swap]
    &&&
    \K^2
    \arrow[ur, "\F \times \F"description]
    \arrow[r, "\otimes _\K", swap]
    &
    \K
    \arrow[ur, "\F", swap]
\end{tikzcd}
\end{equation}
\begin{displaymath}
\begin{tikzcd}[column sep=.25in]
    \K
    \arrow[dr, "1\times I"']
    \arrow[rr, "\F\times I"{name=doma}]
    \arrow[rrd, bend right=80, "1"', "\cong"]
    \arrow[rrr, bend left=30, "\F"]
    \arrow[rrr, phantom, bend left=15, "{\scriptstyle\cong}"description] 
    && 
    \L\times\L
    \arrow[r, "\otimes _\L"'] 
    & 
    \L
    \arrow[Rightarrow, from=doma, to=coda, "1\times\mu_0\;\;"', shorten <=.5em, shorten >=.5em]
    \\& 
    |[alias=coda]|
    \K\times\K
    \arrow[r, "\otimes _\K"']
    \arrow[ur, "\F\times\F"description]
    \arrow[urr, phantom, "\Downarrow{\scriptstyle \mu}"description] 
    &
    \K
    \arrow[ur, "\F"'] 
    &
\end{tikzcd}
\stackrel{\zeta}{\Rrightarrow}
\begin{tikzcd}[column sep=.25in]
    \\
    \K\arrow[rr, bend left, "\F"]\arrow[rr, bend right, "\F"']\arrow[rr, phantom, "\Downarrow{\scriptstyle1}"description] && \L \\
\end{tikzcd}
\stackrel{\xi}{\Rrightarrow}
\begin{tikzcd}[column sep=.25in]
    \K
    \arrow[dr, "I\times1"']
    \arrow[rr, "I\times\F"{name=doma}]
    \arrow[rrd, bend right=80, "1"', "\cong"]
    \arrow[rrr, bend left=30, "\F"]
    \arrow[rrr, phantom, bend left=15, "{\scriptstyle\cong}"description] 
    &&
    \L \times \L
    \arrow[r, "\otimes _\L"'] 
    & 
    \L
    \arrow[Rightarrow, from=doma, to=coda, "\mu_0\times1\;\;"', shorten <=.5em, shorten >=.5em]
    \\& |[alias=coda]|
    \K\times\K
    \arrow[r, "\otimes _\K"']
    \arrow[ur, "\F\times\F"description]
    \arrow[urr, phantom, "\Downarrow{\scriptstyle \mu}"description] 
    &
    \K
    \arrow[ur, "\F"'] 
\end{tikzcd}
\end{displaymath}
subject to coherence conditions which can be found in Definition 2 in \cite{Monoidalbicatshopfalgebroids}. A \define{monoidal pseudonatural transformation} $\tau \maps \F \Rightarrow \G$ between two lax monoidal pseudofunctors $(\F, \mu, \mu_0)$ and $(\G, \nu, \nu_0)$ is a pseudonatural transformation equipped with two invertible modifications
\begin{equation}\label{eq:monpseudonat}
\begin{tikzcd}[column sep = .2in, row sep = .1in]
   \K \times \K 
   \arrow[rr, bend left, "\F \times \F"] 
   \arrow[rr, bend right, "\G \times \G"'] 
   \arrow[dd, "\otimes "'] 
   \arrow[ddrr, phantom, bend right = 10, "\Downarrow{\scriptstyle \nu}"'] 
   \arrow[rr, phantom, description, "\Downarrow{\scriptstyle \tau \times \tau}"] 
   && 
   \L \times \L
   \arrow[dd, "\otimes "] 
   &&
   \K \times \K 
   \arrow[rr, "\F \times \F"]
   \arrow[dd, "\otimes "']
   \arrow[ddrr, phantom, bend left=10, "\Downarrow{\scriptstyle \mu}"] 
   && 
   \L \times \L 
   \arrow[dd, "\otimes "] 
   \\&&& 
   \stackrel{u}{\Rrightarrow} 
   &&&\\
   \K
   \arrow[rr, "\G"'] 
   && 
   \L 
   &&
   \K
   \arrow[rr, bend left, "\F"]
   \arrow[rr, bend right, "\G"']
   \arrow[rr, phantom, description, "\Downarrow{\scriptstyle\tau}"] 
   && 
   \L
\end{tikzcd}
\end{equation}
\begin{displaymath}
\begin{tikzcd}[row sep = .1in, column sep = .3in]
   \1
    \arrow[ddrr, phantom, "\Downarrow{\scriptstyle \nu_0}" description]
    \arrow[rr, "I_\L"]
    \arrow[dd, "1_\K"'] 
    && 
    \L 
    &&
    \1
    \arrow[ddrr, phantom, near start, "\Downarrow{\scriptstyle \mu_0}" description]
    \arrow[rr, "I_\L"]
    \arrow[dd, "I_\K"']
    && 
    \L 
    \\&&&
    \stackrel{u_0}{\Rrightarrow} 
    &&&\\
   \K
    \arrow[uurr, bend right = 40, 
    "\G"']  
    &&
    \phantom{A}
    && 
     \K 
    \arrow[uurr, "\F" description]
    \arrow[uurr, bend right=60, "\G"']
    \arrow[uurr, phantom, bend right, "\Downarrow{\scriptstyle \tau}" description]
    && 
    \phantom{A}
\end{tikzcd}
\end{displaymath}
that consist of natural isomorphisms with components
\begin{equation}\label{eq:monpseudocomponents}
    u_{a, b}\maps \nu_{a, b} \circ (\tau_a \otimes \tau_b) \xrightarrow{\sim} \tau_{a \otimes b} \circ \mu_{a, b}, 
    \quad
    u_0\maps \nu_0 \xrightarrow{\sim}\tau_I \circ \mu_0 
\end{equation}
satisfying coherence conditions which can be found in \cite[Section 3.3]{CoherenceTricats}. 

The above notions of course generalize those of an ordinary monoidal category, lax monoidal functor and monoidal natural transformation. However, in our higher dimensional setting, there is now room for a structure not present for monoidal 1-categories.

A \define{monoidal modification}
between two monoidal pseudonatural transformations $(\tau, u, u_0)$ and $(\sigma, v, v_0)$ is a modification
\begin{displaymath}
\begin{tikzcd}
    \K
    \arrow[rr, bend left=40, "\F", ""'{name = F}]
    \arrow[rr, bend right=40, "\G"', ""{name = G}] 
    \arrow[rr, phantom, "\stackrel{m}{\Rrightarrow}"description] 
    && 
    \L
    \arrow[from = F, to = G, Rightarrow, "\tau"', bend right=50]
    \arrow[from = F, to = G, Rightarrow, "\sigma", bend left=50]
\end{tikzcd}
\end{displaymath}
which consists of pseudonatural transformations $m_a\maps \tau_a\Rightarrow\sigma_a$ compatible with the monoidal structures, in the sense that
\begin{equation}\label{eq:monoidalmodaxioms}
\begin{tikzcd}[column sep=.27in]
    & 
    \G a \otimes \G b
    \arrow[dr, bend left, "\nu_{a, b}"] 
    &&&
    \G a \otimes \G b 
    \arrow[dr, bend left, "\nu_{a, b}"] 
    \\
    \F a\otimes \F b
    \arrow[ur, bend left, "\sigma_a\otimes \sigma_b"]
    \arrow[dr, bend right, "\mu_{a, b}"']
    \arrow[rr, phantom, near start, "\Downarrow{\scriptstyle v_{a, b}}"description] 
    && 
    \G (x\otimes y) 
    \arrow[r, phantom, "="description] 
    & 
    \F a\otimes \F b
    \arrow[dr, bend right, "\mu_{a, b}"']
    \arrow[ur, bend right, "\sigma_a\otimes \sigma_b"']
    \arrow[rr, phantom, near end, "\Downarrow{\scriptstyle u_{a, b}}"description]
    \arrow[ur, bend left, "\tau_a\otimes\tau_b"]
    \arrow[ur, phantom, "\Downarrow{\scriptstyle m_x\otimes m_y}"description] 
    && 
    \G (a\otimes b) 
    \\
    & 
    \F(a\otimes b)
    \arrow[ur, bend left, "\tau_{a\otimes b}"]
    \arrow[ur, bend right, "\sigma_{a\otimes b}"']
    \arrow[ur, phantom, "\Downarrow{\scriptstyle m_{a\otimes b}}"description] 
    &&& 
    \F(a\otimes b)
    \arrow[ur, bend right, "\tau_{a\otimes b}"']
\end{tikzcd}
\end{equation}
\begin{displaymath}
\begin{tikzcd}[row sep=.2in]
    I
    \arrow[dr, bend right, "\mu_0"']
    \arrow[rr, bend left, "\nu_0"]
    \arrow[rr, phantom, near start, "\Downarrow{\scriptstyle v_0}"description] 
    && 
    \G(I)
    \arrow[r, phantom, "="description] 
    & 
    I
    \arrow[rr, phantom, "\Downarrow{\scriptstyle u_0}"description]
    \arrow[dr, bend right, "\mu_0"']
    \arrow[rr, bend left, "\nu_0"]
    && 
    \G(I) 
    \\& 
    \F(I)
    \arrow[ur, bend left, "\tau_I"]
    \arrow[ur, bend right, "\sigma_I"']
    \arrow[ur, phantom, "\Downarrow{\scriptstyle m_I}"description] 
    &&& 
    \F(I)
    \arrow[ur, bend right, "\tau_I"']
\end{tikzcd}
\end{displaymath}

For any monoidal 2-categories $\K, \L$ there are 2-categories $\MonTCat_\pse(\K, \L)$ denoted by $\namedcat{WMonHom}(\K, \L)$ in \cite{Monoidalbicatshopfalgebroids} for bicategories. If we take lax monoidal 2-functors and monoidal 2-transformations, the corresponding sub-2-category is denoted by $\MonTCat(\K, \L)$. 

\section{Pseudomonoids}

A \define{pseudomonoid} in a monoidal 2-category $(\K, \otimes , I)$ is an object $a$ equipped with multiplication $m \maps a\otimes a\to a$, unit $j \maps I \to a$, and invertible 2-cells
\begin{equation}
\label{alphalambdarho}
\begin{tikzcd}
    a \otimes a \otimes a
    \arrow[r, "1 \otimes m"]
    \arrow[d, "m \otimes 1"']
    \arrow[dr, phantom, "\scriptstyle \stackrel{\assoc}{\cong}"description]
    & 
    a \otimes a \arrow[d, "m"] \arrow[d, "m"]
    & a \otimes I \arrow[r, "1 \otimes j"] \arrow[dr, "\sim"'] 
    & a \otimes a \arrow[d, "m", "\stackrel{\lambda}{\cong}\quad"'near start, "\quad\;\stackrel{\rho}{\cong}"near start] 
    & I \otimes a \arrow[l, "j \otimes 1"']
    \arrow[dl, "\sim"]  \\
    a \otimes a\arrow[r, "m"'] 
    & a && a &
\end{tikzcd}
\end{equation} 
expressing associativity and unitality up to isomorphism, that satisfy appropriate coherence conditions. A \define{lax morphism} between pseudomonoids $a, b$ is a 1-cell $f\maps a\to b$ equipped with 2-cells
\begin{equation}\label{eq:laxmorphism}
\begin{tikzcd}[row sep=.5in, column sep=.5in]
    a \otimes a\arrow[d, "m"']\arrow[r, "f\otimes f"] & |[alias=doma]| b\otimes b\arrow[d, "m"] \\
    |[alias=coda]| a\arrow[r, "f"'] & b
    \arrow[Rightarrow, from=doma, to=coda, "\phi"', shorten >=.35in, shorten <=.35in]
\end{tikzcd}
\qquad
\begin{tikzcd}[row sep=.5in, column sep=.5in]
    I\arrow[d, "j"']\arrow[dr, bend left, "j"{name=doma}] \\
    |[alias=coda]| a\arrow[r, "f"'] & b
    \arrow[Rightarrow, from=doma, to=coda, "\phi_0"', shorten >=.15in, shorten <=.15in]
\end{tikzcd}
\end{equation}
such that the following conditions hold:
\begin{equation}\label{eq:axiomslaxmorphism}
\adjustbox{scale=.9, center}{
    \begin{tikzcd}[column sep=.3in]
        & b\otimes b\otimes b\arrow[r, "m\otimes1"]\arrow[d, phantom, "\Downarrow{\scriptstyle\phi\otimes 1_f}"description] & b\otimes b\otimes a\arrow[dr, "m"] & \\
        a\otimes a\otimes a\arrow[r, "m\otimes 1"]\arrow[dr, "1\otimes m"']\arrow[ur, "f\otimes f\otimes f"] &
        a\otimes a\arrow[dr, "m"]\arrow[ur, "f\otimes f"']\arrow[rr, phantom, "\Downarrow{\scriptstyle\phi}"description]
        \arrow[d, phantom, "\stackrel{\assoc}{\cong}"description] && b \\
        & a\otimes a\arrow[r, "m"'] & a\arrow[ur, "f"']
    \end{tikzcd} 
    =
    \begin{tikzcd}[column sep=.3in]
        & b\otimes b\otimes b\arrow[r, "m\otimes1"]\arrow[dr, "1\otimes m"']\arrow[dd, phantom, "\Downarrow{\scriptstyle1_f\otimes \phi}"description] &
        b\otimes b\arrow[dr, "m"]\arrow[d, phantom, "\stackrel{\alpha}{\cong}"description] & \\
        a\otimes a\otimes a\arrow[ur, "f\otimes f\otimes f"]\arrow[dr, "1\otimes m"'] & & b\otimes b\arrow[r, "m"]
        \arrow[d, phantom, "\Downarrow{\scriptstyle\phi}"description] & b \\
        & a\otimes a\arrow[ur, "f\otimes f"]\arrow[r, "m"'] & b\otimes b\arrow[ur, "f"']
    \end{tikzcd}
}
\end{equation}

\begin{displaymath}
\adjustbox{scale=.9, center}{
    \begin{tikzcd}[column sep=.25in]
        a\cong a\otimes I\arrow[dr, "1\otimes j"']\arrow[rr, "f\otimes j"{name=doma}]\arrow[rrd, bend right=60, "1_a"', "\lambda\cong"]
        \arrow[rrr, bend left=30, "f"]\arrow[rrr, phantom, bend left=15, "{\scriptstyle 1_f\otimes\lambda\cong}"description] && b\otimes b\arrow[r, "m"'] & b
        \arrow[Rightarrow, from=doma, to=coda, "1_f\otimes \phi_0\;\;"', shorten <=.5em, shorten >=.5em]\\
        & |[alias=coda]|a\otimes a\arrow[r, "m"']\arrow[ur, "f\otimes f"description]\arrow[urr, phantom, "\Downarrow{\scriptstyle \phi}"description] &
        a\arrow[ur, "f"'] &
    \end{tikzcd}
    =
    \begin{tikzcd}[column sep=.25in]
        \\
        a\arrow[rr, bend left, "f"]\arrow[rr, bend right, "f"']\arrow[rr, phantom, "\Downarrow{\scriptstyle1_f}"description] && b \\
    \end{tikzcd}
    =
    \begin{tikzcd}[column sep=.25in]
        a\cong I\otimes a\arrow[dr, "j\otimes1"']\arrow[rr, "j\otimes 1"{name=doma}]\arrow[rrd, bend right=60, "1_a"', "\rho\cong"]
        \arrow[rrr, bend left=30, "f"]\arrow[rrr, phantom, bend left=15, "{\scriptstyle \rho\otimes1_f\cong}"description] && b\otimes b\arrow[r, "m"'] & b
        \arrow[Rightarrow, from=doma, to=coda, "\phi_0\otimes 1_f\;\;"', shorten <=.5em, shorten >=.5em]\\
        & |[alias=coda]|a\otimes a\arrow[r, "m"']\arrow[ur, "f\otimes f"description]\arrow[urr, phantom, "\Downarrow{\scriptstyle \phi}"description] &
    a\arrow[ur, "f"'] &
    \end{tikzcd}
    }
\end{displaymath}

If $(f, \phi, \phi_0)$ and $(g, \psi, \psi_0)$ are two lax morphisms between pseudomonoids $a$ and $b$, a \define{2-cell} between them $\sigma \maps f \Rightarrow g$ in $\K$ which is compatible with multiplications and units, in the sense that
\begin{equation}\label{monoidal2cell}
    \begin{tikzcd}[row sep=.15in]
        & b\otimes b\arrow[dr, bend left=20, "m"] & \\
        a\otimes a\arrow[ur, bend left, "f\otimes f"]\arrow[ur, bend right, "g\otimes g"']\arrow[dr, bend right=20, "m"']
        \arrow[ur, phantom, "\Downarrow{\scriptstyle\sigma\otimes\sigma}"description]
        \arrow[rr, phantom, "{\scriptstyle\psi}\Downarrow"{description, near end}] && b \\
        & a\arrow[ur, bend right=20, "g"'] &
    \end{tikzcd}
    \quad=\quad
    \begin{tikzcd}[row sep=.15in]
        & b\otimes b\arrow[rd, bend left=20, "m"] & \\
        a\otimes a\arrow[ur, bend left=20, "f\otimes f"]\arrow[dr, bend right=20, "m"']
        \arrow[rr, phantom, "{\scriptstyle\phi}\Downarrow"{description, near start}] && b \\
        & a \arrow[ur, phantom, "\Downarrow{\scriptstyle\sigma}"description]
        \arrow[ur, bend left, "f"]\arrow[ur, bend right, "g"'] \\
    \end{tikzcd}
\end{equation}
\begin{displaymath}
        \begin{tikzcd}[row sep=.15in, column sep=.6in]
            I\arrow[rr, bend left, "j"]\arrow[dr, bend right=20, "j"']\arrow[rr, phantom, "{\scriptstyle\phi_0}\Downarrow"{description, near start}] && b \\
            & a\arrow[ur, bend left, "f"]\arrow[ur, bend right, "g"']\arrow[ur, phantom, "\Downarrow{\scriptstyle \sigma}"description] &
        \end{tikzcd}
        \quad=\quad
        \begin{tikzcd}[row sep=.15in, column sep=.6in]
            I\arrow[dr, bend right=20, "j"']\arrow[rr, bend left, "j"]\arrow[rr, phantom, "\Downarrow{\scriptstyle\psi_0}"description] && b \\
            & a\arrow[ur, bend right=20, "g"'] &
        \end{tikzcd}
    \end{displaymath} 

We obtain a 2-category $\PsMon_\lax(\K)$ for any monoidal 2-category $\K$, which is sometimes
denoted by $\Mon(\K)$ \cite{Hopfmonoidalcomonads}. By changing the direction of the 2-cells in \cref{eq:laxmorphism} and the rest of the axioms appropriately, or asking for them to be invertible, we have 2-categories $\PsMon_\opl(\K)$ and $\PsMon(\K)$ of pseudomonoids with \define{oplax} or \define{(strong) morphisms} between them.

\begin{expl}
    The prototypical example is that of the monoidal 2-category $\K = (\Cat, \times, \1)$ of categories, functors, and natural transformations with the cartesian product of categories and the unit category with a unique object and arrow. A pseudomonoid in $(\Cat, \times, \1)$ is a monoidal category, a lax (resp. oplax, strong) morphism between two of these is precisely a lax (resp. oplax, strong) monoidal functor, and a 2-cell is a monoidal natural transformation. Therefore we obtain the well-known 2-categories $\Mon\Cat_\lax$, $\Mon\Cat_\opl$ and $\Mon\Cat$.
\end{expl}

There is an evident similarity between the structures defined above, e.g.\ \cref{eq:weakmonpseudo} and \cref{eq:laxmorphism}, or \cref{eq:monpseudonat} and \cref{monoidal2cell}. This is due to the fact that monoidal 2-categories, lax monoidal pseudofunctors and monoidal pseudonatural transformations are themselves appropriate pseudomonoid-related notions in a higher level; we do not get into such details, as they are not pertinent to the present work.

For our purposes, we are interested in a different observation: any pseudomonoid $a$ in a monoidal 2-category $\K$ can in fact be expressed as a lax monoidal \define{normal} pseudofunctor $A \maps \1 \to \K$ with $A(*) = a$, namely one where $A(1_*)$ is \emph{equal} to $1_a$. Moreover, a monoidal pseudonatural transformation $\tau\maps A\Rightarrow B\maps \1\to\K$ bijectively corresponds to a \emph{strong morphism} between the pseudomonoids $a$ and $b$, and similarly for monoidal modifications and 2-cells. Since every pseudofunctor is equivalent to a normal one, the 2-category of pseudomonoids $\PsMon(\K)$ can be equivalently viewed as $\MonTCat_{\pse}(\1, \K)$, the 2-category of lax monoidal pseudofunctors $\1\to\K$, monoidal pseudonatural transformations and monoidal modifications.

As was already shown in~\cite[Prop.~5]{Monoidalbicatshopfalgebroids}, any lax monoidal 2-functor $\F\maps \K\to\L$ takes pseudomonoids to pseudomonoids, and in fact
\cite{Mccruddencoalgebroids} there is a functor $\PsMon(\F)$ that commutes with the respective forgetful functors
\begin{displaymath}
    \begin{tikzcd}[column sep=.6in]
        \PsMon(\K)\arrow[r, "\PsMon(\F)"]\arrow[d] & \PsMon(\L)\arrow[d] \\
        \K\arrow[r, "\F"'] & \L.
    \end{tikzcd}
\end{displaymath}
Based on the above, and since every pseudofunctor from $\1$ into a 2-category trivially preserves composition on the nose and every pseudonatural transformation is really 2-natural, we can define a hom-2-functor that clarifies these assignments.

\begin{prop}
    There is a 2-functor
    \begin{equation}
    \label{eq:PsMon}
        \PsMon(-) \simeq \MonTCat_{\pse}(\1, -) \maps \MonTCat \to \TCat
    \end{equation}
    which maps a monoidal 2-category to its 2-category of pseudomonoids, strong morphisms and 2-cells between them.
\end{prop}

The theory in \cite{Monoidalbicatshopfalgebroids, Mccruddencoalgebroids} extends the above definitions to the case of \emph{braided} and \emph{symmetric} pseudomonoids in \emph{braided} and \emph{symmetric monoidal 2-categories}. Briefly recall that a \define{braiding} for $(\K, \otimes, I)$ is a pseudonatural equivalence with components $\braid_{a, b} \maps a \otimes b \to b \otimes a$ and invertible modifications, whereas a \define{syllepsis} is an invertible modification 
\[
    a \otimes b \xrightarrow{1} a \otimes b \Rrightarrow a \otimes b \xrightarrow{\braid_{a, b}} b \otimes a \xrightarrow{\braid_{b, a}} a \otimes b
\]
which is called \define{symmetry} if it satisfies extra axioms. With the appropriate notions of \emph{braided} and \emph{symmetric} lax monoidal pseudofunctors and monoidal pseudonatural transformations (and usual monoidal modifications), we have 3-categories $\BrMonTCat_{\pse}$ and $\SymMonTCat_{\pse}$. Indicatively, a lax monoidal pseudofunctor comes equipped an invertible modification with components
\begin{equation}\label{eq:brweakmonpseudo}
\begin{tikzcd}
    \F a \otimes \F b 
    \arrow[r, "\mu_{a, b}"]
    \arrow[d, "\braid_{\F a, \F b}"']
    \arrow[dr, phantom, "\Downarrow{\scriptstyle v_{a, b}}" description] 
    & 
    \F b \otimes \F a 
    \arrow[d, "\F(\braid_{a, b})"] 
    \\
    \F b \times \F a 
    \arrow[r, "\mu_{b, a}"'] 
    & 
    \F(b\otimes a)
\end{tikzcd}
\end{equation}
As earlier, there exist 2-categories of braided and symmetric pseudomonoids with strong morphisms between them, expressed as \[
    \BrPsMon(\K)=\BrMonTCat_{(\pse)}(\1, \K)
\] 
and 
\[
    \SymPsMon(\K)=\SymMonTCat_{(\pse)}(\1, \K).
\]

\begin{prop}
\label{prop:BrPsMon}
    There are 2-functors
    \begin{displaymath}
        \BrPsMon\maps \BrMonTCat\to \TCat, \quad
        \SymPsMon\maps \SymMonTCat\to\TCat 
    \end{displaymath}
    which map a braided or symmetric monoidal 2-category to its 2-category of braided or symmetric pseudomonoids.
\end{prop}

Finally, recall the notion of a monoidal 2-equivalence arising as the equivalence internal to the 2-category $\MonTCat$.

\begin{defn}
    A \define{monoidal 2-equivalence} is a 2-equivalence $\F\maps \K \simeq \L \maps \G$ where both 2-functors
    are lax monoidal, and the 2-natural isomorphisms $1_\K \cong \F \G$, $\G \F \cong 1_\L$ are monoidal. Similarly for \define{braided} and \define{symmetric} monoidal 2-equivalences.
\end{defn}

As is the case for any 2-functor between 2-categories, $\PsMon$ as well as $\BrPsMon$ and $\SymPsMon$ map equivalences
to equivalences.

\begin{prop}\label{prop:2equivpseudomon}
    Any monoidal 2-equivalence $\K\simeq\L$ induces a 2-equivalence between the respective 2-categories
    of pseudomonoids $\PsMon(\K)\simeq\PsMon(\L)$. Similarly any braided or symmetric monoidal 2-equivalence induces $\BrPsMon(\K)\simeq\BrPsMon(\L)$ or $\SymPsMon(\K)\simeq\SymPsMon(\L)$.
\end{prop}
\setcounter{chapter}{2}
\chapter{Fibrations and Indexed Categories}\label{app:fibicat}

We recall some basic facts and constructions from the theory of fibrations and indexed categories, as well as the equivalence between them via the Grothendieck construction. Several indicative references for the general theory are \cite{Grayfibredandcofibred, BenabouFibered, FibredAdjunctions, Handbook2, Jacobs, Elephant1, Vistoli, alaBenabou, 2DCats}.

\section{Fibrations}
\label{sec:fibrations}

Consider a functor $P \maps \A \to \X$. A morphism $\phi \maps a \to b$ in $\A$ over a morphism $f = P(\phi) \maps x \to y$ in $\X$ is called \define{cartesian} if and only if, for all $g \maps x' \to x$ in $\X$ and $\theta \maps a' \to b$ in $\A$ with $P \theta = f \circ g$, there exists a unique arrow $\psi \maps a' \to a$ such that $P \psi = g$ and $\theta = \phi \circ \psi$:
\begin{equation}
\begin{tikzcd}[column sep = huge]
    a'
    \arrow[drr, "\theta"]
    \arrow[dr, dashed, swap, "\exists!\psi"]
    \arrow[dd, dotted, bend right]
    \\&
    a
    \arrow[r, swap, "\phi"]
    \arrow[dd, dotted, bend right]
    &
    b
    \arrow[dd, dotted, bend right]
    &
    \text{in }\A
    \\
    x'
    \arrow[drr, "f \circ g = P \theta"]
    \arrow[dr, swap, "g"]
    \\&
    x
    \arrow[r, swap, "f = P \phi"]
    &
    y
    &
    \text{in }\X
\end{tikzcd}
\end{equation}
For $x \in \Ob\X$, the \define{fibre of $P$ over $x$} written $\A_x$, is the subcategory of $\A$ which consists of objects $a$ such that $P(a) = x$ and morphisms $\phi$ with $P(\phi) = 1_x$, called \define{vertical} morphisms. The functor $P \maps \A \to \X$ is called a \define{fibration} if and only if, for all $f \maps x \to y$ in $\X$ and $b\in\A_Y$, there is a cartesian morphism $\phi$ with codomain $b$ above $f$; it is called a \define{cartesian lifting} of $f$ to $b$. The category $\X$ is then called the \define{base} of the fibration, and $\A $ its \define{total category}.

Dually, the functor $U \maps \C \to \X$ is an \define{opfibration} if $U^\mathrm{op}$ is a fibration, i.e.\ for every $c \in \C _x$ and $h \maps x \to y$ in $\X$, there is a cocartesian morphism with domain $c$ above $h$, the \define{cocartesian lifting} of $h$ to $c$ with the dual universal property:
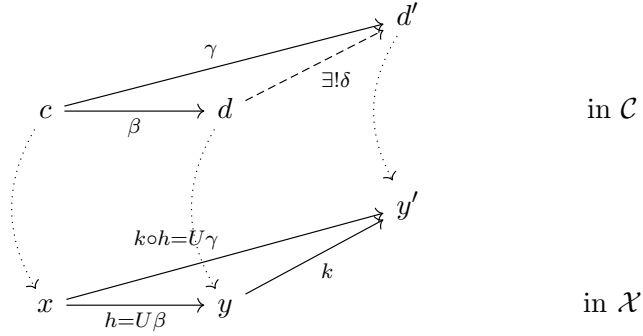
\begin{displaymath}
\begin{tikzcd}[column sep = huge]
    &&
    d'
    \arrow[dd, dotted, bend right]
    \\
    c
    \arrow[urr, "\gamma"]
    \arrow[r, swap, "\beta"]
    \arrow[dd, dotted, bend right]
    &
    d
    \arrow[ur, dashed, swap, "\exists! \delta"]
    \arrow[dd, dotted, bend right]
    &&
    \text{in }\C
    \\&&
    y'
    \\
    x
    \arrow[urr, "k \circ h = U \gamma"]
    \arrow[r, swap, "h = U \beta"]
    &
    y
    \arrow[ur, swap, "k"]
    &&
    \text{in }\X
\end{tikzcd}
\end{displaymath}
A \define{bifibration} is a functor which is both a fibration and opfibration.

If $P\maps \A \to \X$ is a fibration, assuming the axiom of choice we may select a cartesian arrow over each $f\maps x \to y$ in $\X$ and $b \in \A_y$, denoted by $\Cart(f, b) \maps f^*(b) \to b$. Such a choice of cartesian liftings is called a \define{cleavage} for $P$, which is then called a \define{cloven fibration}; any fibration is henceforth assumed to be cloven. Dually, if $U$ is an opfibration, for any $c \in \C_x$ and $h \maps x \to y$ in $\X$ we can choose a cocartesian lifting of $h$ to $c$, $\Cocart(h,c)\maps c \to h_!(c)$. The choice of (co)cartesian liftings in an (op)fibration induces a so-called \define{reindexing functor} between the fibre categories
\begin{equation}\label{reindexing}
    f^*\maps \A _y \to \A _x\quad\textrm{ and }\quad h_! \maps \C _x \to \C _y
\end{equation}
respectively, for each morphism $f\maps x \to y$ and $h \maps x \to y$ in the base category.
It can be verified by the (co)cartesian universal property that $1_{\A _x}\cong(1_x)^*$ 
and that for composable morphism in the base category, $g^*\circ f^*\cong(g\circ f)^* $, as well as $(1_x)_!\cong1_{\C_x}$ and $(k\circ h)_!\cong k_!\circ h_!$. If these isomorphisms are equalities, we have the notion of a \define{split} (op)fibration.

A \define{fibred 1-cell} $(H,F) \maps P \to Q$ between fibrations $P \maps \A \to \X$ and $Q \maps \B \to \Y$ is given by a commutative square of functors and categories
\begin{equation}
\label{commutativefibredcell}
\begin{tikzcd}[row sep = huge]
    \A
    \arrow[rr, "H"]
    \arrow[d, swap, "P"]
    &&
    \B
    \arrow[d, "Q"]
    \\
    \X
    \arrow[rr , "F", swap]
    &&
    \Y
\end{tikzcd}
\end{equation}
where the top $H$ preserves cartesian liftings, meaning that if $\phi$ is $P$-cartesian, then $H\phi$ is $Q$-cartesian. In particular, when $P$ and $Q$ are fibrations over the same base category, we may consider fibred 1-cells of the form $(H,1_{\X})$ displayed as
\begin{equation}
\label{eq:fibredfunctor}
\begin{tikzcd}[row sep = huge]
    \A
    \arrow[rr, "H"]
    \arrow[dr, swap, "P"]
    &&
    \B
    \arrow[dl, "Q"]
    \\&
    \X
\end{tikzcd}
\end{equation}
and $H$ is then called a \define{fibred functor}. Dually, we have the notion of an \define{opfibred 1-cell} and \define{opfibred functor}. Notice that any such (op)fibred 1-cell induces functors between the fibres, by commutativity of \cref{commutativefibredcell}:
\begin{equation}\label{eq:functorbetweenfibres}
    H_{x} \maps \A_x\longrightarrow\B_{Fx}
\end{equation}

A \define{fibred 2-cell} between fibred 1-cells $(H,F)$ and $(K,G)$ is a pair of natural transformations ($\braid\maps H\Rightarrow K,\alpha\maps F\Rightarrow G$) with $\braid$ above $\alpha$, i.e.\ $Q(\braid_a)=\alpha_{Pa}$ for all $a\in\A $, displayed as
\begin{equation}
\label{eq:fibred2cell}
\begin{tikzcd}[row sep = huge]
    \A
    \arrow[rr, bend left, "H"]
    \arrow[rr, phantom, "\Downarrow \beta"]
    \arrow[rr, bend right, swap, "K"]
    \arrow[d, swap, "P"]
    &&
    \B
    \arrow[d, "Q"]
    \\
    \X
    \arrow[rr, bend left, "F"]
    \arrow[rr, phantom, "\Downarrow \alpha"]
    \arrow[rr, bend right, swap, "G"]
    &&
    \Y
\end{tikzcd}
\end{equation}
A \define{fibred natural transformation} is of the form $(\braid,1_{1_{\X}})\maps(H,1_{\X})\Rightarrow(K,1_\X)$
\begin{equation}\label{eq:fibrednaturaltrans}
\begin{tikzcd}[row sep = huge]
    \A
    \arrow[rr, bend left, "H"]
    \arrow[rr, phantom, "\Downarrow \beta"]
    \arrow[rr, bend right, swap, "K"]
    \arrow[dr, swap, "P"]
    &&
    \B
    \arrow[dl, "Q"]
    \\&
    \X
\end{tikzcd}
\end{equation}
Dually, we have the notion of an \define{opfibred 2-cell} and \define{opfibred natural transformation} between opfibred 1-cells and functors respectively.

We thus obtain a 2-category $\Fib$ of fibrations over arbitrary base categories, fibred 1-cells and fibred 2-cells. There is also a 2-category $\Fib(\X)$ of fibrations over a fixed base category $\X$, fibred functors and fibred natural transformations. Dually, we have the 2-categories $\OpFib$ and $\OpFib(\X)$. Moreover, we also have 2-categories $\Fib_\mathrm{sp}$ and $\OpFib_\mathrm{sp}$ of split (op)fibrations, and (op)fibred 1-cells that preserve the cartesian liftings `on the nose'.

Notice that $\Fib$ and $\OpFib$ are both sub-2-categories of $\Cat^\2 = [\2, \Cat]$, the arrow 2-category of $\Cat$. Similarly, $\Fib(\X)$ and $\OpFib(\X)$ are sub-2-categories of $\Cat/\X$, the slice 2-category of functors into $\X$. Due to that, both these (1-)categories form fibrations themselves. Explicitly, the functor $\cod \maps \Fib \to \Cat$ which maps a fibration to its base is a fibration, with fibres $\Fib(\X)$ and cartesian liftings pullbacks along fibrations. In fact, it is a \emph{2-fibration} \cite{HermidaFib, 2Fibs}.

\section{Indexed Categories}
\label{sec:indexedcats}

We now turn to the world of indexed categories. Given an ordinary category $\X$, an $\X$-\define{indexed category} is a pseudofunctor \[\M \maps \X\op \to \Cat\] where $\X$ is viewed as a 2-category with trivial 2-cells; it comes with natural isomorphisms $\delta_{g,f} \maps (\M g) \circ (\M f) \xrightarrow\sim  \M(g \circ f)$ and $\gamma_x \maps 1_{\M x} \xrightarrow\sim \M (1_x)$ for every $x\in\X$ and composable morphisms $f$ and $g$, satisfying coherence axioms.
Dually, an $\X$-\define{opindexed category} is an $\X\op$-indexed category, i.e.\ a pseudofunctor $\X \to \Cat$.
If an (op)indexed category strictly preserves composition, i.e.\ is a (2-)functor, then it is called \define{strict}.

An \define{indexed $1$-cell} $(F, \tau) \maps \M \to \psN$ between indexed categories $\M \maps \X\op \to \Cat$ and $\psN \maps \Y\op \to \Cat$ consists of an ordinary functor $F \maps \X \to \Y$ along with a pseudonatural transformation $\tau \maps \M \Rightarrow \psN \circ F\op$
\begin{equation}\label{eq:indexed1cell}
\begin{tikzcd}[column sep=.7in,row sep=.2in]
    \X\op
    \arrow[dr, "\M"]
    \arrow[dd, "F\op"'] 
    \\ 
    \arrow[r, phantom, "\Downarrow{\scriptstyle\tau}"description] 
    & 
    \Cat 
    \\
    \Y\op \arrow[ur, "\psN"']
\end{tikzcd}
\end{equation}
with components functors $\tau_x \maps \M x \to \psN Fx$, equipped with coherent natural isomorphisms $\tau_f \maps (\psN Ff) \circ \tau_x \xrightarrow{\sim} \tau_y \circ (\M f)$ for any $f \maps x \to y$ in $\X$.
For indexed categories with the same base, we may consider indexed 1-cells of the form $(1_\X, \tau)$
\begin{equation}\label{eq:ifun}
\begin{tikzcd}[column sep=.7in]
    \X\op 
    \arrow[r, bend left, "\M"]
    \arrow[r, bend right, swap, "\psN"]
    \arrow[r, phantom, "\Downarrow \scriptstyle \tau"]
    &
    \Cat
\end{tikzcd}
\end{equation}
which are called \define{indexed functors}. Dually, we have the notion of an \define{opindexed 1-cell} and \define{opindexed functor}.

An \define{indexed 2-cell} $(\alpha,m)$ between indexed 1-cells $(F, \tau)$ and $(G, \sigma)$, pictured as 
\begin{displaymath}
 \begin{tikzcd}[column sep=.6in,row sep=.4in]
    \X\op
    \arrow[drr, "\M"]
    \arrow[drr, ""{name = M}, swap, pos = 0.5]
     \arrow[dd,bend right=40,"F\op" description]
     \arrow[dd, "G\op"description, bend left=40]
    \arrow[dd,phantom,"{\scriptstyle\stackrel{\alpha\op}{\Leftarrow}}"description] &&
    \\
& & \Cat
    \\
    \Y\op
    \arrow[urr,"\psN",swap]
    \arrow[urr,""{name = H}, pos = 0.4]
    \arrow[from = M, to = H, Rightarrow, "\sigma", pos = 0.4, bend left=45]
    \arrow[from = M, to = H, Rightarrow, "\tau", pos = 0.56, bend right=45,swap]
    \arrow[from = M, to = H, phantom, "\scriptscriptstyle\stackrel{m}{\Rrightarrow}"] &&
 \end{tikzcd} 
\end{displaymath}
consists of an ordinary natural transformation $\alpha \maps F \Rightarrow G$ and a modification $m$
\begin{equation}\label{eq:indexed2cell}
\begin{tikzcd}[column sep=.5in,row sep=.2in]
    \X\op
    \arrow[rr,bend left,"\M"]
    \arrow[dr,bend right=10,"F\op"']
    \arrow[rr,phantom,"\Downarrow{\scriptstyle \tau}"description]  
    && 
    \Cat \arrow[r,phantom,"\stackrel{m}{\Rrightarrow}"description] 
    & 
    \X\op
    \arrow[rr,bend left=30,"\M"]
    \arrow[dr,bend left=35,"G\op"]
    \arrow[dr,bend right=35,"F\op"']
    \arrow[rr, phantom, near end, "\Downarrow{\scriptstyle \sigma}"description]
    \arrow[dr,phantom, "\Downarrow{\scriptstyle \alpha\op}"description]
    && 
    \Cat 
    \\
    & 
    \Y\op
    \arrow[ur,bend right=10,"\psN"'] 
    &&& 
    \Y\op
    \arrow[ur,bend right,"\psN"'] &
\end{tikzcd}
\end{equation}
given by a family of natural transformations $m_x \maps \tau_x \Rightarrow\psN \alpha_x \circ\sigma_x$. Notice that taking opposites is a 2-functor $(-)\op \maps \Cat \to \Cat^{co}$, on which the above diagrams rely.
An \define{indexed natural transformation} between two indexed functors is an indexed 2-cell of the form $(1_{1_\X},m)$.
Dually, we have the notion of an \define{opindexed 2-cell} and \define{opindexed natural transformation} between opindexed 1-cells and functors respectively.

Notice that an indexed 2-cell $(\alpha,m)$ is invertible if and only if both $\alpha$ is a natural isomorphism and the modification $m$ is invertible, due to the way vertical composition is formed. 

We obtain a 2-category $\ICat$ of indexed categories over arbitrary bases, indexed 1-cells and indexed 2-cells. In particular, there is a 2-category $\ICat(\X)$ of indexed categories with fixed domain $\X$, indexed functors and indexed natural transformations, which coincides with the functor 2-category $\TCat_\pse(\X\op,\Cat)$.

Dually, we have the 2-categories $\OpICat$ and $\OpICat(\X)=\TCat_\pse(\X,\Cat)$. Notice that due to the absence of opposites in the world of opindexed categories, opindexed 2-cells have a different form than \cref{eq:indexed2cell}, namely
\begin{displaymath}
\begin{tikzcd}[column sep=.5in, row sep=.15in]
    \X
    \arrow[rr, bend left=30, "\M"]
    \arrow[dr, bend left=30, "F"]
    \arrow[dr, bend right=30, "G"']
    \arrow[rr, phantom, near end, "\Downarrow{\scriptstyle\tau}"description]
    \arrow[dr, phantom, "\Downarrow{\scriptstyle\alpha}"description]  && 
    \Cat 
    \arrow[r, phantom, "\stackrel{m}{\Rrightarrow}"description] 
    & 
    \X
    \arrow[rr, bend left, "\M"]
    \arrow[dr, bend right=10, "G"']
    \arrow[rr, phantom, "\Downarrow{\scriptstyle \sigma}"description]
    && 
    \Cat 
    \\& 
    \Y
    \arrow[ur, bend right, "\psN"'] 
    &&& 
    \Y
    \arrow[ur, bend right=10, "\psN"'] 
    &
\end{tikzcd}
\end{displaymath}
Moreover, we have 2-categories of strict (op)indexed categories and (op)indexed 1-cells that consist of strict natural transformations $\tau$ \cref{eq:indexed1cell}, i.e.\
$\ICat(\X)=[\X\op,\Cat]$ and $\OpICat_\mathrm{sp}(\X)=[\X,\Cat]$ the usual functor 2-categories.

Notice that these categories also form fibrations over $\Cat$, this time essentially using the family fibration also seen in \cref{sec:familyfib}. The functor $\ICat \to \Cat$ sends an indexed category to its domain and an indexed 1-cell to its first component. It is a split fibration, with fibres $\ICat(\X)$ and cartesian liftings pre-composition with functors. In fact, it is also a 2-fibration as explained in \cite[2.3.2]{2Fibs}.

\section{The Grothendieck Construction}

In the first volume of the \emph{S\'eminaire de G\'eom\'etrie Alg\'ebrique du Bois Marie}  \cite{SGAI}, Grothendieck introduced a construction for a fibration $P_\M \maps \inta \M \to \X$ from a given indexed category $\M \maps \X\op \to \Cat$ as follows. If $\delta$ and $\gamma$ are the structure pseudonatural transformations of the pseudofunctor $\M$,
the total category $\inta \M$ has
\begin{itemize}
    \item objects $(x,a)$ with $x \in \X$ and $a \in \M x$;
    \item morphisms $(f,k) \maps (x,a) \to (y,b)$ with $f \maps x \to y$ a morphism in $\X$, and $k \maps a \to (\M f)(b)$ a morphism in $\M x$;
    \item composition $(g, \ell) \circ (f, k) \maps (x,a) \to (y,b) \to (z,c)$ is given by $g \circ f \maps a \to b \to c$ in $\X$ and 
    \begin{equation}\label{eq:comp_intM}
     a\xrightarrow{k}(\M f)(b)\xrightarrow{(\M g)(\ell)}(\M g\circ\M f)(c)\xrightarrow{(\delta_{f,g})_c}\M(g\circ f)(c)\quad\textrm{in }\M x;
    \end{equation}
    \item unit $1_{(x,a)} \maps (x,a) \to (x,a)$ is given by $1_x \maps x \to x$ in $\X$ and \[a=1_{\M x}a\xrightarrow{(\gamma_x)_a}(\M 1_x)(a)\quad\textrm{in }\M x.\]
\end{itemize}

The fibration $P_\M \maps \inta \M \to \X$ is given by $(x,a) \mapsto x$ on objects and $(f,k) \mapsto f$ on morphisms, and the cartesian lifting of 
any $(y,b)$ in $\inta \M$ along $f \maps x \to y$ in $\X$ is precisely $(f,1_{(\M f)b})$. Its fibres are precisely $\M x$ and the reindexing functors between them are $\M f$.

In the other direction, given a (cloven) fibration $P \maps \A \to \X$, we can define an indexed category $\M_P \maps \X\op \to \Cat$ that sends each object $x$ of $\X$ to its fibre category $\A_x$, and each morphism $f \maps x \to y$ to the corresponding reindexing functor $f^* \maps \A_y \to \A_x$ as in \cref{reindexing}. The isomorphisms of cartesian liftings $f^* \circ g^* \cong (g \circ f)^*$ and $1_{\A_x} \cong 1_x^*$  render this assignment pseudofunctorial.

Details of the above, as well as the correspondence between 1-cells and 2-cells, can be found in the provided references. Briefly, given a pseudonatural transformation $\tau \maps \M \to \psN\circ F\op$ \cref{eq:indexed1cell} with components $\tau_x \maps \M x \to \psN Fx$, define a functor $P_\tau \maps \inta\M \to \inta\psN$ mapping $(x\in\X,a\in\M x)$ to the pair $(Fx\in\Y,\tau_x(a)\in\psN Fx)$ and accordingly for arrows. This makes the square
\begin{equation}
\label{eq:inducedfibred1cell}
\begin{tikzcd}
    \inta\M
    \arrow[r,"P_\tau"]
    \arrow[d,"P_\M"'] 
    & 
    \inta\psN
    \arrow[d,"P_\psN"] 
    \\
    \X\arrow[r,"F"'] 
    & 
    \Y
\end{tikzcd}
\end{equation}
commute, and moreover $P_\tau$ preserves cartesian liftings due to pseudonaturality of $\tau$. Moreover, given an indexed 2-cell $(\alpha,m) \maps (F,\tau)\Rightarrow(G,\sigma)$ as in \cref{eq:indexed2cell}, we can form a fibred 2-cell
\begin{equation}\label{eq:inducedfibred2cell}
\begin{tikzcd}[column sep=.8in,row sep=.6in]
    \inta\M
    \arrow[r,bend left,"P_\tau"]
    \arrow[r,bend right,"P_\sigma"']
    \arrow[r,phantom,"\Downarrow{\scriptstyle P_m}"description]
    \arrow[d,"P_\M"'] 
    & 
    \inta\psN
    \arrow[d,"P_\psN"] 
    \\
    \X
    \arrow[r,bend left,"F"]
    \arrow[r,bend right,"G"']
    \arrow[r,phantom,"\Downarrow{\scriptstyle\alpha}"description] 
    & 
    \Y
\end{tikzcd}
\end{equation}
where $\alpha \maps F\Rightarrow G$ is piece of the given structure, whereas $P_m$ is given by components 
\[
    (P_m)_{(x,a)} \maps P_\tau(x,a)=(Fx,\tau_xa) \to P_\sigma(x,a)=(Gx,\sigma_xa)\quad\textrm{in } \inta \psN
\]
explicitly formed by $\alpha_x \maps Fx \to Gx$ in $\Y$ and $(m_x)_a \maps \tau_xa \to (\psN\alpha_x)\sigma_xa$ in $\psN Fx$.

The following theorem summarizes these standard results.
\begin{thm}\label{thm:Grothendieck}
    \leavevmode
    \begin{enumerate}
        \item Every fibration $P \maps \A \to \X$ gives rise to a pseudofunctor $\M_P \maps \X\op { \to }\Cat$.
        \item Every indexed category $\M \maps \X\op \to \Cat$ gives rise to  a fibration $P_\M \maps \inta\M \to \X$.
        \item The above correspondences yield an equivalence of 2-categories 
        \begin{displaymath}
            \ICat(\X) \simeq \Fib(\X)
        \end{displaymath}
        so that $\M_{P_\M} \simeq \M$ and $P_{\M_P} \simeq P$.
        \item The above 2-equivalence extends to one between 2-categories of arbitrary-base fibrations and arbitrary-domain indexed categories
        \begin{equation}\label{eq:Gr_equiv}
        \ICat\simeq\Fib    
        \end{equation}
    \end{enumerate}
\end{thm}

If we combine the above with the fact that the 2-categories $\Fib$ and $\ICat$ are fibred over $\Cat$ with fibres $\Fib(\X)$ and $\ICat(\X)$ respectively, we obtain the following $\Cat$-fibred equivalence
\begin{equation}\label{FibICat}
\begin{tikzcd}
    \ICat
    \arrow[rr, "\simeq"]
    \arrow[dr]
    &&
    \Fib
    \arrow[dl]
    \\&
    \Cat
\end{tikzcd}
\end{equation}
There is an analogous story for opindexed categories and opfibrations that results into a 2-equivalences $\OpICat(\X)\simeq\OpFib(\X)$ and $\OpICat\simeq\OpFib$, as well as for the split versions of (op)indexed and (op)fibred categories.

\section{Examples}

\subsection{Fundamental Fibration}

Let $2$ denote the category with two objects, and one non-identity morphism $\star \to \bullet$. For a category $\X$, the functor category $\X^2$ then consists of the arrows of $\X$ as objects, and commuting squares between them as the morphisms.

For any category $\X$, the \emph{codomain} or \emph{fundamental} opfibration is the usual functor from its arrow category
\[\cod \maps \X^2 \longrightarrow \X\] mapping every morphism to its codomain and every commutative square to its right-hand side leg. It uniquely corresponds to the strict opindexed category, i.e.\ functor
\begin{equation}
\begin{tikzcd}[row sep=.05in]
    \X\arrow[r] & \Cat \\
    x\arrow[r,mapsto]\arrow[dd,"f"'] &     \X/x\arrow[dd,"f_!"] \\
     \\
    y\arrow[mapsto, r] & \X/y
\end{tikzcd}
\end{equation}
that maps an object to the slice category over it and a morphism to the post-composition functor $f_!=f\circ-$ induced by it.

\subsection{Graphs}
\label{expl:grphsoverset}

Consider (directed, multi) graphs, i.e.\ presheaves on the category $G = \begin{tikzcd} V \arrow[r, shift right, swap, "t"] \arrow[r, shift left, "s"] & E\end{tikzcd}$. For a presheaf $g \maps G\op \to \Set$, the set $g_V$ is the set of vertices of the graph, the set $g_E$ is the set of edges of the graph, and the maps $g_s, g_t \maps g_E \to g_V$ assign to an edge its starting and terminating vertex respectively. Let $\Grph$ denoted the category of graphs, $\Set^{G\op}$. Sometimes it is helpful to think of a graph as a single map of the form $(g_s, g_T) \maps g_E \to g_V \times g_V$. When convenient, we will abuse notation by simply referring to this map as $g \maps g_E \to g_V^2$.

Consider the inclusion of a terminal category $1$ into $G$ which selects the object $V$. This induces a functor $\mathsf V \maps \Grph \to \Set$ by precomposing, which sends a graph $g$ to its vertex set $g_V$. As we show below, this functor is in fact a bifibration. The idea here is that if you have a function $f \maps x \to y$, you can pull a graph on $y$ back to a graph on $x$, and you can also push a graph on $x$ forward to a graph on $y$.

\begin{prop}
\label{prop:pullbackcartesian}
    A morphism $\phi\maps g \to h$ in $\Grph$ is $\mathsf V$-cartesian if and only if the square
    \[
    \begin{tikzcd}
        g_E
        \arrow[d, swap, "{(g_s, g_t)}"]
        \arrow[r, "\phi_E"]
        &
        h_E
        \arrow[d, "{(h_s, h_t)}"]
        \\
        g_V^2
        \arrow[r, swap, "\phi_V^2"]
        &
        h_V^2
    \end{tikzcd}
    \]
    is a pullback in $\Set$.
\end{prop}
\begin{proof}
    A simple manipulation shows that the universal property of $\phi$ forming a pullback square is the same as the universal property for it to be $\mathsf V$-cartesian.
\end{proof}

\begin{prop}
\label{prop:Vertexfibration}
    The functor $\mathsf V \maps \Grph \to \Set$ is a fibration.
\end{prop}
\begin{proof}
    Let $f \maps x \to y$ be a function, and $g \in \Grph$ with $g_V=y$. Then we can take a pullback of the following diagram.
    \[x^2 \xrightarrow{f^2} y^2 = h_V^2 \xleftarrow{(h_s,h_t)} h_E\]
    By \cref{prop:pullbackcartesian}, this map is a cartesian lift of $f$.
\end{proof}

By the Grothendieck correspondence, there is a indexed category $\Set\op \to \Cat$.This pseudofunctor assigns to a set $X$ the category $\Grph_X$ of graphs which have vertex set $X$, and graph morphisms which fix the vertices. Given a function $f \maps X \to Y$, this pseudofunctor gives a functor $f^* \maps \Grph_Y \to \Grph_X$ which sends a graph $g$ over $Y$ to the pullback, as in the proof of \cref{prop:Vertexfibration}. Since there is also an opindexed category with the same fibres, we denote this by $\Grph^*$, referring to the action on morphisms. 

To show that $\mathsf V$ is also an opfibration, it is actually easier to construct an explicit splitting. We can derive a characterization of the cocartesian maps from there.

\begin{prop}
    The functor $\mathsf V \maps \Grph \to \Set$ is an opfibration.
\end{prop}
\begin{proof}
    Let $g \in \Grph$, $y \in \Set$, and $f \maps g_V \to y$ a function. Then we can obtain a graph with vertex set $y$ by taking the following composite.
    \[g_E \xrightarrow{g} g_V^2 = g_V^2 \xrightarrow{f^2} y^2\]
    We claim the induced map of graphs displayed below is in fact a cocartesian lift of $f$. 
    \[
    \begin{tikzcd}
        g_E
        \arrow[r, "1"]
        \arrow[d, swap, "g"]
        &
        g_E
        \arrow[d, "f^2 \circ g"]
        \\
        g_V^2
        \arrow[r, swap, "f^2"]
        &
        y^2
    \end{tikzcd}
    \]
    Let $h$ be a graph, $\phi \maps g \to h$ a map of graphs, and $\phi \maps y \to h_V$.
    \[
    \begin{tikzcd}[column sep = huge]
        &&
        h_E
        \arrow[d, "h"]
        \\&&
        h_V^2
        \arrow[ddll, bend right = 15, leftarrow, swap, "\phi_V^2"]
        \\
        g_E
        \arrow[uurr, bend left = 15, "\phi_E"]
        \arrow[r, -, white, line width = 5]
        \arrow[r, "1", pos = 0.4]
        \arrow[d, swap, "g"]
        &
        g_E
        \arrow[d, swap, "f^2 \circ g"]
        \arrow[uur, -, white, line width = 5]
        \arrow[uur, dashed, ""]
        \\
        g_V^2
        \arrow[r, swap, "f^2"]
        &
        y^2
        \arrow[uur, swap, "1"]
    \end{tikzcd}
    \]
    The only map which may take the place of the dashed arrow is $\phi_E$. 
\end{proof}

\begin{cor}
\label{prop:Vertexopfibration}
    A morphism $\phi\maps g \to h$ in $\Grph$ is $\mathsf V$-cocartesian if and only if it is bijective on edges.
\end{cor}

By the Grothendieck correspondence, there is a corresponding opindexed category $\Grph_* \maps \Set \to \Cat$, again referring to the action on morphisms. This must have the same fibres $\Grph_X$ as the indexed category $\Grph^*$ above. Given a function $f \maps X \to Y$, this pseudofunctor gives a functor $f_* \maps \Grph_X \to \Grph_Y$ which sends a graph $g$ over $X$ to the composite, as in the proof of \cref{prop:Vertexopfibration}.

\begin{cor}
    The functor $\mathsf V \maps \Grph \to \Set$ is a bifibration.
\end{cor}

\subsection{Ring Modules}

For a ring $R$, denote by $\Mod_R$ the category of $R$-modules and their homomorphisms. Given a ring homomorphism $f \maps R \to S$, and an $S$-module $N$, we can give the underlying abelian group of $N$ the structure of an $R$-module, denoted $f^\ast N$, by the formula
\[r.x := f(r).x\]
where $r \in R$ and $x \in N$.
This pullback construction is functorial:
\[f^\ast \maps \Mod_S \to \Mod_R\]
and preserves ring homomorphism composition.
\[(f\circ g)^\ast = g^\ast f^\ast \]
Indeed, the above defines a functor $\Mod_- \maps \Ring\op \to \Cat$, a (strict) indexed category. Note, one could choose to be persnickety about size here, but we do not. We can then apply the Grothendieck construction, resulting in a category $\Mod := \int \Mod_-$ where 
\begin{itemize}
    \item an object is a pair $(R \in \Ring, M \in \Mod_R)$
    \item a morphism is a pair $(f, \phi)$ where $f \maps R \to S$ is a ring homomorphism, and $\phi \maps M \to f^\ast N$ is an $S$-module homomorphism.
\end{itemize}
The category $\Mod$ admits a fibration $\Mod \to \Ring$ which forgets the module in a ring-module pair.

\setcounter{chapter}{3}
\chapter{Species and Operads}
\label{app:speciesandoperads}

\section{Combinatorial Species}
\label{sec:species}

Combinatorial species were introduced by Joyal \cite{especies}. A standard reference for the combinatorial perspective is \cite{BLL}. In the previous section, we noted that the category $\Set$ can be given both a cartesian and cocartesian monoidal structure. Moreover, these satisfy a distributive law reminiscent of rings.
\[A \times (B + C) \cong A \times B + A \times C\]
Consider the subcategory $\FinBij$ consisting of finite sets and bijections. This subcategory is closed under both finite sums and products, and thus inherits both monoidal structures. However, $\FinBij$ lacks the maps that would be the projections and inclusions necessary for these structures to be cartesian or cocartesian themselves. By abuse of notation, we will continue to denote them by $+$ and $\times$. 

\begin{defn}
    A \define{combinatorial species} or simply \define{species} is a functor $F \maps \FinBij \to \Set$. The category of species is the functor category $\Set^\FinBij$.
\end{defn}

There are several operations which have been defined on species. These operations make up the building blocks of a calculus for counting families of combinatorial gadgets. 

\begin{defn}
    Being a presheaf category, $\Set^\FinBij$ has colimits given pointwise, thus giving it a cocartesian structure. We refer to this operation simply as \define{addition}. On objects, this operation is given by $(F+G)(U) = F(U) + G(U)$.
\end{defn}

\begin{defn}
    If we apply Day convolution as in \cref{Dayconvolution} to the $+$ monoidal structure on $\FinBij$, we get the operation which we refer to as \define{multiplication} of species.
    On objects, this operation is given by the following formula.
    \[(F \cdot G)(U) = \sum_{V \subseteq U} F(V) \times G(U \setminus V)\]
\end{defn}

\begin{defn}
    Being a presheaf category, $\Set^\FinBij$ has products which are given pointwise, thus giving a cartesian monoidal structure. This is called the \define{Hadamard product}. On objects, it is given by $(F \times G) (U) = F(U) \times G(U)$. This tends to be less useful than multiplication, but it certainly has its purposes.
\end{defn}

\begin{defn}
    We define the \define{Dirichlet product} on species to be the Day convolution (as in \cref{Dayconvolution}) of the $\times$ monoidal structure on $\FinBij$.
\end{defn}

To define differentiation of species, we will need to make use of the \define{shift operator}, denoted $+1$, on $\FinBij$, which is defined as the composite
\[
\begin{tikzcd}[column sep = huge]
    \FinBij
    \arrow[r, dashed, "+1"]
    \arrow[d, swap, "\sim"]
    &
    \FinBij
    \\
    \FinBij \times 1
    \arrow[r, swap, "1_{\FinBij} \times \Delta1"]
    &
    \FinBij \times \FinBij
    \arrow[u, swap, "+"]
\end{tikzcd}
\]
where the map $\Delta1 \maps 1 \to \FinBij$ is the monoidal unit with respect to $\times$.

\begin{defn}
    The \define{differentiation operator} on species is given by $D = +1^\ast \maps \Set^\FinBij \to \Set^\FinBij$. In other words, for a given species $F$, the \define{derivative} of $F$ is given by $F' = F \circ +1$, or $F'(U) = F(U+1)$ on objects. The motivation for calling this operation differentiation is that it actually corresponds to taking the formal derivative of its generation series.
\end{defn}

\begin{defn}
\label{def:substitution}
    The \define{composition product} or \define{substitution product} is given by the following formula. 
    \[(F \circ G)(U) = \sum_{\pi \text{ partition of } U} \left( F(\pi) \times \prod_{p \in \pi} G(p) \right)\]
    The species which acts as unit for this product is the singleton indicator functor, i.e.\ $I_\circ (U)$ is a singleton is $U$ is, and is empty otherwise. This monoidal structure is not symmetric.
\end{defn}

\section{Operads}
\label{app:operads}

An operad is a generalization of category which incorporates the notion of an arrow having multiple inputs. A category is an \emph{arrow-like} compositional system, consisting of a collection of directed arrows, a collection of labels for the endpoints called objects, and a rule for turning a path of such arrows into a single arrow which is associative and unital. An operad is also a compositional system, but now \emph{tree-like}. An operad consists of a collection of directed ``short'' trees
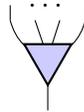
\begin{equation*}
\begin{tikzpicture}
\begin{pgfonlayer}{nodelayer}
	\node [style=none] (3) at (0, 0.7) {};
	\node [style=none] () at (0.5, 0.7) {\dots};
    \node [style=none] (1) at (1, 0.7) {};
	\node [style=none] (2) at (0, 0.5) {};
	\node [style=none] (6) at (0.33, 0.5) {};
	\node [style=none] (7) at (0.66, 0.5) {};
	\node [style=none] (4) at (1, 0.5) {};
	\node [style=downtri] (mu) at (0.5, 0) {};
	\node [style=none] (5) at (0.5, -0.7) {};
\end{pgfonlayer}
\begin{pgfonlayer}{edgelayer}
	\draw (1.center) to (4.center);
	\draw (3.center) to (2.center);
	\draw [bend right] (2.center) to (mu.center);
	\draw [bend left] (4.center) to (mu.center);
	\draw (6.center) to (mu.center);
	\draw (7.center) to (mu.center);
	\draw (mu.center) to (5.center);
\end{pgfonlayer}
\end{tikzpicture}
\end{equation*}
a collection of labels for the endpoints called objects, and a rule for turning a big tree of short trees
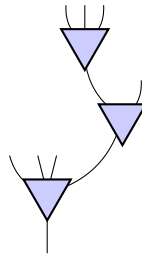
\begin{equation*}
\begin{tikzpicture}
\begin{pgfonlayer}{nodelayer}
	\node [style=downtri] (a) at (0.5, 2.5) {};
	\node [style=downtri] (b) at (1, 1.5) {};
	\node [style=downtri] (c) at (0, 0.5) {};
	\node [style=none] (1) at (0.25, 3) {};
	\node [style=none] (2) at (0.5, 3) {};
	\node [style=none] (3) at (0.75, 3) {};
	\node [style=none] (5) at (1.25, 2) {};
	\node [style=none] (6) at (-0.5, 1) {};
	\node [style=none] (7) at (-0.125, 1) {};
	\node [style=none] (8) at (0.125, 1) {};
	\node [style=none] (9) at (0, -0.3) {};
\end{pgfonlayer}
\begin{pgfonlayer}{edgelayer}
	\draw [bend right] (1.center) to (a.center);
	\draw (2.center) to (a.center);
	\draw [bend left] (3.center) to (a.center);
	\draw [bend right] (a.center) to (b.center);
	\draw [bend left] (5.center) to (b.center);
	\draw [bend left] (b.center) to (c.center);
	\draw [bend right] (6.center) to (c.center);
	\draw (7.center) to (c.center);
	\draw (8.center) to (c.center);
	\draw (c.center) to (9.center);
\end{pgfonlayer}
\end{tikzpicture}
\end{equation*}
into a single short tree
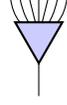
\begin{equation*}
\begin{tikzpicture}
\begin{pgfonlayer}{nodelayer}
	\node [style=downtri] (a) at (0, 0.5) {};
	\node [style=none] (1) at (-0.375, 1) {};
	\node [style=none] (2) at (-0.25, 1) {};
	\node [style=none] (3) at (-0.125, 1) {};
	\node [style=none] (4) at (0, 1) {};
	\node [style=none] (5) at (0.125, 1) {};
	\node [style=none] (6) at (0.25, 1) {};
	\node [style=none] (7) at (0.375, 1) {};
	\node [style=none] (8) at (0, -0.3) {};
\end{pgfonlayer}
\begin{pgfonlayer}{edgelayer}
	\draw [bend right] (1.center) to (a.center);
	\draw [bend right = 15] (2.center) to (a.center);
	\draw (3.center) to (a.center);
	\draw (4.center) to (a.center);
	\draw (5.center) to (a.center);
	\draw [bend left = 15] (6.center) to (a.center);
	\draw [bend left] (7.center) to (a.center);
	\draw (8.center) to (a.center);
\end{pgfonlayer}
\end{tikzpicture}
\end{equation*}
which is associative and unital. There are several good references for the theory of operads \cite{MarklShniderStasheff, Yau, SetOperads, HDAIII, KellyonMay, HigherOperads}. Here, we follow the treatment given by Yau in \cite{HomotopQFT}. 

\subsection{Definition of Operad}

Let $C$ be a non-empty set, whose elements we call \define{colors}. Recall that we denote the free symmetric monoidal category on $C$ by $\S(C)$. Below, we define operads to be monoids in the presheaf category $\Set^{\S(C) \times C}$ with respect to a certain monoidal structure. First we must define this monoidal structure, which is somewhat involved. 

For this section, we denote objects of $\S(C)$ by either $\underline c$ or $(c_1, \dots, c_n)$ depending on context. We denote the monoidal structure on $\S(C)$ by $+$. For an object $X \in \Set^{\S(C)\op \times C}$, we denote the set assigned to $(\underline c, d) \in \S(C)\op \times C$ by $X(\underline c; d)$.

\begin{defn}
    Let $X, Y \in \Set^{\S(C) \times C}$. For each $\underline c \in \S(C)$, define $Y^{\underline c}$ by the following coend formula.
    \[Y^{\underline c} (\underline b) = \int^{\{\underline a_j\} \in \prod_{j=1}^m \S(C)\op} \S(C)\op (\underline a_1 + \dots + \underline a_m, \underline b) \times \left[\prod_{j=1}^m Y(\underline a_j; c_j) \right]\]
    The \define{$C$-colored circle product} of $X$ and $Y$ is given by
    \[X \circ Y(\underline b; d) = \int^{\underline c \in \S(C)} X(\underline c; d) \otimes Y^{\underline c}(\underline b)\]
    Define the unit object $I$ as follows.
    \[I(\underline c; d) = \begin{cases} 
        1 & \text{if } \underline c = d\\
        \emptyset & \text{otherwise}
    \end{cases}\]
\end{defn}

This reduces to the composition monoidal product of species in the case where $C \cong 1$. 

\begin{prop}
    $(\Set^{\S(C) \times C}, \circ, I)$ is a monoidal category.
\end{prop}

\begin{defn}
\label{def:operad}
    Let $C$ be a set. Define the category of operads by \[\Opd_C = \Mon(\Set^{\S(C) \times C}, \circ, I).\] We refer to an object of $\Opd_C$ as a \define{$C$-colored operad}, and a morphism as a \define{color-fixing $C$-operad functor}. Let $f \maps C \to D$ be a function. Let $f(\underline c)$ denote $(f(c_1), \dots, f(c_n))$ for $\underline c \in \S(C)$.
    For a $D$-operad $P$, we can pullback along $f$ to get a $C$-operad given by $f^*P(\underline c; d) = P(f(\underline c); f(d))$. For a color-fixing $D$-operad functor $\phi \maps P \to Q$, we get a color-fixing $C$-operad functor $f^*\phi \maps f^*P \to f^*Q$ which sends an operation $\theta \in f^*P(\underline c; d) = P(f(\underline c); d)$ to $\phi\theta \in Q(f(\underline c); d)$. These assignments give a functor $f^* \maps \Opd_D \to \Opd_D$, and we get an indexed category $\Opd_- \maps \Set\op \to \Cat$. Define $\Opd = \int \Opd_-$. We refer to an object of $\Opd$ as an \define{operad}, and to a morphism as an \define{operad functor}.
\end{defn}

\subsection{Operads from symmetric monoidal categories}

There is a standard method of constructing an single-colored operad from an object $x$ in a strict symmetric monoidal category $\C$. Namely, we define the set of $n$-ary operations to be $\hom_{\C}(x^{\otimes n}, x)$, and compose these operations using composition in $\C$. This gives the so-called \define{endomorphism operad} of $x$. Here we give the generalization of this idea to the multi-color case, using \emph{all} the objects of $\C$ as the objects of the operad. In what follows we let $\Ob(\C)$ be the set of objects of a small category $\C$.

\begin{prop}
\label{prop:operad_from_symmoncat}
    If $\C$ is a small strict symmetric monoidal category then there is an $\Ob(\C)$-colored operad $\Op(\C)$ for which:
    \begin{itemize}
        \item the set of operations $\Op(\C)(c_1, \dots, c_k; c)$ is defined to be $\hom_\C(c_1 \otimes \dots \otimes c_k, c)$,
        \item given operations   
        \[f\in \hom_\C(c_1 \otimes \cdots \otimes c_k; c) \]
        and 
        \[g_i \in \hom_\C(c_{ij_1} \otimes \cdots \otimes c_{ij_i},c_i) 
        \]
        for $1 \le i \le k$, their composite is defined by
        \begin{equation}
        \label{eq:composition_of_operations}
        f \circ (g_1, \dots, g_k) = f \circ (g_1 \otimes \cdots \otimes g_k) .
        \end{equation}
        \item identity operations are identity morphisms in $\C$, and
        \item the action of $S_k$ on $k$-ary operations is defined using the braiding in $\C$.
    \end{itemize}
\end{prop}

\begin{proof}
    The various axioms of a colored operad can be checked for $\Op(\C)$ using the corresponding laws in the definition of a strict symmetric monoidal category.  The associativity axiom for $\Op(\C)$ follows from associativity of composition and the functoriality of the tensor product in $\C$.  The left and right unit axioms for $\Op(\C)$ follow from the unit laws for composition and the functoriality of the tensor product in $\C$. The two equivariance axioms for $\Op(\C)$ follow from the laws governing the braiding in $\C$. 
\end{proof}

\begin{prop}
\label{prop:functoriality_of_operads_from_ssmcs}
    The assignment $\Op \maps \Sym\Mon\Cat_\spl \to \Opd$ defined on objects as in \cref{prop:operad_from_symmoncat} and sending any strict symmetric monoidal functor $F \maps \C \to \C'$ to the operad morphism $\Op(F) \maps \Op(\C) \to \Op(\C')$ that acts by $F$ on types and also on operations:
    \[
        \Op(F) = F \maps \hom_C(c_1 \otimes \cdots \otimes c_n, c) \to \hom_{\C'}(F(c_1) \otimes \cdots \otimes F(c_n), F(c))
    \] is a functor.
\end{prop}

\begin{proof} 
    This is a straightforward verification.
\end{proof}

\subsection{Operad Algebras}

As a sort of monoid, operads exist to act. The elements of $\O(\underline c;d)$ for some operad $\O$ are meant to be thought of as ``abstract operations'' with $\underline c$ as the input types, and $d$ as the output type. When $\O$ acts on something, it is meant to be thought of as realizing these abstract operations as real operations on some family of sets indexed by the elements of $C$. 

\begin{defn}
    Let $C$ be a set. A \define{$C$-colored set} is a functor $C \to \Set$, where $C$ is thought of as a discrete category. This is of course the same as a function $C \to \ob\Set$. For $\underline c = (c_1, \dots, c_n) \in \S(C)$, let $X_{\underline c}$ denote the set $\prod_{j=1}^n X_{c_j}$. A \define{map of $C$-colored sets} $f \maps X \to Y$ is a natural transformation $f \maps X \To Y$. This is the same as a family of functions $\{f_c \maps X_c \to Y_c\}_{c \in C}$ with no further conditions. For $\underline c = (c_1, \dots, c_n) \in \S(C)$, let $f_{\underline c} \maps X_{\underline c} \to Y_{\underline c}$ denote the function $\prod_{j=1}^n f_{c_j} \maps \prod_{j=1}^n X_{c_j} \to \prod_{j=1}^n Y_{c_j}$.
\end{defn}

\begin{defn}
    Let $\O$ be a $C$-colored operad, with operad composition denoted by $\gamma$, and unit operation denoted by $I$. An \define{$\O$-algebra} consists of
    \begin{itemize}
        \item a $C$-colored set $X \maps C \to \Set$, also denoted $\{X_c\}_{c \in C}$
        \item for $\underline c \in \S(C)$ and $d \in C$, a map $\theta \maps \O(\underline c;d) \times X_{\underline c} \to X_d$
    \end{itemize}
    which makes the following diagrams commute for $c,d, c_j \in C$, $\underline c, \underline b_j \in \S(C)$, $\underline b = \sum^n \underline b_j$, and $\sigma \in S_n$.
    \begin{itemize}
        \item associativity:
        \[
        \begin{tikzcd}
            \O(\underline c, d) \times \prod_{j=1}^n \O(\underline b_j, c_j) \times X_{\underline b}
            \arrow[d, "\cong", "\text{permute}"']
            \arrow[r, "\gamma \times 1"]
            &
            \O(\underline b; d) \times X_{\underline b}
            \arrow[dd, "\theta"]
            \\
            \O(\underline c; d) \times \prod_{j=1}^n [\O(\underline b_j, c_j) \times X_{\underline b_j}]
            \arrow[d, swap, "1 \times \prod_j \theta"]
            \\
            \O(\underline c; d) \times X_{\underline c}
            \arrow[r, swap, "\theta"]
            &
            X_d
        \end{tikzcd}\]
        \item unity:
        \[
        \begin{tikzcd}
            &
            1 \times X_c
            \arrow[dl, swap, "I \times 1"]
            \arrow[dr, "\cong"]
            \\
            \O(c;c) \times X_c 
            \arrow[rr, swap, "\theta"]
            &&
            X_c
        \end{tikzcd}\]
        \item equivariance:
        \[
        \begin{tikzcd}
            \O(\underline c; d) \times X_{\underline c}
            \arrow[dr, swap, "\theta"]
            \arrow[rr, "\sigma \times \sigma\inv"]
            &&
            \O(\underline c \sigma; d) \times X_{\underline c \sigma}
            \arrow[dl, "\theta"]
            \\&
            X_d
        \end{tikzcd}\]
    \end{itemize}
\end{defn}

\begin{defn}
    A \define{map of $\O$-algebras} $(X, \theta) \to (Y, \xi)$ consists of a map of $C$-colored sets $\alpha \maps X \to Y$ such that the following diagram commutes.
    \[
    \begin{tikzcd}
        \O(\underline c; d) \times X_{\underline c}
        \arrow[r, "1 \times f_{\underline c}"]
        \arrow[d, swap, "\theta"]
        &
        \O(\underline c; d) \times Y_{\underline c}
        \arrow[d, "\xi"]
        \\
        X_d
        \arrow[r, swap, "f"]
        &
        Y_d
    \end{tikzcd}\]
    Let $\Alg(\O)$ denote the category of $\O$-algebras and maps of $\O$-algebras.
\end{defn}
}
\bibliographystyle{alpha}
\bibliography{references.bib}

\end{document}